\documentclass[11pt,notitlepage]{amsart} 
\usepackage{amsmath}
\usepackage{amsthm}
\usepackage{amsfonts}
\usepackage{amssymb,amscd,epsf,verbatim,mathtools,framed}
\usepackage{mathrsfs}
\usepackage{graphicx}
\usepackage{latexsym}
\usepackage{lscape}
\usepackage[colorlinks=true]{hyperref}
\hypersetup{colorlinks, citecolor=blue, filecolor=black, linkcolor=purple, urlcolor=violet}
\usepackage{tikz}
\usetikzlibrary{calc}
\usetikzlibrary{matrix,arrows,decorations.pathmorphing}
\usepackage{tikz-cd}
\usepackage{color}
\usepackage{multirow}  
\usepackage{scalefnt}
\usepackage{fancyhdr}
\usepackage[margin=1.2in]{geometry}
\usepackage[T1]{fontenc}

\usepackage{relsize} 
\usepackage[bbgreekl]{mathbbol} 
\DeclareSymbolFontAlphabet{\mathbb}{AMSb} 
\DeclareSymbolFontAlphabet{\mathbbl}{bbold} 
\newcommand{\Prism}{{\mathlarger{\mathbbl{\Delta}}}}

\newcommand{\Z}{\mathbb{Z}}
\newcommand{\F}{\mathbb{F}}
\newcommand{\N}{\mathbb{N}}

\newcommand{\Q}{\mathbb{Q}}

\renewcommand{\L}{\mathbb{L}}

\newcommand{\mC}{\mathcal{C}}
\newcommand{\mD}{\mathcal{D}}

\newcommand{\mF}{\mathcal{F}}
\newcommand{\mG}{\mathcal{G}}

\newcommand{\mI}{\mathcal{I}}

\newcommand{\mM}{\mathcal{M}}

\newcommand{\mO}{\mathcal{O}}

\newcommand{\fm}{\mathfrak{m}} 

\newcommand{\tu}{\textup}
\newcommand{\cl}{\overline}
\newcommand{\ul}{\underline}

\newcommand{\ra}{\rightarrow}

\newcommand{\sq}{\widetilde}
\newcommand{\minus}{\backslash}

\DeclareMathOperator{\im}{im}
\DeclareMathOperator{\Hom}{Hom}

\DeclareMathOperator{\midwedge}{\mathlarger{\mathlarger{\mathlarger{\wedge}}}} 

\DeclareMathOperator{\Ainf}{\text{A}_{\textup{inf}}}
\DeclareMathOperator{\Acrys}{\text{A}_{\textup{crys}}}
\DeclareMathOperator{\crys}{crys}
\DeclareMathOperator{\QSyn}{QSyn}
\DeclareMathOperator{\qSyn}{qSyn}
\DeclareMathOperator{\QRSP}{QRSPerfd} 
\DeclareMathOperator{\grade}{gr}
\DeclareMathOperator{\Fil}{Fil}
\DeclareMathOperator{\opp}{\textup{op}}
\DeclareMathOperator{\lra}{\: \longrightarrow \:} 
\DeclareMathOperator{\isom}{\;\xrightarrow{\: {}_{\sim} \:} \;}
\DeclareMathOperator{\perf}{perf}
\DeclareMathOperator{\Perf}{Perf}

\newcommand{\gr}[1]{\langle {#1} \rangle} 
\DeclareMathOperator{\spec}{Spec}
\DeclareMathOperator{\spf}{Spf}
\DeclareMathOperator{\spa}{Spa}
\DeclareMathOperator{\spd}{Spd}
 
\DeclareMathOperator{\init}{init}
\DeclareMathOperator{\ett}{\textup{\'et}}
\DeclareMathOperator{\ket}{\textup{k\'et}}
\DeclareMathOperator{\proet}{\textup{pro\'et}}
\DeclareMathOperator{\qproet}{\textup{qpro\'et}}
\DeclareMathOperator{\qpket}{\textup{qpk\'et}}

\newcommand{\bi}{\begin{itemize}}
\newcommand{\ei}{\end{itemize}}
\newcommand{\bt}{\begin{theorem}}
\newcommand{\et}{\end{theorem}}
\newcommand{\bbt}{\begin{theorem*}}
\newcommand{\eet}{\end{theorem*}}
\newcommand{\bp}{\begin{proposition}}
\newcommand{\ep}{\end{proposition}}
\newcommand{\bl}{\begin{lemma}}
\newcommand{\el}{\end{lemma}}
\newcommand{\bbl}{\begin{lemma*}}
\newcommand{\eel}{\end{lemma*}}
\newcommand{\bc}{\begin{corollary}}
\newcommand{\ec}{\end{corollary}}
\newcommand{\beg}{\begin{example}}
\newcommand{\eeg}{\end{example}}
\newcommand{\br}{\begin{remark}}
\newcommand{\er}{\end{remark}}
\newcommand{\bbr}{\begin{remark*}}
\newcommand{\eer}{\end{remark*}}
\newcommand{\bd}{\begin{definition}}
\newcommand{\ed}{\end{definition}}
\newcommand{\be}{\begin{enumerate}}
\newcommand{\ee}{\end{enumerate}}
\newcommand{\bex}{\begin{exercise}}
\newcommand{\eex}{\end{exercise}}
\newcommand{\bproof}{\begin{proof}}
\newcommand{\eproof}{\end{proof}}

\theoremstyle{theorem}
\newtheorem{theorem}{Theorem}[section]
\newtheorem{mainthm}{Theorem} 
\newtheorem{maindef}[mainthm]{Definition}

\newtheorem*{theorem*}{Theorem}

\theoremstyle{definition}
\newtheorem{mainrmk}[mainthm]{Remark}
\newtheorem{example}[theorem]{Example}

\newtheorem{definition}[theorem]{Definition}
\newtheorem{proposition}[theorem]{Proposition}

\newtheorem{construction}[theorem]{Construction}

\newtheorem{lemma}[theorem]{Lemma}
\newtheorem{corollary}[theorem]{Corollary}

\newtheorem{notation}[theorem]{Notation}
\newtheorem{remark}[theorem]{Remark}

\newtheorem{assumption}[theorem]{Assumption}
\newtheorem{convention}[theorem]{Convention}

 \usepackage{etoolbox}
\patchcmd{\section}{\scshape}{\bfseries}{}{}
\makeatletter
\renewcommand{\@secnumfont}{\bfseries}
\makeatother

\title{Logarithmic prismatic cohomology II}
\author{Teruhisa Koshikawa}
\email{teruhisa@kurims.kyoto-u.ac.jp}
\address{Research Institute for Mathematical Sciences, Kyoto University}
\author{Zijian Yao}
\email{zijian.yao.math@gmail.com}
\address{Department of Mathematics, University of Chicago}

\numberwithin{equation}{section}
\begin{document}

\begin{abstract}
We continue to study the logarithmic prismatic cohomology defined by the first author, and complete the proof of the de Rham comparison and \'etale comparison generalizing those of Bhatt and Scholze. We prove these comparisons for a derived version of logarithmic prismatic cohomology, and, along the way, we construct a suitable Nygaard filtration and explain a relation between $F$-crystals and $\Z_p$-local systems in the logarithmic setting. 
\end{abstract}

\maketitle
 \thispagestyle{empty}
 
\vspace{-0.2in}

\tableofcontents

\setlength{\parskip}{0.3em}



\section{Introduction} \label{section:intro}

This article is a continuation of \cite{Koshikawa} and further develops the theory of logarithmic prismatic cohomology, generalizing the work of Bhatt and Scholze in \cite{BS}.  In \cite{Koshikawa}, the first author introduced the notion of log prisms and log prismatic site, and showed that the log prismatic cohomology enjoys crystalline and Hodge--Tate comparisons. 

Let us briefly recall the setup. Roughly, a $\delta_{\log}$-\textit{ring} is pre-log ring $(A, \alpha\colon M_A \ra A)$ equipped with additional structures: a map $\delta\colon A \ra A$ such that $(A, \delta)$ becomes a $\delta$-ring and a map $\delta_{\log}\colon M \ra A$ which satisfies the following conditions 
\bi
\item $\delta_{\log} (0) = 0,\quad $ $\delta_{\log}(m+m') = \delta_{\log} (m) + \delta_{\log}(m') + p \delta_{\log} (m) \delta_{\log}(m') $ 
\item $\alpha(m)^p \cdot \delta_{\log}(m) = \delta (\alpha(m))$
\ei
(see \cite[Definition 2.2]{Koshikawa}). A \textit{pre-log prism} is a triple $(A, I, M_A)$ where $(A, M_A)$ is a $\delta_{\log}$-ring and $(A, I)$ is a prism in the sense of \cite{BS}. A \textit{log prism} is a bounded\footnote{This means that $(A, I)$ is a bounded prism in the sense of \cite[Defintiion 1.6]{BS}, in other words, $A/I$ has bounded $p^{\infty}$-torsion.} pre-log prism $(A, I, M_A)$ considered up to taking the associated log structure. To a pre-log prism $(A, I, M_A)$, one naturally associates a log prism $(A, I, M_A)^a=(A, I, M_{\spf A})$ (\cite[Corollary 2.15]{Koshikawa}). 

For the rest of the introduction, we fix a bounded pre-log prism $(A, I, M_A)$ with $M_A$ being an integral monoid. Let $(X,M_X)$ be an integral log $p$-adic formal scheme that is smooth\footnote{Here smoothness is defined in \cite[Definition A.11]{Koshikawa}, which is slightly different from the usual notion of log smoothness defined in \cite[Section 3.3]{Kato}, in the sense that it relaxes the usual finiteness condition but requires the maps on monoids to be integral (see Conventions in the end of the introduction for more details).} over $(A/I, M_A)$. To ease notation, in the introduction, we denote the log formal scheme by $X = (\mathring{X}, M_X)$, where $\mathring{X}$ denotes its underlying formal scheme. 

\begin{maindef}[The logarithmic prismatic site {\cite[Definition 4.1]{Koshikawa}}]
 Let  $(X/(A,M_A))_{\Prism}$ be the opposite of the category of integral log prisms $(B, IB, M_{\spf B})$ equipped with 
\be
\item a map $g\colon (A, I, M_A) \ra (B, IB, \Gamma (\spf B, M_{\spf B}))$ of pre-log prisms, 
\item a map $f\colon \spf (B/I) \ra \mathring{X}$ of formal schemes, and
\item a strict\footnote{Recall that strict here means that the pullback of $M_{\spf B}$ is isomorphic to $f^*M_X$ as log structures on $\spf B/I$. Here this condition is equivalent to saying that the closed immersion is exact. Recall that a map of integral monoids $h\colon P \ra Q$ is exact if $P = (h^{\textup{gp}})^{-1} (Q)$ in $P^{\textup{gp}}$. A map between integral log (formal) schemes is exact if at every point the induced map of monoids is exact stalk-wise. A closed immersion of integral quasi-coherent log schemes is exact precisely when it is strict.} 
closed immersion $(\spf (B/I), f^* M_X) \ra (\spf B, M_{\spf B})$ of log formal schemes (here the target is a log $(p, I)$-adic formal scheme).
\ee 
A morphism $(B, IB, M_{\spf B}) \ra (C, IC, M_{\spf C})$ in $(X/(A,M_A))_{\Prism}$ is an \'etale cover if $B \ra C$ is $(p, I)$-completely \'etale and faithfully flat, and the map on the associated log $(p, I)$-adic formal schemes is strict. 
The logarithmic prismatic site of $X$ over $(A, I, M_A)$ is $(X/(A,M_A))_{\Prism}$ equipped with the \'etale topology. It has a structure sheaf $\mO_{\Prism}$ (resp. reduced structure sheaf $\cl \mO_{\Prism}$) defined by sending 
$(B, IB, M_{\spf B}) \mapsto B$ (resp. $(B, IB, M_{\spf B}) \mapsto B/IB$). The logarithmic prismatic cohomology of $X$ is defined as  
\[
R \Gamma_{\Prism}(X/(A, M_A)) \coloneqq R \Gamma ((X/(A, M_A))_{\Prism}, \mO_{\Prism}).
\]
This is an $E_\infty$-$A$-algebra equipped with a $\phi_A$-semilinear endomorphism $\phi$, where $\phi_A$ denotes the Frobenius on $A$ given by $\phi_A(x)= x^p + p \delta(x)$. 
\end{maindef}  

One of the major goals of this article is to prove the analogue of \cite[Theorem 1.8 and  Theorem 1.16]{BS}. Let us start with the following 

\begin{mainthm} \label{mainthm:comparisons}
Let $X = (\mathring{X},M_X)$ be an integral log $p$-adic formal scheme that is smooth over $(A/I, M_A)$. Assume moreover that the underlying formal scheme of $X$ is quasicompact and quasiseparated (qcqs). 

\be 
\item Hodge--Tate comparison \textup{(}\cite{Koshikawa}\textup{)}: There exists a canonical isomorphism 
\[
\eta_X\colon \Omega^\bullet_{X/(A/I, M_A)} \isom H^\bullet (\cl \Prism_{X/(A, M_A)})\{\bullet\}
\]
of sheaves of differential graded $A/I$-algebras. Here 
$\cl \Prism_{X/(A, M_A)} $ denotes the \'etale sheaf $R \nu_* \cl \mO_{\Prism}$ where 
\begin{equation}\label{eq:map_of_topoi_from_prismatic}
\nu\colon \textup{Shv} ((X/(A, M))_{\Prism}) \ra \textup{Shv} (\mathring{X}_{\ett})
\end{equation} is the morphism of topoi constructed in \cite[Remark 4.4]{Koshikawa}.  
The right hand side of the isomorphism is equipped with the Bockstein differential, where $\{i\}$ denotes the Breuil--Kisin twist $\otimes I^i/I^{i+1}$.  
\item Base change \textup{(}\cite{Koshikawa}\textup{)}: Let $(A, I, M_A) \ra (A', I A', M_{A'})$ be a map of bounded pre-log prisms with integral monoids, and let \[
X' \coloneqq X \times_{(\spf A/I, M_A)^a} (\spf A'/I, M_{A'})^a
\] be the (integral) log $p$-adic formal scheme obtained by base change. Then we have a natural base change isomorphism 
\[ 
R \Gamma_{\Prism} (X/(A, M_A)) \widehat \otimes_{A}^\L A' \isom R \Gamma_{\Prism} (X'/(A', M_{A'})).
\]
\item Crystalline comparison \textup{(}\cite{Koshikawa}\textup{)}: Suppose $I = (p)$ and further assume that $X$ is of Cartier type over $(A/I, M_A)$. Then there is a canonical $\phi$-equivariant isomorphism 
\[
R \Gamma_{\textup{logcrys}} (X/(A, M_A)) \cong R \Gamma_{\Prism} (X/(A, M_A)) \widehat \otimes^\L_{A, \phi_A} A. 
\]
of $E_\infty$-A-algebras. 
\item 
de Rham comparison: Assume that the mod $p$ fiber of $X$ is of Cartier type over $(A/(p, I), M_A)$. There is a canonical isomorphism
\[
R \Gamma_{\textup{logdR}} (X/(A/I, M_A)) \cong R \Gamma_{\Prism} (X/(A, M_A)) \widehat \otimes^\L_{A, \phi_A} A/I
\]
of $E_\infty$-A-algebras. 
\item \'Etale comparison: Suppose that the associated log prism $(A, I, M_A)^a$ is a perfect\footnote{This means the Frobenius lift is an isomorphism of log formal scheme.} log prism. 
Further assume that $X \ra \spf (A/I, M_A)^a$ is the base change of an fs log $p$-adic formal scheme 
\[X_0 \lra \spf (A/I, M_0)^a\] 
that is smooth over $(A/I, M_0)$ whose mod $p$ fiber is of Cartier type,  where $M_0$ is an fs monoid. Let $X_{\eta}$ be the generic fiber of $X$  over $\spa (\Q_p, \Z_p)$.\footnote{Strictly speaking, the generic fiber $X_{\eta}$ may not exist as a log adic space since the underlying pre-adic space might not be sheafy. Instead, the generic fiber $X_\eta$ is viewed as a log diamond over $\spd \Q_p$. See Corollary \ref{etale_comp_diaomnd_global} for a precise statement of the \'etale comparison. Also see Theorem \ref{thm:etale_comp} for a cleaner statement in the affine case (with no restrictions on singularities)}
For every $n \ge 1$, there is a canonical isomorphism
\[
R \Gamma_{\textup{k\'et}} (X_{\eta}, \Z/p^n) \cong \Big( R \Gamma_{\Prism} (X/(A, M_A)) [\frac{1}{I}] /p^n \Big)^{\phi = 1}. 
\]
of $E_\infty$-$\Z/p^n\Z$-algebras. 
\item Assume that the mod $p$ fiber of $X$ is of Cartier type over $(A/(p,I), M_A)$, then Frobenius is an isogeny: the linearized Frobenius map 
\[
\phi\colon R \Gamma_{\Prism} (X/ (A, M_A)) \widehat \otimes^\L_{A, \phi_A} A \lra R \Gamma_{\Prism} (X/ (A, M_A))
\]
becomes an isomorphism upon inverting $I$. 
More precisely, suppose that $I = (d)$ is principal, then for each $i \ge 0$ there is a map 
\[V_i\colon H^i_{\Prism} (X/ (A, M_A)) \ra H^i (R \Gamma_{\Prism} (X/ (A, M_A)) \widehat \otimes^\L_{A, \phi_A} A)
\] such that $\phi \circ V_i = V_i \circ \phi = d^i$.  
Moreover, if $\mathring{X}$ is affine, then the map $\phi$ factors as 
\[
R \Gamma_{\Prism} (X/ (A, M_A)) \widehat \otimes^\L_{A, \phi_A} A \isom L \eta_I R \Gamma_{\Prism} (X/ (A, M_A)) \lra R \Gamma_{\Prism} (X/ (A, M_A))
\]
where $L\eta_I$ denotes the d\'ecalage operator (see Subsection \ref{ss:derived_Nygaard_and_dR}). 
\ee
\end{mainthm}

As indicated above, Part (1), (2), (3) of the comparisons  have already been established in \cite{Koshikawa}, see Theorem 5.3, Theorem 5.5 and Theorem 6.3 of \textit{loc.cit.} respectively.

The primary goal of this article, therefore, is to study the \'etale comparison of log prismatic cohomology and to develop the theory of derived log prismatic cohomology, which we shall use to study the Nygaard filtration and to prove the de Rham comparison theorem, as well as the isogeny statement on Frobenius. In the rest of the introduction, we explain some of these ideas in slightly more detail.

The first part of the article is devoted to establish a ``flat descent'' result for log cotangent complexes, which generalizes a theorem of Bhatt in \cite{Bhatt_descent} (also see \cite{BMS2}) and is of independent interest. In this direction, we say that a map $(R, P) \ra (S, M)$ between pre-log rings is \textit{homologically log flat}\footnote{This terminology is suggested to us by Gabber.} if $R \ra S$ is flat and for any monoid map $P \ra Q$, the homotopy pushout $M \oplus^\L_{P} Q$ agrees with the naive pushout. An important class of examples is when $P \hookrightarrow M$ is an injective integral map of integral monoids. A homologically log flat map $(R, P) \ra (S, M)$ is called a homologically log flat covering if the underlying map $R \ra S$ on rings is in addition faithfully flat. 
We upgrade the category $\text{Alg}^{\textup{prelog}}$ of pre-log rings to a site with covering maps given by homologically log flat coverings and refer to this topology as the hlf (= homologically log flat) topology. In this article we show that the log cotangent complex (recalled in Section \ref{sec:cotangent_complex}) introduced by Gabber satisfies homologically log faithfully flat descent. More precisely, 

\begin{mainthm} \label{mainthm:flat_descent}
Fix a base pre-log ring $(R, P)$. The functor $(S, M) \mapsto \L_{(S, M)/(R, P)}$ is an hlf sheaf on $\textup{Alg}^{\textup{prelog}}_{(R, P)/}$ taking values in $\mD(R)$.
\end{mainthm}

Theorem \ref{mainthm:flat_descent} plays a key role in our study of a derived version of log prismatic cohomology, which is ``derived'' (or ``animated'') from the functor that sends the $p$-completed ``log free'' algebra $\Sigma_{S, T} \coloneqq (A/I [\N^{S}, \N^{T}]^{\wedge}, M_A \oplus \N^T)$ to its log prismatic cohomology 
\[
 \Prism_{\Sigma_{S, T}/(A, M_A)} \coloneqq R \Gamma_{\Prism} (\spf (\Sigma_{S, T}, \N^T)^a/(A, M_A)) \in \mD (A).
\]
From Theorem \ref{mainthm:comparisons} Part (1) one obtains a derived version of Hodge--Tate comparison (see Proposition \ref{prop:derived_HT}), which essentially says that to control the derived log prismatic cohomology of $(R, P)$ it often suffices to control the log cotangent complex $\L_{(R, P)/(A/I, M_A)}$.  It is in this context that we shall apply Theorem \ref{mainthm:flat_descent}. 
Moreover, we introduce a logarithmic version of the quasisyntomic site studied by \cite{BMS2}. 

\begin{maindef}[The logarithmic quasisyntomic site] \noindent 

\noindent 
\be
\item  A pre-log ring $(A, M_A)$ is quasisyntomic if $A$ is $p$-complete and has bounded $p^\infty$-torsion and $\L_{(A, M_A)/\Z_p} \otimes^\L_{A} A/p \in \mD(A/p)$ has Tor amplitude in  $[-1, 0]$. 
\item A map $(A, M_A) \ra (B, M_B)$ of quasisyntomic  pre-log rings is a quasisyntomic  map (resp. cover) if $B \otimes_A^\L A/p \cong B/p$, the map $(A/p, M_A) \ra (B/p, M_B)$ is homologically log flat (resp. log faithfully flat) and $\L_{(B/p, M_B)/(A/p, M_A)}  \in \mD(B/p)$ has Tor amplitude in  $[-1, 0]$.
\ee  Let $\QSyn^{\textup{prelog}}$ denote the category of quasisyntomic pre-log rings, (the opposite category of) which forms a site with topology given by quasisyntomic covers (see Subsection \ref{subsec:log_qSyn}).  
\end{maindef}

By Theorem \ref{mainthm:flat_descent}, we know that log cotangent complexes (and thus derived log prismatic cohomology) satisfies quasisyntomic descent. Similar to the (nonlog) quasisyntomic site, an important feature of $\QSyn^{\textup{prelog}}$ is that it is locally log quasiregular semiperfectoid. In other words, $\QSyn^{\textup{prelog}}$ has a basis given by pre-log rings $(S, M)$ that are semiperfect which admits a map $\sq S \ra S$ from a perfectoid ring $\sq S$ such that $\L_{(S, M)/\sq S} \otimes^\L_S S/p \in \mD(S/p)$ has Tor amplitude concentrated in degree $[-1]$ (see Subsection \ref{subsec:log_qRSP}, in particular Definition \ref{definition:qrspd_log} and Lemma \ref{lemma:log_BMS_4.25}). Such pre-log rings are particularly nice since by the last condition (and the derived Hodge--Tate comparison) its derived log prismatic cohomology relative to $\Ainf (\sq S)$ is a discrete ring (concentrated in degree $0$).

Using the derived theory, we are able to equip log prismatic cohomology with the Nygaard filtration. To state the result, let $(R, P)$ a (simplicial) pre-log ring over $(A/I, M_A)$ and let $\Prism^\L_{(R, P)/(A, M_A)}$ denote its derived log prismatic cohomology (when the context is clear we often drop the supscript and write $\Prism_{(R, P)/(A, M_A)}$ for $\Prism^\L_{(R, P)/(A, M_A)}$). 

\begin{mainthm} \label{mainthm:Nygaard}
There is a functorial decreasing multiplicative filtration $\Fil_N$ by $(p, I)$-complete objects on $\Prism_{(R, P)/(A, M_A)}^{(1)}\coloneqq \Prism_{(R, P)/(A, M_A)} \widehat \otimes^\L_{A, \phi_A} A$, equipped with an isomorphism
\[
\grade^i_N \Prism_{(R, P)/(A, M_A)}^{(1)}\cong \Fil_i \overline{\Prism}_{(R, P)/(A, M_A)}\{i\}
\]
on the $i^{th}$ graded piece. 
\end{mainthm}

The presence of log structures causes additional complications for the construction of the log Nygaard filtration, compared to the construction given in \cite{BS} in the nonlog case. The main issue is the following: suppose that  $(S, M)$ is a quasiregular semiperfectoid pre-log ring and let $(A, I, M_A)$ be a perfect prism with $M_A$ being perfectoid (see Definition \ref{def:perfectoid_monoid}), equipped with a map $(A/I, M_A) \ra (S, M)$, so 
\begin{equation} \label{eq:Frob_twist_of_derived_prismatic}
\Prism_{(S, M)/(A, M_A)}^{(1)} \coloneqq  \Prism_{(S, M)/(A, M_A)} \otimes^\L_{A, \phi_A} A \cong \Prism_{(S, M)/(A, M_A)}
\end{equation}
lives in degree $0$. However, in contrast with \cite[Proposition 12.11]{BS}, the derived Nygaard filtration in Theorem \ref{mainthm:Nygaard} may fail to be discrete (in some natural way) even in this case. This is essentially caused by the exactification procedure in the process of forming log prismatic envelops, which introduces non-perfect base prisms into the picture. To be slightly more precise, suppose that the map $(A/I, M_A) \ra (S, M)$ is surjective on both rings and monoids.  Let 
\[\sq M \coloneqq (h^{\textup{gp}})^{-1} (M) \subset (\sq M)^{\textup{gp}}
\]be the exactification of the surjection $h\colon M_A  \ra M$. Set $\sq A \coloneqq A \widehat \otimes_{\Z_p} \Z_p[\sq M]$, which we naturally promote to an integral pre-log prism $(\sq A, I, \sq M)$. We remark that the prism $\sq A$ is usually not perfect. Also note that we naturally get a map $(\sq A/I, \sq M) \ra (S, M)$ which is surjective both on rings and monoids. From Subsection \ref{ss:universal_object_bdd}, we have $\Prism_{(S, M)/(A, M_A)} \cong \Prism_{S/\sq A}$ where the right hand side is the usual nonlog derived prismatic cohomology.  It turns out that the condition 
\[
\phi(-) \in I^{i} \Prism_{(S, M)/(A, M_A)}\] does not pick out the correct log Nygaard filtration on it. 
We refer the reader to the example in Subsection \ref{ss:toy_example} for more detail. The key observation is that, in order to understand the  Nygaard filtrations in the logarithmic world, it suffices to understand the (nonlog) Nygaard filtrations on 
\[\Prism_{S/\sq A} \otimes_{\sq A, \phi_{\sq A}} \sq A
\]
for sufficiently many such maps $(A/I, M_A) \ra (S, M)$ described above (note that this is different from $\Prism_{(S, M)/(A, M_A)}^{(1)}$ in (\ref{eq:Frob_twist_of_derived_prismatic}) since $\sq A$ is generally not perfect).  On $\Prism_{S/\sq A} \otimes_{\sq A, \phi_{\sq A}} \sq A$, the $i^{th}$ Nygaard filtration \textit{is} in fact defined by the condition $\phi (-) \in I^i \Prism_{S/\sq A}$. Theorem \ref{mainthm:Nygaard} will then be proven using quasisyntomic descent and a certain ``Frobenius descent'' lemma. We refer the reader to Section \ref{sec:Nygaard} for more detail.

The existence of the Nygaard filtration will eventually allow us to establish Part (4) and (6) of Theorem \ref{mainthm:comparisons} on de Rham comparison and the image of the Frobenius. 

Next let us briefly explain the proof of the \'etale comparison in the affine case. This time, the additional complication due to the presence of log structures is caused by the saturation procedure. Recall that an integral monoid $M$ is \textit{saturated} if for any $a \in M^{\text{gp}}$, $a \in M$ whenever $na \in M$ for some integer $n \ge 1$. For an integral monoid $M$, its saturation is given by $M^{\text{sat}}\coloneqq\{x \in M^{\textup{gp}} \: | \:  n x \in M \text{ for some } n \ge 1 \},$ which describes the left adjoint of the forgetful functor from saturated monoids to integral monoids. Now let $(A, I, M_A)$ be a pre-log prism with $(A, I)$ being  perfect and $M_A$ being perfectoid, and let $(R, P)$ be a pre-log ring over $(A/I, M_A)$. Suppose that $(R, P) \ra (R_\infty, P_\infty)$ is a sufficiently nice covering using which we can write the right hand side of the \'etale comparison as a limit of 
\begin{equation} \label{eq:rhs_etale}
(\Prism_{(R_\infty^\bullet, P_\infty^\bullet)/(A, M_A)}[\frac{1}{I}]/p^n)^{\phi = 1},
\end{equation}
where $(R_\infty^\bullet, P_\infty^\bullet)$ denotes the naive \v{C}ech nerve of the cover. However, when dealing with log (or Kummer log) \'etale cohomology of log schemes, one needs to work with the category of saturated log schemes. In other words, we need to consider the saturated \v{C}ech nerve $(R_\infty^{\bullet, \text{sat}}, P_\infty^{\bullet, \text{sat}})$. The key observation is that the  Frobenius fixed points of the  log prismatic cohomology in (\ref{eq:rhs_etale}) does not change when we replace $(R_\infty^i, P_\infty^i)$ by its ($p$-)saturation (this uses the fact from \cite{BS} that the expression (\ref{eq:rhs_etale}) does not change under taking the colimit perfection of the derived prismatic cohomology). To relate the Kummer \'etale cohomology and the Frobenius fixed points of the  log prismatic cohomology for $(R_\infty^{i, \text{sat}}[\frac{1}{p}], P_\infty^{i, \text{sat}})$, we further reduce to the nonlog situation and apply the nonlog version of the \'etale comparison to conclude. See Section \ref{sec:etale} for more detail and in particular, see Subsection \ref{ss:example_etale_free} for an illustrative example of the logarithmic affine line. To globalize the \'etale comparison, we extend the theory of diamonds to the logarithmic world and introduce the notion of quasi-pro-Kummer-\'etale topology on (saturated) log diamonds. For a $p$-complete fs pre-log ring $(R, P)$ with bounded $p^{\infty}$-torsion, the Kummer \'etale cohomology of the log scheme associated to $(R[\frac{1}{p}], P)$ with torsion coefficients agrees with the quasi-pro-Kummer-\'etale cohomology of the associated log diamond  (see Theorem \ref{comparing_diamond_to_spec}), which also agrees with the Kummer \'etale cohomology of the associated log adic space if the latter exists (see Lemma \ref{lemma:comparing_adic_generic_to_spec} for the precise statement). The quasi-pro-Kummer-\'etale cohomology of log diamonds then allows us to formulate the global version of the \'etale comparison as stated in Theorem \ref{mainthm:comparisons}. 

The theory of log diamonds also provides a convenient framework to study $F$-crystals on the absolute (saturated) log prismatic site. In particular, we prove the following generalization of a result of Bhatt and Scholze in \cite{BS_crystal}, which can be viewed as a first step towards understanding a logarithmic version of crystalline local systems. 

\begin{mainthm} 
Let $X$ be an fs log bounded $p$-adic formal scheme and let $X_\eta$ denote its log diamond generic fiber.  
There is a natural equivalence
\[
\textup{Vect}(X_{\Prism}, \mO_{\Prism}[1/\mI]_p^{\wedge})^{\phi=1} \simeq 
\textup{Loc}_{\Z_p}(X_{\eta})  
\]
from the category  $\textup{Vect}(X_{\Prism}, \mO_{\Prism}[1/\mI]_p^{\wedge})^{\phi=1}$  of Laurent $F$-crystals over $X$ (see Definition \ref{def:Laurent_F_crystal}) to the category of (quasi-pro-Kummer-\'etale) $\Z_p$-local systems on $X_\eta$. 
\end{mainthm}

Next consider an algebraically closed nonarchimedean extension $C$ of $\Q_p$. The perfectoid ring $\mO_C$ corresponds to a perfect prism $(A, I)$ where $A = \Ainf (\mO_C) = W(\mO_C^\flat)$ and $I = (\xi)$. Here $\xi = \mu/\phi_A^{-1} (\mu)$ and $\mu = [\epsilon] - 1$ are elements in $A$, where $\epsilon = (1, \zeta_p, \zeta_{p^2}, ...) \in \mO_C^\flat$ for a fixed choice of $p$-power roots of unity. Now suppose that $(A, I, M_A)^a$ is a perfect log prism and let $X$ be a log $p$-adic formal scheme as in the statement of the \'etale comparison (cf.  Theorem \ref{mainthm:comparisons} Part (5)). Then the log prismatic cohomology is an $\Ainf$-module equipped with a Frobenius operator. In this setting, we prove the following ``\'etale comparison'' and a structural result of the log prismatic cohomology. 

\begin{mainthm} \label{mainthm:BKF}
Suppose that $X$ is a smooth log $p$-adic formal scheme over $(A/I, M_A)$ satisfying the assumptions in Theorem \ref{mainthm:comparisons} Part (5). Assume that $X$ is proper over $A/I = \mO_C$ and let  $X_{\eta}$ be its log (diamond) generic fiber. Then for each $i \ge 0$, we have a natural isomorphism\footnote{Again, the Kummer \'etale cohomology of $X_{\eta}$ should be interpreted as the quasi-pro-Kummer-\'etale cohomology of the associated log diamond (cf. Theorem \ref{etale_comp_diaomnd_global}). }
\[
H^i_{\Prism} (X/(A, M_A)) \otimes_{A} A [ {1}/{\phi^{-1}(\mu)}] \cong H^i_{\ket} (X_{\eta}, \Z_p) \otimes_{\Z_p} A [{1}/{\phi^{-1} (\mu)}]. 
\]
Moreover, the (Frobenius twisted) log prismatic cohomology group 
\[
\phi_A^* H^i_{\Prism} (X/(A, M_A)) \cong H^i_{\Prism} (X/(A, M_A)) \otimes_{A, \phi_A} A. 
\]
is a Breuil--Kisin--Fargues module (cf. Definition \ref{definiton:BKF_module}). 
\end{mainthm}

In \cite{DY}, a similar result is obtained by comparing the Frobenius twisted log prismatic cohomology to a logarithmic version of $\Ainf$-cohomology introduced in \cite{BMS1}, making use of a primitive comparison theorem for Kummer \'etale cohomology of log adic spaces. Our proof of Theorem \ref{mainthm:BKF} is rather different (even in the non-log case, compared to \cite{BS}). Let $M = R \Gamma_{\Prism} (X/(A, M_A))$ denote the log prismatic cohomology of $X$ and $T = R \Gamma_{\ket} (X_{\eta}, \Z_p)$ denote the Kummer \'etale cohomology of $X_{\eta}$. To prove Theorem \ref{mainthm:BKF}, we pass to the $p$-completed perfection $(M_{\perf})_p^\wedge$ of $M$ and construct an isomorphism 
\[ 
(M_{\perf})_p^\wedge \otimes_{A}^\L A [{1}/{\phi^{-1} (\mu)}] \isom T \otimes_{\Z_p}^\L A[{1}/{\phi^{-1} (\mu)}].
\]
For this we make use of a mod $p$ Riemann--Hilbert correspondence from \cite{Bhatt_Lurie}. To relate $(M_{\perf})_p^\wedge$ back to $M$, we observe that, the pole of the Frobenius along $\phi^{-1}(\xi)$ is bounded by our study of the Nygaard filtration. Therefore, taking the ($p$-completed) perfection introduces bounded amount of poles along $\phi^{-1} (\mu)$,  thus we have 
\[
M \otimes_{A}^\L A [{1}/{\phi^{-1} (\mu)}] \isom (M_{\perf})_p^\wedge \otimes_{A}^\L A [{1}/{\phi^{-1} (\mu)}].  
\]
We refer the reader to Section \ref{section:BKF_module} for the detail of the proof.

Finally, we remark that, the theory of log prismatic cohomology (in particular, of the existence of the Nygaard filtration) allows extensions of many recent developments in integral $p$-adic Hodge theory to the logarithmic world. To showcase some of these immediate generalizations, let us record some simple applications, which we prove in Section \ref{section:apps}.

Our first example is an extension of the main result in \cite{Yao_generic_fiber} to the logarithmic setting. For the setup, let us retain the notation and assumptions of Theorem \ref{mainthm:BKF}. Let $k = \mO_C/\fm$ be the residue field of $C$.  Recall that $X$ is a smooth and proper log $p$-adic formal scheme over $(\mO_C, M_A)$. Let $X_s$ be the log special fiber of $X$ over $\ul k = (k, M_A)$ and $X_{\eta}$ be the log generic fiber over $\ul C = (C, M_A)$ as before. 

Let us state a rough version of Theorem \ref{theorem:Newton_Hodge}. 
\begin{mainthm} \label{mainthm:Newton_Hodge}
Let $X$ be as above. For each $i$, the Newton polygon of the $F$-isocrystal 
\[H^i_{\textup{logcrys}}(X_s/W(\ul k))[\frac{1}{p}]
\]
lies on or above the Hodge polygon of $X_{\eta},$ defined by the Hodge numbers 
\[h^j \coloneqq \dim_K H^{i-j} (X_{\eta}, \Omega^j_{X_\eta/\ul C}).
\]
\end{mainthm}

Keep the notation from Theorem \ref{mainthm:Newton_Hodge}. One of the original motivations of \cite{BMS1} to introduce prismatic cohomology over $\Ainf$ is to study integral structures of $p$-adic cohomologies, for example, relations between various torsion submodules. In this direction, let us first record the following result on $p$-torsions, generalizing results from \cite{BMS1, CK_semistable}.

\begin{mainthm} \label{mainthm:torsion}
Let $X$ be as above. 
\be 
\item We have the following inequality. For each $i$, 
\[
\dim_{\F_p} H^i_{\ket}(X_{\eta}, \F_p) \le \dim_k H^i_{\textup{logdR}} (X_s/\ul k).
\]
More generally, for each $n \ge 1$, we have 
\[
\textup{length}_{\Z_p} (H^i_{\ket}(X_{\eta}, \Z/p^n))  \le \textup{length}_{W(k)} (H^i_{\textup{logcris}}(X_s/W_n(\ul{k}))).  
\]
A similar result holds for log de Rham cohomology over $\mO_C/p^n$ (using the normalized length, cf. Theorem \ref{theorem:torsion}).  
\item  $H^i_{\textup{logcris}}(X_s/W(\ul{k}))$ is $p$-torsion free if and only if $H^i_{\textup{logdR}}(X)$ is $p$-torsion free, in which case $H^i_{\ket}(X_{\eta}, \Z_p)$ is $p$-torsion free.
\ee 
\end{mainthm} 

In fact, in Section \ref{section:apps} we prove a stronger version of Theorem \ref{mainthm:torsion} that bounds the length of torsion submodule $ H^i_{\ket}(X_{\eta}, \Z_p)_{\textup{tor}}$ by that of log crystalline (and log de Rham) cohomology (cf. Theorem \ref{theorem:torsion} (1) and (2)). 

While inequalities in Theorem \ref{mainthm:torsion} can be strict in general as shown in \cite{BMS1}, it can be improved for small $i$ in the Breuil--Kisin setting following \cite{LiuLi1, LiuLi2}. 
Let $K$ be a discretely valued extension of $\Q_p$ with perfect residue field $k$. Fix a uniformizer $\pi \in \mO_K$. Let $\mathfrak S = W(k)[\![u]\!]$ and equip $\mathfrak S$ with a pre-log structure $\N \ra S$ sending $1 \mapsto u$. Then $(\mathfrak S, \N)$ upgrades to a $\delta_{\log}$-ring of rank $1$ with the Frobenius structure sending $u \mapsto u^p$ on $\mathfrak S$. Let $E$ be the Eisenstein polynomial of $\pi$, so $(\mathfrak S, (E), \N)$ becomes a pre-log prism. Let $X$ be a proper log  $p$-adic formal scheme which is smooth over $(\mO_K, \pi^\N)$ and with mod $p$ fiber being of Cartier type. We have the following generalization of \cite[Theorem 3.5]{LiuLi2}, which says that in small degrees, the log prismatic cohomology of $X$ over the Breuil--Kisin prism has a simple structure that resembles the structure of a $\Z_p$-module. 

\begin{mainthm} \label{mainthm:BK_module_structure}
Let $e$ be the ramification index of $K$, let $i$ be a positive integer satisfying 
\[e \cdot (i -1) < p -1, 
\]
then there is a (non-canonical) isomorphism 
\[
H^i_{\Prism}(X/(\mathfrak S, \N)) \cong H^i_{\ket}(X_{\cl \eta}, \Z_p) \otimes_{\Z_p} \mathfrak S,
\]
where $X_{\cl \eta}$ denotes the log (diamond) generic fiber of $X$ over $C_K = \widehat{\cl K}$. 
In particular, the log prismatic cohomology $H^i_{\Prism}(X/(\mathfrak S, \N))$ has the following structure 
\[
H^i_{\Prism}(X/(\mathfrak S, \N)) \cong \mathfrak S^{\oplus r_0} \oplus (\mathfrak S/p)^{\oplus r_1} \oplus \cdots \oplus (\mathfrak S/p^m)^{\oplus r_m}
\]
for some $r_0, ..., r_m \in \N$. Moreover, for each $n \ge 1$, we have a (non-canonical) isomorphism 
\[
H^i(R\Gamma_{\Prism}(X/(\mathfrak S, \N)) \otimes^\L_{\Z} \Z/p^n) \cong H^i_{\ket}(X_{\cl \eta}, \Z/p^n) \otimes_{\Z_p} \mathfrak S. 
\]
In particular, inequalities in Theorem \ref{mainthm:torsion} become equalities if $e\cdot i < p-1$. 
\end{mainthm}

In the non-log case, this was first obtained by Min \cite{Min} under a slightly stronger assumption $e\cdot i < p-1$, and then improved by Li--Liu \cite{LiuLi2}. It is related to more classical works of Fontaine--Messing, Kato, Breuil, and Caruso and we refer the reader to \cite{Min, LiuLi2} for related references and some history. 
Our proof of Theorem \ref{mainthm:BK_module_structure} follows the strategy of \cite{LiuLi2}. In particular, we establish the some analogous results on the structure of $u^\infty$-torsion inside $H^i (X/(\mathfrak S, \N))$, for example see Theorem \ref{theorem:boundary_u_torsion}. We refer the reader to Subsection \ref{ss:boundary_u_torsion} for more details. 
  
\begin{mainrmk}[Stacky approaches] 
Let us comment on our understanding of stacky approaches to log prismatic cohomology. We are aware of two approaches. 

The first approach, due to Bhatt--Mathew, uses (variants of) infinite root stacks of Talpo--Vistoli (at least) when the base log structure, i.e., $M_A$, is trivial: the principle is that the cohomology of a certain stack (without log structure) over the underlying scheme of a given log scheme gives the cohomology of the log scheme if the log structure is sufficiently nice. Therefore, one could use results of \cite{BS} once they are generalized to formal stacks. In particular, this would be enough to define and study the derived log prismatic cohomology as it only needs free pre-log rings, and it seems to fit well with \'etale comparison as \'etale comparison can be stated without base log structure. However, the formalism of derived log prismatic cohomology itself would not change. Let us also emphasize that log prisms and $\delta_{\log}$-rings are good classes of base objects or coefficients; a prism with a log structure may not have a $\delta_{\log}$-structure, and may not admit a Frobenius lift as a log ring.

The second approach is a combination of the stacky approach of Bhatt--Lurie and Drinfeld to prismatic cohomology, Olsson's stack classifying log structures, and an observation of Pridham from \cite{Pridham} (this idea is due to the first author). More precisely, Olsson's stack will allow us to generalize the (relative) prismatization \cite{prismatization} in the log setting: we will need a correct formalism of log derived geometry, which is much more subtle than one might expect in the first author's opinion; nevertheless, Pridham \cite{Pridham} pointed out that Olsson's stack can be used to define log structures of derived schemes and we believe this is the right approach here.\footnote{Pridham gives the ``third'' definition of log cotangent complexes for log schemes. It agrees with Olsson's cotangent complex under some flatness assumption. Contrary to what Pridham indicates, it seems to the first author that it agrees with Gabber's log cotangent complex in general (or under a mild assumption).} This second approach would clarify the proof of Hodge--Tate comparison, for instance, and it is potentially useful given a recent development of \emph{prismatic $F$-gauges} of Bhatt--Lurie. This approach seems to be less directly related to log \'etale cohomology, while two approaches seem to be very closely related. 
\end{mainrmk}

\subsection*{Conventions} \noindent 

\noindent 
We largely follow conventions and notations from \cite{BS} and \cite{Koshikawa}. In particular, all monoids and rings are assumed to be commutative with units. A pre-log ring is a pair $(R, \alpha\colon M \ra R)$ where $R$ is a ring and $\alpha$ is a map of monoids (where $R$ is viewed as a multiplicative monoid). We often abbreviate a pre-log ring as $(R, M)$ or even $\ul R$ to ease notations. A pre-log ring $(R, \alpha\colon M \ra R)$ is called a log ring if it satisfies the additional requirement that the map  $\alpha^{-1} (R^\times) \ra R^\times$ is an isomorphism. The category of pre-log rings will be denoted by $\tu{Alg}^{\tu{prelog}}$. 
For a pre-log ring $(R, M)$, the associated log ring is denoted by $(R, M)^a$, and we use both $\spec (R, M)^a$, $(\spec R, M)^a$ to denote the log scheme associated to $(\spec R, \ul{M} \ra \mO_{\spec R, \ett})$. If $(R, M)$ is classically $p$-complete, we similarly use both $\spf (R, M)^a$, $(\spf (R), M)^a$ to denote the associated log $p$-adic formal scheme. 
In order to simplify notations, we used $X$ to denote a log $p$-adic formal scheme in the introduction, and wrote its underlying formal scheme as $\mathring{X}$. In the rest of the article, we usually use the notation $(X, M_X)$ to denote a log $p$-adic formal scheme, with the underlying formal scheme $X$ and a map $\alpha\colon M_X \ra \mO_{X_{\ett}}$ of \'etale sheaves of monoids (where the map $\alpha$ is often hidden from the notation).  

For the notion of perfectoid rings, we follow the convention of \cite{BMS1}. 
In Section \ref{sec:cotangent_complex}, we define a notion of perfectoid monoid, which is related to the ``tilt'' $M^\flat \coloneqq \varprojlim_{x \mapsto x^p} M$ of a monoid $M$ (cf. Definition \ref{def:perfectoid_monoid}). A stronger notion than being perfectoid is being perfect, which means that the monoid is uniquely $p$-divisible (cf. Remark \ref{remark:perfect_monoid}). In addition, we define a convenient weaker notion for a monoid to be pseudo-perfectoid (cf. Remark \ref{remark:pseudo_perfectoid}). A pre-log ring $(R, M)$ is said to be perfectoid (resp. pseudo-perfectoid) if $R$ is a perfectoid ring and $M$ is a perfectoid (resp. pseudo-perfectoid) monoid. It turns out that for an integral log ring, being perfectoid is equivalent to being pseudo-perfectoid (cf. Remark \ref{remark:perfectoid_log_same_as_pseudo}), thus we never need to use the terminology ``pseudo-perfectoid log ring'' in the paper. 

Let us also point out the following potentially confusing terminology we use in the paper. We use the term ``log prism'' to denote a bounded pre-log prism $(A, I, M_A)$ where the underlying pre-log ring $(A, M_A)$ is a log ring. Here the quotation mark is a part of the notation (to distinguish it from the notion of a log prism), and the quotation mark disappears by taking the associated log structure. With this terminology, we will show that the category of perfect ``log prisms'' is equivalent to the category of perfectoid log rings. 

Suppose $I \subset A$ is an ideal in a ring $A$. Both the notion of being classically $I$-complete and derived $I$-complete appear in the article. We try to make the distinction clear in the text. Note that, if $I = (f)$ for some element $f \in A$, then the notion of being classically $f$-complete and derived $f$-complete is equivalent for an $A$-module $M$ that has bounded $f^\infty$-torsion (see, for example, \cite[Lemma III.2.4]{Bhatt_notes}).  
The notation $(-)^\wedge_{I}$ (or $(-)^\wedge$ if $I$ is understood) will be used to denote derived $I$-completion unless otherwise noted. Moreover, for the log cotangent complex $\L_{-}$ and its variants, we use $\widehat \L_{-}$ to denote the derived $p$-completion. We use the notation $\gr{-}$ to denote the classically completed polynomial (or monoid) algebras, except in Section \ref{sec:log_diamonds}, where $\Q_p \gr{-}$ is used to denote the $\Q_p$-algebra $\Z_p \gr{-}[\frac{1}{p}]$ (in particular, our notation $\Q_p \gr{t_1, ..., t_n}$ agrees with the standard notation for the Tate algebra with $n$-variables over $\Q_p$).

It is worth emphasizing that we adopt the following non-standard convention for the notion of smoothness of log formal schemes (introduced in \cite[Definition A.11]{Koshikawa}). More precisely, let $(B, M_B)$ be an integral pre-log ring where $B$ is classically $p$-complete, then we say that a log $p$-adic formal scheme $(X, M_X)$ is smooth over $(B, M_B)$ if \'etale locally on $X$, there exists a chart $P \ra \Gamma (X, M_X)$ over $M_B$ satisfying the following conditions
\be 
\item  $P$ is integral and finitely generated over $M_B$ (in other words, there exists a surjection $M_B \oplus \N^{\oplus k} \ra P$ for some integer $k \in \N$). Moreover, the morphism of monoids $M_B \ra P$ is integral.
\item the map $M_B^{\tu{gp}} \ra P^{\tu{gp}}$ is injective and the torsion part of its cokernel is a finite group of order coprime to $p$, and the induced map 
\[
X \ra \spf B\gr{P} \times_{\spf B\gr{M_B}} \spf B
\] 
is \'etale. 
\ee 
Note that Condition (1) is an additional requirement compared to Kato's definition of a ``log smooth morphism'' from \cite{Kato}, which, together with (2), guarantees that the morphism $X\to \spf B$ is $p$-completely flat. This class of morphisms is stable under (naive) base changes on $(B, M_B)$. In practice, we are only interested in the case base changed from some fine monoid. 

Finally, derived categories that appear in the paper are derived $\infty$-categories. For a ring (or more generally, an $E_\infty$-algebra) $A$, we use $\mD(A)$ to denote its derived $\infty$-category and $\mathcal{DF}(A)$ to denote the filtered derived $\infty$-category of $A$-modules in the sense of \cite[Section 5]{BMS2} (also see \cite[Appendix D]{APC}). We refer the reader to \cite{APC, BMS2} for the definition and further discussions of the Beilinson $t$-structure $(\mathcal{DF}^{\le 0}(A), \mathcal{DF}^{\ge 0}(A))$ on $\mathcal{DF}(A)$

\subsection*{Acknowledgements}
T.K. was supported by JSPS KAKENHI Grant Number 20K14284. Z.Y. was partially supported by CNRS and ERS Grant Number 851146. It is a pleasure to thank Dori Bejleri, K\k{e}stutis \v{C}esnavi\v{c}ius, Lin Chen, Yuchen Fu, Ofer Gabber, Ildar Gaisin, Luc Illusie, Akhil Mathew, Takeshi Tsuji for helpful discussions in the preparation of this paper.

\newpage

\newpage 

\section{Log cotangent complexes and flat descent} \label{sec:cotangent_complex}

In this section we recall the definition and basic properties of log cotangent complexes for pre-log rings constructed by Gabber, following \cite{Bhatt_dR, Olsson_log, SSV} and prove the flat descent for them. We also introduce perfectoid (pre-)log rings and perfect ``log prisms'' and relate them.

\begin{notation}
 
\be 
\item Let $\textup{Alg}^\textup{prelog}$ be the category of pre-log rings $(M_R\to R)$, which we often denote by $\ul R$ or $(R, M_R)$ if the monoid or the monoid map is understood.  
\item  For a map of pre-log rings
\[(M_R\overset{\beta_R}{\to} R)\to (M_S\overset{\beta_S}{\to} S),\] the $S$-module of \emph{log K\"ahler differentials} is denoted
by $\Omega^1_{\ul S/\ul R}$ and is given by 
\[
\Omega^1_{\underline{S}/\underline{R}} \cong  (\Omega^1_{S/R}\oplus (\tu{cok}(M_R^{\tu{gp}}\to M_S^{\tu{gp}})\otimes_{\mathbb{Z}} S)/( (d\beta_S(m),0)-(0,m\otimes \beta_S(m)) ).
\]
\item  We write $\Omega^{\bullet}_{\underline{S}/\underline{R}}$ for the corresponding \emph{log de Rham complex}, which is equipped with a multiplication and a descending Hodge filtration, with graded pieces given by $\Omega^i_{\underline{S}/\underline{R}}\coloneqq\midwedge^i_{S} \Omega^1_{\underline{S}/\underline{R}}$.
\ee 
\end{notation}

\subsection{Log cotangent complexes} \noindent 

\noindent 
The log differential $\Omega^1_{\underline{S}/\underline{R}}$ and log de Rham complex $\Omega^\bullet_{\underline{S}/\underline{R}}$ both have derived versions, which we now recall. 

\begin{notation} 
We write 
\begin{align*} 
 \textup{Forget}^{\textup{Alg}^\textup{prelog}}_{\textup{Set}\times \textup{Set}}&\colon\textup{Alg}^\textup{prelog} \lra \textup{Set}\times \textup{Set} \; \qquad  \textup{and}  \\
 \textup{Forget}^{\textup{Alg}^\textup{prelog}}_{\textup{Mon}\times \textup{Alg}}&\colon\textup{Alg}^\textup{prelog} \lra  \textup{Mon}\times \textup{Alg}  
\end{align*}
for the forgetful functors sending $(M_R\to R) \longmapsto (M_R,R)$. We use similar notation for simplicial objects.  
\end{notation}

The category $s\textup{Mon}$ (resp. $s\textup{Alg}$, resp. $s\textup{Alg}^\textup{prelog}$) of simplicial commutative monoids (resp. rings, resp. pre-log rings) admits the structure of a combinatorial model category such that $M_R\to M_S$ (resp. $R\to S$, resp.$\ul R\to \ul S$) is a (trivial) fibration if and only if so is $M_R\to M_S$ (resp. $R\to S$, resp. $R\to S$, $M_R\to M_S$) as simplicial sets. This is explained in \cite{Bhatt_dR, SSV}. In particular, the one on $s\textup{Alg}^\textup{prelog}$ is given in \cite[Proposition 5.3]{Bhatt_dR} and \cite[Proposition 3.3]{SSV}. 

\begin{proposition}
Under the model structures above, both forgetful functors $\textup{Forget}^{s\textup{Alg}^\textup{prelog}}_{s\textup{Set}\times s\textup{Set}}$ and $\textup{Forget}^{s\textup{Alg}^\textup{prelog}}_{s\textup{Mon}\times s\textup{Alg}}$ are right Quillen functors. Moreover, the functor $\textup{Forget}^{s\textup{Alg}^\textup{prelog}}_{s\textup{Mon}\times s\textup{Alg}}$ is also a \emph{left} Quillen functor. In particular, it commutes with homotopy colimits.
\end{proposition} 
\begin{proof} 
The first part is clear by construction, and the rest is \cite[Proposition 5.5]{Bhatt_dR}. 
\end{proof} 

\begin{notation} 
Following \cite{Bhatt_dR}, let us denote the left adjoint of the functor $\textup{Forget}^{s\textup{Alg}^\textup{prelog}}_{s\textup{Set}\times s\textup{Set}}$ (resp. $\textup{Forget}^{s\textup{Alg}^\textup{prelog}}_{s\textup{Mon}\times s\textup{Alg}}$) by  $\textup{Free}_{s\textup{Alg}^\textup{prelog}}^{s\textup{Set}\times s\textup{Set}}$ (resp.  $\textup{Free}_{s\textup{Alg}^\textup{prelog}}^{s\textup{Mon}\times s\textup{Alg}}$). 
\end{notation}

Now let $\underline{R}\to \underline{S}$ be a map in $\textup{Alg}^\textup{prelog}$. By a \emph{projective resolution} of $\underline{S}$ as an $\underline{R}$-algebra, we mean a trivial fibration $\underline{P}_\bullet\to \underline{S}$ with $\underline{P}_\bullet$ cofibrant in $s\textup{Alg}^\textup{prelog}_{\underline{R}/}$. By the previous proposition, it remains to be a projective resolution of rings and monoids after applying  $\tu{Forget}^{s\textup{Alg}^\textup{prelog}}_{s\tu{Set}\times s\tu{Set}}$. Such resolutions always exist  (for example, the adjunction 
\[(\tu{Free}_{s\textup{Alg}^\textup{prelog}_{\underline{R}/}}^{s\tu{Set}\times s\tu{Set}}, \:\:\: \tu{Forget}^{s\textup{Alg}^\textup{prelog}_{\underline{R}/}}_{s\tu{Set}\times s\tu{Set}})\] provides a projective resolution,  called the canonical free resolution). Any two projective resolutions are homotopy equivalent.

\bd 
 Let $\underline{R}\to \underline{S}$ be a map in $\textup{Alg}^\textup{prelog}.$ For a projective resolution $\underline{P}_\bullet\to \underline{S}$ as above, consider the simplicial $P_\bullet$-module $\Omega^1_{\underline{P}_\bullet/\underline{R}}$. The \emph{log cotangent complex} $\mathbb{L}_{\underline{S}/\underline{R}}$ is defined as 
 \[ \mathbb{L}_{\underline{S}/\underline{R}} = \Omega^1_{\underline{P}_\bullet/\underline{R}}\otimes_{ P_\bullet } S, \] viewed as an object in the derived $\infty$-category $\mathcal \mD(S)$ of $S$-modules (via the Dold-Kan correspondence). Note that the resulting object does not depend on the choice of the projective resolution.
\ed

This definition agrees with those in \cite[Section 8]{Olsson_log} (applied to pre-log rings), \cite{Bhatt_dR}, \cite{SSV}. 

\begin{construction} \label{construction:derived_log_dR}
Similarly, the \emph{derived log de Rham complex} $\mathbb{L}\Omega_{\underline{S}/\underline{R}}$ is the simple complex (in other words,  homotopy colimit) associated to the simplicial complex $n\mapsto \Omega^\bullet_{ \underline{P}_n/\underline{R} }$.  Equivalently, one may obtain a bicomplex from $\Omega^\bullet_{ \underline{P}_\bullet/\underline{R} }$ via the Dold-Kan correspondence, and take $\mathbb{L}\Omega_{\underline{S}/\underline{R}}$ to be the associated total complex. By construction, $\L \Omega_{\ul S/\ul R}$ is an $E_\infty$-$R$-algebra equipped with a decreasing multiplicative \emph{Hodge filtration} $\Fil^\bullet_{\tu{Hodge}} \L \Omega_{\underline{S}/\underline{R}}$ whose graded piece is given by $\mathbb{L}\Omega^i_{\underline{S}/\underline{R}}\coloneqq\midwedge^i_{S} \mathbb{L}_{\underline{S}/\underline{R}}$.  Note that $\L \Omega_{\ul S/\ul R}$ is in general different from the completion (in other words,  homotopy limit) 
with respect to the Hodge filtration, which we denote by $\L\Omega^{\tu{hc}}_{\underline{S}/\underline{R}}.$
The log de Rham complex $\L \Omega_{\ul S/\ul R}$ is also equipped with an increasing exhaustive \emph{conjugate filtration} $\Fil^{\tu{conj}}_{\bullet} \L \Omega_{\underline{S}/\underline{R}}$ (see \cite[Proposition 6.9]{Bhatt_dR}). When $R$ is an $\F_p$-algebra, the graded piece of the conjugate filtration can be described using the \emph{derived Cartier isomorphism} 
\begin{equation} \label{eq:derived_Cartier}
C_i: \grade_i^{\tu{conj}} \L \Omega_{\ul S/\ul R} \cong \L \Omega^i_{\ul S^{(1)}/\ul R}
\end{equation}
from \cite[Proposition 7.4]{Bhatt_dR}, where $\ul S^{(1)}$ denotes the (derived) Frobenius base change of $\ul S$ over $\ul R$ defined in Notation \ref{notation:Frob_base_change} below. 
We write $\widehat \L \Omega_{\ul S/\ul R}$ (resp. $\widehat \L \Omega^i_{\ul S/\ul R}$) for the derived $p$-adic completion of $ \L \Omega_{\ul S/\ul R}$ (resp. $ \L \Omega_{\ul S/\ul R}^i$), which is used throughout this article.   
\end{construction}

\begin{remark} The above constructions generalize to the case where $\underline{S}$ is replaced by an object $\underline{S}_\bullet$ in $s\textup{Alg}^\textup{prelog}_{\underline{R}/}$. Namely, we form the simplicial complex $n\mapsto \mathbb L \Omega_{\underline{S}_n/\underline{R}}$ and then take (homotopy) colimit. 
\end{remark}

\begin{remark}[Animated pre-log rings]\label{animated pre-log rings}
It is also possible to phrase the definition of cotangent complexes using animation \cite{CS}. Fix a pre-log ring $\ul R=(R, M_R)$. The category $\textup{Alg}_{\ul R/}^\textup{prelog}$ is cocomplete and generated under colimits by compact projective objects
\[
\Sigma_{T_0, T_1} \coloneqq (R [\N^{T_0}, \N^{T_1}], M_R \oplus \N^{T_1}), 
\]
where $T_0, T_1$ are finite sets. The animation of $\textup{Alg}_{\ul R/}^\textup{prelog}$ is the $\infty$-category freely generated under sifted colimits by such compact projective objects. It can be described as the $\infty$-category of functors from the opposite of the category of compact projective objects to anima that preserve finite products. Via \cite[Corollary 5.5.9.3]{Lurie}, this is equivalent to the $\infty$-category obtained from simplicial pre-log rings with the simplicial model structure from \cite[Proposition 5.5.9.1]{Lurie} and \cite[Proposition 1.1.5.10]{Lurie}.\footnote{$\mathbf{A}^{\circ}$ in \cite[Corollary 5.5.9.3]{Lurie} denotes the subcategory of fibrant-cofibrant objects of $\mathbf{A}$ and the simplicial nerve $\textup{N}(\mathbf{A}^{\circ})$ is the underlying $\infty$-category of $\mathbf{A}$. See the beginning of \cite[A.2]{Lurie}.} Indeed, in the notation there, our $\mathcal{C}$ is the category of compact projective objects and $\mathbf{A}$ identifies with the category of simplicial pre-log $\ul R$-algebras by evaluating on $(R[\N], M_R)$ and $(R[\N], M_R\oplus\N)$. Moreover, the class of (trivial) fibrations agrees with the one we introduced before. In this way, we regard animated pre-log $\ul R$-algebras as simplicial pre-log $\ul R$-algebras up to weak equivalences. (Let us omit the discussion on change of $\ul R$.) Note that the homotopy colimits can be used to compute colimits in $\infty$-categories \cite[Theorem 4.2.4.1]{Lurie}. 

The cotangent complex as an animated $R$-module is the left Kan extension of the functor
\[
\Sigma_{T_0, T_1}\mapsto \bigoplus_{t\in T_0} R[\N^{T_0}, \N^{T_1}]dX_t \oplus \bigoplus_{t\in T_1} R[\N^{T_0}, \N^{T_1}]d\log X_t,
\]
and a similar claim holds for exterior powers. As $\mathbb{L}\Omega^0_{\underline{S}/\underline{R}}\cong S$ for $\ul S=(S, M_S)$, we can regard $\L\Omega^i_{\ul S/\ul R}$ as an animated $S$-modules, cf. \cite[Remark B.2]{APC}, and this indeed recovers the definition above. Moreover, as explained in \cite[Remark 6.6]{Bhatt_dR} and \cite[Section 4]{SSV}, the animated $S$-module $\L_{\ul S/\ul R}$ co-represents
\[
\textup{Map}_{\ul R/ \ul S}(\ul S, (S\oplus J, N\oplus J))
\]
for animated $S$-modules $J$. 
This is analogous to \cite[5.1.8]{CS}.   
\end{remark}

\subsection{Functorial properties} \label{ss:log_cotangent_functoriality} \noindent 

\noindent 
Let $\ul R \ra \ul S$ be a map in $\tu{Alg}^{\textup{prelog}}$.  We have a morphism $\L_{S/R} \ra \L_{\ul S/\ul R}$ 
(resp. $\mathbb{L}\Omega_{S/R} \to  \mathbb L \Omega_{\underline{S}/\underline{R}}$) from the usual cotangent complex (resp. derived de Rham complex) to the log version induced by the obvious map $\Omega_{S/R}^1\to\Omega^1_{\underline{S}/\underline{R}}$.

\bl \label{lemma:no_additional_log}
The maps  
\[\L_{S/R} \ra \L_{\ul S/\ul R} \quad \text{ and } \quad \mathbb{L}\Omega_{S/R} \to  \mathbb L \Omega_{\underline{S}/\underline{R}}\]
are isomorphisms if $M_R\to M_S$ is an isomorphism.
\el 
\bproof
This is \cite[Lemma 8.22]{Olsson_log} (also see \cite[Proposition 6.5]{Bhatt_dR}). 
\eproof

\bl \label{lemma:pre_log_cotangent_same_as_log}
Let $(R, M) \ra (R, N)$ be a map of pre-log rings which induces an isomorphism on the associated log rings, then $\L_{(R, N)/(R, M)} = 0$. 
\el 

\bproof 
This follows from \cite[Theorem 8.20]{Olsson_log}. 
\eproof 

The log cotangent complex and derived log de Rham complex share similar functoriality properties as the usual ones, among which the following two are important for us. 

\begin{theorem} \label{thm:functorial_cotangent_complex}
\noindent 
\begin{itemize} 
	\item (transitivity triangle) For a sequence $\underline{R}\to \underline{S}_1 \to \underline{S}_2$ in $\textup{Alg}^\textup{prelog}$, the canonical sequence 
	\[\mathbb{L}_{\underline{S}_1/\underline{R}}\otimes^\L_{S_1} S_2 
	\lra \mathbb{L}_{\underline{S}_2/\underline{R}} 
	\lra  \mathbb{L}_{\underline{S}_2/\underline{S}_1}\] is a homotopy fiber sequence.

	\item (base-change)
Let $\underline{R}\to \underline{S}_i$ for $i=1,2$ be two objects in $\textup{Alg}^\textup{prelog}_{\underline{R}/}$, and let $\underline{R}\to \underline{S}$ be their homotopy coproduct, which is an object in $s\textup{Alg}^\textup{prelog}_{\underline{R}/}$. Then the natural map      
$\L_{\underline{S}_1/\underline{R}} \ra \L_{\underline{S}/\underline{S}_2}$ (resp. $\mathbb{L}\Omega_{\underline{S}_1/\underline{R}} \ra \mathbb{L}\Omega_{\underline{S}/\underline{S}_2}$) induces isomorphisms 
	\begin{align*} 
	& \mathbb{L}_{\underline{S}_1/\underline{R}}\otimes^\L_{R} S_2 \simeq \mathbb{L}_{\underline{S}_1/\underline{R}}\otimes^\L_{S_1} S \simeq \mathbb{L}_{\underline{S}/\underline{S}_2} \\
	(\tu{resp. } & \mathbb{L}\Omega_{\underline{S}_1/\underline{R}}\otimes^\L_{R} S_2 \simeq \mathbb{L}\Omega_{\underline{S}_1/\underline{R}}\otimes^\L_{S_1} S \simeq \mathbb{L}\Omega_{\underline{S}/\underline{S}_2}).
	\end{align*}
\end{itemize}
\end{theorem}

\begin{proof} 
The first claim is \cite[Theorem 8.18]{Olsson_log}, 
and the second claim is \cite[Proposition 4.12 (ii)]{SSV} (resp. \cite[Proposition 6.12]{Bhatt_dR} for the claim on derived log de Rham complexes). 
\end{proof}

\br \label{remark:compare_Gabber_with_Olsson}
Gabber's construction works more generally for morphisms between pre-log ringed structures in a topos $\mathcal R$ with enough points (see \cite[Section 8]{Olsson_log}). In particular one has $\L_{X/Y}$ for a map $X \ra Y$ between log schemes \cite[8.29]{Olsson_log}. In the case where $X \ra Y$ comes from a map $(R, P) \ra (S, Q)$ of pre-log rings, where $Y = (\spec R, P)^a$ and  $X = (\spec S, Q)^a$ are affine log schemes on charts $P$ and $Q$, there is a natural isomorphism \cite[Theorem 8.20]{Olsson_log}:
\[
\L_{(\spec S, \ul{Q})/(\spec R, \ul{P})} \xrightarrow{\cong} \L_{X/Y}, 
\]
where $\ul{Q}$ and $\ul{P}$ are the associated constant \'etale \emph{sheaves} regarded as pre-log structures. In fact, the cotangent complex $\L_{X/Y}$ viewed as a sheaf of animated $S$-modules is the quasi-coherent complex over $X$ attached to the log cotangent complex $\L_{\ul S/\ul R} \in \mathcal \mD(S)$ of pre-log rings discussed above, and $\L_{\ul S/\ul R}$ may be recovered as the (derived) global section of $\L_{X/Y}$. Using that the same is true for cotangent complexes of affine schemes such as $\L_{\spec S/\spec R}$, one can deduce this from \cite[Lemma 8.28]{Olsson_log} (see also the proof of \cite[Theorem 8.32]{Olsson_log}).\footnote{Note that $\mO_X\otimes^\L_{\Z[\ul P]}\L_{(\Z [\ul P], \ul P)/(\Z [\ul Q], \ul Q)}$ is quasi-coherent and associated with $\L_{(\Z[P], P)/(\Z[Q], Q)}$. This is because one can use a free resolution of $P$ over $Q$ to compute it. See also the proof of \cite[8.23 (ii)]{Olsson_log}.}
\er

\br \label{remark:log_cotangent_complex_nondiscrete}
As we have seen, Gabber's log cotangent complexes as defined above enjoy similar functoriality properties as the usual cotangent complexes. However, even in the case where $\ul R \ra \ul S$ induces a log smooth map between the associated log schemes in the sense of Kato \cite[3.3]{Kato}, the log cotangent complex $\L_{\ul S/\ul R}$ may fail to be discrete. One such example is given in \cite[7.3]{Olsson_log}, where $\ul R = (k, P)$  with $P \subset \N^2$ being the submonoid generated by $(2, 0), (0, 2), (1, 1)$, and $\ul S = (k[x, y]/(x^2, y^2, xy), \N^2)$. The pre-log structure on $R$ sends all nontrivial elements of $P$ to $0$, while $\N^2 \ra S$ is given by $(1, 0) \mapsto x, (0, 1) \mapsto y$. The map $P \ra \N^2$ is given by the inclusion. When $k$ has characteristic away from $2$, the map $\ul R \ra \ul S$ induces the log \'etale map. However, $\L_{(\ul S/\ul R)} \ne \Omega^1_{(\ul S/\ul R)} = 0$ (in fact it is unbounded on the left).  

In the same paper \textit{loc.cit.}, Olsson introduces another version of log cotangent complexes for log schemes, which is simpler for log smooth maps, but in general it differs from Gabber's complex $\L_{X/Y}$ (except in small degrees). 
\er 

Nevertheless, we have the following 
\bl \label{lemma:compare_Gabber_with_Olsson}
Let $\ul R \ra \ul S$ be an integral morphism of integral pre-log rings such that the induced map on the associated log schemes $(\spec S, M_S)^a \ra (\spec R, M_R)^a$ is smooth, then 
\[ \L_{\ul S/\ul R} \cong \Omega^1_{\ul S/\ul R}. \]
\el 

\bproof 
Let $Y = (\spec S, M_S)^a$ and $X = (\spec R, M_R)^a$. By Remark \ref{remark:compare_Gabber_with_Olsson} it suffices to show that $\L_{X/Y} \cong \Omega^1_{X/Y}$. Since $X \ra Y$ is integral, Gabber's log cotangent complex $\L_{X/Y}$ agrees with Olsson's construction by \cite[Theorem 8.32]{Olsson_log}, which agrees with the log differentials for log smooth maps by \cite[1.1(iii)]{Olsson_log}. 
(One can argue more directly without using Olsson's construction: the proof of \cite[Theorem 8.32]{Olsson_log} actually implies $\L_{X/Y} \cong \Omega^1_{X/Y}$, cf. the proof of \cite[Proposition 5.1]{Koshikawa}.)
\eproof

\subsection{Log cotangent complexes for perfectoid pre-log rings}  \noindent 

\noindent
In this subsection, we introduce the notion of perfectoid pre-log rings, and show that they have trivial $p$-completed log cotangent complexes. 

\begin{notation} \label{notation:Frob_base_change}
\be
\item Let $\ul R = (M \ra R)$ be a pre-log ring over $\F_p$, then we have the (absolute) Frobenius map $F_{\ul R}\colon \ul R \ra \ul R $ given by multiplication by  $p$ on $M$ and the usual $p$-power Frobenius on $R$. \item For a map $f\colon \ul R \ra \ul S$ in $\textup{Alg}^{\textup{prelog}}_{\F_p/}$, denote  by $\ul S^{(1)}$  the homotopy pushout of $f$ along the Frobenius  $F_{\ul R}$ and by 
\[F =  F_{\ul S/\ul R}\colon \ul S^{(1)} \lra \ul S\]
the relative Frobenius. 
\ee 
\end{notation}

The following definition is a slight generalization of {\cite[Definition 7.9]{Bhatt_dR}}. (This generalization is included for completeness, and we will only need the original definition to study perfectoid log rings.) 

\bd
A map $f$ in $\text{Alg}^{\textup{prelog}}_{\F_p/}$ as above is \emph{relatively perfect} if the relative Frobenius\[ F_{\ul S/\ul R}\colon \ul S^{(1)}\ra \ul S
\]
is a log equivalence in the sense of \cite[p.379]{SSV}, i.e., induces a weak equivalence on the associated log simplicial rings via \cite[Construction 3.9]{SSV}. 
\ed

\bl  \label{lemma:cotangent_for_perfect_maps}
If $f\colon \ul R \ra \ul S$ is a relatively perfect map of pre-log $\F_p$-algebras, then $\L_{\ul S/\ul R} = 0 $ and the derived log de Rham complex is simply given by $\L \Omega_{\ul S/\ul R} \cong S$. 
\el 

\bproof 
This is a slight generalization of \cite[Corollary 7.11]{Bhatt_dR}. The map
\[
\L_{\ul S^{(1)}/\ul R} \to \L_{\ul S/\ul R}
\]
is always $0$ as this is clear if $\ul S$ is free over $\ul R$. On the other hand, 
\[
\L_{\ul S^{(1)}/\ul R} \xrightarrow{\cong} \L_{\ul S^{(1), a}/\ul R}, \quad
\L_{\ul S/\ul R} \xrightarrow{\cong} \L_{\ul S^a/\ul R}
\]
by the interpretation of log cotangent complexes using derivations \cite[Definition 4.7]{SSV} and \cite[Lemma 11.9]{Rognes}. Therefore,  
\[
\L_{\ul S^{(1)}/\ul R} \cong \L_{\ul S/\ul R} = 0. 
\]
\eproof

For any monoid $M$, let us define its \emph{tilt} by
\[M^\flat \coloneqq \varprojlim_{m\mapsto m^p} M.\]
Note that  $(M^\flat)^\times =\varprojlim_{m\mapsto m^p} M^\times$ as taking units in rings commutes with limits. Also note that $M^\flat$, and hence $M^\flat /(M^\flat)^\times$,  is uniquely $p$-divisible.

\bd[Perfectoid monoids and pre-log rings] \label{def:perfectoid_monoid}
A monoid $M$ is called perfectoid if the canonical map 
\[M^\flat /(M^\flat)^\times \lra M/M^\times\]
is an isomorphism. A pre-log ring $\ul R = (R, M)$ is \emph{perfectoid} if both $R$ and $M$ are perfectoid.  
If $M$ is in addition integral, we say that $(R, M)$ is a perfectoid integral pre-log ring. 
\ed

For a perfectoid pre-log ring $(R, M)$, the natural map 
\[ (R, M^\flat) \lra (R, M) \] 
induces an isomorphism on the associated log rings. 

\br[Perfect monoids] \label{remark:perfect_monoid} A variant of Definition \ref{def:perfectoid_monoid} is the following more restrictive notion: we say that a monoid $M$ is \emph{perfect} if it is uniquely $p$-divisible, in other words, if the canonical map $M^\flat \ra M$ is an isomorphism. Perfect monoids are clearly perfectoid. One advantage of Definition \ref{def:perfectoid_monoid} is that it includes monoids of the form $\mO_C \minus \{0\}$ (as a multiplicative monoid) where $C$ is an algebraically closed perfectoid field. 
\er 

\br[Uniquely divisible monoids] Our notion of perfect(oid) monoids differs from the one in \cite{DLLZ}. We say that a monoid $M$ is \emph{divisible} (resp. \emph{uniquely divisible}) if it is  $m$-divisible (resp. uniquely $m$-divisible) for all integers $m$. Such monoids play an important role in the development of the pro-Kummer-\'etale site in \textit{loc.cit} and Section \ref{sec:log_diamonds}. A prototypical example is the following: let $P$ be a sharp fs monoid and let $P_{\Q_{\geq 0}} \coloneqq \textup{colim}_{[m]} P$ where the colimit is taken over the multiplication by $m$ map $[m]$ over all integers $m \ge 1$. Then $P_{\Q_{\geq 0}}$ is uniquely divisible. 
\er

\br[Pseudo-perfectoid monoids] \label{remark:pseudo_perfectoid} 
There is also a less restrictive variant of  Definition \ref{def:perfectoid_monoid}: we say that a monoid $M$ is \emph{pseudo-perfectoid} if $M/ M^{\times}$ is perfect, namely, uniquely $p$-divisible.  Perfectoid monoids are pseudo-perfectoid but not conversely. For example, the following monoid 
\[
M=\gr{x_0, x_1, x_2, \dots, y_1^{\pm 1},  y_2^{\pm 1}, \dots}/\gr{(x_1)^p=x_0 \cdot y_1, (x_2)^p= x_1 \cdot y_2, \dots}
\]
generated by $\{x_i, y_j\}_{i \ge 0, j \ge 1}$ with relations $x_j^p = x_{j-1} y_j$ (written multiplicatively) is pseudo-perfectoid but not perfectoid ($M/M^\times \cong \N[\frac{1}{p}]$ but $M^\flat = 0$). We say that a pre-log ring $(R, M)$ is pseudo-perfectoid if $R$ is perfectoid and $M$ is pseudo-perfectoid. 
\er

\begin{construction}
Let  $(M \xrightarrow{\alpha} R)$ be a perfectoid pre-log ring.  As $M^\flat$ is uniquely $p$-divisible, we obtain a natural monoid map 
\[
\alpha^\flat\colon M^\flat \ra R^\flat,
\]
sending $m = (m_0, m_1, m_2, ...) \in M^\flat$ to \[ \alpha^\flat (m) \coloneqq  (\alpha(m_0), \alpha(m_1),  \alpha(m_2), ...) \in R^\flat.\]  
The perfectoid pre-log ring $(R^\flat, M^\flat)$ is the \emph{tilt} of $(R, M)$.   
If we further compose it with the Teichm\"uller lift, we obtain a monoid map
\[
[\alpha^\flat]\colon M^\flat \ra \Ainf (R) =  W(R^\flat)
\]
by sending $m \mapsto  [\alpha^\flat (m)] \in W(R^\flat)$. We denote the pre-log ring $(\Ainf(R), M^\flat)$ by $\Ainf(\ul R)$. 
\end{construction}

\br  \label{remark:tilting_perfectoid_log_rings}
This construction can be extended to the case where $(R, M)$ is an integral log ring and $M$ is pseudo-perfectoid in the sense of Remark \ref{remark:pseudo_perfectoid}.  Note that the monoid $M$ itself may not be $p$-divisible. However if we set
\[
(R/p^n, M_{R/p^n})\coloneqq (R/p^n, M)^a 
\]
for every integer $n\geq 1$, 
then $M_{R/p}$ is $p$-divisible and $M/R^\times \cong M_{R/p^n}/ (R/p^n)^{\times}$. 
We obtain a pre-log ring $(\Ainf(R), M_{R/p}^\flat)$. 
Define a map 
\[\theta_n \colon M_{R/p}^\flat \to M_{R/p^n}
\] by sending $(m_{i})_{i \geq 0}$ to $\widetilde{m}_{n-1}^{p^{n-1}}$, where $\widetilde{m}_{n-1} \in M_{R/p^n}$ is a lift of $m_{n-1}$ and $\widetilde{m}_{n-1}^{p^{n-1}}$ is independent of the choice of the lift, as any two such lifts differ by some 
\[1+pr \in (R/p^{n-1})^\times= M_{R/p^{n-1}}^\times. \]  
Since $M\cong \varprojlim_n M_{R/p^n}$, their limit gives rise to a map 
\begin{equation} \label{eq:monoid_theta} 
\theta\colon M_{R/p}^\flat\to M
\end{equation} compatible with the usual map $\theta \colon \Ainf (R)\to R$. This induces a map of pre-log rings
\[
\theta\colon (\Ainf (R), M_{R/p}^\flat) \to (R, M)
\]
and in turn an exact surjection of integral log rings
\[
\theta\colon (\Ainf(R), M_{\Ainf (R)}) \coloneqq (\Ainf (R), M_{R/p}^\flat)^a \to (R, M).  
\]
\er

\br \label{remark:perfectoid_log_same_as_pseudo} 
The construction in Remark \ref{remark:tilting_perfectoid_log_rings} implies that for an integral log ring $(R, M)$, the monoid $M$ is perfectoid if (and only if) it is pseudo-perfectoid in the sense of Remark \ref{remark:pseudo_perfectoid}. Suppose that $M$ is pseudo-perfectoid, then modulo units, the map 
 $\theta$ in (\ref{eq:monoid_theta}) from Remark \ref{remark:tilting_perfectoid_log_rings} factors as the composite map 
\[
M^\flat_{R/p}/(R^\flat_{R/p})^\times \ra  M^\flat/M^{\flat, \times} \ra M/M^\times,
\]
which is surjective as $M_{R/p}$ is $p$-divisible. Thus the map $M^\flat/M^{\flat, \times} \ra M/M^\times$ is surjective. It is also injective since $M/M^\times$ is uniquely $p$-divisible. The only if direction is clear from definition.  
\er

\bd[Perfectoid log rings] \label{definition:perfectoid_log_ring}
An integral log ring $(R, M)$ is \emph{perfectoid} if it is perfectoid as a pre-log ring, or, equivalently (by Remark \ref{remark:perfectoid_log_same_as_pseudo}), if $R$ is a perfectoid ring and $M/M^\times$ is uniquely $p$-divisible (in other words, if $(R, M)$ is pseudo-perfectoid). 
\ed

\br \label{remark:relative_perfect_over_invertible_submonoid}
Let $(R, M)$ be a perfectoid log ring.  Then $(R/p, R^\times)\ra (R/p, M)$ is relatively perfect\footnote{In fact, $M^{(1)}$ is discrete so that $M^{(1)}\cong M$ as $R^\times \ra M$ is injective and integral.} and $M$ is $p$-saturated. More generally, if $(S, M)$ is pre-log ring over  $\F_p$ where $M$ is pseudo-perfectoid, then $(S, M^{\times})\to (S, M)$ is relatively perfect. 
\er

\br 
Let $(R, P)$ be a perfectoid (or more generally, pseudo-perfectoid) integral pre-log ring, then its associated log ring is a perfectoid log ring.
\er

\bl \label{lemma:perfectoid_monoid}
 Let $\ul R = (R, M)$ be a perfectoid (or more generally, pseudo-perfectoid) pre-log ring.  Let $\Z_p = (0 \ra \Z_p)$ be the trivial pre-log ring. Then the natural map 
\[\widehat \L_{R/\Z_p} \isom \widehat \L_{\ul R/\Z_p}\]
is an isomorphism. Equivalently, $\widehat \L_{\ul R/ R}=0$. In particular, we have $\widehat \L_{\ul R/\Z_p}[-1] \{-1\} \cong R$. 
\el 

\bproof
We shall prove $\widehat \L_{\ul R/ R}=0$. By derived Nakayama, it suffices to show that \[\widehat \L_{\ul R/ R}\otimes^\L_R R/p\cong \L_{(R/p, M)/ R/p}=0.\] This follows from the  isomorphism
\[
\L_{(R/p, M)/ R/p} \cong \L_{(R/p, M)/ (R/p, M^\times)}
\]
and Lemma \ref{lemma:cotangent_for_perfect_maps}, since the map $(R/p, M^\times)\to (R/p, M)$ is relatively perfect (see Remark \ref{remark:relative_perfect_over_invertible_submonoid}). 
\eproof

\bc \label{cor:cotangent_for_perfectoid}
Let $f\colon \ul R \ra \ul S$ be a map of perfectoid pre-log rings, then the $p$-completed log cotangent complex $ \widehat \L_{\ul S/\ul R} = 0 $. 
\ec 

\bproof 
Using Lemma \ref{lemma:perfectoid_monoid}, one computes that
\[
\widehat \L_{\ul S/\ul R} = \widehat \L_{\ul S/ R} =\widehat \L_{ S/ R}=0. 
\]
\eproof

We conclude this subsection with a few examples. Let us start with a ``non-perfectoid'' example. 

\beg \label{example:2}
  Let $\ul R = (\N \ra R)$ be the pre-log ring with pre-log structure $1 \mapsto 0$, then $\L_{\ul R/R}$ is concentrated in degree $[-1, 0]$. To see this, consider the transitivity triangle
\[ \L_{\ul{R[t]}/R} \otimes^\L_{R[t]} R \lra \L_{\ul R/ R} \lra \L_{\ul R/\ul{R[t]}} \]
 for $R \ra \ul{R[t]} =  (\N \xrightarrow{1 \mapsto t} R[t]) \ra \ul R$ and observe that
\bi
\item $\L_{\ul{R[t]}/R} \cong \Omega^1_{\ul{R[t]}/R} = R[t] \textup{ dlog} (t) \cong R[t]$, thus $ \L_{\ul{R[t]}/R} \otimes^\L_{R[t]} R  \cong R$;
\item $\L_{\ul R/\ul{R[t]}} \cong \L_{R/R[t]} \cong R[1]$
\ei 
In fact, we have 
\begin{equation} \label{eq:log_ci_N}
\L_{\ul R/R} \cong (R \xrightarrow{ \: 0 \: } R). 
\end{equation}
More generally, let $\ul R, \ul S$ be any pre-log algebras and $\ul R \ra \ul S$ be a log complete intersection in the following sense: suppose that the map factors through 
\[
 \ul R = (R, M_R) \lra  \ul{\sq S} = (\sq S, M_S) \lra \ul S = (S, M_S), 
 \]
 where $\sq S \twoheadrightarrow S$ is surjective, $I = \ker (\sq S \ra S)$ is a regular ideal, and $\ul R \ra \ul{\sq S}$ is smooth (on the induced map of associated log schemes).  
Then 
 \begin{equation} \label{eq:log_ci}
 \L_{\ul S/\ul R} \cong \big(I/I^2 \xrightarrow{\: - d \:} \Omega^1_{\ul{\sq S}/\ul R} \otimes_{\sq S} S \big) 
 \end{equation}
 where $I/I^2$ is placed in degree $-1$, and the differential $d$ is induced from $d\colon \sq S \ra \Omega^1_{\ul{\sq S}/\ul R}$.  
 \eeg
 
Next, let $C$ be an algebraically closed perfectoid field over $\Q_p$ and fix a choice $p^{\Q} \in C$ of rational $p$-power roots in $C$.

\beg  First let us note that, if we write $\ul \mO_C = (\Q_{\ge 0} \xrightarrow{\alpha \mapsto p^{\alpha}} \mO_C)$, then  
\[\widehat \L_{\ul \mO_C/\Z_p} [-1] \cong \mO_C \{1\} \cong \xi \Ainf/\xi^2 \Ainf. \]
This is a special case of Lemma \ref{lemma:perfectoid_monoid}. 
\eeg

 \beg
   Let 
\bi 
\item  $\ul S = (\mO_C \gr{T^{1/p^\infty}}, \N[\frac{1}{p}])$ be the perfectoid ring equipped with $\frac{a}{p^b} \mapsto T^{\frac{a}{p^b}}$,  and 
\item  $\ul R = (\N \xrightarrow{1 \mapsto T} \mO_C \gr{T})$.
\ei 
Let $\ul R \ra \ul S$ be natural inclusion of pre-log rings. The $p$-completed log cotangent complex $\widehat \L_{\ul S/\ul R}$ is concentrated in degree $-1$. 
 To see this, consider the transitivity triangle
\[ \widehat \L_{\ul R /\mO_C} \widehat \otimes_{\mO_C\gr{T}} \mO_C \gr{T^{1/p^\infty}} \lra \widehat \L_{\ul S/\mO_C} \lra \widehat \L_{\ul S/\ul R} \]
for $\mO_C \ra \ul R \ra \ul S$.  By Corollary \ref{cor:cotangent_for_perfectoid}, we know that the middle term $\widehat \L_{\ul S/\mO_C} = 0$ vanishes. Now, as $\widehat \L_{\ul R/\mO_C} \cong \widehat \Omega^1_{(\mO_C[T], \N)/\mO_C} \cong \mO_C \gr{T}$, we conclude that 
\begin{equation} \label{eq:example_5}
\widehat \L_{\ul S/\ul R} \cong \mO_C \gr{T^{1/p^\infty}} [1].\footnote{Note that although $\widehat \L_{\ul S/\ul R}$ is isomorphic to the usual cotangent complex $\widehat \L_{S/R}$ as a complex over $\mO_C \gr{T^{1/p^\infty}}$, the natural map $\widehat \L_{S/R} \ra \widehat \L_{\ul S/\ul R}$ is not an isomorphism. In fact this map is the base change of $\widehat \L_{R/\mO_C}[1] \ra \widehat \L_{\ul R/\mO_C}[1]$, given by $\mO_C \gr{T} dT \hookrightarrow \mO_C \gr{T} \textup{ dlog} T$ after a degree shift.  
}
\end{equation}
In fact, this shows that $\widehat \L_{\ul S/\ul R}$ has $p$-complete Tor amplitude concentrated in $[-1]$. 
\eeg

\beg \label{example:5}
 In this example, we let $M^{(1)}\coloneqq \N[\frac{1}{p}] \oplus_{\N} \N[\frac{1}{p}]$ be the pushout of the natural inclusion $\N \ra \N[\frac{1}{p}]$ along itself. Let  
 \[S = \mO_C \gr{M^{(1)}} \cong \mO_C \gr{X^{1/p^\infty}, Y^{1/p^\infty}}/(X - Y)
 \]
 and $\ul S = (S, M^{(1)})$, where the pre-log structure is given by $(\alpha,\beta) \mapsto X^\alpha Y^\beta$. Let $\ul R = (\mO_C \gr{T}, \N)$ be as in the previous example and regard $\ul S$ as pre-log $\ul R$-algebra by the natural map $\N \ra M^{(1)}$.  Then the $p$-completed log cotangent complex $\widehat \L_{\ul S/\mO_C}$ is given by 
$\widehat \L_{\ul S/\mO_C} \cong  S [1].$ Moreover, both $\widehat \L_{\ul S/\Z_p}$ and $\widehat \L_{\ul S/\ul R}$ are concentrated in degree $-1$. To see this, let us consider the exact triangle given by 
\[\mO_C \ra \ul{S_0} \coloneqq (\mO_C \gr{\N[\frac{1}{p}]}, \N[\frac{1}{p}]) \ra \ul S,
\] where the second map is induced from either of the maps from $\N[\frac{1}{p}] \ra M^{(1)}$, so we have
\[
\widehat \L_{\ul S/\mO_C} \cong \widehat \L_{\ul S/\ul{S_0}} \cong \widehat \L_{\ul{S_0}/\ul R} \widehat \otimes_{S_0} S \cong \widehat \L_{\ul R/\mO_C} [1] \widehat \otimes_{R} S  \cong S [1]. 
\]
From this it follows that $\widehat \L_{\ul S/\Z_p}$ lives in degree $-1$. A similar argument shows that $\widehat \L_{\ul S/\ul R}$ also lives in degree $-1$. Again, similar to the previous example, even though $\widehat \L_{S/\mO_C} \cong S[1]$ as an $S$-module as well, the map $\widehat \L_{S/\mO_C} \ra \widehat \L_{\ul S/\mO_C}$ is not an isomorphism (the cofiber of this map is nonzero and lives in degree $-1$). 
\eeg 

\subsection{Perfect ``log prisms''} \noindent

\noindent
Recall that the category of perfect prisms and the category of perfectoid rings are equivalent \cite[Theorem 3.10]{BS}. We shall define perfect ``log prisms'' and relate them to perfectoid log rings. (Unlike \cite{Koshikawa}, we do not pass to the associated log structure. It is also easy to obtain results for them.)

\begin{convention}
A bounded pre-log prism $(A, I, M_A)$ is said to be a ``log prism'' if $(A, M_A)$ is a log ring. Any bounded pre-log prism has the associated ``log prism'' $(A, I, M_A)^a$; taking the associated log structure further, we then obtain the associated log prism $(A, I, M_{\spf A})$ in the sense of \cite{Koshikawa}. Note also that, if $(A, I, M_A)$ is a ``log prism'', the $\delta_{\log}$-structure provides a Frobenius lift $\phi$ of $(A, M_A)$ via the formula
\[
\phi(m)=m^p (1+p\delta_{\log}(m))
\]
as $A$ is classically $p$-adically complete \cite[Remark 2.3]{Koshikawa}.  
\end{convention}

\bd
Let $(A, I, M_A)$ be a ``log prism''. We say that $(A, I, M_A)$ is \emph{perfect} if $M_A$ is integral and $\phi$ on $(A, M_A)$ is an isomorphism. 
\ed

\br
If $(A, I, M_A)$ is a perfect ``log prism'', then $M_A$ is $p$-saturated. Indeed, $\phi$ on $M_A$ induces the $p$-th power on $M_A/A^\times$, so $M_A/A^\times$, hence $M_A$, is clearly $p$-saturated. By a similar reason, suppose that $(A, I)$ is perfect, then a ``log prism'' $(A, I, M_A)$ is perfect if and only if $M_A/A^\times$ is uniquely $p$-divisible. 
\er

\beg
Let $(A, I, M_A)$ be an integral ``log prism''. We can consider the associated pre-log prism $(A_{\perf}, IA_{\perf}, M_A)$ whose underlying prism is perfect and the $\delta_{\log}$-structure induces a map of monoids
\[
\varinjlim_{\phi} M_A \to A_{\perf}. 
\]
By the formula of $\phi$, its associated log ring $(A_{\perf}, M_{A, \perf})$ satisfies that $M_{A,\perf}/A_{\perf}^\times$ is uniquely $p$-divisible. Therefore, $(A_{\perf}, IA_{\perf}, M_{A, \perf})$ is a perfect ``log prism'', which we call the ``perfection'' of $(A, I, M_A)$. 
\eeg

\beg
If $(R, M)$ is a perfectoid integral pre-log ring, then $\delta ([\alpha^\flat (m)])=0$ for every $m\in M^\flat$. Therefore, we may regard $(\Ainf(R), M^\flat)$ as a $\delta_{\log}$-ring of rank 1 in the sense of \cite{Koshikawa} and it is easy to check that $(\Ainf (R), \ker \theta, M^\flat)^a$ is perfect in the above sense. 
\eeg

\bp
Consider the category of perfect ``log prisms'' $(A, I, M_A)$. 
Sending 
\[(A, I, M_A) \longmapsto (A/I, M_A)^a\] induces an equivalence of categories with the category of perfectoid log rings. 
In particular, any perfect ``log prism'' $(A, I, M_A)$ admits a chart $N\to A$ of rank 1, i.e., $\delta_{\log}(N)=0$. 
\ep

\bproof
For a perfectoid log ring $(R, M)$, a quasi-inverse is given by
\[
(\Ainf (R), \ker \theta, M_{\Ainf (R)})=(\Ainf (R), \ker\theta, M_{R/p}^{\flat})^a\cong (\Ainf (R), \ker\theta, M^{\flat})^a.
\]
To check that it is indeed the quasi-inverse, the nontrivial direction is to show that, for any $(A, I, M_A)$, the above functor sends $(A/I, M_A)^a$ to $(A, I, M_A)$. This follows from Remark \ref{remark:variant_category_C} below. 
\eproof

Fix a perfectoid integral pre-log ring $(R, M)$. Let  $\mC_{/(R, M)}$ denote the category consisting of diagrams of the form 
\[
 \begin{tikzcd}
& (A, I , M_A)  \arrow[d] \\
(R, M) \arrow[r] & (A/I, M_A),
\end{tikzcd}
\]
where $(A, I, M_A)$ is a ``log prism''. Then our next lemma, which is an analogue of \cite[Lemma 4.8]{BS}, says that the category $\mC_{/(R, M^\flat)}$ admits an initial object. 

\bl \label{perfectoid pre-log prism is initial}
Let $(R, M)$ be a perfectoid integral pre-log ring. 
Let $(A, I, M_A)$ be an integral ``log prism''. Then any map 
\[(R, M)\to (A/I, M_A)^a
\]of pre-log rings (such as ones attached to $(R, M)\to (A/I, M_A)$) lifts uniquely to a map
\[
(\Ainf(R), \ker \theta, M^\flat) \to (A, I, M_A), \quad
\]
of pre-log prisms. In particular, the category $\mC_{/(R, M^\flat)}$  admits an initial object,  given by 
\[
(\Ainf(R), \ker \theta, M^\flat)^a \ra (R, M^\flat)^a. 
\]
\el 

\br \label{remark:boundedness_initial}
The boundedness condition on $A$ seems to be a technical one and can be probably removed if \cite[Theorem 5.6]{SSV} extends to the non-discrete case; it is the case for simplicial rings and this is why \cite[Lemma 4.8]{BS}  does not assume boundedness on $(B, J)$. 
\er

\begin{proof}
This proof uses results from \cite{SSV}. For every integer $m\geq 1$, the cotangent complex 
\[
\L_{(\Ainf(\ul R)/p^m)^a /(\Z/p^m, (\Z/p^m)^\times)} \cong \L_{(\Ainf(\ul R)/p^m)/(\Z/p^m)} = 0  
\]
 vanishes. 
 Now apply \cite[Theorem 5.6]{SSV} (and its proof), we know that the map 
\[(\Z/p^m, (\Z/p^m)^\times) \ra (\Ainf (\ul R)/p^m)^a\]
is formally log-\'etale in the sense of \cite[Definition 5.1]{SSV}. As $A/I$, $A$ are both classically $p$-complete by the boundedness assumption, the composition 
\[ (\Ainf(\ul R))^a \to (A/I, M_A)^a \to (A/(p,I), M_A)^a \]
lifts uniquely to a map
\[ (\Ainf(\ul R))^a \to (A, M_A)^a =  (A, M_A). \]
In particular,  we get a map of pre-log rings $\Ainf(\ul R) \to (A, M_A).$ 
Since $\Ainf(\ul R) = (\Ainf(R), M^{\flat})$ and $(A, M_A)$ have Frobenius lifts as in \cite[Remark 2.3]{Koshikawa}, and $(\Ainf (R), M^{\flat})$ is perfectoid, we can use the argument of the proof of \cite[Lemma 4.8]{BS}  and consider the unique factorization
\[(\Ainf(R), M^{\flat}) \to \varprojlim_{\phi} (A, M_A) \to (A, M_A),  \]
where the second map is a map of $\delta_{\log}$-rings. 
Using that $\varprojlim_{\phi} A$ is $p$-torsionfree, we see that $\Ainf \to A$ is a map of $\delta$-rings. 
As the monoid $M_A$ is integral, the compatibility of Frobenius lifts on monoids implies that, for any $m \in M^{\flat}$ with the image $m'$ in $\varprojlim_{\phi} M_A$, we have $1+p\delta_{\log} (m')=1$. 
Again using that $\varprojlim_{\phi} A$ is $p$-torsionfree, we see that $\delta_{\log} (m')=0$. 
Thus, we conclude that $(\Ainf(R), M^{\flat})\to (A, M_A)$ is a map of $\delta_{\log}$-rings.  
For the last claim, note that perfect prisms are bounded. 
\end{proof}

\br[Variant of the category $\mC_{/(R, M)}$] \label{remark:variant_category_C}
From the point of view of log geometry, it is more natural to consider log rings rather than pre-log rings. To this end, for the perfectoid integral pre-log ring 
$(R,M)$, we let $\mC_{/(R, M)}^a$ denote the category whose objects are diagrams of the form 
\[
 \begin{tikzcd}
& (A, I , M_A)  \arrow[d] \\
(R, M) \arrow[r] & (A/I, M_A)^a,
\end{tikzcd}
\]
where $(A, I, M_A)$ is an integral ``log prism''. This category includes more general ``log prisms'' (since not all maps $(R, M) \ra (A/I, M_A)^a$ come from ones on the chosen chart $M_A \ra A/I$ of the target). The $\mC_{/(R, M)}^a$ shares the ``same'' initial object as $\mC_{/(R, M^\flat)}$, given by 
\[ (\Ainf(R), \ker \theta, M^\flat)^a \ra (R, M^\flat)^a.\] 
In particular, if $(R, M)$ is a perfectoid log ring and $(A, I, M_A)$ is an object of $\mC_{/(R, M)}^a$, the map $(R, M)\to (A/I, M_A)^a$ lifts uniquely to 
\[
(\Ainf(R), \ker \theta, M^\flat)^a \to (A, I, M_A). 
\]
\er 

\br[Another variant of the category $\mC_{/(R, M)}$] \label{remark:variant_category_C_2}
There is another variant of the category $\mC_{/(R, M)}$ where we only allow strict morphisms of log rings (i.e., without changing log structures) from $(R, M)$, which is related to the definition of relative prismatic site in \cite{Koshikawa}. Let $\mC_{/(R, M)}^{\tu{str}}$ denote the category whose objects are diagrams of the form 
\[
 \begin{tikzcd}
& (A, I , M_A)  \arrow[d, "i"] \\
(R, M) \arrow[r] & (A/I, M)^a,
\end{tikzcd}
\]
where $(A, I, M_A)$ is an integral ``log prism'' and $i$ is an exact surjection of log rings. The initial object of this category can be identified with the initial object in  $\mC_{/(R, M)}^a$, as the natural map $(R,M^\flat)^a \isom (R, M)^a$ is an isomorphism. 
\er

\subsection{Flat descent for log cotangent complexes}  \label{ss:flat_descent} 
\noindent 

\noindent In this subsection we prove a result on flat descent of log cotangent complexes.  First we define a suitable notion of (faithful) flatness for pre-log rings, under which it is convenient to form coproduct in the (nonabelian) derived category. This will be crucial in our definition of the log quasisyntomic site and for our proof of the descent result. 

\begin{definition}  \label{definition:hom_flat}
\be 
\item A map $\underline{R}\to\underline{S}$ of pre-log rings is \emph{homologically log flat}\footnote{This terminology is suggested to us by Gabber.} if for all maps $\underline{R}\to \underline{S'}$, the canonical map 
\[ \underline{S'}\oplus^{\L}_{\underline{R}} \underline{S} \lra \underline{S'}\oplus_{\underline{R}} \underline{S}\]
from the homotopy pushout to the naive pushout is an isomorphism. 
\item A map $\underline{R}\to\underline{S}$ in $\textup{Alg}^\textup{prelog}$ is \emph{homologically log faithfully flat} if it is homologically log flat and the underlying map $R\to S$ is faithfully flat.
\ee 
\end{definition}

\begin{remark}
Since $\textup{Forget}^{s\textup{Alg}^\textup{prelog}}_{s\textup{Mon}\times s\textup{Alg}}$ commutes with homotopy colimits, $\underline{R}\to\underline{S}$ is homologically log flat if and only if $R\to S$ is a flat map of rings and $M_R\to M_S$ is a flat map of monoids in the sense of \cite[Definition 4.8]{Bhatt_dR}. In particular, by \cite[Proposition 4.9]{Bhatt_dR} and \cite[Proposition 4.1]{Kato}, if $R\to S$ is flat and $M_R\to M_S$ is an injective and integral morphism of integral monoids, then $\underline{R}\to\underline{S}$ is homologically log flat. 
\end{remark}

\begin{remark} 
The notion of homologically log flat is different from Kato's notion of log flatness (naturally extended to pre-log rings) from \cite[Definition 1.10]{Kato2}. For example, a map $(k[P], P) \ra (k[Q], Q)$ induced from a map of monoids $P \ra Q$ that is injective but non-integral is log flat but not homologically log flat. In particular, we may take $P \subset Q = \N^2$ to be the submonoid generated by $(2, 0), (0, 2), (1, 1)$ as in Remark \ref{remark:log_cotangent_complex_nondiscrete}. On the other hand, let $(k, \N)$ (resp. $(k, \N^2)$) be the pre-log ring where all nonzero elements in the monoid map to $0 \in k$. Then the map $(k, \N) \ra (k, \N^2)$ induced by the diagonal $\N \xrightarrow{1 \mapsto (1, 1)} \N^2$ is homologically log flat but not log flat. 
\end{remark}

By definition, homologically log flat maps are closed under base change and compositions. Therefore $\textup{Alg}^\textup{prelog}$ can be upgraded to a site with covering maps given by homologically log faithfully flat maps. We refer this Grothendieck topology as the hlf topology, where ``hlf'' stands for homologically log flat. 

Our goal is to show that the log cotangent complexes satisfy descent for the hlf topology, which is analogous to the proof of \cite[Theorem 3.1]{BMS2}. 
 
\begin{proposition} 
\label{lem-descent-hlf-hodge-graded-piece} Fix a base pre-log ring $\underline{R}$. For each $i\ge 0$, the functor $\underline{S}\mapsto \midwedge^i_S \mathbb{L}_{\underline{S}/\underline{R}}$ is an hlf sheaf on $\textup{Alg}^\textup{prelog}_{\underline{R}/}$ with values in $\mD(R)$. In other words,  if $\underline{S}^{-1}\to \underline{S}^0$ is a homologically log faithfully flat map of $\underline{R}$-pre-log algebras, then we have a natural isomorphism
\[\midwedge^i_{S^{-1}} \mathbb{L}_{\underline{S}^{-1}/\underline{R}} \simeq \tu{Tot}( \midwedge^i_{S^\bullet} \mathbb{L}_{\underline{S}^\bullet/\underline{R}} )\] where $\underline{S}^\bullet$ is the homotopy \v{C}ech nerve (which coincides with the naive \v{C}ech nerve by the homologically log flat assumption) of $\underline{S}^{-1}\to \underline{S}^0$.
\end{proposition}

\bproof The cosimplicial sequence $\underline{R}\to \underline{S}^{-1} \to \underline{S}^\bullet$ induces a cosimplicial fiber sequence 
\[
 \mathbb{L}_{\underline{S}^{-1}/\underline{R}} \otimes^\L_{S^{-1}} S^\bullet \lra \mathbb{L}_{\underline{S}^{\bullet}/\underline{R}} \lra \mathbb{L}_{\underline{S}^{\bullet}/\underline{S}^{-1}}. \] Applying $\midwedge^i$ to it, we obtain a length $i+1$ 
filtered cosimplicial object in $\mD(R)$:
\begin{multline*}
0=\tu{Fil}^{i+1}( \midwedge^i_{{S}^\bullet} \mathbb{L}_{\underline{S}^{\bullet}/\underline{R}} )  \to  \midwedge^i_{{S}^{-1}}\mathbb{L}_{\underline{S}^{-1}/\underline{R}} \otimes^\L_{S^{-1}} S^\bullet  = \tu{Fil}^i( \midwedge^i_{{S}^\bullet} \mathbb{L}_{\underline{S}^{\bullet}/\underline{R}} ) \to \cdots \\
\cdots \to \tu{Fil}^1( \midwedge^i_{{S}^\bullet} \mathbb{L}_{\underline{S}^{\bullet}/\underline{R}} ) \to \tu{Fil}^0( \midwedge^i_{{S}^\bullet} \mathbb{L}_{\underline{S}^{\bullet}/\underline{R}} ) = \midwedge^i_{{S}^\bullet} \mathbb{L}_{\underline{S}^{\bullet}/\underline{R}},
\end{multline*}
with graded pieces 
\[\tu{gr}^j( \midwedge^i_{{S}^\bullet} \mathbb{L}_{\underline{S}^{\bullet}/\underline{R}} ) \simeq \midwedge^j_{{S}^{-1}}\mathbb{L}_{\underline{S}^{-1}/\underline{R}} \otimes^\L_{S^{-1}}\midwedge^{i-j}_{{S}^\bullet} \mathbb{L}_{\underline{S}^{\bullet}/\underline{S}^{-1}}.\]
This induces a length $i+1$ filtration on $\tu{Tot}( \midwedge^i_{S^\bullet} \mathbb{L}_{\underline{S}^\bullet/\underline{R}} )$. 
Since homotopy totalization commutes with taking fibers (which is a homotopy pullback), we have 
\[\tu{gr}^j\tu{Tot}( \midwedge^i_{S^\bullet} \mathbb{L}_{\underline{S}^\bullet/\underline{R}} ) \simeq \tu{Tot}( \tu{gr}^j (\midwedge^i_{S^\bullet} \mathbb{L}_{\underline{S}^\bullet/\underline{R}}) ).\]
In particular, 
\[\tu{gr}^i\tu{Tot}( \midwedge^i_{S^\bullet} \mathbb{L}_{\underline{S}^\bullet/\underline{R}} ) \simeq \tu{Tot}( \tu{gr}^i (\midwedge^i_{S^\bullet} \mathbb{L}_{\underline{S}^\bullet/\underline{R}}) ) \simeq  \tu{Tot}(  \midwedge^i_{{S}^{-1}}\mathbb{L}_{\underline{S}^{-1}/\underline{R}} \otimes^\L_{S^{-1}} S^\bullet  ) \simeq \midwedge^i_{{S}^{-1}}\mathbb{L}_{\underline{S}^{-1}/\underline{R}},\]
where the last isomorphism comes from faithfully flat descent for modules. 

Therefore,  it suffices to show that $\tu{Tot}( \tu{gr}^j (\midwedge^i_{S^\bullet} \mathbb{L}_{\underline{S}^\bullet/\underline{R}}) )$ vanishes for $j<i$. By faithfully flat descent for modules, we have
\[ \tu{Tot}_{[n]\in\Delta}( \tu{gr}^j (\midwedge^i_{S^n} \mathbb{L}_{\underline{S}^n/\underline{R}})) \simeq \tu{Tot}_{[n]\in\Delta}\tu{Tot}_{[m]\in\Delta}( \tu{gr}^j (\midwedge^i_{S^n} \mathbb{L}_{\underline{S}^n/\underline{R}}) \otimes^\L_{S^{-1}} S^m ).\]
Tautologically, we have
\[ \tu{Tot}_{[n]\in\Delta}\tu{Tot}_{[m]\in\Delta}( \tu{gr}^j (\midwedge^i_{S^n} \mathbb{L}_{\underline{S}^n/\underline{R}}) \otimes^\L_{S^{-1}} S^m )\simeq \tu{Tot}_{[m]\in\Delta}\tu{Tot}_{[n]\in\Delta}( \tu{gr}^j (\midwedge^i_{S^n} \mathbb{L}_{\underline{S}^n/\underline{R}}) \otimes^\L_{S^{-1}} S^m ).\]
Hence it suffices to show that $\tu{Tot}_{[n]\in\Delta}( \tu{gr}^j (\midwedge^i_{S^n} \mathbb{L}_{\underline{S}^n/\underline{R}}) \otimes^\L_{S^{-1}} S^m )$ vanishes.  
Unwinding definitions, we have 
\[\tu{gr}^j (\midwedge^i_{S^n} \mathbb{L}_{\underline{S}^n/\underline{R}}) \otimes^\L_{S^{-1}} S^m \simeq \midwedge^j_{{S}^{-1}}\mathbb{L}_{\underline{S}^{-1}/\underline{R}} \otimes^\L_{S^{-1}} ( \midwedge^{i-j}_{{S}^n} \mathbb{L}_{\underline{S}^{n}/\underline{S}^{-1}}\otimes^\L_{S^{-1}} S^m).\]
By the base change isomorphisms 
\[
\midwedge^{i-j}_{{S}^n} \mathbb{L}_{\underline{S}^{n}/\underline{S}^{-1}} \otimes^\L_{S^{-1}} S^m \simeq \midwedge^{i-j}_{{S}^{n+m+1}} \mathbb{L}_{\underline{S}^{n+m+1}/\underline{S}^{m}},
\] 
and it is enough to show that 
\[\tu{Tot}_{[n]\in \Delta}(  \midwedge^j_{{S}^{-1}}\mathbb{L}_{\underline{S}^{-1}/\underline{R}} \otimes^\L_{S^{-1}} \midwedge^{i-j}_{{S}^{n+m+1}} \mathbb{L}_{\underline{S}^{n+m+1}/\underline{S}^{m}} )\] vanishes. 
At this point, note that this cosimplicial diagram is obtained by applying the functor (from the category $ \textup{Alg}^\textup{prelog}_{S^{-1}/} \to \mD(R)$)
\[T\mapsto \midwedge^j_{{S}^{-1}}\mathbb{L}_{\underline{S}^{-1}/\underline{R}} \otimes^\L_{S^{-1}} \midwedge^{i-j}_{T} \mathbb{L}_{\underline{T}/\underline{S}^{m}} \]
to the cosimplicial diagram 
\[\Delta\to \textup{Alg}^\textup{prelog}_{S^{-1}/},\;[n]\mapsto \underline{S}^{n+m+1},\]
where a map $f\colon [n_1]\to [n_2]$ is sent to the composition
\[ \underline{S}^{n_1+m+1} \simeq \underline{S}^{n_1}\oplus_{\underline{S}^{-1}}\underline{S}^{m} \to \underline{S}^{n_2}\oplus_{\underline{S}^{-1}}\underline{S}^{m} \simeq \underline{S}^{n_2+m+1} \]
with the middle map induced by $f$.

This cosimplicial diagram can be extended to a split augmented cosimplicial diagram with its $[-1]$-term given by $\underline{S}^m$. Since any functor preserves split augmented cosimplicial diagrams, the desired cosimplicial diagram $\tu{Tot}_{[n]\in \Delta}(  \midwedge^j_{{S}^{-1}}\mathbb{L}_{\underline{S}^{-1}/\underline{R}} \otimes^\L_{S^{-1}} \midwedge^{i-j}_{{S}^{n+m+1}} \mathbb{L}_{\underline{S}^{n+m+1}/\underline{S}^{m}} )$ can be extended to a split augmented cosimplicial diagram with its $[-1]$-term given by $\midwedge^j_{{S}^{-1}}\mathbb{L}_{\underline{S}^{-1}/\underline{R}} \otimes^\L_{S^{-1}} \midwedge^{i-j}_{S^m} \mathbb{L}_{\underline{S}^m/\underline{S}^{m}}\simeq 0$. This proves the desired vanishing result because the underlying augmented cosimplicial diagram of any split augmented cosimplicial diagram is a limit diagram (also see \cite[Lemma 6.1.3.16]{Lurie}). 
\eproof 

\br \label{remark:simplicial_base_change}
The same argument above in fact shows the following slight extension of Proposition \ref{lem-descent-hlf-hodge-graded-piece}. Let   $\underline{S}^{-1}\to \underline{S}^0$ be a homologically log faithfully flat map of pre-log algebras over $\ul R$. Let $\ul R \ra \ul R'$ be a map of pre-log rings and let ${\ul{S}^{-1,}}'$ (resp. ${\ul{S}^{0,}}'$) be the homotopy pushout of $\ul{S}^{-1}$ (resp. $\ul{S}^0$) along $\ul R \ra \ul R'$. Then we have a natural isomorphism 
\[\midwedge^i  \mathbb{L}_{{\ul{S}^{-1,}}'/\ul R'} \simeq \tu{Tot}( \midwedge^i  \mathbb{L}_{{\ul S^{\bullet,}}'/\ul R'})\]
where ${\ul S^{\bullet,}}'$ denotes the \v{C}ech nerve of ${\ul S^{-1,}}' \ra {\ul S^{0,}}'$.
\er 

The following is a log version of \cite[Proposition E.16]{APC}.

\begin{corollary} \label{cor:qsyn_descent_cotangent} Fix a base pre-log ring $\underline{R}$. The functor $\underline{S}\mapsto  \mathbb{L}\Omega^{\tu{hc}}_{\underline{S}/\underline{R}}$ sending $\ul S$ to the Hodge-completed derived log de Rham complex is an hlf sheaf on $\textup{Alg}^\textup{prelog}_{\underline{R}/}$ with values in $\mD(R)$ (or 
in the $\infty$-category of $E_\infty$-$R$-algebras).
\end{corollary}

\proof Since  totalization commutes with completion for the Hodge filtration (which is a limit), it suffices to show that for each $i \ge 0$ the functor 
\[\underline{S}\longmapsto \L \Omega_{\underline{S}/\underline{R}}/ \tu{Fil}_{\tu{Hodge}}^{i}{\mathbb{L}}\Omega_{\underline{S}/\underline{R}}\] is an hlf sheaf with values in $\mD(R)$. The statement is clear for $i=0$.  
By induction it is then enough to show that the functor sending 
\[\underline{S}\longmapsto \tu{gr}_{\tu{Hodge}}^{i}{\mathbb{L}}\Omega_{\underline{S}/\underline{R}}\] is an hlf sheaf with values in $\mD(R)$, but this is Proposition \ref{lem-descent-hlf-hodge-graded-piece}.
\qed
 
\br See Corollary \ref{lemma:qSyn_descent_of_logdR} for a variant of Corollary \ref{cor:qsyn_descent_cotangent} for the $p$-completed log de Rham complex $\widehat \L\Omega_{-/\ul R}$.
\er

\newpage

\section{The log quasisyntomic site} \label{sec:qSyn}
In this section we introduce the quasisyntomic topology on the category of pre-log rings, which is a logarithmic analogue of the quasisyntomic site introduced in \cite{BMS2}. 

\subsection{Quasisyntomic pre-log rings and quasisyntomic maps} \label{subsec:log_qSyn} \noindent 

\bd 
We say that a map of pre-log rings $\ul A \ra \ul B$ is $p$-completely homologically log flat (resp. faithfully flat) if 
\be
\item $ B \otimes^\L_A A/p \cong B/p$ is concentrated in degree $0$, and 
\item the map $\ul A/p \lra \ul B/p$ is homologically log flat (resp. faithfully flat). 
\ee
\ed 

\bd 
A pre-log ring $\ul A = (A, M_A)$ is \emph{quasisyntomic} if 
\be 
\item $A$ is $p$-complete and has bounded $p^{\infty}$-torsion, and
\item the log cotangent complex $\L_{\ul A/\Z_p} \in \mD (A)$ has $p$-complete Tor amplitude in $[-1, 0]$, i.e., its mod $p$ reduction $\L_{\ul A/\Z_p} \otimes^\L_A A/p$  has Tor amplitude in $[-1, 0]$. 
\ee

Let $\QSyn^{\textup{prelog}}$ denote the category of quasisyntomic pre-log rings. 
\ed

\bd 
Let $\ul A$ and $\ul B$ be $p$-complete pre-log rings with bounded $p^{\infty}$-torsion. A map $\ul A \ra \ul B$ is \emph{quasisyntomic (resp. cover)} if 
\be 
\item $\ul A \ra \ul B $ is $p$-completely homologically log flat (resp. faithfully flat), and 
\item $\L_{\ul B /\ul A} \otimes^\L_{B} B/p$ has Tor amplitude in $[-1, 0]$. 
\ee
\ed 

 Let us give some examples of quasisyntomic pre-log rings and quasisyntomic maps. 
\begin{example} \label{example:log_qsyn}
\be
\item For rings equipped with the trivial pre-log structures, quasisyntomic objects and maps are the usual quasisyntomic rings and maps defined in \cite{BMS2}. In particular, $p$-adic completions of smooth algebras over $\Z_p$ or $\mO_C$ (equipped with trivial pre-log structures) are quasisyntomic; quasiregular semiperfectoid and perfectoid rings such as
\[ \mO_C/p, \quad \mO_C\gr{T^{1/p^\infty}}, \quad \mO_C\gr{T^{1/p^\infty}}/(T) \]
are quasisyntomic. 
\item If $R$ is a (classically) $p$-complete $\Z_p$-algebra, and $\Z_p \ra \ul R = (R, M)$ is a map of pre-log rings that induces a smooth map on the associated log $p$-adic formal schemes (so $R$ is necessarily $p$-torsion free), then $\ul R \in \QSyn^{\textup{prelog}}$. More generally, if $\ul{R}$ is the $p$-adic completion of a $p$-torsion free log complete intersection over $\Z_p$, then $\ul R \in \QSyn^{\textup{prelog}}$ by the discussion in Example \ref{example:2}. 
\item The pre-log ring 
\[\Big( \N[\frac{1}{p}]\oplus_\N \N[\frac{1}{p}] \lra  \mO_C \gr{X^{1/p^\infty}, Y^{1/p^\infty}}/(X-Y)\Big)\]
described in Example \ref{example:5}  is quasisyntomic by the computation thereof. 
\ee 
\end{example}

\bl \label{lemma:qsyn_cover_of_qsyn_is_qsyn}
Let $\ul A \ra \ul B$ be a quasisyntomic cover of $p$-complete pre-log rings with bounded $p^{\infty}$-torsion. Then $\ul A \in \QSyn^{\textup{prelog}}$ if and only if $\ul B \in \QSyn^{\textup{prelog}}$.  
\el 

\bproof 
By considering maps $\Z_p \ra \ul A \ra \ul B$ of pre-log rings, we have 
\[ \L_{\ul A/\Z_p} \otimes^\L_{A} B/p  \lra \L_{\ul B/\Z_p} \otimes^\L_{B} B/p \lra \L_{\ul B/\ul A} \otimes^\L_B B/p.\]
If $\ul A$ is quasisyntomic, then both $ \L_{\ul A/\Z_p} \otimes^\L_{A} B/p$ and $\L_{\ul B/\ul A} \otimes^\L_B B/p$ have Tor amplitudes in $[-1, 0]$, so $ \L_{\ul B/\Z_p} \otimes^\L_{B} B/p$ has Tor amplitude in $[-1, 0]$. 
Conversely, if $\ul B$ is quasisyntomic, then $ \L_{\ul A/\Z_p} \otimes^\L_{A} B/p$, thus $ \L_{\ul A/\Z_p} \otimes^\L_{A} A/p$, has Tor amplitude in $[-1, 1]$ since $A/p \ra B/p$ is faithfully flat by assumption. This further implies that $ \L_{\ul A/\Z_p} \otimes^\L_{A} A/p$ has Tor amplitude in $[-1, 0]$.
\eproof

\bl \label{lemma:composition_pushout_QSyn} Assume the underlying rings of all the relevant pre-log rings are $p$-complete with bounded $p^{\infty}$-torsion. 
\be
\item The composition $\ul A \ra \ul B \ra \ul C $ of two quasisyntomic maps is again quasisyntomic. 
\item Let $\ul A \ra \ul B$ be a quasisyntomic map and $\ul A \ra \ul C$ be any map of pre-log rings, then the $p$-completed pushout $\ul D = \ul B \widehat \otimes_{\ul A}^\L \ul C$ is a discrete pre-log ring, $D$ has bounded $p^{\infty}$-torsion, and the map 
\[\ul C \lra \ul D\]
is quasisyntomic. 
\ee
\el
\bproof 
We first note that 
\[ C \otimes^\L_A A/p \cong C \otimes^\L_B (B \otimes^\L_A A/p) \cong  C \otimes^\L_B B/p \cong C/p \]
is discrete (concentrated in degree $0$), and the composition 
\[ \ul{A}/p \lra \ul{B}/p \lra \ul{C}/p \]
is homologically log flat, so $\ul A\ra \ul C$ is $p$-completely homologically log flat. It is then clear from the transitivity triangle that $\L_{\ul C/\ul A} \otimes^\L_C C/p$ has Tor amplitude in $[-1, 0]$. This proves (1).  

For (2), we first check that 
\begin{align*} 
D \otimes^\L_C C/p  =  (B \otimes^\L_A C) \otimes^\L_C C/p 
 = & \: \: B \otimes^\L_A  C/p  \\
 = & \:  (B \otimes^\L_A A/p) \otimes^\L_{A/p} C/p 
=  B/p \otimes^\L_{A/p} C/p,
\end{align*}
where we use that $B \otimes^\L_A A/p = B/p$ is discrete for the last equality. As $A/p \ra B/p$ is flat, we conclude that $D \otimes^\L_{C} C/p = B/p \otimes_{A/p} C/p$ is concentrated in degree $0$, and that $\ul D/p$ is homologically log flat over $\ul C/p$. Since $D$ is $p$-completely flat over $C$, which has bounded $p^\infty$-torsion, we know that $D$ has bounded $p^\infty$-torsion by \cite[Corollary 4.8]{BMS2}. Finally, the claim on Tor amplitude follows from the base change functoriality of the log cotangent complex. 
\eproof

\bc 
$\QSyn^{\textup{prelog}, \opp}$ forms a site with the topology given by quasisyntomic covers. 
\ec 

\bproof 
This is an immediate consequence of Lemma \ref{lemma:composition_pushout_QSyn}. 
\eproof 

\br
Let $(R, M)$ be a perfectoid pre-log ring, which is quasisyntomic. We define $\QSyn_{(R, M)}$ to be the slice category of $\QSyn^{\textup{prelog}, \opp}$ over $(R, M)$ with the quasisyntomic topology. For any object $(S, N)$ of $\QSyn_{(R, M)}$, the derived $p$-completion of $\midwedge^i \L_{(S, N)/(R, M)}[-i]$ lives in $\mD^{\geq 0}(S)$. 
\er

\br
Let $(R, M)$ be a $p$-complete pre-log ring with bounded $p^\infty$-torsion. 
The small quasisyntomic site $\qSyn_{(R, M)}$ consists of quasisyntomic maps $(R, M)\to (S, N)$ with the quasisyntomic topology. For any object $(S, N)$ of $\qSyn_{(R, M)}$, the derived $p$-completion of $\midwedge^i \L_{(S, N)/(R, M)}[-i]$ lives in $\mD^{\geq 0}(S)$. If $(R, M)$ is quasisyntomic, then all objects of $\qSyn_{(R, M)}$ are quasisyntomic. 
\er

The following is a log version of some aspect of \cite[Variant E.17]{APC}. 

\bc \label{lemma:qSyn_descent_of_logdR} 
Let $\ul R = (R, M)$ be a $p$-complete pre-log ring with bounded $p^\infty$-torsion. The functor $\ul S \longmapsto  \widehat{\L}\Omega_{\underline{S}/\underline{R}}$ is a sheaf on $\qSyn_{\underline{R}}$ 
 with values in the $\infty$-category of $E_\infty$-$R$-algebras. 
\ec 

\bproof 
Let $\ul{S}^{-1} \ra \ul{S}^0$ be a quasisyntomic cover in $\qSyn_{\ul R}$ and let $\ul S^\bullet$ denote its (derived) $p$-completed \v{C}ech nerve. We need to show that the map $  \widehat{\L}\Omega_{\ul S^{-1}/\underline{R}} \lra \tu{Tot} (\widehat{\L}\Omega_{\ul S^{\bullet}/\underline{R}})$ is an isomorphism in $\mD(R)$. For which it suffices to check after taking derived reduction mod $p$. In other words, we want to show that, after tensoring with $\otimes^\L_{R} R/p$, the canonical map 
\[ \L\Omega_{(\ul S^{-1}/p)/(\underline{R}/p)} \lra \tu{Tot} (\L\Omega_{(\ul S^{\bullet}/p)/(\underline{R}/p)}) \]
is an isomorphism (here we use that $S^i \otimes^\L_{R} R/p \cong S^i/p$ from the quasisyntomic assumption).  
To this end, we first claim that for each $j \ge 0$, we have  
\[ \Fil_j^{\tu{conj}} \L \Omega_{(\ul S^{-1}/p)/(\underline{R}/p)} \isom  \tu{Tot} (\Fil_j^{\tu{conj}} \L \Omega_{(\ul S^{\bullet}/p)/(\underline{R}/p)}) \]
on the conjugate filtration (see Construction \ref{construction:derived_log_dR}). By considering the graded pieces (and using \ref{eq:derived_Cartier}), it suffices to show that 
\[  \L \Omega^j_{ (\ul S^{-1}/p)^{(1)}/(\underline{R}/p)} \isom  \tu{Tot} (  \L \Omega^j_{(\ul S^{\bullet}/p)^{(1)}/(\underline{R}/p)}),\]
but this follows from Lemma \ref{lem-descent-hlf-hodge-graded-piece} since Frobenius base change $(\ul S^{\bullet}/p)^{(1)}$ can be identified with the \v{C}ech nerve of $(\ul S^{-1}/p)^{(1)} \ra (\ul S^{0}/p)^{(1)}$ (note that, despite the Frobenius twist, $(\ul S^{-1}/p)^{(1)} \ra (\ul S^{0}/p)^{(1)}$ is a homologically log faithfully flat map of \emph{discrete} pre-log rings, since both $\ul S^{-1}/p$ and $\ul S^{0}/p$ are homologically log flat over $\ul R/p$ by assumption). 
Now by the quasisyntomic assumption (and Lemma \ref{lemma:composition_pushout_QSyn}), the object $\Fil_j^{\tu{conj}} \L \Omega_{(\ul S^{i}/p)/(\underline{R}/p)}$ lives in degree $\ge 0$ for each $i \ge 0$. The claim thus follows since totalization commutes with filtered colimit for complexes bounded from below. 
\eproof

\subsection{Quasiregular semiperfectoid (pre-)log rings} \label{subsec:log_qRSP} \noindent 

\noindent 
Now we introduce a version of quasiregular semiperfectoid pre-log rings -- roughly these are given by a pair $(S, M_S)$ where $S$ is a quotient of perfectoid rings by a quasiregular ideal and $M_S$ is a quotient of a uniquely $p$-divisible monoid by a ``quasiregular'' equivalence relation.

\bd \label{definition:qrspd_log}
A $p$-complete pre-log ring $\ul S = (S, M)$ is \emph{semiperfectoid} if 
\be
\item There exists a map $R \ra S$ from a perfectoid ring to $S$;
\item  $S/p$ is semiperfect (meaning that the map $\textup{Frob}\colon S/p \lra S/p$ is surjective); 
\item  the natural map  
\begin{align} \label{eq:semiperfect_monoid}
  M^\flat \lra M/M^\times 
\end{align}
is surjective. 
\ee
If, moreover, 
\begin{enumerate}
    \item[(4)] $\ul S$ is quasisyntomic
\end{enumerate}
then we say that $\ul S$ is \emph{quasiregular semiperfectoid}.
The category of quasiregular semiperfectoid pre-log rings is denoted by $\QRSP^{\textup{prelog}}$. 

It is sometimes useful to require the monoid $M$ in the pre-log ring $(S,M)$ to be integral (see Lemma \ref{perfectoid pre-log prism is initial} for example). 
\ed 

\br
Pre-log rings satisfy condition (2) and (3) in Definition \ref{definition:qrspd_log} are called log-semiperfect. They imply that
\[
\L_{(S, M)/ R}\otimes^\L_S {S/p}\in \mD^{\leq -1}(S/p)
\]
for any map $R\to S$. 
So, if $(S, M)$ is quasiregular semiperfectoid, then $\widehat \L_{(S, M)/ \Z_p}[-1]$ is $p$-completely flat.

There are other related conditions one may require on the monoids instead of asking the map $M^\flat \ra \cl M \coloneqq M/M^\times $ to be surjective. For example one could consider conditions 
\be
\item[(5)] The natural map $M^\flat \ra M$ is surjective, or equivalently, the multiplication by $p$ map $[p]\colon M \ra M $ is surjective. 
\item[(6)]  The natural map $(M/M^\times)^\flat \ra M/M^\times$ is surjective, or equivalently, the multiplication by $p$ map $[p]\colon M/M^\times \ra M/M^\times $ is surjective.
\ee
Note that condition (5) implies condition (3), which implies condition (6), and for sharp monoids they are equivalent. Also note that, if $\ul S$ is any pre-log ring where $S/p$ is semiperfect and $M$ satisfies condition $(6)$ (the weakest condition), then already we have $\Omega^1_{(\ul S/p)/\F_p} = 0$. Moreover, if we were to replace condition (3) by condition (5) or (6) in the definition of quasiregular semiperfectoid pre-log rings, then most assertions in the subsection would still hold. The condition (3) is chosen so that it is more convenient to discuss quasiregular semiperfectoid log rings rather than just pre-log rings. 
\er

\br 
If $(S, M)$ is semiperfectoid with a map $R \ra S$ from a perfectoid ring $R$, then
\[
(R', M')\coloneqq (R\widehat{\otimes}_{\Z_p}W(S^\flat)\widehat{\otimes}_{\Z_p} \Z_p \langle M^\flat \rangle, M^\flat) \to (S, M)
\]
is a map from a perfectoid pre-log ring that is surjective on rings and surjective modulo invertible elements on monoids. If $(S, N)$ is quasiregular semiperfectoid, then the $p$-completion of $\L_{(S, N)/ (R', M')}[-1]$ is $p$-completely flat. 

Conversely, any $(S, M)$ satisfying (4) in Definition \ref{definition:qrspd_log} is semiperfectoid if there exists a map $(R, N)\to (S, M)$ from a perfectoid pre-log ring $(R, N)$ that is surjective on rings and surjective modulo invertible elements on monoids. 
\er

\br \label{remark:quotient_semiperfect}
Let $M$ be a monoid satisfying (\ref{eq:semiperfect_monoid}) (i.e., condition (3) in Definition \ref{definition:qrspd_log}), then any quotient monoid of $M$ also satisfies (\ref{eq:semiperfect_monoid}). 
To see this, suppose $M \twoheadrightarrow N$ is surjective, then the map $M  \twoheadrightarrow N/N^\times$ and thus $M/M^\times  \twoheadrightarrow N/N^\times$ is surjective. The claim follows easily from this observation. In particular, if $M \ra P$ and $M \ra Q $ are two monoid maps and $P, Q$ both satisfy (\ref{eq:semiperfect_monoid}), then the pushout $P \oplus_M Q$ also satisfies (\ref{eq:semiperfect_monoid}). 
\er 

\begin{example}
\be
\item Let $\ul R = (R, M)$ with $R \in \QRSP$ and $M$ is uniquely $p$-divisible, then $\ul R \in \QRSP^{\textup{prelog}}$. 
\item  The pre-log ring from Example \ref{example:log_qsyn} (3)
\[
\Big( \N[\frac{1}{p}]\oplus_\N \N[\frac{1}{p}] \lra  \mO_C \gr{X^{1/p^\infty}, Y^{1/p^\infty}}/(X-Y)\Big)
\]
is log quasiregular semiperfectoid. 
\ee
\end{example}

The analogue of \cite[Lemma 4.25]{BMS1} still holds in the logarithmic setup.   
\bl \label{lemma:log_BMS_4.25}
Let $\ul S$ be a pre-log algebra and assume $S$ is $p$-complete with bounded $p^{\infty}$-torsion and $\ul S/p$ is log-semiperfect. Then $\ul S \in \QRSP^{\textup{prelog}}$ if and only if there exists a map $R \ra \ul S$ from a perfectoid $R$ (equipped with the trivial pre-log structure) such that $\L_{\ul{S}/R} \otimes^\L_{S} S/p \in \mD(S/p)$ has Tor amplitude concentrated in degree $[-1]$. In this case, the latter condition holds for any choice of $R \ra \ul S$ with $R$ being perfectoid. 
\el  

\bproof 
Suppose $\ul S \in \QRSP^{\textup{prelog}}$ and let $R \ra \ul{S}$ be a map from a perfectoid $R$ (which exists by definition). We have to show that $\L_{\ul S/R}$ has $p$-complete Tor amplitude concentrated in degree $[-1]$. The maps $\Z_p \ra R \ra \ul{S}$ give rise to the transitivity triangle 
\[ \L_{R/\Z_p} \otimes^\L_{R} S/p \lra \L_{\ul S/\Z_p} \otimes^\L_{S} S/p \lra \L_{\ul S/ R} \otimes^\L_S S/p .\]
Note that the first two terms have Tor amplitude in degree $[-1]$, as in the proof of \cite[Lemma 4.25]{BMS1}, we have to show that the first arrow is pure on $\pi_1$, namely, the map 
\[ \beta = \beta_{S/p}\colon \pi_1 (\L_{R/\Z_p} \otimes^\L_R S/p) \cong S/p \lra \pi_1 (\L_{\ul S/\Z_p} \otimes^\L_{S} S/p) \]
is injective after tensoring with any discrete $S/p$-module. By \cite[Lemma 4.26]{BMS1}, it suffices to show that $\beta_{S/p} \otimes k$ is injective for any perfect field $k$ over $S/p$. Now, by functoriality, $\beta_{S/p} \otimes k$ factors the map
\[ \beta_k\colon \pi_1 (\L_{R/\Z_p} \otimes^\L_{R} k ) \cong k \lra \pi_1 (\L_{\ul k/\Z_p}),\]
thus it suffices to show that $\beta_k$ is injective (equivalently, nonzero). To this end, we may replace $\Z_p$ by $W(k)$ (as the $p$-complete cotangent complex $\widehat \L_{W(k)/\Z_p} = 0$ vanishes) and consider the homotopy pushout square 
\[ 
\begin{tikzcd}
W(k) \arrow[r] \arrow[d] & W(\ul k) \arrow[d]  \\
k \arrow[r] & \ul k
\end{tikzcd}
\]
where $W(\ul k)$ is given by $W(\ul k) = (M_S \xrightarrow{m \mapsto [\alpha(m)]} W(k))$, with $\ul k = (k, M_S \ra S \ra k)$. Now $\beta_k$ is the map given by $\pi_1$ of the natural map 
\[ \L_{k/W(k)} \cong k[1] \lra \L_{\ul k / W(k)}\]
(coming from the transitivity triangle from the lower two arrows in the diagram).  If we compose this map with 
$\L_{\ul k/ W(k)}  \lra \L_{\ul k /W(\ul k)} $ coming from the right vertical arrow, we end up with the natural map 
\[ \L_{k /W( k)} \lra \L_{\ul k/W (\ul k)} \]
which is the morphism coming from base change. This is a quasi-isomorphism by Lemma \ref{lemma:no_additional_log}. In particular, this implies that $\beta_k$, thus $\beta_{S/p} \otimes k$, is injective, and we are done. 
\eproof

\bl
Let $\ul A \ra \ul B, \ul A \ra \ul C$ be maps in $\QRSP^{\textup{prelog}}$,  where $\ul A \ra \ul B$ is a quasisyntomic cover.  Then their pushout exists in $\QRSP^{\textup{prelog}}$ and is a quasisyntomic cover of $\ul C$. In particular, the category $\QRSP^{\textup{prelog}, \opp}$ forms a site with the topology given by quasisyntomic covers. 
\el

\bproof 
Let $\ul D \coloneqq  \ul B \widehat \otimes^\L_{\ul A} \ul C$ be the $p$-completed pushout in the category of pre-log algebras. Then by Lemma \ref{lemma:composition_pushout_QSyn}, $\ul D$ is a quasisyntomic cover in $\QSyn^{\textup{prelog}}$ (it particular it is discrete and has bounded $p^\infty$-torsion). As $\ul D/p = (B/p \otimes_{A/p} C/p, M_B \oplus_{M_A} M_C)$, we see that $\ul D/p$ satisfies condition (3) and (4) in Definition \ref{definition:qrspd_log} by Remark \ref{remark:quotient_semiperfect}, so $\ul D$ lies in $\QRSP^{\textup{prelog}}$.  
\eproof 

The following lemmas will imply that quasiregular semiperfectoid objects form a basis for the quasisyntomic site. 

\bl \label{lemma:cover_QSyn_by_QSRP}
Let $\ul R = (M \xrightarrow{\alpha} R) \in \QSyn^{\textup{prelog}}$ be a quasisyntomic pre-log algebra, then there exists a quasisyntomic cover $\ul R \ra \ul S$ with $\ul S \in \QRSP^{\textup{prelog}}$. One could choose $\ul S = (S, N)$ such that $N$ is $p$-divisible.  
\el 

\bproof 
Choose surjections $\Z_p [X_i]_{i \in I} \twoheadrightarrow R$ and $\N^{\oplus J} \twoheadrightarrow M$. Let $\Z_p\gr{X_i, Y_j}_{I, J}$ be the $p$-adic completion of $\Z_[X_i, Y_j]_{I, J}$ and equip it with a pre-log structure from $\N^{\oplus J}$ by sending $1_j \mapsto Y_j$, where $1_j$ is the generator of the $j^{th}$-copy of $\N$ in $\N^{\oplus J}$. Now consider 
\[ \Big(\Z_p\gr{X_i, Y_j}_{I, J}, \N^{\oplus J}\Big)  \lra (R, M) \]
given by the surjections chosen above. Now we take the $p$-completed base change along the quasisyntomic cover 
\[ \Big(\Z_p\gr{X_i, Y_j}_{I, J}, \N^{\oplus J}\Big) \lra \Big(\mO_C\gr{X_i^{1/p^\infty}, Y_j^{1/p^\infty}}_{I, J}, \N[\frac{1}{p}]^{\oplus J}\Big) \]
and denote this $p$-completed base change by $\ul S = (S, M_S)$. By Lemma \ref{lemma:qsyn_cover_of_qsyn_is_qsyn} and Lemma \ref{lemma:composition_pushout_QSyn} we know that $\ul S$ is quasisyntomic and that $\ul R \ra \ul S$ is a quasisyntomic cover. To finish the proof it suffices to note that $\ul S/p$ is log-semiperfect, as $S$ is a quotient of a perfectoid integral pre-log ring while $M_S$ is a quotient of $\N[\frac{1}{p}]^{\oplus J}$. 
\eproof 

\bl \label{lemma:preserve_Cech}
Let $\ul R \ra \ul S = (S, M_S)$ be a quasisyntomic cover in $\QSyn^{\textup{prelog}}$ with $\ul S \in \QRSP^{\textup{prelog}}$. Then every term in the  $\check{\textup{C}}$ech  nerve of the cover lies in $\QRSP^{\textup{prelog}}$.
\el 

\bproof 
Each term $\ul S^i = (S^i, M^i)$ in the $\check{\textup{C}}$ech nerve  lives in $\QSyn^{\textup{prelog}}$ by Lemma \ref{lemma:composition_pushout_QSyn}. As $M^i$ is a quotient of $ (M_S )^{\oplus (i+1)}$ and $S^i/p$ is a quotient of $(S/p)^{\otimes_{\F_p} (i+1)}$, the pre-log ring $\ul S^i/p $ is log-semiperfect. The lemma thus follows. 
\eproof

\bc 
Let $\mC$ be any presentable $\infty$-category, then the canonical restriction of sites $\QRSP^{\textup{prelog}, \opp} \ra \QSyn^{\textup{prelog}, \opp} $ induces an equivalence of sheaves valued in $\mC$
\[ \textup{Shv}_{\mC} (\QSyn^{\textup{prelog}, \opp}) \isom \textup{Shv}_{\mC} (\QRSP^{\textup{prelog}, \opp}).\]
\ec 

The corollary follows from a general result of this form in \cite[Appendix A.3, in particular Proposition A.3.11]{Mann}. For completeness we give an argument below similar to that of \cite[Proposition 4.31]{BMS2}. 

\bproof
It suffices to prove the universal case where $\mC$ is the $\infty$-category of spaces. We have a natural functor (by restriction)
\[ \iota\colon  \textup{Shv}  (\QSyn^{\textup{prelog}, \opp}) \lra  \textup{Shv} (\QRSP^{\textup{prelog}, \opp}).\]
 To define its inverse, we consider the composition of the above functor with the Yoneda embedding to obtain a functor $h\colon  \QSyn^{\textup{prelog}, \opp} \lra  \textup{Shv} (\QRSP^{\textup{prelog}, \opp})$, sending $A \mapsto h_A$. As in \cite{BMS2}, this functor sends covers to effective epimorphisms (by Lemma \ref{lemma:composition_pushout_QSyn} and Lemma \ref{lemma:cover_QSyn_by_QSRP}) and preserves their $\check{\text{C}}$ech nerves, thus we get a functor 
\[ \rho\colon  \textup{Shv} (\QRSP^{\textup{prelog}, \opp}) \lra  \textup{Shv}  (\QSyn^{\textup{prelog}, \opp})\]
which sends $\mF \mapsto (\rho(\mF)\colon A \mapsto \Hom_{\textup{Shv} (\QRSP^{\textup{prelog}, \opp})} (h_A, \mF))$. In other words, the presheaf $\rho(\mF)$ is a sheaf. To show that $\rho$ is indeed the inverse of $\iota$, one has to check that $\rho \circ \iota$ is the identity functor, which in turn follows from Lemma \ref{lemma:preserve_Cech}. 
\eproof 

\newpage

\section{Derived log prismatic cohomology} \label{sec:derived}

Building on earlier work \cite{Koshikawa}  of the first author, we introduce ``derived'' log prismatic cohomology for ($p$-complete) pre-log rings by deriving from the free case and sheafify in the \'etale topology.  The derived log prismatic cohomology and the cohomology of the structure sheaf on the log prismatic site \cite{Koshikawa} agree in reasonable situations, at least after \'etale sheafification.

\subsection{Derived log prismatic cohomology}
\noindent 

\noindent 
In this subsection we define the derived log prismatic cohomology (compare with \cite[Construction 7.6]{BS} for the nonlog case and \cite[6.8]{Bhatt_dR} for the derived log de Rham cohomology).  

\begin{construction}
Fix a bounded pre-log prism $(A, I, M_A)$. 
The category of pre-log rings over $(A/I, M_A)$ is cocomplete and generated under colimits by pre-log rings of the form of
\[
(A/I [ (X_s)_{ s\in S}, \N^T ], M_A\oplus\N^T)
\]
for finite sets $S,T$, and we will consider the $\infty$-category of simplicial pre-log rings over $(A/I, M_A)$ (or animated pre-log rings as in Remark \ref{animated pre-log rings}). 
We have a functor
\[
(A/I [ (X_s)_{ s\in S}, \N^T ], M_A\oplus\N^T)
\longmapsto
\Prism_{(A/I \langle (X_s)_{ s\in S}, \N^T \rangle, M_A\oplus \N^T)/(A, M_A)}
\]
to the $\infty$-category $\mathcal{D}(A)$ of $A$-modules, using the log prismatic site in \cite{Koshikawa}. In fact, each value is a $(p, I)$-complete commutative algebra in $\mathcal{D}(A)$ and equipped with a $\phi_A$-semilinear map $\phi$. We often regard the functor above as a functor to the $\infty$-category of such objects. 

\begin{remark}
One might think that the construction makes sense only assuming $(A, I)$ is a bounded prism and $(A, M_A)$ is a pre-log ring by using $\Prism_{(A/I \langle (X_s)_{ s\in S}, \N^T \rangle, \N^T)/A}$. However, it is not clear that it is functorial. 
\end{remark}

\begin{definition}
The \emph{derived log prismatic cohomology} is the functor deriving (or animating) the functor above to \emph{all} simplicial pre-log rings over $(A/I, M_A)$. 
For a simplicial pre-log ring $(R, P)$ over $(A/I, M_A)$, we simply write $\Prism_{(R, P)/(A, M_A)}$ with 
\[
\phi\colon \Prism_{(R, P)/(A, M_A)}\to \phi_{A,*}\Prism_{(R, P)/(A, M_A)}
\]
for the derived log prismatic cohomology. We sometimes write $\Prism^\L_{(R, P)/(A, M_A)}$ instead to distinguish it from the one in \cite{Koshikawa}.
We similarly define $\overline{\Prism}_{(R, P)/(A, M_A)}$, $ \overline{\Prism}^\L_{(R, P)/(A, M_A)}$. 
It follows easily that $\overline{\Prism}_{(R, P)/(A, M_A)}\cong \Prism_{(R, P)/(A, M_A)}\otimes^\L_A A/I$. 
\end{definition}
\end{construction}

\begin{remark}
In fact, it suffices to consider the subcategory of derived $p$-complete simplicial pre-log rings $(R, P)$ (i.e., the underlying simplicial ring $R$ is derived $p$-complete), as $\Prism_{(R, P)/(A, M_A)}$ only depends on the derived $p$-completion of $(R, P)$. This  follows from the Hodge--Tate comparison below and the fact that cotangent complex and its exterior powers commute with derived base change on $A$; the completed cotangent complex and its completed exterior powers are invariant under the derived $p$-completion of $(R, P)$. 
Moreover, the derived log prismatic cohomology on derived $p$-complete simplicial pre-log rings can also be obtained as the left Kan extension from $p$-complete pre-log rings of the form $\Sigma_{S, T} \coloneqq (A/I \langle (X_s)_{ s\in S}, \N^T \rangle, M_A \oplus\N^T)$. 
\end{remark}

We can further upgrade $\overline{\Prism}_{(R, P)/(A, M_A)}$ to a filtered object (We follow \cite{BMS2} for the discussion on filtered derived categories, but see also  \cite{GP}, \cite[Appendix D]{APC}).\footnote{Note that the $\infty$-category spanned by $p$-complete (resp. $(p,I)$-complete) objects of $\mD(A/I)$ (resp. $\mD(A)$) is a presentably closed symmetric monoidal stable $\infty$-category with $p$-completed (resp. $(p,I)$-completed) tensor products, hence the discussion there applies.}
For $(R, P)=(A/I [ (X_s)_{ s\in S}, \N^T ], M_A\oplus\N^T)$, we consider the canonical filtration on $\overline{\Prism}_{(R, P)/ (A, M_A)}$. This is a functorial increasing exhaustive multiplicative filtration by $p$-complete objects.
From the Hodge--Tate comparison proved in \cite{Koshikawa}, we have an functorial isomorphism 
\[
H^i (\overline{\Prism}_{(R, P)/(A, M_A)}) \cong 
(\Omega^i_{(R, P)/ (A, M_A)})^{\wedge}\{-i\} \cong 
(\midwedge^i \L_{(R, P)/ (A, M_A)})^{\wedge}\{-i\}. 
\]

Therefore, by deriving this filtration, we obtain

\begin{proposition}[(Derived) Hodge--Tate comparison] \label{prop:derived_HT}
Let $(R, P)$ be a simplicial pre-log ring over $(A/I, M_A)$. There exists  an increasing exhaustive multiplicative filtration 
\[
\Fil_\bullet \overline{\Prism}_{(R, P)/(A, M_A)} \lra \overline{\Prism}_{(R, P)/(A, M_A)}
\] on $\overline{\Prism}_{(R, P)/(A, M_A)}$  by (derived) $p$-complete objects, such that its graded pieces satisfy
\[
\grade_i \overline{\Prism}_{(R, P)/(A, M_A)} \cong (\midwedge^i \L_{(R, P)/ (A/I, M_A)} \{-i\}[-i])^{\wedge}. 
\]
This filtration is called the \emph{conjugate filtration}. 
\end{proposition}

\bc [The smooth case]  \label{cor:derived_vs_non_derived_smooth}
Assume that $M_A\to P$ is a smooth chart of the associated log $p$-adic formal affine scheme $(X, M_X)$ over $(A/I, M_A)$ with $X = \spf R$. 
There is a natural isomorphism 
\[
\Prism^\L_{(R, P)/ (A, M_A)} \to R\Gamma_{\Prism} ((X, M_X)/(A, M_A))
\]
compatible with the Hodge--Tate comparison maps. 
\ec

\begin{proof}
The existence of the map follows from the functoriality of the Hodge--Tate comparison map (see \cite{Koshikawa}). 
As the Hodge--Tate comparison holds for both $\Prism^\L_{(R, P)/(A, M_A)}$ and the log prismaic cohomology of $(X, M_X)$,  it is clear by derived Nakayama that it is an isomorphism. 
\end{proof}

As a consequence of the Hodge--Tate comparison, let us list some functorial properties of the derived prismatic cohomology for pre-log rings.

\bp 
\be
\item For any (derived $p$-complete) simplicial ring $R$ over $A/I$, there is a natural isomorphism
\[\Prism_{R/ A} \cong \Prism_{(R, M_A) / (A, M_A)}. \]
\item The derived log prismatic cohomology on pre-log rings is invariant upon taking the associated log rings. More precisely, let $(R, P)$ be a pre-log ring over $(A/I, M_A)$ with the associated log ring $(R, P)^a$, then we have a natural isomorphism
\[\Prism_{(R, P)/ (A, M_A)} \cong \Prism_{(R, P)^a / (A, M_A)}. \]
\item (Base change) Let $(A, I, M_A)\to (A', IA', M_{A'})$ be a map of bounded pre-log prisms. 
For a simplicial pre-log ring $(R, P)$ over $(A/I, M_A)$, let $(R', P')$ denote the (homotopy) base change to $(A'/IA', M_{A'})$. 
There is a natural isomorphism
\[
\Prism_{(R, P)/(A, M_A)}\widehat{\otimes}^\L_A A' \cong \Prism_{(R', P')/(A', M_{A'})}. 
\]
\item (Multiplicativity)\label{symmetric monoidal}
Let $(R_1, P_1)$, $(R_2, P_2)$ be simplicial pre-log rings  over $(A/I, M_A)$, and let $(R_3, P_3)$ denote the homotopy cofiber product.  
The natural map
\[
\Prism_{(R_1, P_1)/(A, M_A)}\widehat{\otimes}^\L_A \Prism_{(R_2, P_2)/(A, M_A)} \cong
\Prism_{(R_3, P_3)/(A, M_A)}
\]
is an isomorphism and compatible with conjugate filtrations after $\otimes^\L_A A/I$, where the filtration on the left hand side is given by the Day convolution. In fact, $\Prism_{-/(A, M_A)}$ commutes with all colimits. 
\ee

\ep

\bproof 
For (1), note that for any finite set $S$, there is a functorial isomorphism
\[
\Prism_{A/I \langle (X_s)_{ s\in S} \rangle / A} \cong 
\Prism_{(A/I \langle (X_s)_{ s\in S}, M_A \rangle) / (A, M_A)} 
\]
compatible with the Hodge--Tate comparison maps. 
This induces the desired map, which is an isomorphism by the Hodge--Tate comparison and Lemma \ref{lemma:no_additional_log}. 
(2) follows from the Hodge--Tate comparison and Lemma \ref{lemma:pre_log_cotangent_same_as_log}. 
For (3), by the Hodge--Tate comparison and derived Nakayama, it suffices to prove the corresponding statement for the cotangent complexes, which is part (2) of Theorem \ref{thm:functorial_cotangent_complex}.  
 Finally, for part (4), it suffices to consider the case where $(R_i, P_i)$, $i=1, 2$ have the form of 
\[
(A/I [ (X_s)_{ s\in S_i}, \N^{T_i} ], M_A\oplus\N^{T_i})
\]
for finite sets $S_i$, $T_i$; one can then pass to the general case.  
Using the Hodge--Tate comparison once again (and its functoriality), the claim reduces to a standard computation of differential forms (compare with \cite[6.12]{Bhatt_dR}). 
\eproof

Finally, we observe that derived log prismatic cohomology satisfies quasisyntomic descent.  

\begin{proposition}[Quasisyntomic descent]
Let $(A, I, M_A)$ be a bounded pre-log prism. 
On the log quasisyntomic site $\qSyn_{(A/I, M_A)}$, the presheaf
\[
(R, P) \to \Prism_{(R, P)/ (A, M_A)}
\]
is a sheaf. 
If $(A/I, M_A)$ is a perfectoid pre-log ring, the same holds true on $\QSyn_{(A/I, M_A)}$. 
\end{proposition}

\begin{proof}
In both cases,  the $p$-completion of $\midwedge^i \L_{(R, P)/(A/I, M_A)}[-i]$ lies in $\mD^{\geq 0}(R)$ for every $i$. Therefore, by derived Nakayama and the Hodge--Tate comparison, the problem reduces to the quasisyntomic descent for exterior powers of the cotangent complex, which follows from 
Proposition \ref{lem-descent-hlf-hodge-graded-piece}. 
\end{proof}

\subsection{Globalization of derived log prismatic cohomology} \noindent 
\label{ss:global_derived_logprismatic}

\noindent 
Now we work with a log $p$-adic formal scheme rather than a pre-log ring. Let us note that there is a subtle difference between sections of a log structure and its (fixed) chart. 
Let $(X, M_X)$ be a log $p$-adic formal scheme over $(A/I, M_A)$. 
Then we have an assignment on affine objects of the \'etale site $X_{\ett}$:
\begin{equation} \label{eq:derived_prismatic_presheaf}
U=\tu{Spf} (R) \longmapsto \Prism_{(X, M_X)/(A, M_A)}^{\tu{pre}}(U)\coloneqq \Prism^\L_{(R, \Gamma (U, M_X))/ (A, M_A)}. 
\end{equation} 
This presheaf valued on $\mD(A)$ may not be an \'etale sheaf; this is related to why one uses the \'etale topoi to define the cotangent complex for log schemes \cite[8.29]{Olsson_log}.   

\begin{definition}[Derived prismatic cohomology]
Let $(X, M_X)$ be a log $p$-adic formal scheme over $(A/I, M_A)$. 
We define the derived log prismatic cohomology $\Prism^\L_{(X, M_X)/(A, M_A)}$ to be the $(p,I)$-complete \'etale sheafification of the presheaf $\Prism_{(X, M_X)/(A, M_A)}^{\tu{pre}}$. We often denote this \'etale sheaf by $\Prism^\L_{(X, M_X)}$ when the base bounded pre-log prism $(A, I, M_A)$ is understood.  
\end{definition}

\br  \label{remark:global_derived_conjugate_filtration}
Similarly, we may define $\cl \Prism^\L_{(X, M_X)/(A, M_A)}$ (resp. the derived conjugate filtration $\Fil_i \cl \Prism^\L_{(X, M_X)/(A, M_A)}$) as the $p$-complete \'etale sheafification of the presheaf sending 
\[
U = \spf (R) \longmapsto \cl \Prism^\L_{(R, \Gamma (U, M_X))/ (A, M_A)} \quad (\tu{resp. } 
U \longmapsto \Fil_i \cl \Prism^\L_{(R, \Gamma (U, M_X))/ (A, M_A)}).
\]
Note that $\cl \Prism^\L_{(X, M_X)/(A, M_A)}$ may also be obtained as the base change $\Prism^\L_{(X, M_X)/(A, M_A)} \otimes_A^\L A/I$. The (global) derived conjugate filtration has grade pieces 
\begin{align}
    \grade_i \cl \Prism^\L_{(R, \Gamma (U, M_X))/ (A, M_A)})  \isom & \widehat \L \Omega^i_{(X, M_X)/(A/I, M_A)}  \{-i\}[-i] \\
  \nonumber  & \quad \coloneqq   (\midwedge^i \L_{(X, M_X)/(A/I, M_A)})^\wedge \{-i\}[-i]. 
\end{align}
Here $\L_{(X, M_X)/(A/I, M_A)}$ denotes the cotangent complex for log formal schemes\footnote{This depends only on $(\spf A/I, M_A)^a$ and $(X, M_X)$ but the chart $M_A\to A/I$ (or its associated constant sheaf) can be used to compute the cotangent complex by \cite[Theorem 8.20]{Olsson_log}.} and $(-)^\wedge$ denotes the derived $p$-adic completion. This follows from Proposition \ref{prop:derived_HT} and the fact that the $p$-complete \'etale sheafification of the functor 
\[
U=\spf(R) \longmapsto \widehat \L_{(R, \Gamma (U, M_X))/(A/I, M_A)}
\]
is naturally isomorphic to the derived $p$-adic completion $\widehat \L_{(X, M_X)/(A/I, M_A)}$; this can be checked using the canonical free resolutions over $(A/I, M_A)$. 
\er 

\begin{remark}
We take the \'etale sheafification in above constructions mainly because the identity map $\Gamma (U, M_X)\to \Gamma (U, M_X)$ may not be a chart even if there is a chart on $U$, as pointed by Tsuji. If it is a chart, we can control the log cotangent complex without taking sheafification by Remark \ref{remark:log_cotangent_complex_nondiscrete}. Assuming there is a chart on $U$, we can achieve this condition Zariski locally on $U$ if the reduced scheme underlying $X_{A/(p,I)}$ is locally noetherian and $M_X$ is fine by \cite[Lemmas 1.3.2, 1.3.3]{Tsuji_cst} (see also \cite[Lemma A.9]{Koshikawa}), in which case we may consider only such $U$ as they form a basis of \'etale topology of $X$. 
\end{remark}

If $(X, M_X)$ is smooth over $(A/I, M_A)$, we can compare it with the cohomology of the log prismatic site.

\begin{proposition} \label{prop:sections_of_derived_prismatic}
Let $(A, I, M_A)$ be a bounded pre-log prism with $M_A$ being integral and $(X, M_X)$ a log $p$-adic formal scheme which is  smooth over $(A/I, M_A)$. 
\be
\item For any affine object $U=\spf (R)$ of $X_{\ett}$, the natural map
\begin{equation} \label{eq:compare_derived_with_nonderived}
\Prism^\L_{(X, M_X)/(A, M_A)}(U) \to R\Gamma (((U, M_U)/(A, M_A))_{\Prism}, \mO),  
\end{equation}
where $M_U$ is the restriction of $M_X$, is an isomorphism in $\mD(A)$. 
\item If $P\to \Gamma (U, M_X)$ is a smooth chart for the log structure, then the natural map
\begin{equation}
\Prism^\L_{(R, P)/(A, M_A)} \to
\Prism^\L_{(X, M_X)/(A, M_A)}(U)
\end{equation}
is an isomorphism. 
\ee
\end{proposition}

\begin{proof}
The map (\ref{eq:compare_derived_with_nonderived}) comes from the functoriality of the log prismatic site as the association
\[
U\mapsto R\Gamma (((U, M_U)/(A, M_A))_{\Prism}, \mO)
\]
is an \'etale sheaf (see \cite{Koshikawa}).  
Since this map is compatible with the Hodge--Tate comparison,  
by assumption we are reduced to prove the following statement:
the $p$-complete \'etale sheafification of the functor 
\[
U=\spf(R) \mapsto \widehat \L_{(R, \Gamma (U, M_X))/(A/I, M_A)}
\]
is naturally isomorphic to the sheaf $\underline{\Omega}^1_{(X, M_X)/(A/I, M_A)}$. 
As mentioned above, this \'etale sheafification is isomorphic to $\widehat\L_{(X, M_X)/(A/I, M_A)}$, so just apply \cite[Proposition 6.1]{Koshikawa}. 

For part (2), first note that, if $P\to \Gamma (U, M_X)$ is a smooth chart, the natural map
\[
\widehat \L_{(R, P)/(A/I, M_A)} \to \widehat \Omega^1_{(R, P)/(A/I, M_A)}
\]
is an isomorphism (Lemma \ref{lemma:compare_Gabber_with_Olsson}). As $\widehat \Omega^1_{(R, P)/(A/I, M_A)}$ is recovered from the locally free sheaf $\underline{\Omega}^1_{(X, M_X)/(A/I, M_A)}$ of finite rank by taking sections on $U$, the claim follows from the Hodge--Tate comparison.  
\end{proof}

\br \label{remark:comparing_derived_and_nonderived} 
Let $\Prism_{(X, M_X)}= \Prism_{(X, M_X)/(A, M_A)} $ denote the \'etale sheaf $R \nu_* \mO_{\Prism}$ where 
\[ 
\nu\colon \text{Shv} ((X/(A, M_A))_{\Prism}) \ra \text{Shv} (X_{\ett})
\] 
is the map of topoi (\ref{eq:map_of_topoi_from_prismatic}) from the introduction. Then part (1) of Proposition \ref{prop:sections_of_derived_prismatic} says that there is a natural isomorphism of \'etale sheaves 
\[\Prism^\L_{(X, M_X)} \isom \Prism_{(X, M_X)} \] 
so that
\[
R\Gamma_{\Prism}((X, M_X)/(A, M_A))\coloneqq R\Gamma (((X, M_X)/(A, M_A))_{\Prism}, \mO)\cong R\Gamma (X_{\ett}, \Prism^\L_{(X, M_X)}).
\]
In this case (under the assumption of Proposition \ref{prop:sections_of_derived_prismatic}), we also have 
\[
\cl \Prism^\L_{(X, M_X)} \isom \cl \Prism_{(X, M_X)},
\]
where $\cl \Prism_{(X, M_X)} =\cl \Prism_{(X, M_X)/(A, M)} \coloneqq R \nu_* \cl \mO_{\Prism}$. Moreover, under this isomorphism, the derived conjugate filtration $\Fil_i \cl \Prism^\L_{(X, M_X)} $ gets identified with the canonical filtration $\tau_{\le i} \cl \Prism_{(X, M_X)}$.
\er


\subsection{Universal pre-log prisms of semiperfectoid pre-log rings.} \label{ss:universal_object_bdd}
\noindent 

\noindent
Given a semiperfectoid pre-log ring, we discuss an ``initial'' pre-log prism and its relation to derived log prismatic cohomology. 

Let $\ul S = (S, N)$ be a semiperfectoid integral pre-log ring. Let $\ul R \to \ul S$ be a map from a perfectoid integral pre-log ring $\ul R = (R, M)$ that is surjective on rings and surjective modulo invertible elements on monoids. 
Take the exactification 
\[M^{\flat} \to \widetilde{M} \to N\] of the composition $M^{\flat}\to M \to N$ and set
\begin{gather*}
\widetilde{R}\coloneqq R\widehat{\otimes}_{\Z_p} \Z_p \langle \widetilde{M} \rangle, \quad
\Ainf(R, M^{\flat})\coloneqq \Ainf(R)\widehat{\otimes}_{\Z_p} \Z_p \langle M^\flat \rangle, \\
\Ainf(R, \widetilde{M}) \coloneqq \Ainf(R)\widehat{\otimes}_{\Z_p} \Z_p \langle \widetilde{M} \rangle.
\end{gather*}
The ring $\Ainf(R, M^{\flat})$ is perfectoid and $(\Ainf(R, M^{\flat}), (\xi), M^{\flat})$ is a perfect pre-log prism of rank 1. 
We also have a bounded pre-log prism $(\Ainf(R, \widetilde{M}), (\xi), \widetilde{M})$ of rank 1 with a map
\[
(\Ainf(R, M^{\flat}), (\xi), M^{\flat}) \to
(\Ainf(R, \widetilde{M}), (\xi), \widetilde{M}). 
\]
The map of pre-log rings
\[ (\Ainf(R, \widetilde{M}), \widetilde{M}) \to
(S, N) \]  induces an exact surjection on the associated log rings. 
Applying the existence of (nonlog) prismatic envelopes to the surjection $\Ainf(R, \widetilde{M})\to S$, we obtain a pre-log prism 
\[(\Prism^{\init}_{\ul S/ \ul R}, (\xi), M^{\init}_{\ul S/ \ul R}\coloneqq \widetilde{M})\] over $(\Ainf(R, \widetilde{M}), (\xi), \widetilde{M})$ equipped with a map $S\to \Prism^{\init}_{\ul S/ \ul R}/\xi$ and an exact surjection
\[
(\Prism^{\init}_{\ul S/ \ul R}, M^{\init}_{\ul S/ \ul R})^a \to (\Prism^{\init}_{\ul S/ \ul R}/\xi, N\to \Prism^{\init}_{\ul S/ \ul R}/\xi)^a, 
\]
where the superscript $a$ denotes the associated log ring. 
Similarly, we have a map
\[
S \to \Prism^{\init}_{\ul S/ \ul R, \tu{perf}}/\xi\coloneqq(\varinjlim_{\phi} \Prism^{\init}_{\ul S/ \ul R})^{\wedge}_{(p,\xi)}/\xi
\]
with an exact surjection
\[
(\Prism^{\init}_{\ul S/ \ul R, \tu{perf}}, M^{\init}_{(S,N)/(R, M)})^a \to (\Prism^{\init}_{\ul S/ \ul R, \tu{perf}}/\xi, N\to \Prism^{\init}_{\ul S/ \ul R, \tu{perf}}/\xi)^a. 
\]

\begin{proposition}\label{initial object}
For any integral ``log prism'' $(A, I, M_A)$, equipped with a map $S\to A/I$ and an exact surjection
\[
(A, M_A) \to (A/I, N\to S \to A/I)^a,
\]
there is a unique map of pre-log prisms
\[
(\Prism^{\init}_{\ul S/ \ul R}, (\xi), M^{\init}_{\ul S/ \ul R}) \to  (A, I, M_A)
\] 
compatible with exact surjections above.  The same universality with an additional condition that $A$ being perfect (resp. $(A, I, M_A)$ being perfect) holds for the perfection 
\[
(\Prism^{\init}_{\ul S/ \ul R, \tu{perf}}, (\xi), M^{\init}_{\ul S/ \ul R}) \quad
(\textnormal{resp.}~ (\Prism^{\init}_{\ul S/ \ul R, \tu{perf}}, (\xi), M^{\init,\perf}_{\ul S/ \ul R})).
\]
\end{proposition}

\begin{proof}
By the assumption, we have
\[
(A / \xi, M_A )^a \cong (A/ \xi, N \to S \to A/I)^a. 
\]
In particular, we have a map $(S, N)^a \to (A/\xi, M_A)^a$. 
Therefore, we can consider the following composition
\[
(\Ainf(R, M^{\flat}), M^\flat)^a \to (S, N)^a \to (A/\xi, M_A)^a \to (A/(p, \xi), M_A)^a. 
\]
By Lemma \ref{perfectoid pre-log prism is initial}, the composite lifts uniquely to a map
\[
(\Ainf(R, M^{\flat}), (\xi), M^{\flat})^a \to (A, I, M_A)
\]
of pre-log prisms, which extends to $(\Ainf(R, \widetilde{M}), (\xi), \widetilde{M})$ and factors through the prismatic envelope
\[
(\Ainf(R, \widetilde{M}), (\xi), \widetilde{M}) \to
(\Prism^{\init}_{\ul S/ \ul R}, (\xi), M^{\init}_{\ul S/ \ul R}). 
\]
Therefore, we obtain a map
\[
(\Prism^{\init}_{\ul S/ \ul R}, (\xi), M^{\init}_{\ul S/ \ul R}) \to 
(A, I, M_A)
\]
as desired. The claim for the perfection follows immediately. 
\end{proof}

\begin{remark} \label{remark:``initial''_object}
While the notation suggests that $(\Prism^{\init}_{\ul S/ \ul R}, (\xi), M^{\init}_{\ul S/ \ul R})^a$ is ``initial'', this pre-log prism is in general not the initial object in the category 
$\mC_{/(S, N)}^{\tu{str}}$ considered in Remark \ref{remark:variant_category_C_2}. 
The issue is that $(\Prism^{\init}_{\ul S/ \ul R}, (\xi))$ is not necessarily bounded. On the other hand, in the important special case when $\ul S$ is   quasiregular semiperfectoid, the ``log prism'' $(\Prism^{\init}_{\ul S/ \ul R}, (\xi), M^{\init}_{\ul S/ \ul R})^a$ turns out to be literally initial and independent of the choice of $\ul R \ra \ul S$. We will see this from Proposition \ref{prop:compare_derived_prismatic_with_initial} below. 
\end{remark}

\br 
Moreover, the perfect counterpart $(\Prism^{\init}_{\ul S/ \ul R, \tu{perf}}, (\xi), M^{\init,\perf}_{\ul S/ \ul R})^a$ is also initial (in the relevant category of perfect ``log prisms'') as any perfect prism is bounded. In particular, $\Prism_{\ul S/ \ul R, \tu{perf}}$ is independent of the choice of $\ul R$ and the surjection $\ul R \ra \ul S$, and we shall write 
\[
\Prism_{\ul S, \tu{perf}} \coloneqq \Prism_{\ul S/ \ul R, \tu{perf}}
\]
in the rest of the article. 
\er 

If $\ul S = (S, N)$ is a semiperfectoid integral pre-log ring, then the derived log prismatic cohomology 
\[
\Prism_{\ul S/\Ainf(\ul R)} \cong
\Prism_{\ul S/(\Ainf(R, M^{\flat}), M^{\flat})}
\]
is discrete by the Hodge--Tate comparison and the vanishing of $\Omega^1_{\ul S / \ul R}$. Analogously to \cite[Proposition 7.10]{BS}, we can compare it to the ``initial object'' above if $\ul S$ is moreover quasiregular. 

\begin{proposition} \label{prop:compare_derived_prismatic_with_initial}
Assume that $\ul S = (S, N)$ is a quasiregular semiperfectoid integral pre-log ring. Then $\Prism_{\ul S/ \ul R}$ is discrete and naturally has the structure of a $\delta$-ring, and there is an isomorphism of $\delta$-rings
\[
\Prism_{\ul S/\Ainf(\ul R)}  \cong \Prism^{\init}_{\ul S/ \ul R}. 
\]
Moreover, $\Prism^{\init}_{\ul S/ \ul R}$ is bounded and independent of the choice of $\ul R \ra \ul S$. In particular, we obtain an initial ``log prism'' 
\[ 
(\Prism^{\init}_{\ul S}, (\xi), M^{\init}_{\ul S})\coloneqq (\Prism^{\init}_{\ul S/ \ul R}, (\xi), M^{\init}_{\ul S/ \ul R})^a
\]
in the category $\mC_{/(S, N)}^{\tu{str}}$ considered in Remark \ref{remark:variant_category_C_2}, as indicated in Remark \ref{remark:``initial''_object}.
\end{proposition}

\begin{proof} 
By assumption we know that $S$ has bounded $p^{\infty}$-torsion and $\L_{\ul S/ \ul R}[-1]$ is a $p$-completely flat $S$-module. 
By the Hodge--Tate comparison, $\overline{\Prism}_{\ul S/ (\Ainf(\ul R)}$ is discrete and $p$-completely flat. 
We have isomorphisms
\begin{align*}
\Prism_{\ul S/ (\Ainf(\ul R))} 
&\cong \Prism_{\ul S/(\Ainf(R, M^{\flat}), M^{\flat})} \\
&\cong \Prism_{\ul S/(\Ainf(R, \widetilde{M}), \widetilde{M})} \\
&\cong \Prism_{S/\Ainf(R, \widetilde{M})}. 
\end{align*}
Applying \cite[Lemma 7.7]{BS} to $\Prism_{S/\Ainf(R, \widetilde{M})}$, we obtain the structure of a $\delta$-ring on it, and know that the bounded prism $(\Prism_{S/\Ainf(R, \widetilde{M})}, (\xi))$ is weakly initial in the unbounded prismatic site $(S/ \Ainf(R, \widetilde{M}))_{\Prism}^{\textnormal{unbdd}}$ whose objects are not assumed to bounded. 
Given Proposition \ref{initial object}, we need only show that $(\Prism_{S/\Ainf(R, \widetilde{M})}, (\xi))$ is actually initial, i.e., isomorphic to the prismatic envelope of the surjection $\Ainf(R, \widetilde{M}) \to S$, which follows from Lemma \ref{sufficient condition for being initial}. 
\end{proof}

We used the following version of \cite[Proposition 7.10]{BS}:

\begin{lemma}\label{sufficient condition for being initial}
Let $(A, I)$ be a bounded prism and $A/I \to S$ a surjection to a ring with bounded $p^{\infty}$-torsion. 
If $\L_{S/ (A/I)}[-1]$ is $p$-completely flat $S$-module, then $(\Prism_{S/A}, I\Prism_{S/A})$ defines a prism and it is initial in the unbounded prismatic site $(S/A)_{\Prism}^{\textnormal{unbdd}}$. 
\end{lemma}

This is also contained in \cite[Remark 4.3.9, Proposition 4.3.10]{APC} (with the same proof), and is strengthen in \cite[Corollary 4.3.14]{APC} by dropping the surjectivity assumption. 

\begin{proof}
The argument is the same as in \cite[Proposition 7.10]{BS}: using \cite[Lemma 7.7]{BS}, one replaces $S$ by a ring of the form of
\[
A/I \langle X_1^{1/p^{\infty}}, \dots, X_r^{1/p^{\infty}}\rangle / (X_1 - f_1, \dots, X_r - f_r). 
\]
Then replace $A$ by $A\langle X_1^{1/p^{\infty}}, \dots, X_r^{1/p^{\infty}}\rangle$ with $\delta (X_i^{1/p^k})=0$ for all $i, k$. The resulting case is settled in \cite[Example 7.9]{BS}. 
\end{proof}

For a general semiperfectoid integral pre-log ring $\ul S$, we can conclude the same after taking the perfection.  
\begin{proposition}\label{perfect log prismatic cohomology is initial}
Assume that $\ul S = (S, N)$ is a semiperfectoid integral pre-log ring. Then $\Prism_{\ul S/ \ul R, \perf}$ is discrete and naturally has the structure of a perfect prism, and there is an isomorphism of perfect $\delta$-rings
\[
\Prism_{\ul S/\Ainf(\ul R), \perf}  \cong \Prism^{\init}_{\ul S, \perf}. 
\]
\end{proposition}

\begin{proof}
It follows from the assumption and the proof of \cite[Lemma 8.4]{BS} that $\Prism_{\ul S/ \ul R, \perf}\otimes^\L_A A/I$ is discrete. 
So, by arguing as before, it remains to check the following variant of Lemma \ref{sufficient condition for being initial}. 
\end{proof}

\begin{lemma}
Let $(A, I)$ be a bounded prism and $S$ a derived $p$-complete $A/I$-algebra. 
If $\Prism_{S/A, \perf}\otimes^\L_A A/I$ is discrete, then $(\Prism_{S/A, \perf}, I\Prism_{S/A, \perf})$ defines a perfect prism and it is initial in the perfect prismatic site $(S/A)_{\Prism}^{\perf}$. 
\end{lemma}

\begin{proof}
The proof of \cite[Lemma 7.7, Lemma 8.4]{BS} still applies to this setting and we see that the reduction $\Prism_{S/A, \perf}\otimes^\L_A A/I$ is discrete and $(\Prism_{S/A, \perf}, I\Prism_{S/A, \perf})$ is a perfect prism. 

If $(A, I)$ is a perfect prism, then the reduction is the perfectoidization $S_{\textnormal{perfd}}$ of $S$ by definition. 
It follows from \cite[Corollary 8.14]{BS} that the perfectoid ring $S_{\textnormal{perfd}}$ is initial in the perfect prismatic site $(S/A)_{\Prism}^{\perf}$. 

In general, let $(A_{\perf}, IA_{\perf})$ denote the perfection of $(A, I)$. 
There are identifications
\begin{align*}
\Prism_{S/A, \perf} 
&\cong
\bigl(\varinjlim A_{\perf}\widehat{\otimes}^\L_{A} \Prism_{S/A}\bigr)^{\wedge} \\
&\cong 
\Prism_{(A_{\perf}\widehat{\otimes}^\L_A S) / A_{\perf}, \perf} \\
&\cong
\Prism_{\pi_0(A_{\perf}\widehat{\otimes}^\L_A S) / A_{\perf}, \perf}, 
\end{align*}
where the last one comes from \cite[Proposition 8.5]{BS}. 
So, it is initial in the perfect prismatic site $(\pi_0(A_{\perf}\widehat{\otimes}^\L_A S)/A_{\perf})_{\Prism}^{\perf}$. 
This clearly implies that it is also initial in the perfect prismatic site $(S/A)_{\Prism}^{\perf}$. 
\end{proof}

\bc \label{cor:p_saturation_induces_iso_on_perfect_prism}
Assume that $\ul S = (S, N)$ is a semiperfectoid integral pre-log ring and $N$ is semiperfect in the sense that $N^\flat \to N$ is surjective. 
Let $\ul S^{p\text{-sat}}$ denote the (classically) $p$-completed $p$-saturation of $\ul S=(S, N)$. 
Then, $\ul S^{p\text{-sat}}$ is semiperfectoid and the natural map
\[
\Prism_{\ul S/ \ul R, \perf} \to \Prism_{\ul S^{p\text{-sat}}/\ul R, \perf}
\]
is an isomorphism. 
\ec

\bproof
It is easy to see that $p$-saturation $N^{p\text{-sat}}$ of $N$ is again semiperfect. Therefore,
\[
S/p \otimes_{\F_p [N]}\F_p [N^{p\text{-sat}}]
\]
is semiperfect and $\ul S^{p\text{-sat}}$ is semiperfectoid. 
By Proposition \ref{perfect log prismatic cohomology is initial}, $(\Prism_{\ul S/ \ul R, \perf}/(\xi), M_{\ul S/ \ul R}^{\init,\perf})^a$ is the universal perfectoid log ring over $\ul S$. Recall that the monoid of any perfectoid log ring is $p$-saturated. In particular, there is a unique map
\[
\ul S^{p\text{-sat}} \ra (\Prism_{\ul S/ \ul R, \perf}/(\xi), M_{\ul S/ \ul R}^{\init,\perf})^a, \]
and this map exhibits the target as the universal perfectoid log ring over $\ul S^{p\text{-sat}}$. Therefore, again by Proposition \ref{perfect log prismatic cohomology is initial}, we see that
\[
\Prism_{\ul S/ \ul R, \perf} \cong \Prism_{\ul S^{p\text{-sat}}/\ul R, \perf}.
\]
\eproof


\newpage

\section{Nygaard filtrations and de Rham comparison} \label{sec:Nygaard}

In this section, we construct a Nygaard filtration for free pre-log rings and then get a \emph{derived} Nygaard filtration for arbitrary pre-log rings over a fixed base. The main theorem (in the local setting) in this section is the following. 
\bt[Derived Nygaard filtration] \label{derived_nygaard_HT}
Let $(A, I, M_A)$ be an integral bounded pre-log prism and $(R, P)$ a simplicial pre-log ring over $(A/I, M_A)$. 
There is a decreasing multiplicative filtration $\Fil_N$ by $(p, I)$-complete objects on $\Prism_{(R, P)/(A, M_A)}^{(1)}$ equipped with an isomorphism
\[
\grade^i_N \Prism_{(R, P)/(A, M_A)}^{(1)}\cong \Fil_i \overline{\Prism}_{(R, P)/(A, M_A)}\{i\}.
\]
We call it the \emph{derived Nygaard filtration}. 
\et
As a corollary, we prove the de Rham comparison for derived log prismatic cohomology (see Corollary \ref{cor:dR_comparison}).

\subsection{A toy example}  \label{ss:toy_example}
\noindent 

\noindent Our strategy is similar to that of \cite{BS}, but with a key difference, which arises because of the exactification procedure in the process of forming log prismatic envelops. This introduces non-perfect base prisms into the picture. 
As a result,  the derived Nygaard filtration may not be discrete even for ``large'' pre-log rings (in contrast with \cite[Proposition 12.11]{BS}), which  makes our definition less canonical compared to \cite[Theorem 15.3]{BS}. 
Let us illustrate the difference with an example. 
 
\beg \label{example:log_Nygaard}
Fix $(\Ainf, I) = (\Ainf, (\xi))$ as our base prism with the trivial pre-log structure. Let $\ul R = (\mO_C \gr{x}, \N)$ and let $ \ul S \coloneqq (\mO_C \gr{x^{1/p^\infty}}, \N[\frac{1}{p}])$. Following \cite[Section 12]{BS}, a natural strategy to endow the Nygaard filtration on $\Prism_{\ul{R}/\Ainf}$ is to consider the \v{C}ech nerve $\ul S^\bullet$ of the quasisyntomic cover $\ul R \ra \ul S$: one first constructs a suitable Nygaard filtration on each $\Prism_{\ul S^i/\Ainf}$ and then take the totalization. For $i = 1$, we have the following pre-log algebra 
\[ \ul S^1 = \big(\mO_C \gr{x^{1/p^\infty}, y^{1/p^\infty}}/(x - y), \textstyle \N[\frac{1}{p}] \oplus_{\N} \N[\frac{1}{p}]   \big). \]
which is log quasiregular semiperfectoid. To ease notation, let us write $\ul B = (B, N)$ for $\ul S^1$.  To compute the derived log prismatic cohomology let us start with the following surjection from a perfectoid pre-log ring 
\begin{equation} \label{eq:surjection_iota} \iota\colon  \ul{B}_\infty \coloneqq (\mO_C\gr{x^{1/p^\infty}, y^{1/p^\infty}}, \textstyle \N[\frac{1}{p}]^{\oplus 2}) \lra (B, N).
\end{equation}
Let $\ul{\sq B} = (\sq B, \sq N)$ be the exactification of the surjection $\iota$ in (\ref{eq:surjection_iota}), which takes the form 
\bi
\item $\sq B = \mO_C \gr{x^{1/p^\infty}, y^{1/p^\infty}, x/y, y/x}$ 
\item $\sq N = \{(a, b) \in \Z[\frac{1}{p}]^{\oplus 2} \: | \:   a + b \in  \N [\frac{1}{p}]\}$ 
\ei
and $\iota$ factors as 
\[ \ul B_\infty \lra \ul{\sq B} \lra \ul B \]
where the first map is log \'etale and the second map is an exact surjection of pre-log rings. 
Let   $(\Ainf(B_\infty), I, N_\infty = \N[\frac{1}{p}]^{\oplus 2})$ be the (perfect) pre-log prism of rank 1 coming from $\Ainf (\ul B_\infty)$, and let 
\[(\sq A, I, \sq N)\] be the pre-log prism of rank 1 where 
\bi 
\item $\sq A = \Ainf (B_\infty) \widehat \otimes_{\Z_p \gr{N_\infty}} \Z_p \gr{\sq N} \cong \Ainf \gr{X^{1/p^\infty}, Y^{1/p^\infty}, X/Y, Y/X}.$
\item $\delta (X/Y) = \delta (Y/X) = 0$. 
\ei 
From Hodge--Tate comparison (see also Subsection \ref{ss:universal_object_bdd}) 
we know that 
\begin{align*}
 \Prism_{\ul B/\Ainf} & \isom \Prism_{\ul B/ (\Ainf(B_\infty), N_\infty) } 
 \isom  \Prism_{\ul B/ (\sq A, \sq N) } \isom  \Prism_{B/ \sq A } 
\end{align*}
where the last isomorphism comes from the exactness of the map $\ul{\sq B} \ra \ul B$. Again by the Hodge--Tate comparison, we see that $\Prism_{\ul B/\Ainf} \cong \Prism_{B/\sq A}$ is discrete. It is therefore tempting to consider a naive Nygaard-like filtration 
\[ \sq \Fil^i \coloneqq \{x \in \Prism_{\ul B/\Ainf} | \phi(x) \in I^i  \Prism_{\ul B/\Ainf} \}.
\] However, this filtration is \textit{not} the one that we are after --- it will not satisfy the analogue of \cite[Theorem 12.2]{BS}. To see this, let us observe that $\Prism_{B/\sq A}$ can be computed by the prismatic envelope of the map 
\[\Ainf \gr{X^{1/p^\infty}, Y^{1/p^\infty}, X/Y, Y/X} \ra \mO_C \gr{x^{1/p^\infty}, y^{1/p^\infty}}/(x - y)\]
which sends $X^{1/p^m} \mapsto x^{1/p^m}, Y^{1/p^m} \mapsto y^{1/p^m},$ and $X/Y \mapsto 1$. This is isomorphic to the prismatic envelope of 
\[\Ainf \gr{u, v, X^{1/p^\infty}, (uX)^{1/p^\infty}}/(uv - 1) \lra \mO_C\gr{u, v, x^{1/p^\infty}, (ux)^{1/p^\infty}}/(uv - 1, u - 1).\]
Therefore, we have (up to completion)
\[ \Prism_{B/\sq A} \cong \frac{\Ainf \gr{u, v, X^{1/p^\infty}, (uX)^{1/p^\infty}}}{(uv - 1)} \Big\{\frac{(\xi, u-1)}{\xi}\Big\}^{\delta}.\]
Now the element $\frac{u-1}{\xi}$ lies in $\Fil_1 \cl \Prism_{B/\sq A}$ but not in the image of the map 
$\textstyle{``\frac{\phi}{\xi}"} \colon \sq \Fil^1 \ra \cl \Prism_{B/\sq A}$ as $u$ does not admit $p$-power roots in $\sq A$.  

The issue here is that the prism $(\sq A, I)$ is not perfect. In fact, one should always consider Nygaard filtrations on Frobenius twisted prismatic cohomology $\Prism_{-/A}^{(1)}$ (i.e., base changed along the Frobenius of the base prism). When the base prism $(A, I)$ is perfect, we have $\Prism_{-/A}^{(1)} \cong \Prism_{-/A}$ and we can view $\Fil^\bullet_{N}$ as a filtration on $\Prism_{-/A}$ (for example, in \cite[Definition 12.1]{BS}). In the example above, if we consider the filtration 
\[\Fil^i \coloneqq \{x \in    \Prism_{B/\sq A}^{(1)}  \: | \: \phi (x) \in I^i \Prism_{B/\sq A} \} \]
on $ \Prism_{B/\sq A}^{(1)} = \Prism_{B/\sq A} \widehat \otimes_{\sq A, \phi} \sq A$,\footnote{
Note that while $\Prism_{B/\sq A} \cong \Prism_{\ul B/\Ainf}$, the Frobenius twisted prismatic cohomology $\Prism_{B/\sq A}^{(1)}$ \textit{does not only} depend on $\ul B$ -- it also depends on $\sq A$, which comes from a choice of the surjection $\ul B_\infty \ra \ul B$.
} 
then the analogue of \cite[Theorem 12.2]{BS} continues to hold, in other words, we have 
\[ \textstyle  \frac{\phi}{\xi^i}\colon \grade^i \Prism_{B/\sq A}^{(1)} \isom \Fil_i \cl \Prism_{B/\sq A}.\]
Therefore, to equip the Nygaard filtration on the log prismatic cohomology $\Prism_{\ul R/\Ainf}$ (or rather, its isomorphic copy $\Prism_{\ul R/\Ainf}^{(1)}$) from the beginning of this example, we need to rewrite it as the totalization of the log prismatic cohomology of not just the \v{C}ech nerve $\ul S^\bullet$ above, but certain Frobenius twisted variants of $\ul S^\bullet$. 
This is possible essentially by the result in \cite[Appendix B]{Koshikawa} (see also \cite[Lemma 5.4]{BS}), which we explain in detail later in this section.  
\eeg

\subsection{A non-log case: lci quotients over non-perfect bases} \label{ss:lci_quotients_nonperfect}
\noindent 

\noindent 
As illustrated in Example \ref{example:log_Nygaard}, we need to slightly extend the discussion of \cite[Section 12]{BS} to the case of certain lci quotients over non-perfect bases. 

Fix a bounded prism $(A, I)$ and consider 
\[A_0\coloneqq A \langle X_1, \dots X_d \rangle\] as a $\delta$-ring over $A$ with $\delta (X_i)=0$. The pair $(A_0, IA_0)$ gives rise to a bounded prism over $(A, I)$, which will be the base prism in this discussion. In other words, we will consider the (derived) prismatic cohomology $\Prism_{-/ A_0}$ relative to $A_0$ (not $A$). 

Let $R$ be an $A_0$-algebra of the form 
\[ R= A \langle X_1, \dots X_d \rangle /(I, f_1, \dots, f_r)\] with $f_1, \dots, f_r$ giving a $p$-completely regular sequence on $A /I \langle X_1, \dots, X_d \rangle$ relative to $A/I$. 
By \cite[Example 7.9]{BS},\footnote{Our setting is not covered actually as the sequence $f_1, \dots, f_r$ may not be Koszul-regular, but their argument still applies.} $\Prism_{R/ A_0}$ is identified with the prismatic envelope 
\[A_0\{\tfrac{(I, f_1, \dots , f_r)}{I}\}^{\delta}, \]
which is discrete, concentrated in degree $0$, and $(p,I)$-completely flat over $A$.  Recall that 
\[
\Prism_{R/A_0}^{(1)}\coloneqq \Prism_{R/A_0}\widehat{\otimes}^\L_{A_0, \phi } A_0
\cong \Prism_{R^{(1)} / A_0}
\]
denotes the base change of $\Prism_{R/A_0}$ along the Frobenius on $A_0$. Here we write $R^{(1)} \coloneqq R\widehat{\otimes}^\L_{A_{0},\phi}A_0$.  

\begin{assumption} \label{assumption:p_I_flat}
In this section, we further assume that 
$\Prism_{R/A_0}^{(1)}$ is $(p, I)$-completely flat. 
\end{assumption}

\br 
For instance, Assumption \ref{assumption:p_I_flat} is satisfied if $(A, \phi (I)A)$ is also a bounded prism (e.g., the universal oriented prism) and 
\[\phi_{A_0} (f_1), \dots \phi_A (f_r)\] forms a $p$-completely regular sequence in $A_0/\phi(I)A_0$ relative to $A/\phi(I)A$. An important example is $f_i=X_i-1$ for all $i$, which will be used later. 
\er 

\bd \label{definition:Nygaard_filtration}
We define the \emph{Nygaard filtration} $\Fil_N^\bullet$ on $\Prism_{R/A_0}^{(1)}$ by the formula
\begin{equation} \label{eq:Nygaard_definition}
\Fil_N^i \Prism_{R/A_0}^{(1)}=\{x\in \Prism_{R/A_0}^{(1)} \mid \phi_{R/A_0} (x) \in I^i \Prism_{R/A_0} \}, 
\end{equation}
where  $\phi_{R/A_0}\colon \Prism_{R/A_0}^{(1)} \to \Prism_{R/A_0}$ is the relative Frobenius (relative to $A_0$). 
The formation of Nygaard filtration commutes with $(p, I)$-completely flat base changes on $A$. 
\ed

\begin{notation}
Let $\Fil_\bullet$ denote the conjugate filtration of $\overline{\Prism}_{R/A_0}$. 
It commutes with $(p, I)$-completely flat base change on $A_0$. 
Moreover, $\Fil_i \overline{\Prism}_{R/A_0}$ is a $p$-completely flat $R$-module for every $i$. 
\end{notation}

Now we state a variant of \cite[Theorem 12.2, Theorem 15.2]{BS} (which treats quasiregular semiperfectoid rings and \emph{large} quasisyntomic $A'_0/IA'_0$-algebras).

\bp\label{lci case} 
Keep notations from above. Let $A_0 \to A'_0$ be a $(p, I)$-completely \emph{flat} map of bounded prisms. We regard $(A'_0, IA'_0)$ as a bounded prism over $(A, I)$ and let $R'$ denote $R\widehat{\otimes}_A A'_0$. Then the Nygaard filtration\footnote{Note that, the derived prismatic cohomology 
$\Prism_{R'/A'_0}$ and its base change $\Prism_{R'/A'_0}^{(1)}\coloneqq \Prism_{R'/A'_0}\widehat{\otimes}^\L_{A'_0}A'_0$ are both discrete. In particular, the definition of Nygaard filtration $\Fil^\bullet_N$ makes sense for $\Prism_{R'/A'_0}^{(1)}$.  }
 $\Fil_N^\bullet$ on $\Prism_{R'/A_0'}^{(1)}$ agrees with the derived Nygaard filtration from \cite[Section 12]{BS}, and  the relative Frobenius $\phi_{R'/A'_0}$ induces an isomorphism
\[
\grade_N^i \Prism_{R'/A'_0}^{(1)} \cong \Fil_i \overline{\Prism}_{R'/A'_0}\{i\}. 
\] 
where $\{i\}$ denotes the Breuil--Kisin twist.
\ep

\begin{proof} It suffices to prove the second claim, which we argue by flat descent. The argument is in spirit close to that of \cite[Theorem 15.2]{BS}. 

\noindent \textit{(1)}. First we consider the case where 
\begin{equation*}
A'_0=A_{\infty}\coloneqq A_0\langle X_1^{1/p^{\infty}}, \dots, X_d^{1/p^{\infty}}\rangle
\end{equation*}
with $\delta (X_i^{1/p^m})=0$ for all $m\geq 0$ and $i$, and write $R_{\infty}$ for $R'$. 
In this case, there are compatible identifications
\begin{equation*}
\Prism_{R_{\infty}/A_{\infty}}\cong \Prism_{R_{\infty}/A}, \quad
\Prism_{R_{\infty}/A_{\infty}}^{(1)}\cong \Prism_{R_{\infty}/A}\widehat{\otimes}^\L_{A,\phi} A
\end{equation*}
as \emph{$A$-modules}. 
Note that $R_{\infty}$ is a large quasisyntomic $A/I$-algebra in the sense of \cite[Definition 15.1]{BS}. Our Nygaard filtration on $\Prism_{R_{\infty}/A_{\infty}}^{(1)}$ and the one in \cite[Theorem 15.2]{BS} agree under this identification, thus, by \emph{loc.cit.}, we have an isomorphism
\begin{equation*}
\grade_N^i \Prism_{R_{\infty}/A_{\infty}}^{(1)} \cong \Fil_i \overline{\Prism}_{R_{\infty}/A_{\infty}}\{i\}. 
\end{equation*}

\vspace*{0.1in}
\noindent \textit{(2)}. Now suppose that $A'_0$ lives over some $A_{\infty}$ from the preceding paragraph. By the discussion above,  the derived base change of the Nygaard filtration of $\Prism_{R_{\infty}/A_{\infty}}^{(1)}$ along $A_{\infty}\to A'_0$ is discrete and gives rise to a filtration whose graded pieces are $\Fil_i \overline{\Prism}_{R'/A'_0}\{i\}$ (argue by induction on $i$.) Thus this filtration has to agree with the Nygaard filtration $\Fil_N^\bullet$ of $\Prism_{R'/A'_0}^{(1)}$ (again by induction on $i$). This also implies that
\begin{equation*}
\grade_N^i \Prism_{R'/ A'_0}^{(1)} \cong \Fil_i \overline{\Prism}_{R'/ A'_0}\{i\} 
\end{equation*}
by base change. (Compare with the beginning of the proof of \cite[Theorem 15.2]{BS}.)

\vspace*{0.2in}
\noindent \textit{(3)}.
Next, we consider the case $A'_0=A_0$. We will show the following claims simultaneously
\bi
\item The formation of $\Fil^{i}_N$  commutes with (completed) derived base change along the map $A_0\to A_{\infty}$. 
\item the image of the relative Frobenius
\begin{equation} \label{eq:relative_frob_on_Nygaard}
\phi_{R/A_0}\colon \Fil^i_N \Prism_{R/A_0}^{(1)} \to \overline{\Prism}_{R/ A_0}\{i\}
\end{equation}
is $\Fil_i \overline{\Prism}_{R/A_0} \{i\}$. 
\ei 
We proceed by induction on $i$. For $i=0$, $\Fil^{0}_N \Prism_{R/A_0}^{(1)}$ is just $\Prism_{R/A_0}^{(1)}$, which commutes with derived base change. We will show that, if $\Fil^i_N$ commutes with derived base change, then 
\bi
\item The relative Frobenius map in (\ref{eq:relative_frob_on_Nygaard}) factors through 
\begin{equation} \label{eq:relative_frob_on_Nygaard_2}
\phi_{R/A_0}\colon \Fil^i_N \Prism_{R/A_0}^{(1)} \to \Fil_i \overline{\Prism}_{R/ A_0}\{i\}
\end{equation} and is surjective, 
\item $\Fil^{i+1}_N$ commutes with derived base change along $A_0 \ra A_\infty$. 
\ei
Consider $A_0 \ra A_\infty$ where $A_\infty = A_0 \gr{X_1^{1/p^\infty}, ..., X_d^{1/p^\infty}}$.  As both $\Fil^{i}_N$ (by inductive hypothesis) and $\Fil_\bullet $ commute with base change, we get a map
$ \phi_{R/A_0}\colon \Fil^{i}_N \Prism_{R/A_0}^{(1)} \to \Fil_{i} \overline{\Prism}_{R/ A_0}\{ i\} $
by flat descent for the conjugate filtration.  This map is surjective as the corresponding map is surjective for $A_{\infty}$ by the discussion in step (1). (The surjectivity of a map between $(p, I)$-derived complete modules can be checked modulo $(p, I)$ and hence after any $(p, I)$-completed faithfully flat base change.) Now, since the map $\phi_{R/A_0}$ in (\ref{eq:relative_frob_on_Nygaard_2}) is surjective,  we conclude that $\Fil_N^{i+1} = \ker \phi_{R/A_0}$ commutes with derived base change since the conjugate filtration commutes with derived base change.

\vspace*{0.2in}
\noindent \textit{(4)}.
Finally, for a general $A'_0$, we argue as in the preceding paragraph and use assertions from part (2).  
This finishes the proof of the proposition. 
\end{proof}

\br
Let $(A, I) \to (B, IB)$ a map of bounded prisms. 
Let $B_0$ (resp. $S$) denote the base change of $A_0$ (resp. $R$) to $B$. Then, $f_1, \dots, f_r$ is a $p$-completely regular sequence of $B_0$ relative to $B / IB$ and 
\[
\Prism_{S/ B_0}^{(1)} \cong \Prism_{R/A_0}^{(1)} \widehat{\otimes}^\L_{A_0} B_0 \cong 
\Prism_{R/A_0}^{(1)} \widehat{\otimes}^\L_{A} B
\]
is $(p, I)$-completely flat. Using Proposition \ref{lci case}, one can see that the above identification
is compatible with Nygaard filtrations, i.e., the Nygaard filtration commutes with base change $A\to B$. 
(Note that the conjugate filtration commutes with base change $A\to B$.)
\er

\subsection{Extracting $p$-power roots} \label{ss:extracting_p-power_roots}
\noindent
We apply the results of the previous subsection to derived log prismatic cohomology of a certain map arising from the perfection of free monoids. 

\begin{construction}[Extracting $p$-power roots]
Fix a bounded pre-log prism $(A, I, M_A)$ and assume $M_A$ is integral. 
Let $M_A\oplus N \to M_A\oplus M$ be a surjective map of finitely generated free monoids over $M_A$ and let $M_A\oplus N\to \widetilde{N}\to M_A\oplus M$ be its exactification. 
The kernel 
\[G = \ker (M_A^{\tu{gp}}\oplus N^{\tu{gp}}=\widetilde{N}^{\tu{gp}}\to M_A^{\tu{gp}}\oplus M^{\tu{gp}})\] is a finitely generated free abelian group. 
As $M$ is free, the surjection $\widetilde{N}\to M_A \oplus M$ admits a section and it induces a decomposition $\widetilde{N}\cong M_A\oplus M \oplus G$. Set $N_{\infty}\coloneqq N[1/p]$ and define integral monoids $\widetilde{N}_{\infty}$ and $M_{\infty}$ by the pushout diagram 
\[
\begin{tikzcd}
M_A\oplus N \arrow[r] \arrow[d] & \sq N \arrow[d] \arrow[r] & M_A\oplus M \arrow[d] \\
M_A\oplus N_\infty \arrow[r] & \sq N_\infty  \arrow[r] & M_\infty, 
\end{tikzcd} \quad \quad
\]
where pushouts are discrete. (Beware potential confusing notation.)
Note that the factorization $M_A\oplus N_{\infty}\to \widetilde{N}_{\infty}\to M_{\infty}$ is the exactification, and there is a natural isomorphism $\widetilde{N}_{\infty}/ G \cong M_{\infty}$. 
Now we consider a pre-log ring
\[
(A/I)_{\infty}\coloneqq(A/I\widehat{\otimes}_{\Z_p \langle M_A\rangle}\Z_p \langle M_{\infty}\rangle, M_{\infty})
\]
that is quasisyntomic over $(A/I, M_A)$; the $p$-completed tensor product is the same as the $p$-completed derived tensor product as $M_A \to M_{\infty}$ is injective and integral. Let us also define a $\delta_{\log}$-ring
\begin{equation*}
\widetilde{A}_{\infty}\coloneqq A\widehat{\otimes}^{\L}_{\Z_p\langle M_A\rangle} \Z_p\langle \widetilde{N}_{\infty}\rangle \cong A\langle M\oplus G\rangle\widehat{\otimes}^{\L}_{\Z_p\langle N\rangle}\Z_p\langle N_{\infty}\rangle, \quad
\widetilde{N}_{\infty}\ra \widetilde{A}_{\infty},
\end{equation*}
where the $(p,I)$-completed tensor products are again discrete as $M_A \to \widetilde{N}$ is flat and the $\delta_{\log}$-structure is defined by $\delta_{\log} (N_{\infty})=0$; this extends to $\widetilde{N}_{\infty}\to \widetilde{A}_{\infty}$ by \cite[Proposition 2.16]{Koshikawa}. 
Using the Hodge--Tate comparison, we then have isomorphisms
\begin{align} \label{eq:Nygaard_log_and_nonlog}
\Prism_{((A/I)_{\infty}, M_{\infty})/ (A, M_A)} 
&\cong \Prism_{((A/I)_{\infty}, M_{\infty})/ (A\langle N_{\infty}\rangle, M_A\oplus N_{\infty})} \\ 
\nonumber &\cong \Prism_{((A/I)_{\infty}, M_{\infty})/ (\widetilde{A}_{\infty}, \widetilde{N}_{\infty})} \\
\nonumber
&\cong \Prism_{(A/I)_{\infty}/ \widetilde{A}_{\infty}}
\cong \Prism_{A/I\langle M \rangle/ A\langle M\oplus G\rangle }\widehat{\otimes}^{\L}_{A\langle M\oplus G\rangle} \widetilde{A}_{\infty}. 
\end{align} 
In particular, these cohomology are discrete, concentrated in degree $0$, and $(p,I)$-completely flat over $A$ since $G\subset A/I\langle M\oplus G\rangle$ gives rise to a $p$-completely regular sequence relative to $A/I$. 
\end{construction}

\bl \label{semiperfectoid case}
Let $\Fil_N^\bullet $ be the Nygaard filtration defined on 
\[ \Prism_{ (A/I)_{\infty} / \sq A_\infty}^{(1)} \cong \Prism_{(A/I)_{\infty}, M_{\infty})/(A, M_A)}\widehat{\otimes}^\L_{\widetilde{A}_{\infty}, \phi}\widetilde{A}_{\infty},\]
which is discrete, by Formula (\ref{eq:Nygaard_definition}). Then the relative Frobenius induces an isomorphism
\[
\grade_N^i (\Prism_{((A/I)_{\infty}, M_{\infty})/(A, M_A)}\widehat{\otimes}^\L_{\widetilde{A}_{\infty}, \phi}\widetilde{A}_{\infty})
\cong \Fil_i \overline{\Prism}_{((A/I)_{\infty}, M_{\infty})/(A, M_A)}\{i\}.
\]
\el

\bproof 
After choosing a basis we have $G \cong \Z^{c}$ and a map $\N^{c} \hookrightarrow  G$. This allows us to write $A/I \gr{M}$ as a quotient of the (completed) polynomial ring $A \gr{M, x_1, ..., x_c}$ via the surjection 
\begin{equation} \label{eq:presenting_AM_as_lci} 
A \gr{\N^{c} \oplus M} \hookrightarrow A \gr{G \oplus M}  \lra A/I \gr{M}. 
\end{equation} 
Note that the (completed) pushout of the composite 
\[
A \gr{\N^{c} \oplus M} \lra A \gr{G \oplus M} \lra \sq A_\infty,
\]
which is $(p,I)$-completely flat, along the  composition map in (\ref{eq:presenting_AM_as_lci}) is $(A/I)_{\infty}$. Therefore, we can apply Proposition \ref{lci case} and conclude the lemma once we verify Assumption \ref{assumption:p_I_flat}. To verify this, we may work pro-Zariski locally on $A$ and assume $(A, I)$ is orientable, and then it further reduces to the universal oriented case with $M_A$, i.e., the $(p,d)$-completion of $\Z_p\{d\}^{\delta} \{M_A\}^{\delta}_{\log}[\delta(d)^{-1}]$ (in particular $(A, \phi(I))$ is again a prism). In this case, $\phi(x_1)-1, \dots \phi (x_c)-1$ is a $p$-completely regular sequence in $A\langle \N^c\oplus M\rangle$ relative to $A/\phi(I)A$ and Assumption \ref{assumption:p_I_flat} holds. 
\eproof 

\subsection{The log free case}\label{ss:log_free_case}
\noindent 
In this subsection we endow a Nygaard filtration for free pre-log rings in a functorial way. We follow the strategy outlined in Example \ref{example:log_Nygaard}. 

Fix an integral bounded pre-log prism $(A, I, M_A)$. Let 
\[
(A/I\langle M \rangle, M_A\oplus M)=(A/I \langle \N^S \rangle, M_A\oplus \N^S)
\]
denote the free pre-log ring on some finite set $S$. Given a surjection $M_A\oplus N \to M_A\oplus M$ from a finitely generated free monoid over $M_A$, let us consider the quasisyntomic cover
\begin{equation*}
(A/I \langle M \rangle, M_A\oplus M) \lra \textstyle ((A/I)_{\infty}, M_{\infty})
\end{equation*}
from the previous subsection and let $((A/I)_{\infty}^{\bullet}, M_{\infty}^{\bullet})$ denote its (completed) \v{C}ech nerve.  

\begin{construction}[Exactification for free pre-log rings] \label{construction:free_case_extract_roots} \noindent 

\noindent
In order to apply the discussion from Subsection \ref{ss:extracting_p-power_roots}, we further introduce the following cosimplicial objects.  Let $(A/I \langle N^{\bullet} \rangle, N^{\bullet})$ denote the (completed) \v{C}ech nerve of the map
\[A/I \lra (A/I \gr{N} , N).\]
Note that the monoids $N^{\bullet}$ are just coproducts of $N$. Write $N_{\infty}^{\bullet}$ for $N^{\bullet}[1/p]$. 
The following natural maps of monoids (horizontal maps are induced by summation) form a pushout diagram
\begin{equation} \label{diagram:N_infty}
\begin{tikzcd}
M_A\oplus N^\bullet \arrow[r] \arrow[d] & M_A\oplus M \arrow[d]\\
M_A\oplus N_\infty^\bullet \arrow[r] & M_\infty^\bullet.
\end{tikzcd} 
\end{equation}

\noindent 
By taking exactifications of horizontal arrows in Diagram (\ref{diagram:N_infty}), we get
\[
\begin{tikzcd}
M_A\oplus N^\bullet \arrow[r] \arrow[d] & \arrow[r] \arrow[d] \sq N^\bullet & M_A\oplus M \arrow[d]\\
M_A\oplus N_\infty^\bullet \arrow[r] & \sq N_\infty^\bullet \arrow[r] & M_\infty^\bullet,
\end{tikzcd} 
\]
where the left square is also a pushout square. 
Let us set
\begin{equation*}
(A/I)^{\bullet}_{\infty}\coloneqq A/I \widehat{\otimes}_{A/I\langle M_A\rangle} A/I\langle M_{\infty}^{\bullet}\rangle, \quad
\widetilde{A}_{\infty}^{\bullet}\coloneqq A \langle \widetilde{N}_{\infty}^{\bullet} \rangle. 
\end{equation*}

\noindent By Corollary \ref{semiperfectoid case}, we have the Nygaard filtration $\Fil^i_N$ on
\begin{equation}\label{eq:derives_prismatic_log_compares_to_nonlog_Cech_nerve}
\Prism_{((A/I)^{\bullet}_{\infty},M_{\infty}^{\bullet}) / (A, M_A) } \widehat{\otimes}^\L_{\widetilde{A}_{\infty}^{\bullet}, \phi} \widetilde{A}_{\infty}^{\bullet} 
\cong 
\Prism_{(A/I)^{\bullet}_{\infty}  / \sq A_\infty^\bullet } \widehat{\otimes}^\L_{\widetilde{A}_{\infty}^{\bullet}, \phi} \widetilde{A}_{\infty}^{\bullet}, 
\end{equation}
such that the relative Frobenius (using (\ref{eq:Nygaard_log_and_nonlog}) to identify the target) 
\begin{equation} \label{eq:Nygaard_relative_Frob_Cech_nerve}
\phi_{(A/I)^{\bullet}_{\infty}/\widetilde{A}_{\infty}^{\bullet}}\colon 
\Fil^i_N \Prism_{((A/I)^{\bullet}_{\infty},M_{\infty}^{\bullet}) / (A, M_A) } \widehat{\otimes}^\L_{\widetilde{A}_{\infty}^{\bullet}, \phi} \widetilde{A}_{\infty}^{\bullet} \lra 
I^i \Prism_{((A/I)^{\bullet}_{\infty},M_{\infty}^{\bullet}) / (A, M_A) }
\end{equation}
induces an isomorphism
\[
\grade_N^i  \Prism_{((A/I)^{\bullet}_{\infty},M_{\infty}^{\bullet}) / (A, M_A) } \widehat{\otimes}^\L_{\widetilde{A}_{\infty}^{\bullet}, \phi} \widetilde{A}_{\infty}^{\bullet} 
\cong 
\Fil_i \cl \Prism_{((A/I)^{\bullet}_{\infty},M_{\infty}^{\bullet}) / (A, M_A) } \{i\}.
\]
Finally, note that the cosimplicial object $\Prism_{((A/I)^{\bullet}_{\infty},M_{\infty}^{\bullet}) / (A, M_A) }$ is obtained from the $(p,I)$-completed derived base change along the map of cosimplicial $A$-algebras
\[
A\langle N^{\bullet}\rangle\widehat{\otimes}_{\Z_p\langle M_A\oplus N^{\bullet}\rangle}\Z_p \langle \widetilde{N}^{\bullet}\rangle
\ra \widetilde{A}^{\bullet}_{\infty}. 
\]
\end{construction} 

Following \cite[Definition 12.9]{BS}, we will use the totalization of $\Fil^i_N$ above to define a Nygaard filtration on $\Prism_{(A \langle M \rangle , M_A\oplus M ) /(A, M_A)}^{(1)}\coloneqq\Prism_{(A \langle M \rangle , M_A\oplus M ) /(A, M_A)}\widehat{\otimes}^\L_{A, \phi} A$. For this we need the following lemma. 

\begin{lemma} \label{lemma:Niziol}
Let $F^{\bullet}$ be a cosimplicial $A/(p,I)[ N^{\bullet}]\otimes_{\F_p[ M_A\oplus N^{\bullet}]}\F_p [ \widetilde{N}^{\bullet}]$-module. 
\begin{enumerate}
    \item The natural map
\[
F^{\bullet} \ra E^{\bullet}\coloneqq F^{\bullet}\otimes_{\F_p[N^{\bullet}]} \F_p [N_{\infty}^{\bullet}], 
\]
where the target naturally identifies with the base change of $F^{\bullet}$ to $\widetilde{A}_{\infty}^{\bullet}$, induces a quasi-isomorphism on the associated complexes of $A$-modules. 
    \item The natural map
    \[
    F^{\bullet}\otimes_{A/(p,I),\phi}A/(p,I)\ra F^{\bullet} \otimes_{A/(p,I)[ N^{\bullet}]\otimes_{\F_p[ M_A\oplus N^{\bullet}]}\F_p [ \widetilde{N}^{\bullet}],\phi}A/(p,I)[ N^{\bullet}]\otimes_{\F_p[ M_A\oplus N^{\bullet}]}\F_p [ \widetilde{N}^{\bullet}]    
    \]
    induces a quasi-isomorphism on the associated complexes of $A$-modules. 
    \item The natural map
    \[
    E^{\bullet}\otimes_{A/(p,I),\phi}A/(p,I) \ra E^{\bullet}\otimes_{\widetilde{A}_{\infty}^{\bullet},\phi}\widetilde{A}_{\infty}^{\bullet}
    \]
    induces a quasi-isomorphism on the associated complexes of $A$-modules. 
\end{enumerate}
\end{lemma}

\begin{proof}
(1) The map $\F_p [N^{\bullet}]\ra \F_p [N^{\bullet}_{\infty}]$ is the \v{C}ech nerve of the natural map $\F_p [N]\ra \F_p [N_{\infty}]$. The map $N\ra N_{\infty}$ is the filtered colimit of maps of the multiplication by powers of $p$. Each map gives a quasi-isomorphism by \cite[Lemma 5.4]{BS}, and so is the colimit. 

(2) This follows from \cite[Proposition.B.3]{Koshikawa}. More concretely, we apply it to $M_A\to \widetilde{N}$ here (in the notation there this is $M\to Q$), which is injective and integral. Let us check the assumption of \cite[Proposition.B.3]{Koshikawa}. First, the group
\[
G\coloneqq Q^{\textup{gp}}/M^{\textup{gp}}=\widetilde{N}^{\textup{gp}}/M_A^{\textup{gp}}\cong N^{\textup{gp}}
\]
is clearly a finitely generated free abelian group. Secondly, the map
$Q=\widetilde{N}\ra Q^{(1)}$, where $(-)^{(1)}$ denotes the pushout along $M_A\ra M_A; m\mapsto m^p$, is exact since the map $\widetilde{N}\ra M_A\oplus M$ is exact surjective and $M_A\oplus M \ra (M_A\oplus M)^{(1)}$ is clearly exact. 

(3) Combine (1) and (2). 
\end{proof}
Finally, we are ready to construct Nygaard filtrations in the free case. Let 
\[
(R, P) = (A/I \gr{M}, M_A \oplus M)=(A/I\langle \N^S\rangle, M_A\oplus \N^S)
\]
be the pre-log ring as above. By the quasisyntomic descent, we know that the Frobenius twisted prismatic cohomology $\Prism_{(R, P)/(A, M_A)}^{(1)}$ is computed by
\[
\Prism_{((A/I)_{\infty}^{\bullet}, M_{\infty}^{\bullet})/ A} 
\widehat{\otimes}^\L_{A, \phi} A
\cong 
\Prism_{(A/I)^{\bullet}_{\infty}/ \sq A_\infty^\bullet} 
\widehat{\otimes}^\L_{A, \phi} A, 
\]
where all terms are discrete and $(p,I)$-completely flat over $A$. 
Therefore, Lemma \ref{lemma:Niziol} implies that it is also computed by
\[
\Prism_{(A/I)^{\bullet}_{\infty}/ \sq A_\infty^\bullet}
\widehat{\otimes}^\L_{\widetilde{A}_{\infty}^{\bullet}, \phi} \widetilde{A}_{\infty}^{\bullet}. \]

\bd \label{definition:derived_Nygaard} 
For $(R, P)=(A/I \langle \N^S \rangle, M_A\oplus \N^S)$, we define the Nygaard filtration
\[
\Fil^i_N \Prism_{(R, P)/(A, M_A)}^{(1)} 
\lra  \Prism_{(R, P)/(A, M_A)}^{(1)}=\Prism_{(R, P)/(A, M_A)}\widehat{\otimes}^\L_{A, \phi} A
\]
to be the totalization of 
\[
\Fil^i_N 
\Prism_{(A/I)^{\bullet}_{\infty}/ \sq A_\infty^\bullet}
\widehat{\otimes}^\L_{\widetilde{A}_{\infty}^{\bullet}, \phi} \widetilde{A}_{\infty}^{\bullet}, 
\]
which is independent of the surjection $M_A\oplus N\ra M_A\oplus \N^S$.\footnote{Any map between two surjections induces a map between cosimplicial objects that is a quasi-isomorphism on the filtrations as seen by induction and the description of graded pieces. We can take the direct sum of any two surjections, so we can take the sifted colimit. To have a strictly functorial presentation, we can consider the case $N=\N^{S'}$ for a finite generating subset $S'\subset M_A\oplus \N^S$ with the natural map $N\to M_A\oplus \N^S$ and pass to its filtered colimit.}
These Nygaard filtrations are equipped with a ``Frobenius map'' 
\[
\phi\colon \Fil^i_N \Prism_{(R, P)/(A, M_A)}^{(1)} 
\lra 
I^i  \Prism_{(R, P)/(A, M_A)}
\]
coming from (\ref{eq:Nygaard_relative_Frob_Cech_nerve}), and are functorial.  
More generally, if 
\begin{equation} \label{eq:Sigma_S_T}
\ul \Sigma_{S,T}=(A/I \langle (X_s)_{s\in S}, \N^T\rangle, M_A \oplus \N^T)
\end{equation}
for finite sets $S$, $T$, we define the Nygaard filtration using the isomorphism
\[
\Prism_{\ul \Sigma_{S,T}/(A, M_A)}^{(1)}\cong
\Prism_{A/I \langle (X_s)_{s\in S}\rangle /A}^{(1)} \widehat{\otimes}^\L_A
\Prism_{(A/I \langle \N^T\rangle, \N^T)/A}^{(1)}
\]
and the Day convolution of the above filtration with the one in \cite[Definition 12.9, Theorem 15.3]{BS}. Or, equivalently, one can slightly generalize the above construction itself in this setting. 
\ed

We immediately get

\bp \label{prop:derived_nygaard_HT_free}
Let $\ul \Sigma_{S,T}$ be as above. We have a commutative diagram
\[
\begin{tikzcd}
\Fil^i_N \Prism_{\ul \Sigma_{S,T}/(A, M_A)}^{(1)} \arrow[r, "\phi"] \arrow[d]&   I^i \Prism_{\ul \Sigma_{S,T}/(A, M_A)}  \arrow[d]
 \\
\grade^i_N \Prism_{\ul \Sigma_{S,T}/(A, M_A)}^{(1)} \arrow[r] 
& \overline{\Prism}_{\ul \Sigma_{S,T}/ (A, M_A)}\{i\}, 
\end{tikzcd}
\]
where the bottom horizontal map gives rise to a  functorial isomorphism
\begin{equation} \label{eq:isom_of_graded_Nygaard_on_poly}
\grade^i_N \Prism_{\ul \Sigma_{S,T}/(A, M_A)}^{(1)} \cong 
\tau_{\leq i} \overline{\Prism}_{\ul \Sigma_{S,T}/ (A, M_A)}\{i\}. 
\end{equation} 
\ep

\begin{remark}
The existence of the the map 
\[
\grade^i_N \Prism_{\ul \Sigma_{S,T}/(A, M_A)}^{(1)} \to
\tau_{\leq i} \overline{\Prism}_{\ul \Sigma_{S,T}/ (A, M_A)}\{i\}
\]
in (\ref{eq:isom_of_graded_Nygaard_on_poly}) above is a property of the map 
\[
\grade^i_N \Prism_{\ul \Sigma_{S,T}/(A, M_A)}^{(1)} \to \overline{\Prism}_{\ul \Sigma_{S,T}/ (A, M_A)}\{i\}, 
\]
and it is induced by taking the truncation. 
\end{remark}

\subsection{Derived Nygaard filtration and de Rham comparison}  \noindent  \label{ss:derived_Nygaard_and_dR}

\noindent 
Deriving the above filtration, we obtain the derived Nygaard filtration $\Fil_N^\bullet \Prism_{(R, P)/(A, M_A)}^{\L, (1)}$ on any (simplicial) pre-log ring $(R, P)$ over $(A/I, M_A)$, and a map of filtered complexes
\[
\phi\colon \Fil_N^\bullet \Prism_{(R, P)/(A, M_A)}^{\L, (1)} \to 
I^\bullet \Prism^{\L}_{(R, P)/(A, M_A)}. 
\]
The map above induces a canonical isomorphism 
\begin{equation}\label{eq:derived_graded_of_derived_Nygaard}
    \grade_N^{i} \Prism_{(R, P)/ (A, M_A )}^{\L, (1)} \isom  \Fil_i \overline{\Prism}^\L_{(R, P)/ (A, M_A)}\{i\}
\end{equation}
where the right hand side is the conjugate filtration from Proposition \ref{prop:derived_HT}. Theorem \ref{derived_nygaard_HT} then follows directly from the log free case (Proposition \ref{prop:derived_nygaard_HT_free}).

\br 
We warn the reader that, even when $(R, P)$ is ``large'', the derived Nygaard filtration may not be discrete contrary to \cite[Theorem 15.2]{BS}. 
This makes our definition somewhat non-canonical even in the smooth case, and we do not have a non-derived definition like \cite[Theorem 15.3]{BS}. However, under an additional assumption, e.g., the Cartier type condition, we have another description using $L{\eta}_I$ and the Beilinson $t$-structure \cite[5.3, 5.4, 5.8]{BMS2} (or, a general version of $L\eta_{I}$ considered by Berthelot--Ogus).
\er 

More precisely, 

\bc \label{decalage_affine} \label{cor:Frobenius_factors_through_L_eta_affine}
Let $(A, I, M_A)$ be a bounded pre-log prism with $M_A$ being integral
and suppose that $(R, P)$ is a $p$-complete pre-log ring over $(A/I, M_A)$ with bounded $p^\infty$-torsion and $M_A\to P$ is a smooth chart. Then there are isomorphisms
\[
\grade_N^{i} \Prism_{(R, P)/ (A, M_A )}^{\L, (1)} \cong \tau_{\leq i} \overline{\Prism^\L}_{(R, P)/ (A, M_A)}\{i\}
\]
for all $i \geq 0$. 
The Frobenius map $\phi$ on $\Prism_{(R, P)/ (A, M_A)}$ factors as
\begin{align*}
\Prism^\L_{(R, P)/ (A, M_A)}\widehat{\otimes}^{\L}_{A, \phi_A} A= \Prism_{(R, P)/ (A, M_A)}^{\L, (1)}
& \lra 
L\eta_I \Prism^\L_{(R, P)/ (A, M_A)} \\ 
&\lra
\Prism^\L_{(R, P)/ (A, M_A)}. 
\end{align*}
If moreover $M_A\to P$ is of Cartier type, then the map $\Prism_{(R, P)/ (A, M_A)}^{\L, (1)}\to L\eta_I \Prism^\L_{(R, P)/ (A, M_A)}$ is an isomorphism identifying the derived Nygaard filtration with the truncation of the $I$-adic filtration of $\Prism^\L_{(R, P)/ (A, M_A)}$ with respect to the Beilinson $t$-structure. 
\ec

\begin{proof}
The first part follows since the conjugate filtration is just the canonical filtration in the case of smooth charts. 
The factorization is obtained by using \cite[5.8]{BMS2} as in the proof of \cite[Theorem 15.3]{BS}. 
For the last claim,  to see that the map
\[
\Prism_{(R, P)/ (A, M_A)}^{\L, (1)}\to L\eta_I \Prism^\L_{(R, P)/ (A, M_A)}
\]
is an isomorphism, it suffices to see that it is so modulo $I$. The right hand side modulo $I$ is identified with the log de Rham complex by \cite[Proposition 6.12]{BMS1} and the Hodge--Tate comparison. 
Moreover, the Cartier type assumption implies that the log de Rham complex is identified with the derived log de Rham cohomology \cite[7.6]{Bhatt_dR}. 
Then we can reduce to the case $(R, P)=(A/I \langle (X_{i})_{i\in S}, \N^T\rangle, M_A\oplus\N^T)$ for finite sets $S$, $T$.  We may also assume that $M_A$ is trivial. Then we can replace $(A, I)$ by the universal oriented prism, and in turn by $(\Z_p, (p))$ using \cite[Construction 6.1]{BS} (cf. the proof of \cite[Theorem 15.3]{BS}). 
This case now follows from the crystalline comparison in \cite{Koshikawa} and its compatibility with the Cartier isomorphism. The claim for the Nygaard filtration follows from the description of graded pieces and \cite[5.4]{BMS2} by induction. 
\end{proof}

In the above proof, we also showed the following:

\bc[The de Rham comparison] \label{cor:dR_comparison}
Assume $P\to R$ is a smooth chart of Cartier type over $(A/I, M_A)$. 
There is a canonical isomorphism
\[
\Prism^\L_{(R, P)/ (A, M_A)}\widehat{\otimes}^\L_{A, \phi}A/I \cong \widehat \Omega^{\bullet}_{(R, P)/ (A/I, M_A)}
\]
of $E_{\infty}$-algebras in $\mD(A/I)$. 
\ec

\br[Derived de Rham comparison] \label{remark:derived_dR_comparison}
Deriving the isomorphism in Corollary \ref{cor:dR_comparison}, we obtain a canonical isomorphism \[
\Prism^\L_{(R, P)/ (A, M_A)} \widehat{\otimes}^{\L}_{A, \phi} A/I \cong \widehat \L \Omega_{(R, P)/ (A/I, M_A)}
\] of $E_\infty$-algebras in $\mD(A/I)$
for all (simplicial) pre-log rings $(R, P)$ over $(A/I, M_A)$
\er

\br[Multiplicative structure]  \label{remark:multiplicative_structure}
For any (simplicial) pre-log ring $(R, P)$ over $(A/I, M_A)$, the multiplicative structure on $\Prism_{(R, P)/(A, M_A)}^{\L, (1)}$ admits a refinement to a multiplicative structure 
\[
\Fil_N^i \Prism_{(R, P)/(A, M_A)}^{\L, (1)} \otimes_A^\L \Fil_N^j \Prism_{(R, P)/(A, M_A)}^{\L, (1)} \lra \Fil_N^{i+j} \Prism_{(R, P)/(A, M_A)}^{\L, (1)}
\]
which is functorial in $(R, P)$. Via left Kan extension, it suffices to construct such a multiplicative structure for pre-log rings of the form $\ul{\Sigma}_{S, T} = (A/I \langle (X_s)_{s\in S}, \N^T\rangle, M_A \oplus \N^T)$  that is functorial in $S$ and $T$. From Definition \ref{definition:derived_Nygaard}, it reduces to the case where $S = \varnothing$ (using the multiplicative structure on $\Prism^{\L, (1)}_{A/I \gr{X_s}_{s \in S}/A}$ from \cite[Remark 5.1.8]{APC}). For this, it suffices to construct a functorial multiplicative structure 
\begin{equation} \label{eq:multiplicativity_on_Cech_nerve}
    \Fil^i_N \Big(
\Prism^{(1)}_{A/I \langle M_{\infty}^{k}\rangle/ \sq A_\infty^k}  \Big) \otimes^\L_{A}  \Fil^j_N \Big(
\Prism^{(1)}_{A/I \langle M_{\infty}^{k}\rangle/ \sq A_\infty^k} \Big)  \lra  \Fil^{i+j}_N \Big(
\Prism^{(1)}_{A/I \langle M_{\infty}^{k}\rangle/ \sq A_\infty^k} \Big)
\end{equation}
for each $k$, where we use the notation
\[
\Prism^{(1)}_{A/I \langle M_{\infty}^{k}\rangle/ \sq A_\infty^k} \coloneqq 
\Prism_{A/I \langle M_{\infty}^{k}\rangle/ \sq A_\infty^k}
\widehat{\otimes}^\L_{\widetilde{A}_{\infty}^{k}, \phi} \widetilde{A}_{\infty}^{k} 
\]
from Subsection \ref{ss:log_free_case}. The desired multiplicative structure then comes from 
\begin{align*}
\Fil_N^i \Prism_{\ul{\Sigma}_{T}/(A, M_A)}^{\L, (1)} \otimes_A^\L \Fil_N^j \Prism_{\ul{\Sigma}_{T}/(A, M_A)}^{\L, (1)}  
\lra
& \varprojlim_{k} \Big( 
 \Fil^i_N \Big(
\Prism^{(1)}_{A/I \langle M_{\infty}^{k}\rangle/ \sq A_\infty^k}  \Big) \otimes^\L_{A}  \Fil^j_N \Big(
\Prism^{(1)}_{A/I \langle M_{\infty}^{k}\rangle/ \sq A_\infty^k} \Big)  
\Big) \\
\lra 
& \varprojlim_{k}  
 \Fil^{i+j}_N \Big(
\Prism^{(1)}_{A/I \langle M_{\infty}^{k}\rangle/ \sq A_\infty^k}  \Big)  \:\: = \Fil_N^{i+j} \Prism_{\ul{\Sigma}_{T}/(A, M_A)}^{\L, (1)}
\end{align*}
Finally, the map (\ref{eq:multiplicativity_on_Cech_nerve}) comes from formula (\ref{eq:Nygaard_definition}). 
\er 

\subsection{Nygaard filtration, $I$-adic  filtration and Hodge filtration.}

\begin{construction} \label{construction:I_adic_filtration_maps_to_Nygaard}
Let $(R, P)$ be a (simplicial) pre-log ring  
over $(A/I, M_A)$. From  Remark \ref{remark:multiplicative_structure} we obtain a map of filtered complexes 
\begin{multline} \label{eq:I_adic_filtration_maps_to_Nygaard}
\quad I^{i} \otimes^\L_A  \Fil_N^\bullet \Prism_{(R, P)/(
A, M_A)}^{\L, (1)} \cong \Fil_{N}^{i} \Prism^{\L, (1)}_{(A/I, M_A)/(A, M_A)} \otimes^\L_{A} \Fil_N^\bullet \Prism_{(R, P)/(
A, M_A)}^{\L, (1)} \\ 
\lra \Fil_{N}^{\bullet + i} \Prism^{\L, (1)}_{(R, P)/(A, M_A)} \quad 
\end{multline}
for each $i \ge 0$, as 
\[\Fil_N^i \Prism_{(A/I, M)/(A, M_A)}^{\L, (1)} \cong \Fil_N^i \Prism_{(A/I)/A}^{\L, (1)} = I^i\] (see \cite[Example 5.1.4]{APC}). From the construction given in Remark \ref{remark:multiplicative_structure}, the composition of  (\ref{eq:I_adic_filtration_maps_to_Nygaard}) with the canonical map 
\[\Fil_N^{\bullet + i} \Prism^{\L, (1)}_{(R, P)/(A, M_A)}  \ra  \Fil_N^{\bullet} \Prism^{\L, (1)}_{(R, P)/(A, M_A)}\] agrees with the natural map induced by $I^i \subset A$. 
\end{construction}

We have the following analogue of \cite[Proposition 5.2.3]{APC}, which says that the natural map 
\[
\Prism^{\L, (1)}_{(R, P)/(A, M_A)} \lra \widehat \L \Omega_{(R, P)/(A/I, M_A)}
\]
from the derived de Rham comparison can be naturally upgraded to a filtered map, where $\Prism^{\L, (1)}_{(R, P)/(A, M_A)} $ is equipped with the Nygaard filtration and the ($p$-completed) derived log de Rham complex $\widehat \L \Omega_{(R,P)/(A/I, M_A)}$ is equipped with the Hodge filtration. 

\bl \label{lemma:Nygaard_map_to_Hodge} 
Let $(R, P)$ be a (simplicial) pre-log ring  over $(A/I, M_A)$.  There is a canonical map of filtered complexes 
\begin{equation} \label{eq:image_of_Nygaard_is_Hodge}
\gamma^\bullet_{(R, P)}\colon \Fil_N^\bullet \Prism^{\L, (1)}_{(R, P)/(A, M_A)} \lra \Fil_H^{\bullet} \widehat \L \Omega_{(R, P)/(A/I, M_A)}
\end{equation}
functorial in $(R, P)$. 
\el

\bproof The proof is similar to that of \cite[Proposition 5.2.3]{APC}. 
Via left Kan extension, it suffices to construct  canonical maps 
\begin{equation} \label{eq:image_of_Nygaard_is_Hodge_for_polynomials}
\gamma^\bullet_{S, T}\colon \Fil_N^\bullet \Prism^{\L, (1)}_{\ul{\Sigma}_{S, T}/(A, M_A)} \lra \Fil_H^{\bullet} \widehat \L \Omega_{\ul{\Sigma}_{S, T}/(A/I, M_A)} = \widehat \Omega^{\ge \bullet}_{\ul{\Sigma}_{S, T}/(A/I, M_A)} \end{equation}
functorial in $S$ and $T$, for $\ul{\Sigma}_{S, T}$ as in (\ref{eq:Sigma_S_T}). Since the filtered complex $\widehat \Omega^{\ge \bullet}_{\ul{\Sigma}_{S, T}/(A/I, M_A)}$ (with its Hodge filtration) lies in the heart of the Beilinson $t$-structure and $\Fil_N^\bullet \Prism^{\L, (1)}_{\ul{\Sigma}_{S, T}/(A, M_A)}$ is connective by Corollary \ref{cor:Frobenius_factors_through_L_eta_affine} and its proof, the map $\gamma^\bullet_{
S, T}$ in (\ref{eq:image_of_Nygaard_is_Hodge_for_polynomials}) corresponds precisely to a map 
\begin{equation} \label{eq:Nygaard_is_Hodge_for_polynomials_Beilinson_truncation}
(H^\bullet (\grade_N^\bullet \big( \Prism^{\L, (1)}_{\ul{\Sigma}_{S, T}/(A, M_A)} )), \beta) \cong 
\tau^{\ge 0}_B \Fil_N^\bullet \big( \Prism^{\L, (1)}_{\ul{\Sigma}_{S, T}/(A, M_A)} \big) \lra  \widehat \Omega^{\ge \bullet}_{\ul{\Sigma}_{S, T}/(A/I, M_A)}, 
\end{equation}
where $\tau_{B}^{\ge 0}$ denotes the truncation with respect to the Beilinson $t$-structure, and $\beta$ denotes the Bockstein differential. Now the canonical map (\ref{eq:Nygaard_is_Hodge_for_polynomials_Beilinson_truncation}) is supplied by the isomorphism 
\[
H^\bullet (\grade_N^\bullet \big( \Prism^{\L, (1)}_{\ul{\Sigma}_{S, T}/(A, M_A)} ))
\isom H^\bullet (\tau_{\le \bullet} \cl \Prism^{\L}_{\ul{\Sigma}_{S, T}/(A, M_A)} \{\bullet\}) 
\isom \widehat \Omega^{ \bullet}_{\ul{\Sigma}_{S, T}/(A/I, M_A)} 
\]
of cochain complexes using Corollary \ref{decalage_affine} and the Hodge--Tate comparison, in particular, the Bockstein differential on 
\[H^\bullet (\tau_{\le \bullet} \cl \Prism^{\L}_{\ul{\Sigma}_{S, T}/(A, M_A)} \{\bullet\})  \cong  H^\bullet (\cl \Prism^{\L}_{\ul{\Sigma}_{S, T}/(A, M_A)} \{\bullet\})
\] gets identified with the de Rham differential on $\widehat \Omega^{\bullet}_{\ul{\Sigma}_{S, T}/(A/I, M_A)}$. 
\eproof

\br
If $M_A = P = \{e\}$, then the map we constructed above agrees with that in \cite[Proposition 5.2.3]{APC}. In particular, for a general (simplicial) pre-log ring $(R, P)$, the map (\ref{eq:image_of_Nygaard_is_Hodge}) is given by 
\[
\grade_N^0 \Prism^{\L, (1)}_{(R, P)/(A, M_A)} \isom \Fil_0 \cl \Prism^{\L}_{(R, P)/(A, M_A)} \cong \grade_H^0  \widehat \L \Omega_{(R, P)/(A/I, M_A)} 
\] on the $0^{\tu{th}}$ graded piece (up to a functorial homotopy). 
\er 

Next we determine the fiber of the map (\ref{eq:image_of_Nygaard_is_Hodge}). 

\bp \label{prop:fiber_of_Nygaard_to_Hodge}
Let $(R, P)$ be a (simplicial) pre-log ring  
over $(A/I, M_A)$. There is a fiber sequence  
\begin{equation} \label{eq:I_adic_and_Nygaard}
I \otimes^\L_A \Fil_N^{\bullet - 1} \Prism^{\L, (1)}_{(R, P)/(A, M_A)}  \lra 
\Fil_N^\bullet \Prism^{\L, (1)}_{(R, P)/(A, M_A)} \xrightarrow{\gamma_{(R, P)}^{\bullet}} \Fil_H^{\bullet} \widehat \L \Omega_{(R, P)/(A/I, M_A)}
\end{equation}
that is functorial in $(R, P)$. Here the first arrow is given by (\ref{eq:I_adic_filtration_maps_to_Nygaard}) in Construction \ref{construction:I_adic_filtration_maps_to_Nygaard} and the second arrow is (\ref{eq:image_of_Nygaard_is_Hodge}). 
\ep

\bproof  The proof is similar to that of \cite[Proposition 5.2.8]{APC}. 
It suffices to show that the functorial sequence (\ref{eq:I_adic_and_Nygaard}) 
is a fiber sequence when $(R, P) = \ul{\Sigma}_{S, T}$ as in the proof of Lemma \ref{lemma:Nygaard_map_to_Hodge}. Now the sequence becomes 
\[
I \otimes^\L_A \Fil_N^{\bullet - 1} \Prism^{\L, (1)}_{\ul{\Sigma}_{S, T}/(A, M_A)}  \lra 
\Fil_N^\bullet \Prism^{\L, (1)}_{\ul{\Sigma}_{S, T}/(A, M_A)} \xrightarrow{\gamma^\bullet_{S, T}} \widehat \Omega^{\ge \bullet}_{\ul{\Sigma}_{S, T}/(A/I, M_A)}.
\]
The composition is nullhomotopic,  since $I \otimes_A \Fil_N^{\bullet - 1} \Prism^{\L, (1)}_{\ul{\Sigma}_{S, T}/(A, M_A)}  \in \mathcal{DF}(A)^{\le -1}$ by Corollary \ref{cor:Frobenius_factors_through_L_eta_affine}, while $ \widehat \Omega^{\ge \bullet}_{\ul{\Sigma}_{S, T}/(A/I, M_A)} \in \mathcal{DF}(A)^{\heartsuit}$ with respect to the Beilinson $t$-structure. Therefore the map $\gamma^\bullet_{S, T}$ descends to a map from the cofiber 
\begin{equation} \label{eq:Nygaard_Hodge_on_cofiber}
\Fil_N^\bullet \Prism^{\L, (1)}_{\ul{\Sigma}_{S, T}/(A, M_A)} / I \otimes^\L_A \Fil_N^{\bullet - 1} \Prism^{\L, (1)}_{\ul{\Sigma}_{S, T}/(A, M_A)} \lra \widehat \Omega^{\ge \bullet}_{\ul{\Sigma}_{S, T}/(A/I, M_A)}.
\end{equation}
It remains to show that filtered map (\ref{eq:Nygaard_Hodge_on_cofiber}) is an isomorphism. For $\bullet=i \le 0$, this follows from the de Rham comparison. For $\bullet=i \ge 0$, by induction, it remains to show that the map (\ref{eq:Nygaard_Hodge_on_cofiber}) induces an isomorphism 
\[
\Fil_i \cl \Prism^{\L}_{\ul{\Sigma}_{S, T}/(A, M_A)} \{i\}/ \Fil_{i-1} \cl \Prism^{\L}_{\ul{\Sigma}_{S, T}/(A, M_A)}\{i\} \isom \widehat \Omega^i_{\ul{\Sigma}_{S, T} /(A/I, M_A)}[-i].
\]
on graded pieces. Now, from the construction of the map $\gamma^\bullet$ in Lemma \ref{lemma:Nygaard_map_to_Hodge}, this map agrees with the isomorphism supplied by the  Hodge--Tate comparison. 
\eproof

\subsection{Global consequences}
\noindent 

\noindent In this subsection, we extend our analysis  on the Nygaard filtration to the global setup. 

\begin{construction}[The ``Frobenius twisted'' sheaf $\Prism^{(1)}_{(X, M_X)/(A, M_A)}$] \label{construction:Frob_twisted_prism}
To start, let $(X, M_X)$ be a log $p$-adic formal scheme over $(A/I, M_A)$ and denote by \[\Prism^{(1)}_{(X, M_X)/(A, M_A)}\] the $(p,I)$-complete \'etale sheafification of the presheaf on $X_{\ett}$ 
\begin{equation} \label{eq:presheaf_twisted_by_Frob}
\Prism^{(1), \tu{pre}}\colon U  \longmapsto 
\big(\Prism_{(X, M_X)/(A, M_A)} (U) \big) \widehat \otimes^\L_{A, \phi_A} A.    
\end{equation}
We similarly define $\Prism^{\L, (1)}_{(X, M_X)/(A, M_A)}$ for the derived version. 
\end{construction}

\begin{construction}[The global Nygaard filtration] \label{construction:global_Nygaard}
Consider the functor on affine objects  of the \'etale site $X_{\ett}$ which sends 
\[
U=\spf (R) \longmapsto \Fil^i_N \Prism^{\L, (1)}_{(R, \Gamma (U, M_X))/ (A, M_A)}.
\]
Taking the $(p,I)$-complete \'etale sheafification, we obtain a Nygaard filtration 
\[
\Fil_N^i \Prism^{\L, (1)}_{(X,M_X)/(A, M_A)} \lra \Prism^{\L, (1)}_{(X,M_X)/(A, M_A)}, 
\]
 equipped with a map of filtered complexes 
\begin{equation} \label{eq:Frob_on_global_Nygaard}
\phi\colon \Fil_N^\bullet \Prism^{\L, (1)}_{(X,M_X)/(A, M_A)} \to I^\bullet \Prism^{\L}_{(X,M_X)/(A, M_A)}
\end{equation}
induced by (refined) Frobenius in Definition \ref{definition:derived_Nygaard}. Moreover, for each $i \ge 0$ there is a natural map  
\begin{equation}  \label{eq:I_adic_filtration_maps_to_Nygaard_global}
I^i  \Fil_N^\bullet \Prism^{\L, (1)}_{(X,M_X)/(A, M_A)} \lra \Fil_N^{\bullet+i}\Prism^{\L, (1)}_{(X,M_X)/(A, M_A)} 
\end{equation}
 of filtered complexes using Construction \ref{construction:I_adic_filtration_maps_to_Nygaard}. 
\end{construction}

\bc \label{cor:global_Nygaard}
The Frobenius map factors through a canonical isomorphism 
\[
\grade_N^i \Prism^{\L, (1)}_{(X,M_X)/(A, M_A)} \isom \Fil_i \cl \Prism^{\L}_{(X,M_X)/(A, M_A)} \{i\}
\]
of \'etale sheaves on $X$, where $\Fil_i$ denotes the derived conjugate filtration from Remark \ref{remark:global_derived_conjugate_filtration}.
If in addition $M_A$ is integral and $(X, M_X)$ is smooth over $(A/I, M_A)$, then we have a Nygaard filtration $\Fil_N^\bullet \Prism^{(1)}_{(X,M_X)/(A, M_A)}$ on $ \Prism^{(1)}_{(X,M_X)/(A, M_A)}$, which satisfies 
\[
\grade_N^\bullet \Prism^{(1)}_{(X,M_X)/(A, M_A)} \isom \tau_{\le \bullet} \cl \Prism_{(X,M_X)/(A, M_A)} \{\bullet\}.
\]
\ec 

Assume $M_A$ is integral. Next, we prove the global counterpart of Corollary \ref{decalage_affine}. Let us start with the following lemma. Let $X_{\ett, \tu{aff}}$ denote the affine (small) \'etale site of $X$, consisting of affine objects in $X_{\ett}$ equipped with topology given by \'etale covers. 

\bl \label{lemma:L_eta_Prism_complete}
Let $(X, M_X)$ be a smooth log $p$-adic formal scheme over $(A/I, M_A)$. 
Then the presheaf 
\[
(L \eta_I \Prism)^{\tu{pre}}\colon U = \spf R \longmapsto L \eta_I (\Prism_{(X, M_X)/(A, M_A)} (U))
\]
is already a sheaf on $X_{\ett, \tu{aff}}$, which we denote by $L \eta_{I} \Prism_{(X, M_X)/(A, M_A), \tu{aff}}$ (or simply $(L \eta_{I} \Prism)_{\tu{aff}}$), which is a derived $(p, I)$-complete sheaf over $X_{\ett, \tu{aff}}$. It determines a unique derived $(p, I)$-complete sheaf on $X_{\ett}$, which we denote by 
\[L \eta_{I} \Prism_{(X, M_X)/(A, M_A)}\]
(or simply $L \eta_{I} \Prism$ if the context is clear).
\el 

\bproof 
Note that the presheaf $(L \eta_I \Prism)^{\tu{pre}}$ takes values in derived $(p, I)$-complete objects (by \cite[Lemma 6.19]{BMS1}). 
Thus it suffices to check the claim after taking derived reduction mod $I$. In other words, it suffices to show that the presheaf sending \[
U = \spf R \longmapsto \Big((L \eta_I \Prism)^{\tu{pre}} (U) \Big) \widehat \otimes_{A}^\L A/I
\]
is a sheaf on $X_{\ett, \tu{aff}}$. 
But this becomes the \'etale sheaf sending 
\[ U = \spf R \longmapsto H^\bullet (\cl \Prism_{(X, M_X)/(A, M_A)} (U))\{i\}\] by \cite[Proposition 6.12]{BMS1}, which is isomorphic to the log de Rham complex $\Omega^{\bullet}_{(X, M_X)/(A/I, M_A)}$. 
\eproof 

Now we are ready to prove 
\bp \label{prop:decalage_global}
Let $(X, M_X)$ be a smooth log $p$-adic formal scheme over $(A/I, M_A)$ with the mod $p$ fiber of Cartier type, then the Frobenius induces a canonical isomorphism of \'etale sheaves 
\begin{equation} \label{eq:L_eta_global}
\Prism_{(X, M_X)/(A, M_A)}^{(1)} \isom L \eta_{I} \Prism_{(X, M_X)/(A, M_A)},
\end{equation}
where $\Prism_{(X, M_X)/(A, M_A)}^{(1)}$ is defined in Construction \ref{construction:Frob_twisted_prism} and $L \eta_{I} \Prism_{(X, M_X)/(A, M_A)}$ is the sheaf introduced in Lemma \ref{lemma:L_eta_Prism_complete}. Moreover, $\Prism^{(1), \tu{pre}}$ is a sheaf on qcqs objects of $X_{\ett}$ and
\[
R\Gamma_{\Prism}((U, M_U)/(A, M_A))^{(1)} \isom L \eta_{I} R\Gamma_{\Prism}((U, M_U)/(A, M_A))
\]
for $U=\spf R$. 
\ep 

\bproof 
Let us first prove (\ref{prop:decalage_global}). 
On each $U =\spf R \in X_{\ett}$, we have a map of filtered complexes 
\begin{equation} \label{eq:map_phi_on_U}
\phi_U\colon \Prism^{(1)}_{(X, M_X)/(A, M_A)} (U) \lra \Prism_{(X, M_X)/(A, M_A)} (U)
\end{equation}
where the left hand side is equipped with the Nygaard filtration 
\[\Fil_N^\bullet \Prism^{(1)}_{(X, M_X)/(A, M_A)} (U)\] while the right hand side is equipped with the $I$-adic filtration. 
By Corollary \ref{cor:global_Nygaard}, the Nygaard filtration on  $\Prism^{(1)}_{(X, M_X)/(A, M_A)} (U)$ is connective with respect to the Beilinson $t$-structure, therefore the map (\ref{eq:map_phi_on_U}) factors canonically through 
\begin{equation} \label{eq:Frobenius_phi_U_map_to_L_eta}
\phi_U\colon \Prism^{(1)}_{(X, M_X)/(A, M_A)} (U) \lra L \eta_I (\Prism_{(X, M_X)/(A, M_A)} (U)),
\end{equation}
which is functorial in $U$. This way we have constructed a natural map 
\[\phi\colon \Prism^{(1), \tu{pre}} \ra (L \eta_I \Prism)_{\tu{aff}}\] as presheaves on $X_{\ett, \tu{aff}}$.  
On the associated derived $(p, I)$-complete sheaves (over $X_{\ett}$) this induces the map as required in (\ref{eq:L_eta_global}).  

To proceed, let $X_{\ett, \textup{sm}} \subset X_{\ett, \tu{aff}}$ denote the full subcategory which consists of affine $U$ \'etale over $X$ that admits a smooth chart   $P \ra \Gamma(U, M_X)$ of Cartier type. Such objects form a basis for the \'etale topology of $X$. Therefore, to show that the map (\ref{eq:L_eta_global}) is an isomorphism, it suffices to show that its restriction to $X_{\ett, \tu{sm}}$ is an isomorphism. To this end, let $\Prism^{(1), \tu{pre}}_{\tu{sm}}$ (resp. $(L \eta_I \Prism)_{\tu{sm}}$) denote the restriction of the presheaf $\Prism^{(1), \tu{pre}}$ (resp. $(L \eta_I \Prism)_{\tu{aff}}$) to $X_{\ett, \tu{sm}}$. By Proposition \ref{prop:sections_of_derived_prismatic} and Corollary \ref{decalage_affine}, the natural map
\begin{equation} \label{eq:Leta_isom_on_small}
\phi\colon \Prism^{(1), \tu{pre}}_{\tu{sm}} \isom (L \eta_I \Prism)_{\tu{sm}}
\end{equation}
is an isomorphism of presheaves on $X_{\ett, \tu{sm}}$. Now taking derived $(p, I)$-complete sheafification we get the desired claim (note that the right hand side of (\ref{eq:Leta_isom_on_small}) is already a derived $(p, I)$-complete sheaf over $X_{\ett, \tu{sm}}$ by Lemma \ref{lemma:L_eta_Prism_complete}). 

For the remaining assertions, it suffices to show that the natural map 
\begin{equation} \label{eq:pulling_out_Frobenius}
 R \Gamma (X_{\ett}, \Prism_{(X, M_X)/(A, M_A)})  \widehat \otimes^\L_{A, \phi_A} A \lra 
  R \Gamma (X_{\ett}, \Prism^{(1)}_{(X, M_X)/(A, M_A)}) 
\end{equation}
is an isomorphism if $X$ is qcqs. We may replace $X_{\ett}$ by $X_{\ett, \tu{sm}}$ and it suffices to show that 
\begin{equation*} 
 R \Gamma (X_{\ett, \tu{sm}}, \Prism)  \widehat \otimes^\L_{A, \phi_A} A \isom 
  R \Gamma (X_{\ett, \tu{sm}}, \Prism \widehat \otimes^\L_{A, \phi_A} A), 
\end{equation*}
where $\Prism$ denotes the sheaf $\Prism_{(X, M_X)/(A, M_A)}$, is an isomorphism. This can be checked after taking $\otimes_A^{\L} A/(p,I)$, and it follows from the projection formula since $X$ is qcqs by assumption. 
\eproof

From Proposition \ref{prop:decalage_global}, we immediately obtain (a sheaf version of) the de Rham comparison stated in the introduction. 

\bc 
Let $(X, M_X)$ be a smooth log $p$-adic formal scheme over $(A/I, M_A)$ with the mod $p$ fiber of Cartier type. Then there is a canonical isomorphism 
\[
\Prism_{(X, M_X)/(A, M_A)} \widehat \otimes_{A, \phi_A}^\L A/I \cong \Omega^{\bullet}_{(X, M_X)/(A/I, M_A)}.
\]
\ec 

\bproof 
This follows from taking the (derived) mod $I$ reduction of the isomorphism (\ref{eq:L_eta_global}) and the Hodge--Tate comparison. 
\eproof

From Proposition \ref{prop:decalage_global}, we further deduce the following more precise statement on the image of Frobenius. 

\bc \label{cor:image_Frob_global}
Let $(X, M_X)$ be a smooth log $p$-adic formal scheme over $(A/I, M_A)$ with the mod $p$ fiber of Cartier type. Then for each $i \ge 0$, we have natural maps 
\begin{equation} \label{eq:V_i_on_complex}
V_i\colon \tau_{\le i} \Prism_{(X, M_X)/(A, M_A)} \otimes^\L_A I^i \lra \tau_{\le i} \Prism_{(X, M_X)/(A, M_A)}^{(1)} 
\end{equation}
such that the composition $\phi \circ V_i$ with  \[\phi\colon \tau_{\le i}\Prism_{(X, M_X)/(A, M_A)}^{(1)}  \ra  \tau_{\le i}\Prism_{(X, M_X)/(A, M_A)}\] 
(resp. the precomposition $V_i \circ (\phi \otimes 1)$ with 
\[\phi \otimes 1\colon  \tau_{\le i}\Prism_{(X, M_X)/(A, M_A)}^{(1)} \otimes^\L_A I^i  \ra  \tau_{\le i}\Prism_{(X, M_X)/(A, M_A)} \otimes^\L_A I^i, \] 
becomes the natural map induced by the inclusion $I^i \hookrightarrow A$. In particular, if the underlying formal scheme $X$ is qcqs, there is a natural map, for each $i \ge 0$, 
\begin{equation} \label{eq:V_i_on_cohomology}
V_i\colon  H^i_{\Prism} ( (X, M_X)/(A, M_A))  \otimes_A I^i \lra  H^i (R \Gamma_{\Prism} ((X, M_X)/(A, M_A))\widehat\otimes^\L_{A, \phi_A} A)  
\end{equation}
which is an inverse of the Frobenius map up to $I^i$. 
\ec 

\bproof 
The existence of the required maps $V_i$ in (\ref{eq:V_i_on_complex}) follows from the first part of Proposition \ref{prop:decalage_global} using \cite[Lemma 6.9]{BMS1} (see \cite[Corollary 15.5]{BS} and its proof). The last assertion follows from the second part of Proposition \ref{prop:decalage_global}: by taking cohomology, we obtain the map $V_i$ in (\ref{eq:V_i_on_cohomology}) as the composition
\[
\begin{tikzcd}[row sep = 1.5em]
  H^i_{\Prism} ( (X, M_X)/(A, M_A))  \otimes_A I^i \arrow[d, equal] \\
   H^i (X_{\ett}, \Prism_{(X, M_X)/(A, M_A)}) \otimes_A I^i \arrow[r, "(\ref{eq:V_i_on_complex})"] &  H^i (X_{\ett}, \Prism_{(X, M_X)/(A, M_A)}^{(1)}) \arrow[d, "\rotatebox{90}{$\sim$}"]\\ 
   & H^i (R \Gamma_{\Prism} ((X, M_X)/(A, M_A)) \widehat \otimes^\L_{A, \phi_A} A) . 
\end{tikzcd}
\]
\eproof 

\br \label{remark:uniform_bound_of_I_power}
In fact, the bound on the power of $I$ needed for the Frobenius to admit an inverse can be chosen uniformly if $\Omega^1_{(X, M_X)/(A/I, M_A)}$ has the finite rank. More precisely, if we let $D$ denote the rank of $\Omega^1_{(X, M_X)/(A/I, M_A)}$, then the map (\ref{eq:V_i_on_complex}) gives us a map 
\[
V\colon  \Prism_{(X, M_X)/(A, M_A)} \otimes^\L_A I^D \lra \Prism_{(X, M_X)/(A, M_A)}^{(1)} 
\]
which provides an inverse of the Frobenius map up to $I^D$. In particular, the Frobenius on $R\Gamma_{\Prism} ((X,M_X)/(A, M_A))$ admits an inverse up to $I^D$. 
\er

Let us also note the following. 

\bc \label{cor:fiber_of_Nygaard_to_Hodge_global}  Let $(A, I, M_A)$ be a bounded pre-log ring and $(X, M_X)$ a log $p$-adic formal scheme over $(A/I, M_A)$. 
For each $i$, let $\Fil_H^i \widehat \L \Omega_{(X, M_X)/(A/I, M_A)}$ denote the derived $p$-complete \'etale sheafification of the presheaf
\[
U = \spf R \longmapsto  \Fil_H^{i} \widehat \L \Omega_{(R, \Gamma(U, M_X))/(A/I, M_A)}
\]
valued in derived $p$-complete objects in $\mD(A)$. Then there is a canonical fiber sequence 
\[ 
I \otimes^\L_A \Fil_N^{\bullet - 1} \Prism^{\L, (1)}_{(X,M_X)/(A, M_A)} \lra \Fil_N^\bullet \Prism^{\L, (1)}_{(X,M_X)/(A, M_A)} \lra \Fil_H^\bullet \widehat \L \Omega_{(X, M_X)/(A/I, M_A)}
\]
of filtered complexes on $X_{\ett}$. 
\ec 

\bproof 
This follows from Proposition \ref{prop:fiber_of_Nygaard_to_Hodge}. 
\eproof 

\br \label{remark:comparing_derived_dR_and_dR_complex}
Suppose that $(X, M_X)$ is smooth over $(A/I, M_A)$ with its mod $p$ fiber of Cartier type, then the fiber sequence in Corollary \ref{cor:fiber_of_Nygaard_to_Hodge_global} becomes \begin{equation}
 I \otimes^\L_A \Fil_N^{\bullet - 1} \Prism^{(1)}_{(X,M_X)/(A, M_A)} \lra \Fil_N^\bullet \Prism^{(1)}_{(X,M_X)/(A, M_A)} \lra \Omega_{(X, M_X)/(A/I, M_A)}^{\ge \bullet}.
 \end{equation}
For this it suffices to show that the natural map \begin{equation} \label{eq:compare_derived_dR_to_dR}
 \widehat \L \Omega_{(X, M_X)/(A/I, M_A)}
 \lra \Omega_{(X, M_X)/(A/I, M_A)}^{\bullet}
\end{equation}
is an isomorphism of \'etale sheaves identifying the Hodge filtration on both sides. 
Let us first show that (\ref{eq:compare_derived_dR_to_dR}) is an isomorphism. If  $A/I$ is an $\F_p$-algebra, then using the conjugate filtration on both sides one reduces to show that the map $\widehat \L_{(X,M_X)/(A/I, M_A)} \ra \Omega_{(X, M_X)/(A/I, M_A)}^1$ is an isomorphism, which follows from the proof of Proposition \ref{prop:sections_of_derived_prismatic}. Slightly more generally, if $A/I$ is an $\Z/p^n$-algebra, we may reduce to the case of $\F_p$-algebras by d\'evissage. In general, since $A/I$ is assumed to have bounded $p^{\infty}$-torsion, we know that the pro-systems $\{A/(I, p^n)\}$ and $\{A/I \otimes^\L_{\Z} \Z/p^n\}$ are pro-isomorphic (see \cite[Lemma 4.6]{BMS2} and the discussion thereafter). It follows that, for any derived $p$-complete object $N \in \mD (A/I)$, we have an isomorphism 
\[
N \isom \varprojlim_{n} N \otimes^\L_{\Z} \Z/p^n \cong  \varprojlim_{n} N \otimes^\L_{A/I} (A/I \otimes^\L_{\Z}\Z/p^n) 
\isom \varprojlim_{n} N \otimes^\L_{A/I} A/(I, p^n). 
\]
Therefore, it suffices to show that for each $n \ge 1$, (\ref{eq:compare_derived_dR_to_dR}) becomes an isomorphism after taking derived base change $-\otimes^\L_{A/I} A/(I, p^n)$. By assumption we know that $(X, M_X)$ is smooth over $(A/I, M_A)$, in particular, this implies that $X$ is flat over $A/I$, thus after derived base change the map becomes 
\[
 \widehat \L \Omega_{(X, M_X)_{n}/(A/(I,p^n), M_A)}
 \lra \Omega_{(X, M_X)_n/(A/(I, p^n), M_A)}^{\bullet},
\]
which we know is an isomorphism. 
Here  $(X, M_X)_{n}$ denotes the base change
\[
 (X, M_X) \times_{(\spf A/I, M_A)^a} (\spec A/(I, p^n), M_A)^a, 
\]
which is a log scheme smooth over $(A/(I, p^n), M_A)$. 
\er 

This implies the following claim, which we need later. 

\bc   \label{cor:Nygaard_complete_for_smooth}
Let $(X, M_X)$ be a smooth log $p$-adic formal scheme over $(A/I, M_A)$. Suppose that $X$ is qcqs and that the mod $p$ fiber of $(X, M_X)$ is of Cartier type. 
\be 
\item Assume $\Omega^1_{(X, M_X)/(A/I, M_A)}$ has the finite rank $D$. If $j \ge D$, then the map 
\[
\Fil_N^j \Prism^{(1)}_{(X,M_X)/(A, M_A)} \otimes_A^\L I^i  \lra \Fil_N^{i+j}\Prism^{(1)}_{(X,M_X)/(A, M_A)} \]
from (\ref{eq:I_adic_filtration_maps_to_Nygaard_global}) induces an isomorphism 
\[
R \Gamma(X_{\ett}, \Fil_N^j \Prism^{(1)}_{(X,M_X)/(A, M_A)}) \otimes^\L_A I^i  \isom R \Gamma(X_{\ett}, \Fil_N^{i+j} \Prism^{(1)}_{(X,M_X)/(A, M_A)}) 
\]
\item  $R \Gamma(X_{\ett}, \Prism^{(1)}_{(X,M_X)/(A, M_A)})$ is complete with respect to the Nygaard filtration given by 
\[
R \Gamma(X_{\ett}, \Fil_N^\bullet \Prism^{(1)}_{(X,M_X)/(A, M_A)}) \lra R \Gamma(X_{\ett}, \Prism^{(1)}_{(X,M_X)/(A, M_A)}).
\]
\ee 
\ec

\bproof 
 Part (1) follows from Corollary \ref{cor:fiber_of_Nygaard_to_Hodge_global} and Remark \ref{remark:comparing_derived_dR_and_dR_complex}. Part (2) follows from Part (1) since the Nygaard filtration on 
 \[\Fil_N^{D} R \Gamma(X_{\ett}, \Prism^{(1)}_{(X,M_X)/(A, M_A)}) = R \Gamma(X_{\ett}, \Fil_N^D \Prism^{(1)}_{(X,M_X)/(A, M_A)})\] becomes the $I$-adic filtration (compare with \cite[Lemma 7.8]{LiuLi1} and its proof). 
\eproof

 
\newpage 
\section{Comparison with Kummer \'etale cohomology}  \label{sec:etale}

In this section, we prove the following \'etale comparison theorem for derived log prismatic cohomology. 
 
\bt \label{thm:etale_comp}
Let $(A, I = (d))$ be a perfect prism where $I \ne (p)$.\footnote{For the case where $I = (p)$, the statement is vacuously true as both sides vanish.} Let $(R, P)$ be a $p$-adically complete pre-log $A/I$-algebra where $P$ is a saturated monoid and $R$ has bounded $p^\infty$-torsion. 
For each $n \ge 1$, there is a canonical isomorphism 
\[ 
R \Gamma_{\textup{k\'et}} (\spec (R[\frac{1}{p}], P)^a, \Z/p^n) \cong 
(\Prism_{(R, P)/A} [\frac{1}{d}] /p^n )^{\phi = 1}
\]
which is functorial in $(R, P)$. 
\et 

For a fine saturated monoid $P$, $R \Gamma_{\textup{k\'et}} (\spec (R[\frac{1}{p}], P)^a, \Z/p^n)$ is the Kummer \'etale cohomology of the fs log scheme $\spec (R[\frac{1}{p}], P)^a$.  
For a general saturated monoid $P$, one may simply define
\[
R \Gamma_{\textup{k\'et}} (\spec (R[\frac{1}{p}], P)^a, \Z/p^n)
\coloneqq
\varinjlim_{P_i \subset P}
R \Gamma_{\textup{k\'et}} (\spec (R[\frac{1}{p}], P_i)^a, \Z/p^n), 
\]
where $P_i$ runs through fine saturated submonoids of $P$ and the colimit is filtered.

Our proof of Theorem \ref{thm:etale_comp} is analogous to the one outlined in \cite{Bhatt_notes}, namely one reduces the statement to a log semiperfectoid situation. The key observation is that a certain log prismatic cohomology can be identified with a (nonlog) prismatic cohomology after taking perfections. 

\br[Pre-log structure on the base prism]
In the theorem above, we may replace the base prism $(A, I)$ by any perfect ``log prism'', or any pre-log prism $(A, I, M_A)$ such that $(A/I, M_A)$ is a perfectoid or even pseudo-perfectoid pre-log ring, in the sense of Definition \ref{def:perfectoid_monoid} and Remark \ref{remark:pseudo_perfectoid}. 
To see this, note that under the (pseudo-)perfectoid assumption, we have a natural isomorphism 
\[\Prism_{(R, P)/A} \isom \Prism_{(R, P)/(A, M_A)}\] by the Hodge--Tate comparison and Lemma \ref{lemma:perfectoid_monoid}. Let us also note that the left hand side does not depend on the base pre-log prism.  
\er 

\br[Log \'etale cohomology] The left hand side of the comparison isomorphism in Theorem \ref{thm:etale_comp} can be replaced by the (full) log \'etale cohomology of $\spec (R[\frac{1}{p}], P)^a$, which can similarly be defined as 
\[
R \Gamma_{\textup{log\'et}} (\spec (R[\frac{1}{p}], P)^a, \Z/p^n)
\coloneqq
\varinjlim_{P_i \subset P}
R \Gamma_{\textup{log\'et}} (\spec (R[\frac{1}{p}], P_i)^a, \Z/p^n).  
\]
This is because for constant torsion coefficients $\Lambda$ (for example $\Lambda = \Z/p^n$), we have a natural isomorphisms
\[
R \Gamma_{\textup{k\'et}} (\spec (R[\frac{1}{p}], P_i)^a, \Lambda) \isom R \Gamma_{\textup{log\'et}} (\spec (R[\frac{1}{p}], P_i)^a, \Lambda)
\]
using \cite[Proposition 5.4]{Nakayama2}, which identifies the\ Kummer \'etale cohomology and the log \'etale cohomology of $\spec (R[\frac{1}{p}], P)^a$ in the colimit. A more geometric approach that recovers the above definition is given in \cite{Dori_Yao}. 
\er 

\br[Relation to cohomology of log adic spaces]  
Suppose that $\spa (R[\frac{1}{p}], R)$ is an adic space (namely it is sheafy as a pre-adic space), then the Kummer \'etale cohomology of $\spec (R[\frac{1}{p}], P)^a$ can be viewed as the Kummer \'etale cohomology of the log adic space $(\spa (R[\frac{1}{p}], R), P)^a$ (note that the definition of Kummer \'etale cohomology of log adic spaces in \cite{DLLZ, DY} typically require some finiteness assumptions). In particular, we have the following 

\bl \label{lemma:comparing_adic_generic_to_spec} Let $\Lambda = \Z/n\Z$. Let $(R, P)$ be a classically $p$-complete 
fs pre-log ring and assume that $R$ is topologically finitely generated over a noetherian ring $A_0$. Let $X =  \spec (R[\frac{1}{p}], P)^a$ and let \[X^{\text{ad}} = (\spa (R[\frac{1}{p}], R), P)^a\] denote the associated log adic space in the sense of \cite{DLLZ}, where $P$ denotes the pre-log structure on $\spa (R[\frac{1}{p}], R)$ induced from $P \ra R$. Then there is a natural isomorphism 
\[
R \Gamma_{\ket} (X, \Lambda) \isom R \Gamma_{\ket} (X^{\text{ad}}, \Lambda). 
\]
\el

\bproof 
Let $\mathring X = \spec (R[\frac{1}{p}])$ (resp. ${\mathring X }^{\text{ad}} = \spa (R[\frac{1}{p}], R)$) be the underlying scheme (resp. adic space) of $X$ (resp. $X^{\text{ad}}$). Write $M_X$ (resp. $\mM_{X^{\text{ad}}}$) for the sheaf of monoids induced from pre-log structure $P$. Then we have maps 
\[
\epsilon\colon X\ra \mathring X  \qquad  (\text{resp. }
\epsilon^{\text{ad}}\colon X^{\text{ad}} \ra {\mathring X }^{\text{ad}}).
\] 
By \cite[Corollary 3.2.3]{Huber}  (and the Grothendieck--Serre spectral sequence), it suffices to show that $R^i \epsilon^{\text{ad}}_* \Lambda \cong f^* R^i \epsilon_* \Lambda$, where $f$ denotes the map $f\colon {\mathring X }^{\text{ad}} \ra {\mathring X }$. For this let us note that we have isomorphisms
\[
\midwedge^i (\cl \mM_{X^{\text{ad}}}^{\text{gp}}/n \cl \mM_{X^{\text{ad}}}^{\text{gp}}) (-i) \isom R^i \epsilon^{\text{ad}}_* \Lambda
\] over the adic space ${\mathring X }^{\text{ad}}_{\ett}$
by \cite[Equation (4.4.28), Lemma 4.4.29]{DLLZ}. On the other hand, we have  
\[
\midwedge^i (\cl M_{X}^{\text{gp}}/n \cl  M_{X}^{\text{gp}} ) (-i) \isom R^i \epsilon_* \Lambda
\] 
in ${\mathring X }_{\ett}$ by \cite[Theorem 2.4]{Kato_Nakayama}. The claim thus follows. 
\eproof 

\er

\subsection{Descendability of log prismatic cohomology}

\bl \label{lemma:descendability_when_fg} Let $I, J$ be two finite sets. Let $T = \cl A \gr{X_i, Y_j}_{i \in I, j \in J}$ and equip $T$ with the pre-log structure $M =\N^{\oplus J} \ra T$, sending each $m \in \N_j$ to $Y_j^m$. Let 
\[(T_\infty = \cl A \gr{X_i^{1/p^\infty}, Y_j^{1/p^\infty}}, \: M_\infty = (\N_\infty)^{\oplus J} )
\]
be the perfectoid pre-log ring over $(T, M)$. Here $\N_\infty$ denotes the monoid $\N[\frac{1}{p}]$. Then the map 
\[ 
\Prism_{(T, M)/A} \lra \Prism_{(T_\infty, M_\infty)/A} \quad \text{and }  
\]
is a descendable map of $E_\infty$-algebras. Moreover, the maps 
\[ 
\Prism_{(T, M)/A, \text{perf}} \lra \Prism_{(T_\infty, M_\infty)/A, \text{perf}} \quad \text{and } \: \: 
\Prism_{(T, M)/A, \text{perf}} [\frac{1}{d}] \lra \Prism_{(T_\infty, M_\infty)/A, \text{perf}} [\frac{1}{d}]  
\]
are  descendable maps of $E_\infty$-algebras. 
\el 

The proof of the lemma is similar to that of \cite[Lemma 8.6]{BS}

\bproof 

Filtered colimits of descendable maps of bounded index is descendable by \cite[Lemma 12.22]{BS_Witt}, so the statement about perfection and localization (inverting $d$) follows from the descendability of the first arrow, which further reduces to the case where $(T, M) = (\cl A \gr{Y}, \N)$ by considering the coproduct (the case where $T = \cl A \gr{X} $ equipped with the trivial pre-log structure is already taken care of by \cite[Lemma 8.6]{BS}).  Now we show that the map 
\begin{equation} \label{eq:descendable_degree_2}
\Prism_{(\cl A \gr{Y}, \N)/A} \lra \Prism_{(\cl A \gr{Y^{1/p^\infty}}, \N[1/p])/A} 
\end{equation}  is descendable of index $\le 2$.  Let $F \in \mD (\Prism_{\cl A \gr{\ul \N}} ) $ be the fiber of the map in (\ref{eq:descendable_degree_2}) where we write $\Prism_{\cl A \gr{\ul \N}}  \coloneqq\Prism_{(\cl A \gr{Y}, \N)/A} $,
 then it suffices to show that 
$ \Hom_{\Prism_{\cl A \gr{\ul \N}}} (F^{\otimes 2}, \Prism_{\cl A \gr{\ul \N}}) \in \mD^{<0}$. By derived completeness, it suffices to prove the statement after reducing mod $(p, \phi^{-1} (d))$. By de Rham comparison and base change, we are reduced to consider the map 
\begin{equation} 
\label{eq:special_case_of_descendable}
\L \Omega_{ (\F_p [Y], \N) /\F_p} \cong \Omega^{\bullet}_{(\F_p [Y], \N)/\F_p} \lra \L \Omega_{(\F_p [Y^{1/p^\infty}], \N[1/p])/\F_p} \cong \F_p [Y^{1/p^\infty}]
\end{equation}
and show that its fiber $F_0$ satisfies $\Hom_{\L \Omega_{\F_p [\ul \N]}} (F_0^{\otimes 2}, \L \Omega_{\F_p [\ul \N]}) \in \mD^{< 0 }$. The same argument of \cite[Lemma 8.6]{BS} applies. By Cartier isomorphism we have $\Omega^{\bullet}_{(\F_p [Y], \N)/\F_p}  \cong \F_p [Y] \oplus \Omega^1 [-1]$ where $\Omega^1 \coloneqq \Omega^1_{(\F_p [Y], \N)/\F_p} \cong \F_p [Y] \text{ dlog } Y$. Following \textit{loc.cit.} let us write $\Omega \coloneqq \F_p [Y] \oplus \Omega^1 [-1]$.  The map (\ref{eq:special_case_of_descendable}) is given by sending $\F_p [Y] \ra \F_p [Y^{1/p^\infty}]$ via Frobenius and sending $\Omega^1[-1]$ to $0$. Since $\text{Frob}: \F_p [T] \ra \F_p [T]$ makes $\F_p [T]$ a free module over itself, we know that the fiber $F_0$ of the map (\ref{eq:special_case_of_descendable}) identifies with $\F_p [Y]^{\oplus S}[-1]$ for some set $S$ as an $\F_p [Y]$ module, so it suffices to show that $\Hom_{\Omega} ((\F_p [Y][-1] )^{\otimes 2}, \Omega) \in \mD^{< 0 }$, which follows from the resolution of $\F_p [Y][-1]$ over the dg-algebra $\Omega$ by 
\[
\Big( \cdots \lra \Omega [-3] \xrightarrow{\text{dlog } Y}    \Omega [-2] \xrightarrow{\text{dlog } Y} \Omega[-1]     \Big ) \isom \F_p [Y][-1]
\]
\eproof

\bc \label{cor:descent_for_perfected_prismatic}
Let $(R,P)$ be a 
classically $p$-complete pre-log $\cl A$-algebra such that $R$ is topologically finitely generated over $\cl A$ with bounded $p^{\infty}$-torsion. 
Choose a map 
\[ 
 (T, M) \coloneqq (\cl A \gr{X_i, Y_j}_{i \in I, j \in J}, \N^{\oplus J}) \ra (R,P)
 \]
with finite sets $I, J$ that is surjective on both rings and monoids. Let 
\[
(T_\infty, M_\infty)\coloneqq (\cl A \gr{X_i^{{1}/{p^\infty}}, Y_j^{{1}/{p^\infty}}}, \N_\infty^{\oplus J}), \quad
(R_\infty, P_\infty)\coloneqq (T_\infty, M_\infty) \oplus_{(T, M)} (R, P);
\]
the latter is the base change of pre-log rings.   
Let $(R_\infty^\bullet, P_\infty^\bullet)$ be the derived $p$-complete \v{C}ech nerve of the map $(R, P) \ra (R_\infty, P_\infty)$, then we have isomorphisms
\[
(\Prism_{(R, P)/A } [\frac{1}{d}] /p^n)^{\phi = 1} \cong  (\Prism_{(R, P)/A, \text{perf}} [\frac{1}{d}] /p^n)^{\phi = 1} \isom \varprojlim ( \Prism_{(R_\infty^\bullet , P_\infty^\bullet)/A, \text{perf}} [\frac{1}{d}] /p^n)^{\phi = 1}.
\]
\ec 

\br 
Each term in the non-completed \v{C}ech nerve of the map $R \ra R_\infty$ has bounded $p^\infty$-torsion (since it is $p$-completely faithfully flat over $R$, see \cite[Lemma 4.7]{BMS2}), thus its derived $p$-completion agrees with the classical $p$-completion and is discrete.
\er 

\bproof 
The first arrow follows from \cite[Lemma 9.2]{BS}. For the descent statement, note that by base change and Lemma \ref{lemma:descendability_when_fg},  the map  
\[  \Prism_{(R, P)/A, \text{perf}}[\frac{1}{d}]/p^n \longrightarrow \Prism_{(R_\infty, P_\infty)/A, \text{perf}}[\frac{1}{d}]/p^n \] is descendable. 
\eproof

\subsection{The Kummer \'etale cohomology} 

\noindent 

\noindent In this subsection we establish a version of Corollary \ref{cor:descent_for_perfected_prismatic} for the Kummer \'etale side. First let us record some observations which will be used later, which roughly say that the \'etale side also behaves well under filtered colimits. 

\bl \label{lemma:ket_equal_et_for_divisible_monoid}
Let $(R, P)$ be a pre-log ring such that $p$ is invertible on $R$ and $P$ is 
a $p$-divisible saturated monoid. Then we have a natural isomorphism  
\[
R \Gamma_{\textup{k\'et}} (\spec (R, P)^a, \Z/p^n) \cong  R \Gamma_{\textup{\'et}} (\spec R, \Z/p^n).
\]
\el 

\bproof 
Write $\Lambda = \Z/p^n$. For each fs submonoid $P_i \subset P$, let $X_i$ denote the fs log scheme $\spec (R, P_i)^a$, and write 
\[\epsilon_i\colon X_{i, \textup{k\'et}} \ra \mathring X_{i, \ett}
\] for the projection from the Kummer \'etale site of $X_i$ to the \'etale site of the underlying scheme. 
Then we have 
\begin{align*}
R \Gamma_{\textup{k\'et}} (\spec (R, P)^a, \Lambda)   = &  
\varinjlim_{P_i \subset P}
R \Gamma_{\textup{k\'et}} (\spec (R , P_i)^a, \Lambda) \\ 
 \cong & \varinjlim_{P_i \subset P}
R \Gamma_{\textup{et}} (\spec R, R {\epsilon_i}_* \Lambda) \\
 \cong & \: \:
R \Gamma_{\textup{et}} (\spec R, \varinjlim_{i} R {\epsilon_i}_* \Lambda),
\end{align*}
where for the last isomorphism we use that the colimit is filtered. Now it suffices to show that the natural map 
\[ 
\Lambda \rightarrow \varinjlim_{i} R {\epsilon_i}_* \Lambda
\]
is a quasi-isomorphism. For this it suffices 
to show that $\varinjlim_{i} R^j{\epsilon_i}_* \Lambda = 0$ for each $j \ge 1$. To this end, by \cite[Theorem 2.4]{Kato_Nakayama} we have natural isomorphisms 
\[
\midwedge^j (\cl M_{X_i}^{\text{gp}}/p^n  \cl  M_{X_i}^{\text{gp}} ) (-j) \isom R^j {\epsilon_i}* \Lambda
\] 
where $M_{X_i}$ denote the log structure on $X_i$ associated to the constant pre-log structure $P_i$. Since $P = \varinjlim_{i} P_i$ is $p$-divisible, it follows that  the colimit $\varinjlim_{i} \cl M_{X_i}$ is $p$-divisible, thus  $\varinjlim_{i} R^j{\epsilon_i}_* \Lambda = 0$ for each $j \ge 1$ and the claim follows. 
\eproof

\bl \label{lemma:log_Fujuwara_Gabber}
Let $\Lambda$ be a torsion abelian group. 
Let $(R, I)$ be a Henselian pair and let $\alpha\colon P \ra R$ be an fs pre-log structure on $R$. Let $f\colon R \ra \widehat R_I$ be the (classical) $I$-adic completion. 
Then $f$ induces an isomorphism   
\[
R\Gamma_{\ket} ((\spec (R)\minus V(I), P)^a, \Lambda) \isom R \Gamma_{\ket} ((\spec \widehat R_I \minus V(I \widehat R_I), P)^a, \Lambda).
\]
Here the log structures in question come from the pre-log structure $P$ on $\spec R$ (resp. $\spec \widehat R_I$). 
\el

\bproof 
Let $X$ (resp. $Y$) be the log scheme with underlying scheme $\spec (R)\minus V(I)$ (resp. $\spec \widehat R_I \minus V(I \widehat R_I)$) and with log structure induced from $\spec (R, P)^a$ (resp. $\spec (\widehat R_I, P)^a$). Let $\mathring{X}$ (resp. $\mathring{Y}$) denote the underlying scheme of $X$ (resp. $Y$), viewed as a log scheme with the trivial log structure. Then we have the following pullback square of fs log schemes  
\[ 
\begin{tikzcd}
Y \arrow[d, swap, "f"] \arrow[r, "h"] & \mathring{Y} \arrow[d, "\mathring{f}"] \\
X \arrow[r, "g"] & \mathring{X}. 
\end{tikzcd}
\]
Thus we have a natural isomorphism of \'etale sheaves 
\[
(\mathring{f})^* R g_* \Lambda \isom R h_* f^* \Lambda \cong R h_*  \Lambda, 
\]
where $\Lambda$ is the constant sheaf on the Kummer \'etale site of $X$, 
by \cite[Theorem 5.1]{Nakayama1} (proper base change for strict morphisms). 
Therefore, we are reduced to show that the natural map 
\[
R\Gamma_{\ett} (\spec (R)\minus V(I), \mG) \isom R \Gamma_{\ett} (\spec \widehat R_I \minus V(I \widehat R_I), (\mathring{f})^* \mG)
\]
is an isomorphism where $\mG = R g_* \Lambda \in \mD(\mathring{X}_{\ett}, \Lambda)$. 
This follows from a theorem of Fujiwara--Gabber (see \cite[Theorem 6.11]{BM_arc}). 
\eproof

\bc \label{cor:colimit_of_cohomology_after_p_completion} 
Let $(R, P)$ be a saturated pre-log ring. Suppose that 
$(R, P) = \varinjlim_{i} (R_i, P_i)$ is a filtered colimit of fs pre-log rings. Then we have 
\[ 
R \Gamma_{\textup{k\'et}} (\spec (\widehat R_p [\frac{1}{p}], P)^a, \Lambda) \cong \varinjlim_{i} R \Gamma_{\textup{k\'et}} (\spec (\widehat{(R_i)}_p [\frac{1}{p}],  P_i)^a, \Lambda).   
\]
Here $\widehat R_p$ denotes the (classical) $p$-adic completion of $R$. 
\ec 
 
\bproof 
By the lemma above, we may replace $\widehat R_p$ by the $p$-henselization $R^h_p$.  
It suffices to show that for each $P_i$, we have 
\[ 
R \Gamma_{\textup{k\'et}} (\spec (R_p^h [\frac{1}{p}], P_i)^a, \Lambda) \cong \varinjlim_{j, j \ge i} R \Gamma_{\textup{k\'et}} (\spec ((R_j)^h_p [\frac{1}{p}],  P_i)^a, \Lambda).   
\]  Fix $i$ and write $X = \spec (R^h_p [\frac{1}{p}], P_i)^a$ (resp. $X_j = \spec ( (R_j)^h_p[\frac{1}{p}], P_i)^a$). 
Let $\epsilon\colon X_{\ket} \ra \mathring X_{\ett}$ (resp. $\epsilon_j\colon X_{j, \ket} \ra \mathring X_{j, \ett}$ ) be the projection from the Kummer \'etale site to the \'etale site of the underlying scheme. Now the right hand side can be rewritten as 
\begin{equation} \label{eq:ket_for_completed_log_rings_colimit} 
\varinjlim_{j, j \ge i} R \Gamma_{\textup{k\'et}} (X_j, \Lambda) 
\cong  
\varinjlim_{j \ge i} R \Gamma_{\textup{\'et}} (\mathring X_j, R {\epsilon_j}_* \Lambda)  
\cong  R \Gamma_{\textup{\'et}} (\mathring X, R {\epsilon}_* \Lambda). 
\end{equation} 
which is isomorphic to $ R \Gamma_{\textup{k\'et}} (X, \Lambda)$ as desired. Note that for the second isomorphism in (\ref{eq:ket_for_completed_log_rings_colimit}) we use that the formation of $p$-hensalization commutes with filtered colimits and proper base change for strict maps of fs log schemes as in the proof of the previous lemma. 
\eproof 

\begin{construction} \label{construction:descent_for_log_etale}

Now let $(R,P)$ be a $p$-complete pre-log algebra over $\cl A$ where $R$ is topologically finitely generated over $\cl A$ and $P$ is an fs monoid. Choose a map 
\begin{equation} \label{eq:surjection_to_R_P}
 (T, M) \coloneqq (\cl A \gr{X_i, Y_j}_{i \in I, j \in J}, \N^{\oplus J}) \ra (R,P)
\end{equation}
that is surjective on both rings and monoids (here $I$ and $J$ are assumed to be finite). 
As in Corollary \ref{cor:descent_for_perfected_prismatic}, let us consider the following ``cover''  
\begin{equation} \label{eq:surjection_to_R_P_2}
(T, M) \ra (T_\infty, M_\infty)\coloneqq (\cl A \gr{X_i^{\frac{1}{p^\infty}}, Y_j^{\frac{1}{p^\infty}}}, \N_\infty^{\oplus J}).
\end{equation}
 Now we form the $p$-completed pushout
\[
\begin{tikzcd}
(T, M) \arrow[r] \arrow[d] & (T_\infty, M_\infty) \arrow[d] \\
(R, P) \arrow[r] & (R_\infty^{\text{sat}}, P_\infty^{\text{sat}})
\end{tikzcd}
\]
in the category of (classically) $p$-complete saturated pre-log rings. Let $(R_\infty^{\bullet, \text{sat}}, P_\infty^{\bullet, \text{sat}})$ denote the \v{C}ech nerve of the map   $(R, P) \ra (R_\infty^{\text{sat}}, P_\infty^{\text{sat}})$, taken in the category $\widehat{\text{Alg}}^{\textup{prelog}}_{\text{sat}}$ of (classically) $p$-complete saturated pre-log rings. 
\end{construction} 

\bp \label{prop:flat_descent_ket}
The map 
\begin{equation} \label{eq:log_v_cover}
(R, P) \ra  (R_\infty^{\text{sat}}, P_\infty^{\text{sat}})
\end{equation} 
constructed above satisfies descent for the functor
$\mF\colon \widehat{\text{Alg}}^{\textup{prelog}}_{\text{sat}} \ra \mD(\Lambda)$  
that sends a classically $p$-complete saturated pre-log ring $(R, P)$ to $R\Gamma_{\ket} (\spec (R[\frac{1}{p}], P)^a, \Lambda)$. 
In other words, we have a natural isomorphism 
\[
R \Gamma_{\text{k\'et}} (\spec (R [\frac{1}{p}], P)^a, \Lambda) \isom \textup{Tot } \Big( R \Gamma_{\text{k\'et}} (\spec (R_\infty^{\bullet, \text{sat}}[\frac{1}{p}], P_\infty^{\bullet, \text{sat}})^a, \Lambda) \Big).
\]
\ep

\br 
There are several ways for arguing this. One convenient argument is to observe that the map in (\ref{eq:log_v_cover}) is a   log $v_p$-cover (in fact even a log $\text{arc}_p$-cover) in the sense of \cite{Dori_Yao}. In that (forthcoming) paper, the authors prove that Kummer \'etale cohomology on the generic fibers satisfies descent with respect to such maps. In what follows let us provide a proof which does not use the log version of $v$-descent. \er

\bproof 
Let $T_{0, \infty} = \cl A \gr{X_i^{1/p^\infty}, Y_i}$, the map in (\ref{eq:surjection_to_R_P_2}) factors as 
\[
(T, M) \lra (T_{0, \infty}, M) \lra (T_{\infty}, M_\infty).
\]
Its (classically $p$-complete) saturated base change along $(T, M) \ra (R, P)$ is given by 
\[
(R,P) \xrightarrow{ f } (R_{0, \infty}, P) \xrightarrow{ g } (R_\infty, P_\infty). 
\]
By \cite[Lemma 3.1.2]{LiuZheng} (and its proof) it suffices to show that the map $f$ satisfies $\mF$-descent and the map $g$ satisfies $\mF$-descent along any (classically $p$-complete)  saturated base change. Let us first treat the map $f$. The argument is similar to the proof of Lemma \ref{lemma:log_Fujuwara_Gabber}. Let $X = \spec (R[\frac{1}{p}], P)^a$ and $X_\infty = \spec (R_{0, \infty}[\frac{1}{p}], P)^a$, then we have a Cartesian square 
\[ 
\begin{tikzcd}
X_\infty \arrow[d, swap, "f"] \arrow[r, "\epsilon_\infty"] & \mathring{X}_{\infty} \arrow[d, "\mathring{f}"] \\
X \arrow[r, "\epsilon"] & \mathring{X}. 
\end{tikzcd}
\]
Then by proper base change (\cite[Theorem 5.1]{Nakayama1}) we have 
\[
(\mathring{f})^* R \epsilon_* \Lambda \isom R (\epsilon_{\infty})_* \Lambda.
\]
Thus the map $f$ satisfies $\mF$-descent, since the functor sending an $R$-algebra $(R \xrightarrow{g} S)$ to $R \Gamma_{\ett} (\spec \widehat S_{p} [\frac{1}{p}], g^* \mG)$ where $\mG = R \epsilon_* \Lambda$ satisfies $\text{arc}_p$-descent by \cite[Corollary 6.17]{BM_arc}. For the second claim, it suffices to show that the map $(\cl A \gr{\N^{\oplus J}}, \N^{\oplus J}) \ra (\cl A \gr{\N_\infty^{\oplus J}}, \N_\infty^{\oplus J})$ satisfies universal $\mF$-descent. This essentially follows from a ``Kummer pro-\'etale'' descent, which we may phrase as follows. Let us write this map as a $p$-completed filtered colimit of maps 
\[g_n\colon (\cl A \gr{Y_j}_{j \in J}, \N^{\oplus J}) \ra (\cl A \gr{Y_j}_{j \in J}, \N^{\oplus J}) 
\]
given by $Y_j \mapsto Y_j^{p^n}$ on the ring and by multiplication by $p^n\colon \N^{\oplus J} \ra \N^{\oplus J}$ on the monoid. Since totalization in $\mD^{\ge 0}$ commutes with filtered colimits, by Corollary \ref{cor:colimit_of_cohomology_after_p_completion}, it suffices to show that each $g_n$ satisfies universal $\mF$-descent. By possibly writing the monoid $N$ as a filtered colimit of fs monoids, it suffices to check the following: for any map 
\[
\psi\colon (\cl A \gr{Y_j}, \N^{\oplus J}) \ra (S, N) 
\]
where $S$ is (classically) $p$-complete and $N$ is an fs monoid, the (classically) $p$-completed saturated base change $(S,N) \ra (S_n, N_n)$ of the map $g_n$ along $\psi$ satisfies $\mF$-descent.

To prove this, we write 
\[(S_n', N_n') \coloneqq 
(S, N) \otimes_{(\cl A \gr{Y_j}, \N^{\oplus J}), g_n} (\cl A \gr{Y_j}, \N^{\oplus J}) 
\]
for the base change in the category of integral pre-log rings, and write 
\[
({S_n'}^{\textup{sat}}, {N_n'}^{\textup{sat}}) = (S_n' \otimes_{\Z_p[N_n']} \Z_p[{N_n'}^{\textup{sat}}],{N_n'}^{\textup{sat}})
\]
for its saturation. Note that $(S_n, N_n) = (\widehat{{S_n'}^{\textup{sat}}}, {N_n'}^{\textup{sat}})$ is given by the (classical) $p$-adic completion. The map $\spec ({S_n'}^{\textup{sat}} [\frac{1}{p}], N_n)^a \ra \spec (S[\frac{1}{p}], N)^a$ is Kummer \'etale (since it is base changed from a standard Kummer \'etale map), thus we are done once we show that 
\[
R \Gamma_{\ket} (\spec ({S_n'}^{\textup{sat}} [\frac{1}{p}], N_n)^a, \Lambda) \cong 
R \Gamma_{\ket} (\spec (S_n [\frac{1}{p}], N_n)^a, \Lambda) 
\]
is an isomorphism (the problem being that we do not know whether after saturation the ring ${S_n'}^{\textup{sat}}$ is classically $p$-complete). To this end, let us observe that since $N_n'= N \otimes_{\N^{\oplus J}, p^n} \N^{\oplus J}$ is a fine monoid, $\Z_p[{N_n'}^{\textup{sat}}]$ is contained in the normalization of $\Z_p [N_n']$ in $\Z_p [N_n'^{\textup{gp}}]$.\footnote{In fact, $\Z_p[{N_n'}^{\textup{sat}}]$ is the normalization of $\Z_p [N_n']$ in $\Z_p [N_n'^{\textup{gp}}]$ by \cite[Theorem 4.11.2]{Gillam_notes}. Indeed, we may assume that $N_n'^{\textup{gp}}$ is torsionfree as $N_n'^{\textup{sat}}$ contains the torsion part of $N_n'^{\textup{gp}}$, and this toric case is well-known.} Therefore, the induced map of rings  $\Z_p[N_n'] \ra \Z_p [{N_n'}^{\textup{sat}}]$
is finite, since $\Z_p [N_n']$ is a Nagata ring. Note that ${S_n'}^{\textup{sat}}$ is the base change of $\Z_p [N] \ra S$ along the composition 
\[
h_n\colon \Z_p [N] \ra \Z_p [N_n'] \ra \Z_p[N_n],
\]
which is a finite morphism of Noetherian rings. By \cite[Tag 0A05]{SP}, the derived base change $S \otimes_{\Z_p [N]}^{\L} \Z_p [N_n]$ is derived $p$-complete, thus ${S_n'}^{\textup{sat}}$ is derived $p$-complete. Finally by \cite[Tag 0G3G]{SP} we know that the surjection 
${S_n'}^{\textup{sat}} \ra S_n = \widehat{{S_n'}^{\textup{sat}}}$ is a nil-thickening. This proves the claim and thus the proposition. 
\eproof

\br \label{remark:getting_rid_of_divisible_monoid_in_descent}
Using Lemma \ref{lemma:control_of_saturated_monoid} below and Lemma \ref{lemma:ket_equal_et_for_divisible_monoid}, Proposition \ref{prop:flat_descent_ket} implies that in fact we have a natural isomorphism 
\[
R \Gamma_{\text{k\'et}} (\spec (R [\frac{1}{p}], P)^a, \Lambda) \isom \textup{Tot } \Big( R \Gamma_{\text{\'et}} (\spec  R_\infty^{\bullet, \text{sat}}[\frac{1}{p}] , \Lambda) \Big),
\]
which is how we will apply this proposition in the proof of Theorem \ref{thm:etale_comp}. One may in fact prove this assertion slightly more directly by rewriting
$R \Gamma_{\ket} (\spec (S_n[\frac{1}{p}], N_n)^a, \Lambda) \cong R \Gamma_{\ett} (\spec S_n[\frac{1}{p}], R {\epsilon_n}_* \Lambda)$ using the map  $\epsilon_n\colon \spec (S_n[\frac{1}{p}], N_n)^a \ra \spec S_n [\frac{1}{p}]$ and take the colimit of \'etale cohomology with coefficients. The claim will then follow from a similar argument as the proof of Lemma \ref{lemma:ket_equal_et_for_divisible_monoid}. 
\er

\subsection{The log free case} \label{ss:example_etale_free} \indent 

\noindent Before we prove Theorem \ref{thm:etale_comp}, let us start with a special case, which illustrates some idea of the proof. 

\bp \label{lemma:etale_comp_special_case}
Theorem \ref{thm:etale_comp} holds for $(R, P) = (\cl A \gr{\N^{J}}, \N^{J})$ with a finite set  $J$.   
\ep

\bproof For notational simplicity let us write the proof assuming $J = \{*\}$ is a singleton. Write $\N_\infty \coloneqq \N[\frac{1}{p}] $. Let $P_\infty^\bullet$ denote the \v{C}ech nerve of $\N \ra \N_\infty$ in the category of monoids, and let $S_\infty^\bullet$ denote the \v{C}ech nerve of $\N \ra \N_\infty$ in the category of saturated monoids, which is given by the saturation $S_\infty^i = P_\infty^{i, \text{sat}}$. By 
Proposition \ref{prop:flat_descent_ket}, 
we know that 
\[
R \Gamma_{\text{k\'et}} (\spec (R [\frac{1}{p}], P)^a, \Z/p^n) \isom \lim R \Gamma_{\text{k\'et}} (\spec (\cl A \gr{S_\infty^\bullet} [\frac{1}{p}], S_\infty^\bullet)^a, \Z/p^n).
\]
By Lemma \ref{cor:descent_for_perfected_prismatic}, we have 
\[
(\Prism_{(R, P)/A}[\frac{1}{d}]/p^n)^{\phi = 1} \cong \lim \Big(\Prism_{(R_\infty^\bullet, P_\infty^\bullet)/A, \text{perf}} [\frac{1}{d}]/p^n \Big)^{\phi = 1}
\]
where $R_\infty^i = \cl A \gr{P_\infty^i}$. 
Thus it suffices to show that there are natural isomorphisms 
\begin{equation} \label{eq:saturated_log_free_nonlog_prismatic}
\Gamma_{\text{k\'et}} (\spec (\cl A \gr{S_\infty^i} [\frac{1}{p}], S_\infty^i)^a, \Z/p^n) \cong 
\Big(\Prism_{(R_\infty^i, P_\infty^i)/A, \text{perf}} [\frac{1}{d}]/p^n \Big)^{\phi = 1}
\end{equation}
for each $i$ (that are compatible with the edge maps $S_\infty^i \ra S_\infty^j$). 
Now observe that the saturated monoids $S_\infty^i$ admits the following description: the map 
\begin{equation}\label{eq:saturated_free_case_monoids}
S_\infty^i = (\N_\infty \oplus_{\N}   \cdots \oplus_{\N} \N_\infty)^{\text{sat}} \longrightarrow \N_\infty \oplus (\N_\infty/\N)^{\oplus i}
\end{equation}
induced from $(x_0, x_1, ..., x_i) \mapsto (x_0 + \cdots x_i, \cl x_1, ..., \cl x_i)$ is an isomorphism. Thus we have $\spec (\cl A\gr{S_\infty^i}[\frac{1}{p}], S_\infty^i)^a \cong \spec (\cl A\gr{S_\infty^i}[\frac{1}{p}], \N_\infty)^a$. Since $\N_\infty$ is uniquely $p$-divisible, the natural map between \'etale and Kummer \'etale cohomology
\[
R \Gamma_{\text{\'et}} (\spec \cl A \gr{S_\infty^i} [\frac{1}{p}], \Z/p^n) \isom R \Gamma_{\text{k\'et}} (\spec (\cl A \gr{S_\infty^i} [\frac{1}{p}], \N_\infty)^a, \Z/p^n)
\]
is an isomorphism by Lemma \ref{lemma:ket_equal_et_for_divisible_monoid}.  
Therefore, by the (nonlog version of the) \'etale comparison of prismatic cohomology (\cite[Theorem 1.8(4)]{BS}), we have 
\begin{equation} \label{eq:reducing_log_free_nonlog_prismatic}
R \Gamma_{\text{\'et}} (\spec \cl A \gr{S_\infty^i} [\frac{1}{p}], \Z/p^n) \cong \big(\Prism_{\cl A \gr{S_\infty^i}/A, \text{perf}} [\frac{1}{d}]/p^n\big)^{\phi = 1} \cong \big(\Prism_{\cl A \gr{S_\infty^i}/A}[\frac{1}{d}] /p^n \big)^{\phi = 1}. 
\end{equation}
It remains to relate the right of (\ref{eq:reducing_log_free_nonlog_prismatic}) with the right hand side of (\ref{eq:saturated_log_free_nonlog_prismatic}), which is the content of Lemma \ref{lemma:relating_log_with_saturated_freecase}.  
\eproof 

\bl \label{lemma:relating_log_with_saturated_freecase}
The natural map $P_\infty^i \ra P_\infty^{i, \text{sat}} = S_\infty^i$ induces an isomorphism 
\[
\Prism_{(R_\infty^i, P_\infty^i)/A, \text{perf}} \isom   \Prism_{(\cl A \gr{S_\infty^i}, S_\infty^i)/A, \text{perf}} \cong \Prism_{\cl A \gr{S_\infty^i}/A, \text{perf}} .
\]
\el 

\bproof 
The second isomorphism is a consequence of the description of $S_\infty^i$  in (\ref{eq:saturated_free_case_monoids}). Let us now prove the first isomorphism. For simplicity, let us treat the case when $i = 1$ (for higher indices, the complication is purely notational).  Let $\sq P_\infty^1$ denote the exactification of the surjection 
\[
\N_\infty^{\oplus 2} \twoheadrightarrow   P_\infty^1 = \N_\infty \oplus_\N \N_\infty,
\]
which is the submonoid of $\Z[1/p]^{\oplus 2}$ generated by $\N_\infty^{\oplus 2}$ and $(1, -1)\cdot \Z$. From Section \ref{sec:derived}, $\Prism_{(R_\infty^1, P_\infty^1)/A}$ can be computed as the $(p, I)$-completed prismatic envelope of 
\[
A \gr{\sq P_\infty^1} \cong A \gr{x^{\frac{1}{p^\infty}}, y^{\frac{1}{p^\infty}}, x/y, y/x} \twoheadrightarrow \cl A \gr{P_\infty^1} \cong \cl  A  \gr{x^{\frac{1}{p^\infty}}, y^{\frac{1}{p^\infty}}}/(x-y).
\]
Now write $u = y/x$, then we have 
\begin{align*}
\Prism_{(R_\infty^1, P_\infty^1)/A} & \cong A \left\langle  x^{\frac{1}{p^\infty}}, (xu)^{\frac{1}{p^\infty}}, u^{\pm 1} \right\rangle \left \{ \frac{u-1}{I} \right\}^{\delta, \wedge}_{(p,d)} \\
& \cong A \left\langle  x^{\frac{1}{p^\infty}}, (xu)^{\frac{1}{p^\infty}}, u \right\rangle \left \{ \frac{u-1}{I} \right\}^{\delta, \wedge}_{(p,d)}. 
\end{align*}
On the other hand, since $\cl A \gr{S_\infty^1} \cong  \cl  A  \gr{x^{\frac{1}{p^\infty}}, u^{\frac{1}{p^\infty}}}/(u-1) $, its derived prismatic cohomology is given by the $(p, I)$-completed prismatic envelope 
\[
\Prism_{\cl A \gr{S_\infty^1}/A}  \cong A \left\langle  x^{\frac{1}{p^\infty}}, (u)^{\frac{1}{p^\infty}} \right\rangle \left \{ \frac{u-1}{I} \right\}^{\delta, \wedge}_{(p,d)}.
\]
By the description of the map $P_\infty^1 \ra S_\infty^1$, the natural map $\Prism_{(R_\infty^1, P_\infty^1)/A} \ra \Prism_{\cl A \gr{S_\infty^1}/A}$ in the statement of the lemma is given by $x^{1/p^m} \mapsto x^{1/p^m}, (xu)^{1/p^m} \mapsto x^{1/p^m} \cdot u^{1/p^m}$. It is clear that this map becomes an isomorphism after taking perfections.  
\eproof

\subsection{The \'etale comparison (the setup)} \label{ss:log_etale_setup}

\noindent 

\noindent Now we begin to prove the \'etale comparison in general. For this subsection let $(R,P)$ be an fs pre-log algebra such that $R$ is (classically) $p$-complete and topologically finitely generated over $\cl A$. Choose a map $
 (T, M) \coloneqq (\cl A \gr{X_i, Y_j}_{i \in I, j \in J}, \N^{\oplus J}) \ra (R,P) $ 
that is surjective on both rings and monoids, where $I$ and $J$ are assumed to be finite. Consider the map 
\[
(T, M) \ra (T_\infty, M_\infty)\coloneqq (\cl A \gr{X_i^{\frac{1}{p^\infty}}, Y_j^{\frac{1}{p^\infty}}}, \N_\infty^{\oplus J}).
\]
As in Construction \ref{construction:descent_for_log_etale}, we first form the pushout diagram
\[
\begin{tikzcd}
(T, M) \arrow[r] \arrow[d] & (T_\infty, M_\infty) \arrow[d] \\
(R, P) \arrow[r] & (R_\infty, P_\infty)
\end{tikzcd}
\]
in the category of (classically) $p$-complete pre-log rings. 
Let $(T_\infty^\bullet, M_\infty^\bullet)$ (resp. $(R_\infty^\bullet, P_\infty^\bullet)$) denote the \v{C}ech nerve of the map $(T, M) \ra (T_\infty, M_\infty)$ (resp. $(R, P) \ra (R_\infty, P_\infty)$), taken in the category of (classically) $p$-complete pre-log rings. Note that all the monoids involved are automatically integral. Also note that $P_\infty^i$ can be obtained as the pushout of $M \ra M_\infty^i$ along the surjective map $M \ra P$ for each $i \ge 0$. The monoids $M_\infty^i$ receives a natural surjection   
\[
(\N_\infty^{\oplus (i+1)})^{\oplus J} \longrightarrow M_\infty^i  \cong (\N_\infty \oplus_{\N} \cdots \oplus_{\N} \N_\infty)^{\oplus J},
\]
which in turn gives rise to a surjection 
\begin{equation} \label{eq:map_on_P_infty}
\pi\colon (\N_\infty^{\oplus (i+1)})^{\oplus J} \longrightarrow P_\infty^i.  
\end{equation} 
Let $\sq M_\infty^i$ be the exactification of the surjection $(\N_\infty^{\oplus (i+1)})^{\oplus J} \ra P_\infty^i$, and let $\sq M_\infty^{i, \text{sat}}$ (resp. $P_\infty^{i, \text{sat}}$) be the saturation of $\sq M_\infty^i$ (resp. of $P_\infty^i$). These monoids fit into the following diagram 
\[
\begin{tikzcd}[column sep = 1em]
& (\N_\infty^{\oplus (i+1)})^{\oplus J} \arrow[d, two heads] \arrow[r, hook]  & \sq M_\infty^i \arrow[ldd, two heads] \arrow[r, hook]
& \sq M_\infty^{i, \text{sat}} \arrow[r, hook] \arrow[ldd]
& (\sq M_\infty^i)^{\text{gp}} = (\Z[\frac{1}{p}]^{\oplus (i+1)})^{\oplus J}
\\
M \arrow[d, two heads] \arrow[r]  & M_\infty^i \arrow[d, two heads] 
\\ 
P \arrow[r] & P_\infty^i \arrow[r, hook] & P_\infty^{i, \text{sat}} \arrow[r, hook] & (P_\infty^i)^{\text{gp}}
\end{tikzcd}
\]
\bl \label{lemma:control_of_saturated_monoid}
\be
\item The monoid $P_\infty^{i, \text{sat}}$ and its group completion $(P_\infty^i)^{\text{gp}}$ are both $p$-divisible. The monoid $\sq M_\infty^{i, \text{sat}}$ is uniquely $p$-divisible. 
\item The map $\sq M_\infty^{i, \text{sat}} \ra P_\infty^{i, \text{sat}}$ is an exact surjection. Likewise, let $\sq M_\infty^{i, p\tu{-sat}}$ (resp. $P_\infty^{i, p\tu{-sat}}$) denote the $p$-saturation of $\sq M_\infty^i$ (resp. $P_\infty^i$), then the map 
$\sq M_\infty^{i, p\text{-sat}} \ra P_\infty^{i, p\text{-sat}}$ is an exact surjection. 
\item The $p$-saturation $\sq M_\infty^{i, p\text{-sat}}$ (resp. $P_\infty^{i, p\text{-sat}}$) of the monoid $\sq M_\infty^i$ (resp. $P_\infty^i$) agrees with its saturation $\sq M_\infty^{i, \text{sat}}$ (resp. $P_\infty^{i, \text{sat}}$).  
\ee
\el

\bproof 
 Note that $P_\infty^i$ is integral, thus we may view $P_\infty^{i, \text{sat}}$ as a submonoid of $(P_\infty^i)^{\text{gp}}$. The group $(P_\infty^i)^{\text{gp}}$ is $p$-divisible since it receives a surjection from $(\sq M_\infty^i)^{\text{gp}} = (\Z[\frac{1}{p}]^{\oplus (i+1)})^{\oplus J}$, from this it  also follows that $P_\infty^{i, \text{sat}}$ is $p$-divisible. As a saturated submonoid of a uniquely $p$-divisible monoid, $\sq M_\infty^{i, \text{sat}}$ is uniquely $p$-divisible. This proves part (1). 
 
 Part (2) follows from the description 
 \begin{align*}
     \sq M_\infty^{i, \text{sat}} 
     & = \{\alpha \in (\sq M_\infty^i)^{\text{gp}} \: | \: m \cdot \alpha \in \sq M_\infty^i \text{ for some } m \}\\
     & = \{\alpha \in (\sq M_\infty^i)^{\text{gp}} \: | \: m \cdot \pi (\alpha) \in P_\infty^i \text{ for some } m \} = \pi^{-1} (P_\infty^i)
 \end{align*}
 where $\pi\colon (\sq M_\infty^i)^{\text{gp}} \twoheadrightarrow (P_\infty^i)^{\text{gp}} $ denotes the map induced from (\ref{eq:map_on_P_infty}), and likewise for the $p$-saturated version. 
 
 For Part (3), we know that $\sq M_\infty^{i, p\tu{-sat}} = \sq M_\infty^{i,\tu{sat}}$ by the description of $(\sq M_\infty^{i})^{\tu{gp}}$. 
 The claim on $P_\infty^{i, p\tu{-sat}}$ follows from the claim on $\sq M_\infty^{i, p\tu{-sat}}$ by Part (2). 
\eproof 

\br \label{lemma:P_sat_infty_is_qrsp}
The commutative square 
\[
\begin{tikzcd}[row sep = 1em]  
\sq M_\infty^i \arrow[r] \arrow[d] & \sq M_\infty^{i, \text{sat}} \arrow[d] \\
P_\infty^i \arrow[r] & P_\infty^{i, \text{sat}}
\end{tikzcd}
\]
is a pushout square of (integral) monoids, and agrees with the derived pushout.  
Moreover, the monoid algebra $\cl A \gr{P_\infty^{i, \text{sat}}}$  (resp. the pre-log algebra $(\cl A \gr{P_\infty^{i, \text{sat}}}, P_\infty^{i, \text{sat}})$)  is quasiregular semiperfectoid. 
 \er

\subsection{End of the proof}  \noindent

 \noindent 

 Now we finish the proof of the \'etale comparison. 

\bproof[Proof of Theorem \ref{thm:etale_comp}]  \noindent 

\noindent 
First let us assume that $(R, P)$ is an fs pre-log algebra where $R$ is $p$-complete, topologically finitely generated  over $\cl A$ with bounded $p^{\infty}$-torsion.  Choose maps as in Subsection \ref{ss:log_etale_setup} (see (\ref{eq:surjection_to_R_P}) and (\ref{eq:surjection_to_R_P_2})). Let 
\[R_\infty^{i, \text{sat}} \coloneqq R^i_\infty \widehat \otimes_{\cl A \gr{P_\infty^i}} \cl A \gr{P_\infty^{i, \text{sat}}},
\]
in other words, $(R_\infty^{\bullet, \text{sat}}, P_\infty^{\bullet, \text{sat}})$ is the (classically) $p$-complete saturated \v{C}ech nerve of the map  $(R, P) \ra (R_\infty^{0, \text{sat}}, P_\infty^{0, \text{sat}})$. 
By Proposition \ref{prop:flat_descent_ket}, we have 
\begin{equation} \label{eq:etale_descent_for_fs}
R \Gamma_{\text{k\'et}} (\spec (R [\frac{1}{p}], P)^a, \Z/p^n) \isom \lim R \Gamma_{\text{k\'et}} (\spec (R_\infty^{\bullet, \text{sat}}[\frac{1}{p}], P_\infty^{\bullet, \text{sat}})^a, \Z/p^n).
\end{equation} 
Since $P_\infty^{i, \text{sat}}$ is $p$-divisible by Lemma \ref{lemma:control_of_saturated_monoid}, we know that the Kummer \'etale cohomology in right hand side of (\ref{eq:etale_descent_for_fs}) can be computed by the usual  \'etale cohomology by Lemma \ref{lemma:ket_equal_et_for_divisible_monoid} (also see Remark \ref{remark:getting_rid_of_divisible_monoid_in_descent}). In other words, we have an isomorphism 
\[
R \Gamma_{\text{\'et}} (\spec R_\infty^{i, \text{sat}} [\frac{1}{p}], \Z/p^n) \isom R \Gamma_{\text{k\'et}} (\spec (R_\infty^{i, \text{sat}} [\frac{1}{p}], P_\infty^{i, \text{sat}})^a, \Z/p^n)
\] 
for each $i$. Therefore, by the nonlog version of the \'etale comparison (\cite[Theorem 1.8(4)]{BS}), we have isomorphisms  
\begin{equation} \label{eq:saturated_prism_compares_to_logeetale}
R \Gamma_{\text{k\'et}} (\spec (R_\infty^{i, \text{sat}} [\frac{1}{p}], P_\infty^{i, \text{sat}})^a, \Z/p^n) \cong 
\Big( \Prism_{R_\infty^{i, \text{sat}}/A} [\frac{1}{d}] /p^n\Big)^{\phi = 1}. 
\end{equation}
By Corollary \ref{cor:descent_for_perfected_prismatic}, the right hand side in the statement of the theorem can be computed by 
\begin{align}  \label{eq:descent_prismatic_P_infty}
    \Big( \Prism_{(R, P)/A} [\frac{1}{d}] /p^n\Big)^{\phi = 1} 
    & \cong \Big( \Prism_{(R, P)/A, \text{perf}} [\frac{1}{d}] /p^n\Big)^{\phi = 1} \nonumber \\
    & \cong \varprojlim  \Big( \Prism_{(R_\infty^i, P_\infty^i)/A, \text{perf}} [\frac{1}{d}] /p^n \Big)^{\phi = 1}.
\end{align}
Combining (\ref{eq:etale_descent_for_fs}), (\ref{eq:saturated_prism_compares_to_logeetale}) and (\ref{eq:descent_prismatic_P_infty}), it suffices to construct (natural) isomorphisms 
\begin{equation} \label{eq:the_map_gamma_comparing_prisms}
\gamma\colon \Prism_{R_\infty^{i, \text{sat}}/A, \text{perf}} \isom \Prism_{(R_\infty^i, P_\infty^i)/A, \text{perf}}
\end{equation} 
between derived prismatic cohomology of $R_\infty^{i, \text{sat}}$ and derived log prismatic cohomology of $(R_\infty^i, P_\infty^i)$ after taking perfections, 
such that these isomorphisms are compatible with the maps $R_\infty^{i} \ra R_\infty^{k}$ induced by all simplicial maps $[i] \ra [k]$. For this, let us consider the following maps 
\[ 
\begin{tikzcd} 
& \Prism_{(R_\infty^i, P_\infty^i)/A} \arrow[d, "\beta"]
\\
\Prism_{R_\infty^{i, \text{sat}}/A} \arrow[r, "\alpha"]
& \Prism_{(R_\infty^{i, \text{sat}}, P_\infty^{i, \text{sat}})/A}
\end{tikzcd}
\]
induced from $R_\infty^{i, \text{sat}} \ra (R_\infty^{i, \text{sat}}, P_\infty^{i, \text{sat}})$ and $ (R_\infty^i, P_\infty^i)\ra (R_\infty^{i, \text{sat}}, P_\infty^{i, \text{sat}})$. The desired isomorphism $\gamma$ in (\ref{eq:the_map_gamma_comparing_prisms}) come from the following claims on the maps $\alpha$ and $\beta$ in the diagram above:
\bi
\item The map $\alpha$ is an isomorphism. 
\item The map $\beta$ induces an isomorphism 
\[
\beta_{\text{perf}}\colon  \Prism_{(R_\infty^i, P_\infty^i)/A, \text{perf}} \isom \Prism_{(R_\infty^{i, \text{sat}}, P_\infty^{i, \text{sat}})/A, \text{perf}}
\]
on taking perfections. 
\ei
For the first claim, it suffices to show that the natural map 
\[
\widehat \L_{R_\infty^{i, \text{sat}}/\cl A} \lra \widehat \L_{(R_\infty^{i, \text{sat}}, P_\infty^{i, \text{sat}})/\cl A}  
\]
is an isomorphism by the Hodge--Tate comparison. This in turn follows from the isomorphism 
\[
 \widehat \L_{\cl A \gr{P_\infty^{i, \text{sat}}} / \cl A } \isom \widehat \L_{(\cl A \gr{P_\infty^{i, \text{sat}}}, P_\infty^{i, \text{sat}})/\cl A}
\]
by base change. 
For the second claim, let us note that $(R_\infty^i, P_\infty^i)$ is a semiperfectoid integral pre-log ring and $P_\infty^i$ is semiperfect, therefore $\beta_{\text{perf}}$ is an isomorphism by Corollary \ref{cor:p_saturation_induces_iso_on_perfect_prism} (using Part (3) of Lemma \ref{lemma:control_of_saturated_monoid}). 

Now we have constructed an isomorphism 
\[
\eta_{(T, M)}\colon R \Gamma_{\ket} (\spec(R[\frac{1}{p}], P)^a, \Z/p^n) \isom \big(\Prism_{(R, P)/A} [\frac{1}{d}]/p^n\big)^{\phi = 1}
\]
which \textit{a priori} depends on the choice of the surjection $(T, M) \ra (R, P)$. In order to get a canonical isomorphism, we follow \cite[Remark 4.3]{Bhatt_notes}. Namely we consider the index category $I$ of a pair consisting of a sufficiently large finite subset of $R$  (resp. of $P$)  whose elements topologically generate $R$ (resp. $P$). Since each $s \in I$ provides a surjection as in (\ref{eq:surjection_to_R_P}), we get an isomorphism $\eta_s$ as above. In the (filtered) colimit we obtain a canonical isomorphism  
\begin{equation} \label{eq:canonical_eta}
\eta = \text{colim } \eta_s\colon R \Gamma_{\ket} (\spec(R[\frac{1}{p}], P)^a, \Z/p^n) \isom \big(\Prism_{(R, P)/A} [\frac{1}{d}]/p^n\big)^{\phi = 1}.
\end{equation} 

For a general saturated pre-log ring $(R, P)$ where $R$ is $p$-complete with bounded $p^\infty$-torsion, we may write it as a filtered colimit of pre-log rings $(R_i, P_i)$ where each $R_i$ is topologically finitely generated over $\cl A$ with bounded $p^{\infty}$-torsion and $P_i$ is an fs monoid. One way to achieve this for the ring $R$ is to write $R = \varinjlim_{i} R_i^{\circ}$ as a filtered colimit of (all of its) finitely generated subrings, and take $R_i \coloneqq \im (\widehat{R_i^{\circ}} \ra R)$ where $ \widehat{R_i^{\circ}}$ is the (classical) $p$-adic completion of $R_i^{\circ}$. Then $R_i$ is a (classically) $p$-complete subring of $R$ that contains $R_i^{\circ}$ and is topologically finitely 
generated over $\cl A$ with bounded $p^\infty$-torsion. Then by taking the filtered colimit of the isomorphisms $\eta_i$ in (\ref{eq:canonical_eta}), we get the desired isomorphism in the statement of the theorem. 
\eproof 
 
 \br 
The above proof of the \'etale comparison implies the following result: if $(R, P)$ is an integral pre-log ring over the perfectoid ring $A/I$ and $R$ is $p$-complete with bounded $p^{\infty}$-torsion, then
\[
\big(\Prism_{(R, P)/A} [\frac{1}{d}]/p^n\big)^{\phi = 1} \cong
\big(\Prism_{(R^{p\text{-sat}}, P^{p\text{-sat}})/A} [\frac{1}{d}]/p^n\big)^{\phi = 1}
\]
for every $n\geq 1$, where $(R^{p\text{-sat}}, P^{p\text{-sat}})$ is the $p$-saturation of $(R, P)$, which is (derived) $p$-complete. 
\er


\newpage
\section{Log diamonds and Kummer \'etale local systems} \label{sec:log_diamonds} \noindent

In this section we develop the theory of log diamonds. The reason for introducing this framework is twofold:  first we use this language to give a reformulation and a globalization of the \'etale comparison theorem (Theorem \ref{thm:etale_comp}). Secondly, we use this framework to relate $F$-crystals on the absolute (saturated) log prismatic site to Kummer \'etale local systems on the log diamond,   generalizing a result of Bhatt--Scholze which we now recall.  

Let $\Perf$ denote the category of perfectoid spaces in characteristic $p$, on which Scholze introduced the following well-behaved (in particular, subcanonical) topologies: the pro-\'etale topology and the v-topology. Also recall that a diamond is a pro-\'etale sheaf on $\Perf$ that is the quotient of a perfectoid space by a pro-\'etale equivalence relation. Any analytic (pre-)adic space $X$ over $\spa \Z_p$ has the associated (locally spatial) diamond $X^{\diamondsuit}$.  
For any diamond $X$, Scholze defined the quasi-pro-\'etale site $X_{\qproet}$ (and the v-site), which is a natural site to consider $\Z_p$-local systems. Bhatt--Scholze found a quite general relation between $\Z_p$-local systems of the generic fiber of any bounded $p$-adic formal scheme $X$ and the absolute prismatic site of $X$ in \cite[Section 3]{BS_crystal}. Here is a rough version of one of their results: 

\begin{theorem*}[Bhatt--Scholze]
Let $X$ be a bounded $p$-adic formal scheme. 
There is a natural equivalence
\[
\textnormal{Vect}(X_{\Prism}, \mO_{\Prism}[1/\mI]_p^{\wedge})^{\phi=1} \simeq 
\textnormal{Loc}_{\Z_p}(X_{\eta, \qproet}^{\diamondsuit}),  
\]
where $\textnormal{Vect}(X_{\Prism}, \mO_{\Prism}[1/\mI]_p^{\wedge})^{\phi=1}$ denotes the category of Laurent $F$-crystals and $\textnormal{Loc}_{\Z_p}(X_{\eta, \qproet}^{\diamondsuit})$ denotes the category of (quasi-pro-\'etale) $\Z_p$-local systems. 
\end{theorem*}

After discussing the \'etale comparison, we will explain a log version of the above theorem (Theorem \ref{etale_realization_is_equivalence}) using pro-Kummer-\'etale local systems. 

\subsection{The quasi-pro-\'etale topology and the structure sheaf} \noindent 

\noindent 
Let us recall the definition of quasi-pro-\'etale maps following \cite{diamond}. 
First recall that a perfectoid space $X$ is strictly totally disconnected if $X$ is qcqs and every \'etale cover of $X$ splits. Any affinoid perfectoid space $X$ admits an affinoid pro-\'etale surjection $\widetilde{X}\to X$ from a strictly totally disconnected perfectoid space $\widetilde{X}$ (see \cite[Lemma 7.18]{diamond}). 

\bd[\cite{diamond}]
\noindent
\be 
\item A locally separated map of diamonds $Y'\to Y$ is quasi-pro-\'etale if for any strictly totally disconnected perfectoid space $X$ over $Y$, the pullback $Y'\times_Y X$ is representable by a perfectoid space and the map $Y'\times_Y X\to X$ is pro-\'etale. 
\item The quasi-pro-\'etale site $Y_{\qproet}$ of $Y$ consists of quasi-pro-\'etale maps $Y'\to Y$, and coverings are jointly surjective maps.
\ee 
\ed 

\br 
Every diamond $Y$ admits a surjective quasi-pro-\'etale map $Y'\to Y$ from a perfectoid space with pro-\'etale equivalence relation, and any pro-\'etale sheaf that admits a surjective quasi-pro-\'etale map from a perfectoid space (or a diamond) is a diamond. 
\er

Next, we consider the structure sheaf on $Y_{\qproet}$ following \cite{MannWerner}. It depends on the following additional datum: let $\spd \Q_p$ denote the diamond attached to $\spa(\Q_p, \Z_p)$, and we fix a map $Y\to \spd \Q_p$.\footnote{If $Y$ comes from an analytic (pre-)adic space over $\spa(\Q_p, \Z_p)$, then it comes naturally with $Y\to \spd \Q_p$.}  The completed integral structure sheaf $\widehat{\mO}^+_Y$ and the completed structure sheaf $\widehat{\mO}_Y$ are defined as follows: for an affinoid perfectoid space $Y'=\spa (R, R^+)$ quasi-pro-\'etale over $Y$, the map to $\spd \Q_p$ determines an untilt $\spa (R^{\sharp}, R^{\sharp+})$, then we have 
\[
\widehat{\mO}^+_Y (Y')=R^{\sharp+}, \quad \text{and }\:   \widehat{\mO}_Y (Y')=R^{\sharp}.
\]

\subsection{Log diamonds} \noindent

\noindent
As the (completed) structure sheaf lives in the quasi-pro-\'etale site, it is also natural to consider a log structure on the quasi-pro-\'etale site. 

\bd
A \emph{log (locally spatial) diamond} over $\Q_p$ is a (locally spatial) diamond $Y$ with a map $Y\to \spd \Q_p$ and a log structure $M_Y \to \widehat{\mO}_Y$ on the quasi-pro-\'etale site. 
\ed

The following treatment of charts is slightly nonstandard from a general theory, and based on the one for log adic spaces in \cite{DLLZ}. 

\bd
Let $(Y, M_Y)$ be a log diamond. 
\be 
\item A chart of $M_X$ is a map of monoids $P\to \Gamma (Y_{\qproet}, M_Y)$ from a monoid $P$ that induces an isomorphism of log structures $P_Y^a \cong M_Y$ and whose composite with $\Gamma (Y_{\qproet}, M_Y)\to \Gamma (Y_{\qproet}, \widehat{\mO}_Y)$ factors over $\Gamma (Y_{\qproet}, \widehat{\mO}_Y^+)$. The chart is said to be integral (resp. saturated, resp. fine, resp. fs) if $P$ is  integral (resp. saturated, resp. fine, resp. fs).  
\item We say that $(Y, M_Y)$ is quasi-coherent (resp. integral quasi-coherent, resp. saturated quasi-coherent, resp. fine, resp. fs) if $M_Y$ admits a (resp. integral, resp. saturated, resp. fine, resp. fs) chart quasi-pro-\'etale locally on $Y$. 
\ee 
\ed

Note that, if $(Y, M_Y)$ is integral quasi-coherent (resp. saturated quasi-coherent), then $\Gamma (U, M_Y)$ is integral (resp. saturated) for every $U \in Y_{\qproet}$. 

\begin{convention}
We will make the following assumption in the rest of the section: for a map $(X, M_{X})\to (Y, M_Y)$ of quasi-coherent (resp. integral quasi-coherent, resp. saturated quasi-coherent, resp. fine, resp. fs) log diamonds, we always assume that quasi-pro-\'etale locally on $X$ and $Y$ there is a map of (resp. integral, resp. saturated, resp. fine, resp. fs) charts
\[
\begin{CD}
Q @>>> \Gamma (Y, M_Y) \\
@VVV @VVV \\
P @>>> \Gamma (X, M_X). 
\end{CD}
\]
Existence of such charts is usually proved under a finiteness condition, but we do not consider this problem and only mention the following example. 
\end{convention}

\begin{example}
Let $X$ be an adic space over $\spa (\Q_p, \Z_p)$ that is either locally noetherian or perfectoid. In this case, we also have the noncompleted structure sheaves $\mO_X^+, \mO_X$ on the \'etale site of $X_{\eta}^{\diamondsuit}$ and the $\widehat{\mO}^+_X$ is the $p$-adic completion of $\nu^{-1}\mO_X^+$ under $\nu\colon X^{\diamondsuit}_{\qproet}\to X^{\diamondsuit}_{\ett}$ (see \cite[Lemma 2.7]{MannWerner}). So, any fs log structure on $X_{\ett}$ in the sense of \cite[Definition 2.3.1]{DLLZ} induces an fs log structure on $X^{\diamondsuit}_{\qproet}$ and gives rise to an fs log diamond $(X, M_X)^{\diamondsuit}$. Any map $(Y, M_Y)\to (X, M_X)$ between such fs log adic spaces induces a map $(Y, M_Y)^{\diamondsuit}\to (X, M_X)^{\diamondsuit}$ of fs log diamonds \cite[Proposition 2.3.22]{DLLZ}. 

Similarly, if $(X, M_X)$ is an fs log $p$-adic formal scheme over $\Z_p$, then the diamond generic fiber $X^{\diamondsuit}_{\eta}$ over $\spd \Q_p$ has a natural fs log structure induced by $M_X$ and gives rise to a log diamond $(X, M_X)_{\eta}^{\diamondsuit}$. This construction is functorial in the above sense. 
\end{example}

We want to define a Kummer analogue of the quasi-pro-\'etale site. 
We first check that the saturation and fiber products exist:

\bl
Let $(X, M_X)$ be an affinoid perfectoid space $X$ over $\Q_p$ with a log structure $M_X$ and a chart $P\to \Gamma (X, \widehat{\mO}^+_X)$. Then, the saturation $(X^{\textup{sat}}, M_X^{\textup{sat}})$ exists as a log locally spatial diamond over $\Q_p$, and its underlying diamond is a perfectoid space. 
\el

\bproof
Write $X=\spa (R, R^+)^{\diamondsuit}$ with $R$ over $\Q_p$. Let 
\[S = R^+ \otimes_{\Z_p [P]} \Z_p [P^{\textup{sat}}]\] denote the tensor product.
As $S$ is integral over the perfectoid ring $R^+$, it has the perfectoidization $S\to S_{\textnormal{perfd}}$ by \cite[Theorem 1.17 (1)]{BS}. 
The perfectoid ring $S_{\textnormal{perfd}}$ and a pre-log structure $P^{\textup{sat}}\to S_{\textnormal{perfd}}$ determines a log affinoid perfectoid $(X^{\textup{sat}}, M_X^{\textup{sat}})$ space over $\Q_p$ satisfying the universality of the saturation. 
\eproof

\bc
Let $(Y, M_Y)$ be a quasi-coherent log (locally spatial) diamond over $\Q_p$. 
Then its saturation $(Y^{\textup{sat}}, M_Y^{\textup{sat}})$ exists as a saturated quasi-coherent log (locally spatial) diamond. 
\ec

\bproof
There is a surjective quasi-pro-\'etale map $X\to Y$ from the disjoint union of affinoid perfectoid spaces $X$ with pro-\'etale equivalence relation. As $M_Y$ is quasi-coherent, we may assume that $M_Y$ admits a chart $P\to M_Y$ on $X$. We have the saturation $(X^{\textup{sat}},M_Y^\textup{sat})$ of $(X, M_Y)$, with underlying perfectoid space. The base changes of both of the projections in 
\[ X\times_Y X\rightrightarrows X\] give the saturation $((X\times_Y X)^{\textup{sat}}, M_Y^{\textup{sat}})$ of $(X\times_Y X, M_Y)$, 
with two strict pro-\'etale maps
\[
((X\times_Y X)^{\textup{sat}}, M_Y^{\textup{sat}})\rightrightarrows
(X^{\textup{sat}},M_Y^\textup{sat}). 
\]
The quotient of $X^{\textup{sat}}$ by the equivalence relation $(X\times_Y X)^{\textup{sat}}$ is a diamond by \cite[Proposition 11.8]{diamond}. Let $Y^{\textup{sat}}$ denote this quotient, with a natural log structure $M_Y^{\textup{sat}}$. The log diamond $(Y^{\textup{sat}}, M_Y^{\textup{sat}})$ is the saturation of $(Y, M_Y)$. 

If $Y$ is spatial, we can take $X\to Y$ to be universally open and then $X^{\textup{sat}}\to Y^{\textup{sat}}$ is also universally open. So, $Y^{\textup{sat}}$ is spatial by \cite[Proposition 11.24]{diamond}. The case of locally spatial $Y$ then follows. 
\eproof

As fiber products exist in the category of (locally spatial) diamonds over $\Q_p$, we deduce the following.

\bc
Fiber products exist in the category of saturated quasi-coherent (resp. fs) log (locally spatial) diamonds over $\Q_p$. 
\ec

\begin{convention}
All log diamonds are saturated quasi-coherent from now on, and in particular we will take the fiber products as saturated quasi-coherent log diamonds and we write $Y'\times_Y^{\textup{sat}}X$ for the underlying diamond.  
\end{convention}

\subsection{The quasi-pro-Kummer-\'etale topology} \noindent

\noindent
We first introduce a logarithmic analogue of the notion of strictly totally disconnected perfectoid spaces. 

\bl
For any strictly totally disconnected perfectoid space $X$, the vanishing of cohomology  $H^1 (X, \widehat{\mO}_X^\times)=0$ holds for the pro-\'etale topology. In particular, for any integral quasi-coherent log structure $M_X$, the map 
\[\Gamma (X, M_X)\to \Gamma (X, M_X/M_X^\times)\] is surjective. 
\el

\bproof
It is known \cite[Theorem II, 3.5.8]{Kedlaya-Liu_II} (or \cite[Lemma 17.1.8]{Berkeley}) that any vector bundle on a perfectoid space for the pro-\'etale topology (or even v-topology) descends to a vector bundle for the \'etale topology (and, in fact, analytic topology). As $X$ is strictly totally disconnected, any line bundle on $X$ for the \'etale topology is trivial, so $H^1 (X, \widehat{\mO}_X^\times)=0$. To deduce the second statement, observe that for any $m\in \Gamma (X, M_X/M_X^{\times})$, the quasi-pro-\'etale sheaf of lifts $\tilde{m}\in \Gamma (-, M_X)$ of $m$ is an $\widehat{\mO}_X^\times$-torsor. We showed that such a torsor is trivial. 
\eproof

\bd
A strictly totally disconnected log perfectoid space is a strictly totally disconnected perfectoid space $X$ with a saturated quasi-coherent log structure $M_X$ such that  $M_X/M_X^\times$ is uniquely divisible. 
In particular, $X$ is affinoid perfectoid and $\Gamma (X, M_X)$ is saturated and divisible (since $\Gamma (X, \widehat{\mO}_X^\times)$ is also divisible). 
\ed

\br
In the definition of a strictly totally disconnected log perfectoid space, it suffices to assume $M_X/M_X^\times$, equivalently $M_X$, is divisible: $M_X/ M_X^\times$ is a sheaf of sharp saturated monoids and hence it must be uniquely divisible.
In particular, if $X$ is a strictly totally disconnected perfectoid space and $P$ is a divisible saturated monoid with a map $P\to \Gamma (X, \widehat{\mO}_X^+)$, then the associated log diamond $(X, P)^a$ is a strictly totally disconnected log perfectoid space.
\er

\bd Let $f\colon (Y', M_{Y'})\to (Y, M_Y)$ be a locally separated map of saturated quasi-coherent log diamonds. We say that $f$ is  quasi-pro-Kummer-\'etale (resp. Kummer-\'etale, resp. finite Kummer-\'etale) if for any map $(X, M_X)\to (Y, M_Y)$ from a strictly totally disconnected log perfectoid space $(X, M_X)$, the saturated fiber product $Y'\times_Y^{\textup{sat}} X$ is a perfectoid space and the map 
\[(Y'\times_Y X, M_{Y'\times_Y X})\to (X, M_X)
\] is strict and pro-\'etale (resp. \'etale, resp. finite \'etale). 
\ed

\br As in \cite{diamond}, we always assume that quasi-pro-Kummer-\'etale maps are locally separated. This is automatic if $Y'$ is a perfectoid space (and $Y$ is a diamond). 
\er

\bp
\begin{enumerate}
    \item A strict map of saturated quasi-coherent log diamonds is quasi-pro-Kummer-\'etale (resp. Kummer-\'etale. resp finite Kummer-\'etale) if and only if its underlying map of diamonds is pro-\'etale (resp. \'etale, resp. finite \'etale). 
    \item Any pullback of a quasi-pro-Kummer-\'etale (resp. Kummer-\'etale, resp. finite Kummer-\'etale) map is quasi-pro-Kummer-\'etale (resp. Kummer-\'etale, resp. finite Kummer-\'etale).  
    \item If $f\colon (Y_1, M_{Y_1}) \to (Y_2, M_{Y_2}), g\colon (Y_2, M_{Y_2})\to (Y_3, M_{Y_3})$ are quasi-pro-Kummer-\'etale (resp. Kummer-\'etale, resp. finite Kummer-\'etale), then so is $g\circ f$. 
    \item Let $f\colon (Y_1, M_{Y_1}) \to (Y_2, M_{Y_2}), g\colon (Y_2, M_{Y_2})\to (Y_3, M_{Y_3})$ be maps of log diamonds. If $g$ and $g\circ f$ are quasi-pro-Kummer-\'etale (resp. Kummer-\'etale, resp. finite Kummer-\'etale), then so is $f$. 
\end{enumerate}
\ep

\bproof
(3) can be proved in the same way as in the proof of \cite[Proposition 10.4 (i)]{diamond}, and the others are more straightforward. 
\eproof

It is convenient to define the surjectivity of a quasi-pro-Kummer-\'etale map as follows:

\bd
Let $(Y', M_{Y'})\to (Y, M_Y)$ be a quasi-pro-Kummer-\'etale map of log diamonds. We say that it is surjective if $Y'\times_Y^{\textup{sat}} X \to X$ is surjective for any strictly totally disconnected log perfectoid space $(X, M_X)$. The class of surjective quasi-pro-Kummer-\'etale maps is stable under base change and composition. 
\ed

\bd[The quasi-pro-Kummer-\'etale site]
Let $(Y, M_Y)$ be a saturated quasi-coherent log diamond. The quasi-pro-Kummer-\'etale site $(Y, M_Y)_{\qpket}$ consists of quasi-pro-Kummer-\'etale maps $(Y', M_{Y'})\to (Y, M_Y)$, and coverings are jointly surjective maps (note that a map between objects are automatically quasi-pro-Kummer-\'etale and the surjectivity for such a map is defined as above). 
\ed

\br
Let $(Z, M_Z) \to (Y, M_Y)$ be a a map of saturated quasi-coherent log diamonds. The pullback functor induces a morphism of sites $(Z, M_Z)_{\qpket}\to (Y, M_Y)_{\qpket}$. 
\er

\br
It is possible to define the Kummer-\'etale site for saturated quasi-coherent log locally spatial diamonds in a similar way, but we do not need it. 
\er

Let us compare this site with more classical ones. We start with a simple covering. 

\bl
\begin{enumerate}
\item For any integer $n\geq 1$, the following natural map
\[
f\colon (\spd (\Q_p \langle T^{1/n} \rangle, \Z_p \langle T^{1/n}\rangle), T^{\N/n})\to (\spd (\Q_p \langle T \rangle, \Z_p \langle T\rangle), T^{\N})
\]
induces a surjective finite Kummer-\'etale map of log diamonds. 
\item The map
\[
(\spd (\Q_p \langle T^{\Q_{\geq 0}} \rangle, \Z_p \langle T^{\Q_{\geq 0}}\rangle), T^{\Q_{\geq 0}})\to (\spd (\Q_p \langle T \rangle, \Z_p \langle T\rangle), T^{\N})
\]
induces a surjective quasi-pro-Kummer-\'etale map of log diamonds. 
\end{enumerate}
\el

\bproof
It suffices to prove (1). Let $(Y, M_Y)\coloneqq (\spd (\Q_p \langle T \rangle, \Z_p \langle T\rangle), T^{\N})^a$ denote the log diamond associated with $(\spd (\Q_p \langle T \rangle, \Z_p \langle T\rangle), T^{\N})$. 
Take a map $(X, M_X)\to (Y, M_Y)$ from a strictly totally disconnected log perfectoid space $(X, M_X)$. As $\Gamma (X, M_X)$ is divisible, there exists $m\in \Gamma (X, M_X)$ such that $m^n$ equals the image of $T$. Then the map from $(X, M_X)$ factors over
\[
g\colon (Z, M_Z)\coloneqq (\spd (\Q_p \langle m \rangle, \Z_p \langle m\rangle), m^{\N})^a \to
(Y, M_Y)=(\spd (\Q_p \langle T \rangle, \Z_p \langle T\rangle), T^{\N})^a, 
\]
which is induced from $T \mapsto m^n$. The saturated fiber product of $f$ and $g$ can be computed using the following standard fact: the saturation of the coproduct of $T^{\N}\hookrightarrow T^{\N/n}$ and $T^{\N}\to m^{\N}$ is isomorphic to $\Z/n\Z \oplus m^{\N}$.  
We see that the saturated fiber product of $f$ and $g$ is isomorphic to the disjoint union of copies of $(Z, M_Z)$ labelled by $\Z/n\Z$. 
Therefore, the saturated fiber product of $f$ and the map $(X, M_X)\to (Y, M_Y)$ is the disjoint union of $(X, M_X)$ labelled by $\Z/n\Z$, and hence the map to $(X, M_X)$ is surjective and finite \'etale. This shows that $f$ is surjective and finite Kummer-\'etale as desired. 
\eproof

More generally, the following holds:

\bp
Let $P$ be an fs monoid. For an integer $n\geq 1$, let $P^{1/n}$ denote $P$ with a map $P\to P; a\mapsto a^n$. Let $P_{\Q_{\geq 0}}$ denote the colimit of all $P^{1/n}$. 
\begin{enumerate}
\item For any integer $n\geq 1$ and a saturated submonoid $Q\subset P^{1/n}$ containing $P$, the following natural map
\[
(\spd (\Q_p \langle Q \rangle, \Z_p \langle Q\rangle), Q)\to (\spd (\Q_p \langle P \rangle, \Z_p \langle P\rangle), P)
\]
induces a surjective finite Kummer-\'etale map of log diamonds. 
\item The map
\[
(\spd (\Q_p \langle P_{\Q_{\geq 0}} \rangle, \Z_p \langle P_{\Q_{\geq 0}}\rangle), P_{\Q_{\geq 0}})\to (\spd (\Q_p \langle P \rangle, \Z_p \langle P\rangle), P)
\]
induces a surjective quasi-pro-Kummer-\'etale map of log diamonds. 
\end{enumerate}
\ep
 
 \bproof
It suffices to prove (1). 
Assume $P$ is a sharp fs monoid. 
We first prove the following: let $f\colon P\to M$ be a map to a divisible saturated monoid $M$ with $M/M^\times$ being uniquely divisible, then there is a map $Q \to M$ for any $n\geq 1$ extending $f$. Observe that there is a unique map $\overline{f}_n \colon Q\to M/M^\times$ extending $\overline{f}\colon P\to M/M^\times$, and we want to lift it to $Q\to M$. Choose any extension 
\[
f^{\textup{gp}}_n \colon Q^{\textup{gp}}\to M^{\textup{gp}}
\]
of $f^{\textup{gp}} \colon P^{\textup{gp}}\to M^{\textup{gp}}$; this is possible since $P^{\textup{gp}}$ is free and $M$ is divisible. The composite of $f^{\textup{gp}}_n$ with $M^{\textup{gp}}\to M^{\textup{gp}}/M^\times$ coincides with $\overline{f}_n^{\textup{gp}}$ as $M/M^\times$ is uniquely divisible. In particular, the composite maps $Q$ to $M/M^\times$, and hence $f^{\textup{gp}}_n$ maps $Q$ to $M$. So, $f^{\textup{gp}}_n$ restricts to an extension of $f$. 

Now, the proof of the previous lemma works for $P$: apply the result in the first paragraph to $M=\Gamma (X, M_X)$ for a strictly totally disconnected log perfectoid $(X, M_X)$ and recall that the saturated self-coproduct of $P\to Q$ is isomorphic to $Q^{\textup{gp}}/ P^{\textup{gp}}\oplus Q$. 

Now consider a general fs monoid $P$. Choose a section $\overline{P}\coloneqq P/P^\times \to P$, and observe that $P\to P^{1/n}$ factors through $P\to P\times_{\overline{P}}\overline{P}^{1/n}$ (note that $P\times_{\overline{P}}\overline{P}^{1/n}$ is also fs ). From what we have shown, the map $P\to P\times_{\overline{P}}\overline{P}^{1/n}$ induces a finite Kummer-\'etale map of log diamonds. The map $P\times_{\overline{P}}\overline{P}^{1/n}\to P^{1/n}$ is induced by $P^\times \to (P^\times)^{1/n}$, and hence it induces a (strict) finite \'etale map of log diamonds. 
\eproof

\bc
Let $P$ be an fs monoid. For any Huber pair $(R, R^+)$ over $(\Q_p, \Z_p)$ with a map $P\to R^+$, 
taking the associated log diamond induces a morphism of sites
\[
(\spd (R, R^+), P)^a_{\qpket} \to (\spec R, P)^a_{\ket}. 
\]
Similarly, if $P$ is a divisible saturated monoid, we have a morphism of sites
\[
(\spd (R, R^+), P)^a_{\qpket} \to (\spec R)_{\ett}
\]
and a natural equivalence
\[
(\spd (R, R^+), P)^a_{\qpket} \cong \spd (R, R^+)_{\qproet}. 
\]
\ec

\bproof
We freely use results from \cite{Kato2}. 
Any morphism in $(\spec R, P)^a_{\ket}$ is \'etale locally a finite Kummer-\'etale map modelled on some map $P\to Q\subset P^{1/n}$ for some $n$. So, it induces a Kummer-\'etale map of associated log diamonds. Recall that coverings in $(\spec R, P)^a_{\ket}$ is defined as jointly surjective maps of underlying schemes, and this class is shown to be stable under (saturated) base change. Let $(Y_i)_{i\in I}\to Y$ be a covering in $(\spec R, P)^a_{\ket}$, and we need to show that it induces the induced Kummer-\'etale map of log diamonds is surjective. It suffices to show the case of affine $Y$, and then we may replace $I$ by a finite subset since the images of $Y_i$ is a Zariski open covering. Now, we can find a map $P\to Q$ of Kummer type such that the (saturated) base change to $(\Z[Q], Q)$ gives a strict \'etale covering. Then, its base change to any strictly totally disconnected space is surjective. So, we obtain a continuous functor of sites. 

Finally, note that saturated fiber products commutes with taking log diamonds; recall that we only use perfectoidization in our proof of existence of saturated fiber products of log diamonds. This implies that the continuous functor induces a morphism of sites. 

Next, assume $P$ is a divisible saturated monoid. The only nontrivial part is the claim on equivalence. But observe that any quasi-pro-Kummer-\'etale map of the form of $(Y', M_{Y'})\to (Y, P)^a$ has the quasi-pro-\'etale underlying map of diamonds: any map $X\to Y$ from a strictly totally disconnected perfectoid space $X$ gives rise to a strict totally disconnected log perfectoid space $(X, P)^a$. Moreover, $Y'\times_X Y$ with pullback of $M_{Y'}$ is already saturated, and hence must be strict over $(X, P)^a$. By taking a suitable union of maps of the form $X\to Y$ to obtain a quasi-pro-\'etale cover, we see that $P\to \Gamma (Y', M_{Y'})$ is a chart of $M_{Y'}$ and $(Y', M_{Y'})\to (Y, P)^a$ is strict. This proves the desired equivalence. 
\eproof

\br
Let $P$ be an fs monoid and $P_{\infty}$ is a divisible saturated monoid over $P$. For any pre-log ring $(R, P)$ and the saturated base change $(S, P_{\infty})$, the base change functor induces a morphism of sites $(\spec S)_{\ett}\to (\spec R, P)_{\ket}$. 
\er

The following result is an evidence that our site is a right one:

\bt[Comparison with Kummer \'etale cohomology of log schemes] \label{comparing_diamond_to_spec}
Let $\Lambda$ be a torsion abelian group. 
Let $R$ be a $p$-complete ring with bounded $p^{\infty}$-torsion, and let $(R[1/p], R^+)$ denote the associated Huber pair. 
For any fs monoid $P$ with a map $P\to R$, the comparison map induces an isomorphism
\[
R\Gamma_{\ket} ((\spec R[1/p], P)^a, \Lambda) \cong
R\Gamma_{\qpket}((\spd (R[1/p], R^+), P)^a, \Lambda). 
\]
\et

\br
A similar comparison holds for the Kummer \'etale cohomology of log adic spaces defined in \cite{DLLZ}, cf. Lemma \ref{lemma:comparing_adic_generic_to_spec}. 
\er

\bproof
Take some surjection $\N^{\oplus J}\to P$ and the associated $P\to P_{\infty}$ as in Proposition \ref{prop:flat_descent_ket}, and let $(S, P_{\infty})$ denote the (classically) completed saturation of the base change of $(R, P)$ along $\Z[P]\to \Z[P_{\infty}]$. 
By Proposition \ref{prop:flat_descent_ket}, the descent of Kummer \'etale cohomology holds for the map
\[
 (\spec S[1/p], P_\infty)^a\to (\spec R[1/p], P)^a. 
\]
On the other hand, the descent of quasi-pro-Kummer-\'etale cohomology holds for the map
\[
(\spd (S[1/p], S^+), P_{\infty})^a \to (\spd (R[1/p], R^+), P)^a
\]
as it is quasi-pro-Kummer-\'etale, where $(S[1/p], S^+)$ denote the Huber pair associated with $S$. 
So, we may replace $(R, P)$ by $(S, P_{\infty})$ and (quasi-pro-)Kummer-\'etale sites by (quasi-pro-)\'etale sites. 
Then, the claim follows from the arc descent \cite[Corollary 6.17]{BM_arc}. 
\eproof

\bc[Reformulation of \'etale comparison] \label{etale_comp_diamond}
Let $(A, I = (d))$ be a perfect prism.  Let $R$ be a $p$-complete $A/I$-algebra with  bounded $p^{\infty}$-torsion and let $P$ be an fs monoid with a map $P\to R$. 
Then, for each $n \ge 1$, there is a canonical isomorphism 
\[ 
R \Gamma ((\spd (R[\frac{1}{p}], R^+), P)^a_{\qpket}, \Z/p^n) \cong 
(\Prism_{(R, P)/A} [\frac{1}{d}] /p^n )^{\phi = 1}
\]
which is functorial in $(R, P)$. 
\ec

\br
If $P$ is only saturated, then the statement continues to hold at least if we define the quasi-pro-Kummer-\'etale cohomology as the filtered colimit of the fs case. We do not consider if the cohomology of the quasi-pro-Kummer-\'etale site works in such generality. 
\er

\subsection{Global \'etale comparison} \label{ss:etale_comp_diaomnd_global}
\noindent 

\noindent 
Now we are ready to globalize the \'etale comparison theorem proved in Section \ref{sec:etale}. For the setup, let us consider the following data: 
\bi 
\item let $(A, I=(d), M_{0})$ be a bounded pre-log prism where $(A, I)$ is perfect and $M_{0}$ is an fs monoid; 
\item let $(X_0, M_{X_0})$ be a smooth fs log $p$-adic formal scheme over $(A/I, M_0)$ where $X_0$ is  qcqs and the mod $p$ fiber of the map 
\begin{equation} \label{eq:saturated_morphism}
(X_0, M_{X_0}) \lra (\spf A/I, M_0)^a
\end{equation}
is of Cartier type; 
\item let $(A, I, M_A)$ be a saturated pre-log prism such that its associated log prism \[(A, I, M_{\spf A})=(A, I, M_A)^{a}\] is a perfect\footnote{As in the definition of perfect ``log prisms'', this means the Frobenius lift $\phi$ is an isomorphism of log formal schemes.} log prism, and is equipped with a map 
\[(A, M_{0})\to (A, \Gamma (\spf A, M_{\spf A})).\]  
\item Let $(X, M_X)$ be the base change of $(X_0, M_{X_0})$ to $(A/I, M_A)$. More generally, for any  submonoid $M_i \subset M_A$ such that $\Gamma(\spf A, M^a_{i, \spf A})$ contains the image of $M_0$, let $(X_i, M_{X_i})$ denote the base change of $(X_0, M_{X_0})$ to $(A/I, M_i)$. They fit in the following commutative diagram of cartesian squares
\[
\begin{tikzcd}
(X_0, M_{X_0})  \arrow[d] 
& (X_i, M_{X_i})   \arrow[d]  \arrow[l]
& (X, M_{X}) \arrow[d] \arrow[l]
\\ 
(\spf A/I, M_0)^a 
& (\spf A/I, M_i)^a  \arrow[l]
& (\spf A/I, M_A)^a \arrow[l] 
\end{tikzcd}
\]
in the category of log $p$-adic formal schemes. Note that the underlying scheme 
\[ X_0 = X_i = X\] does not change. Also note that, by Remark \ref{remark:Cartier_type_saturated} below, the map (\ref{eq:saturated_morphism}) is saturated, thus the base changes above agree with base changes taken in the category of saturated log $p$-adic formal schemes \cite[Proposition II.2.13]{Tsuji}. 
\ei 
\br  \label{remark:Cartier_type_saturated}
In our setting, the assumption that the mod $p$ fiber of the map (\ref{eq:saturated_morphism}) is of Cartier type is equivalent to the condition that (\ref{eq:saturated_morphism}) is a saturated morphism of fs log $p$-adic formal schemes in the sense of \cite[Definition II.2.10]{Tsuji} (naturally generalized to log formal schemes). See \cite[Proposition II.2.14, Theorem II.3.1]{Tsuji}.  
\er 
Given the setup above, we define\footnote{Note that any fs sub log structure of $M_{\spf A}$ is contained in $M_{i, \spf A}^a$ for some $i$ as $\spf A$ is quasicompact, so we can make the definition more canonical and it only depends on $M_{\spf A}$.} the quasi-pro-Kummer-\'etale cohomology of the log diamond $(X, M_X)_{\eta}^{\diamondsuit}$ of the generic fiber of $(X, M_X)$ as follows:
\[
R\Gamma_{\qpket}((X, M_X)_{\eta}^{\diamondsuit}, \Z/p^m\Z) \\
\coloneqq
\varinjlim_i R\Gamma((X_i, M_{X_i})_{\eta, \qpket}^{\diamondsuit}, \Z/p^m\Z), 
\]
here $M_i\subset M_A$ run through all fs submonoids of $M_A$ such that $\Gamma (\spf A, M_{i, \spf A}^a)$ contain the image of $M_0$.

\bt[Globalization of \'etale comparison] \label{etale_comp_diaomnd_global}  In the setup above, there are  functorial isomorphisms
\begin{align*}
R\Gamma_{\qpket} ((X, M_X)^{\diamondsuit}_{\eta}, \Z/p^m \Z)
&\cong
(R\Gamma_{\Prism}((X, M_X) /(A, M_A)))[\frac{1}{d}]/p^m)^{\phi=1} \\
&\cong 
(R\Gamma_{\Prism}((X_0, M_{X_0})/(A, M_0)))[\frac{1}{d}]/p^m)^{\phi=1}. 
\end{align*}
\et

\bproof
Let us first prove that the map of presheaves
\[
R\Gamma_{\qpket} ((U_{\eta}^{\diamondsuit}, \Gamma (U, M_{X}))^a, \Z/p^m) \lra 
R\Gamma_{\qpket} ((U, M_U)_{\eta}^{\diamondsuit}, \Z/p^m), 
\]
where $M_U$ is the restriction of $M_X$, induces an isomorphism after \'etale sheafification on $U$. Assuming that some fs smooth chart $P\to \Gamma (U, M_{X_0})$ (over $M_0$) exists,  we need only show that the natural map
\[
R\Gamma_{\qpket} (((U, P_{M_A})^a)_{\eta}^{\diamondsuit}, \Z/p^m) \to
R\Gamma_{\qpket} ((U_{\eta}^{\diamondsuit}, \Gamma (U, M_{X}))^a, \Z/p^m)
\]
induces an isomorphism after \'etale sheafification on $U$, here $P_{M_A}$ denotes the pushout of monoids $P_{M_A} \coloneqq P \oplus_{M_0} M_A$. By Corollary \ref{etale_comp_diamond}, it is enough to prove that the natural maps
\[
(\Prism^{\L}_{(R, P_{M_A})/A}[1/d]/p^m)^{\phi=1} \to
(\Prism^{\L}_{(R, \Gamma (U, M_{X}))/A}[1/d]/p^m)^{\phi=1}
\]
induce isomorphisms after \'etale sheafification on $U=\spf R$. It suffices to treat $m=1$. The Hodge--Tate comparison implies that, as $P_{M_A} = P\oplus_{M_0} M_A \to R$ is a chart of $M_{X}$, the cone of the map
\[
\Prism^{\L}_{(R, P_{M_A})/A}/p \to
\Prism^{\L}_{(R, \Gamma (U, M_{X}))/A}/p
\]
is uniquely $d$-divisible after \'etale sheafification. Thus taking $(-[1/d])^{\phi=1}$ of the cone is the same as taking $(-)^{\phi=1}$ of the cone after \'etale sheafification. Therefore, it suffices to show that 
\[
(\Prism^{\L}_{(R, P_{M_A})/A}/p)^{\phi=1} \to
(\Prism^{\L}_{(R, \Gamma (U, M_{X}))/A}/p)^{\phi=1}
\]
becomes an isomorphism after \'etale sheafification. By the proof of \cite[Lemma 9.2]{BS}, for any $d$-complete $M\in \mD (A/p)$ with Frobenius $\phi$, $M^{\phi=1}\cong (M/d)^{\phi=1}$. Hence, it suffices to show that 
\[
(\Prism^{\L}_{(R, P_{M_A})/A}/(p,d))^{\phi=1} \to
(\Prism^{\L}_{(R, \Gamma (U, M_{X}))/A}/(p,d))^{\phi=1}
\]
becomes an isomorphism after \'etale sheafification. This follows from the Hodge--Tate comparison; in fact, the map
\[
\Prism^{\L}_{(R, P_{M_A})/A}/(p,d) \to
\Prism^{\L}_{(R, \Gamma (U, M_{X}))/A}/(p,d)
\]
becomes an isomorphism after \'etale sheafification. 

There are natural maps of \'etale sheaves
\[
\Prism^{\L}_{(X, M_X)/A} \ra \Prism^{\L}_{(X, M_X)/(A, M_A)} \ra  \Prism_{(X, M_X)/(A, M_A)}, 
\]
which are both isomorphisms (the first one follows from the Hodge--Tate comparison and the perfectness of $(A, I, M_A)^a$ and the second isomorphism is Remark \ref{remark:comparing_derived_and_nonderived}). 
Moreover, we have the $(p,d)$-complete sheafification map
\[
\Prism^{\L, \textup{pre}}_{(X, M_X)/A} \ra \Prism^{\L}_{(X, M_X)/A}
\]
and, by Corollary \ref{etale_comp_diamond}, the \'etale presheaf 
\[ (\Prism^{\tu{pre}}_{(X, M_X)/A}[\frac{1}{d}]/p^m)^{\phi=1} \colon U   \longmapsto 
(\Prism^{\L}_{(R, \Gamma(U, M_X))/A} [\frac{1}{d}]/p^m)^{\phi=1}
\]
is isomorphic to the following presheaf
\[
U  \longmapsto R\Gamma_{\qpket} ((U_{\eta}^{\diamondsuit}, \Gamma (U, M_X))^a, \Z/p^m). 
\]
By what we prove in the first paragraph and the qcqs assumption on $X_0$, it suffices to show, for any affine $U =\spf R$ \'etale over $X_0 = X$, the composite
\begin{multline*}
R\Gamma_{\qpket} ((U_{\eta}^{\diamondsuit}, \Gamma (U, M_X))^a, \Z/p^m) \cong
(\Prism^{\L}_{(R, \Gamma(U, M_X))/A} [\frac{1}{d}]/p^m)^{\phi=1} \\ 
\ra (\Prism_{(X, M_X)/(A, M_A)} [\frac{1}{d}]/p^m)^{\phi=1}(U)
\end{multline*}
induces an isomorphism after \'etale sheafification. We may assume some fs smooth chart $P\to \Gamma (U, M_{X_0})$ (over $M_0$) exists. Again by what we prove in the first paragraph, we need only observe that there are isomorphisms 
\[
R\Gamma_{\qpket} ((U_{\eta}^{\diamondsuit}, P)^a, \Z/p^m) \cong 
(\Prism^{\L}_{(R, P_{M_A})/(A, M_A)} [\frac{1}{d}]/p^m)^{\phi=1} \cong
(\Prism^{\L}_{(X, M_X)/(A, M_A)} [\frac{1}{d}]/p^m)^{\phi=1}(U). 
\]
\eproof

\bc \label{cor:etale_comparison_over_W(C_flat)}
Further assume that $A/I=\mO_C$ for an algebraically field $C$ so that $A=\Ainf=W(\mO_{C}^\flat)$ and that $X$ is proper over $A/I$. 
Then, 
\[
R\Gamma_{\qpket}((X, M_X)_{\eta}^{\diamondsuit}, \Z_p)\coloneqq
\varprojlim_m R\Gamma_{\qpket}((X, M_X)_{\eta}^{\diamondsuit}, \Z/p^m\Z)
\]
is a perfect complex of $\Z_p$-modules and there is an isomorphism
\[
R\Gamma_{\qpket}((X, M_X)_{\eta}^{\diamondsuit}, \Z_p)\otimes^\L_{\Z_p} W(C^\flat) \cong
R\Gamma_{\Prism} ( (X, M_X)/ (A, M_A))\otimes^\L_{\Ainf} W(C^\flat). 
\] 
A similar assertion holds for mod $p^m$ coefficients. More precisely, there is an isomorphism 
\[
R\Gamma_{\qpket}((X, M_X)_{\eta}^{\diamondsuit}, \Z/p^m)\otimes^\L_{\Z_p} W(C^\flat) \cong
(R\Gamma_{\Prism} ( (X, M_X)/ (A, M_A))/p^m )\otimes^\L_{\Ainf} W(C^\flat). 
\]
\ec

See Proposition \ref{etale_realization_commutes_with_pushforward} below for a generalization. 

\bproof
This follows from \cite[Lemma 8.5]{Bhatt} (or \cite[Example 3.4]{BS_crystal}) as the log prismatic cohomology is perfect if $X$ is proper. 
\eproof

\subsection{Kummer \'etale local systems and Laurent $F$-crystals} \noindent

\noindent
Finally, we discuss (quasi-)pro-Kummer-\'etale local systems. There are several possible approaches, and we follow that of \cite[Section 3]{MannWerner}. 
Let $(X, M_X)$ be an fs log diamond. There are morphisms of sites
\[
\text{pr}\colon (X, M_X)_{\qpket} \to X_{\qproet} \to *_{\proet}. 
\]
A sheaf of sets on $*_{\proet}$ is now known as a condensed set. 

\bd
Let $\Lambda$ be a condensed ring. A sheaf $\mF$ of $\text{pr}^{-1}\Lambda$-modules in $(X, M_X)_{\qpket}$ is constant if $\mF$ is isomorphic to $\text{pr}^{-1}\Lambda^r$, and $\mF$ is locally constant or a $\Lambda$-local system if it is constant quasi-pro-Kummer-\'etale locally on $(X, M_X)_{\qpket}$. 
We write $\text{Loc}_{\Lambda}(X, M_X)$ for the category of $\Lambda$-local systems. 
\ed

\br
We will only consider discrete rings $\Z/p^n\Z$ and $\Z_p$ with profinite topology as condensed rings. 
\er

\bd 
Let $(X, M_X)$ be a bounded fs log $p$-adic formal scheme. The absolute saturated log prismatic site of $(X, M_X)$ is defined as follows: an object is a saturated log prism $(A, I, M_{\spf A})=(A, I, M_A)^a$ with a map $(\spf A/I, M_A)^a \to (X, M_X)$ that admits a saturated chart \'etale locally, and the topology is a strict flat topology.  We denote the absolute saturated log prismatic site of $(X, M_X)$ by $(X, M_X)_{\Prism}$. 
\ed 

\bd \label{def:Laurent_F_crystal} Let $(X, M_X)$ be a bounded fs log $p$-adic formal scheme. A Laurent $F$-crystal is a crystal of vector bundles $E$ on $((X, M_X)_{\Prism}, \mO_{\Prism}[1/\mI]_p^{\wedge})$ equipped with isomorphisms $\phi_E\colon \phi^* E\cong E$. 
We use the notation 
\[\text{Vect}((X, M_X)_{\Prism}, \mO_{\Prism}[1/\mI]_p^{\wedge})^{\phi=1}\] to  denote the category of Laurent $F$-crystals. 
\ed 

We prove the following analogue of \cite[Corollary 3.8]{BS_crystal}.

\bt[Laurent $F$-crystals and pro-Kummer-\'etale local systems] \label{etale_realization_is_equivalence}
Let $(X, M_X)$ be a bounded fs log $p$-adic formal scheme. 
There is a natural equivalence
\[
\textnormal{Vect}((X, M_X)_{\Prism}, \mO_{\Prism}[1/\mI]_p^{\wedge})^{\phi=1} \simeq 
\textnormal{Loc}_{\Z_p}((X, M_X)_{\eta}^{\diamondsuit}). 
\]
\et

\bproof
By a result of Drinfeld--Mathew \cite[Theorem 5.8]{Mathew}, there is an equivalence of categories
\[
\textnormal{Vect}((X, M_X)_{\Prism}, \mO_{\Prism}[1/\mI]_p^{\wedge})^{\phi=1} \simeq
\lim_{(A, I, M_{\spf A})\in (X, M_X)_{\Prism}}\textnormal{Vect}(A[1/I]^{\wedge}_p)^{\phi_A=1},
\]
as in \cite[Proposition 2.7]{BS_crystal}. 
For each $(A, I, M_{\spf A})=(A, I, M_A)^a\in (X, M_X)_{\Prism}$, write 
\[(A_{\perf}, IA_{\perf}, M_{A,\perf})^a
\] 
for its perfection. By \cite[Proposition 3.5]{BS_crystal} and the proof of \cite[Corollary 3.7]{BS_crystal}, 
we have an equivalence
\[
\textnormal{Vect}(A[1/I]^{\wedge}_p)^{\phi_A=1} \simeq
\textnormal{Vect}(A_{\perf}[1/I]^{\wedge}_p)^{\phi_{A_{\perf}}=1}. 
\]
So, if $(X, M_X)_{\Prism}^{\perf}$ denotes the full subcategory of $(X, M_X)_{\Prism}$ of the saturated perfect log prisms, then we have an equivalence
\[
\lim_{(A, I, M_{\spf A})\in (X, M_X)_{\Prism}}\textnormal{Vect}(A[1/I]^{\wedge}_p)^{\phi_A=1} \simeq
\lim_{(A, I, M_{\spf A})\in (X, M_X)_{\Prism}^{\perf}}\textnormal{Vect}(A[1/I]^{\wedge}_p)^{\phi_A=1}. 
\]
Now recall that the category of (saturated) perfect log prisms $(A, I, M_{\spf A})=(A, I, M_A)^a$ is equivalent to the category of (saturated) log $p$-adic formal schemes 
\[(\spf A/I, M_{\spf A/I})=\spf (A/I, M_A)^a
\] attached to (saturated) perfectoid log rings. Therefore, from the equivalence above, we have an equivalence  
\[
\lim_{(A, I, M_{\spf A})\in (X, M_X)_{\Prism}}\textnormal{Vect}(A[1/I]^{\wedge}_p)^{\phi_A=1} \simeq
\lim_{(\spf A/I, M_{\spf A/I}) \to (X, M_X)}\textnormal{Vect}(A[1/I]^{\wedge}_p)^{\phi_A=1},  
\]
where $(\spf A/I, M_{\spf A/I})$ is attached to a saturated perfectoid log ring and 
\[(\spf A/I, M_{\spf A/I})\to (X, M_X)
\] 
is a map of saturated quasi-coherent log $p$-adic formal schemes that admits saturated charts \'etale locally. Write $S=A/I[1/p]$, and let $S^+$ denote the integral closure of $A/I$ in $S$. As $\textnormal{Vect}(A[1/I]^{\wedge}_p)^{\phi_A=1}$ only depends on $S$,  we may only consider integrally closed $A/I$ by replacing $(S, A/I)$ by $(S, S^+)$. 

To proceed, let us study the category $\textnormal{Vect}(A[1/I]^{\wedge}_p)^{\phi_A=1}$. 
Note that $A[1/I]^{\wedge}_p=W(S^\flat)$. 
We have, by \cite[Proposition 3.6, Example 3.4]{BS_crystal} and an induction argument, 
\[
\textnormal{Vect}(W(S^\flat)/p^n )^{\phi_A=1} \simeq
\textnormal{Loc}_{\Z/p^n \Z}(\spec S)
\]
for every integer $n\geq 1$. 
Any $\Z/p^n\Z$-local system on $\spec S$ is represented by a finite \'etale covering of $\spec S$. The same holds for $\spd (S, S^+)$ with quasi-pro-\'etale topology,  as finite \'etale maps satisfy quasi-pro-\'etale descent in the category of diamonds \cite[Proposition 10.11 (iii)]{diamond}. Therefore, we can rewrite the above equivalence as
\[
\textnormal{Vect}(W(S^\flat)/p^n )^{\phi_A=1} \simeq
\textnormal{Loc}_{\Z/p^n \Z}(\spd (S, S^+)). 
\]
By passing to the limits and using \cite[Proposition 3.5]{MannWerner}, we get a natural equivalence
\[
\textnormal{Vect}(W(S^\flat))^{\phi_A=1} \simeq
\textnormal{Loc}_{\Z_p}(\spd (S, S^+)). 
\]
Combining results so far, we now have a functor 
\begin{multline*}
\qquad \quad \lim_{(A, I, M_{\spf A})\in (X, M_X)_{\Prism}}\textnormal{Vect}(A[1/I]^{\wedge}_p)^{\phi_A=1} \\ \lra
\lim_{(\spd (S, S^+), M_S)^a \to (X, M_X)^{\diamondsuit}_{\eta}}\textnormal{Loc}_{\Z_p}(\spd (S, S^+)),  \qquad \qquad 
\end{multline*}
where $\spd (S, S^+)$ is perfectoid, $M_S$ is saturated, and $M_S/M_S^\times$ is divisible; we will prove that this functor is an equivalence at the end of this proof. 

Let us now recall that quasi-pro-\'etale maps to $\spd (S, S^+)$ correspond to quasi-pro-Kummer-\'etale maps to $(\spd (S, S^+), M_S)^a$ as $M_S/M_S^\times$ is divisible. Therefore, there is a natural pullback functor
\[
\textnormal{Loc}_{\Z_p}(X, M_X) \lra 
\lim_{(\spd (S, S^+), M_S)^a \to (X, M_X)^{\diamondsuit}_{\eta}}\textnormal{Loc}_{\Z_p}(\spd (S, S^+)), 
\]
and we shall show that this is an equivalence. 
For this, we observe that log diamonds of the form $(\spd (T, T^+), M_{T})^a$ that are quasi-pro-Kummer-\'etale over $(X, M_X)$ form a basis\footnote{This basis is not necessarily closed under (saturated) fiber products over $(X, M_X)_{\eta}^{\diamondsuit}$.} on $(X, M_X)_{\eta, \qpket}^{\diamondsuit}$. Moreover, their base changes along a map 
\[(\spd (S, S^+), M_S)^a \to (X, M_X)_{\eta}^{\diamondsuit}\]
form a basis of $(\spd (S, S^+), M_S)^a_{\qpket}$, which can be refined by a basis consisting of quasi-pro-Kummer-\'etale maps \[
(\spd (S', S^{'+}), M_{S'})^a\to (\spd (S, S^+), M_S)^a
\] that factor over $(\spd (S, S^+), M_S)^a\times_{(X, M_X)_{\eta}^{\diamondsuit}}(\spd (T, T^+), M_{T})^a$ for some $(\spd (T, T^+), M_{T})^a$. 
Therefore, we have equivalences 
\begin{align*}
\textnormal{Loc}_{\Z_p}(X, M_X) 
&\simeq
\lim_{(\spd (T, T^+), M_T)^a \in (X, M_X)^{\diamondsuit}_{\eta, \qpket}}\textnormal{Loc}_{\Z_p}(\spd (T, T^+)) \\
&\simeq
\lim_{(\spd (S, S^+), M_S)^a \in (X, M_X)^{\diamondsuit}_{\eta}}  \: \: \textnormal{ Loc}_{\Z_p}(\spd (S, S^+)). 
\end{align*}
In fact, the objects of the following form also form a basis on $(X, M_X)_{\eta, \qpket}^{\diamondsuit}$: 
\bi 
\item $(\spd (S, S^+), M_S)^a$ that is  quasi-pro-Kummer-\'etale over $(\spd (R[1/p], R^+), P)^a$ with $M_S=P_{\Q_{\geq 0}}$ for some \'etale map $\spf R\to X$ with an fs chart $P\to R$.
\ei 
This again gives a basis after the base change along any map $(\spd (S, S^+), M_S)^a\to (X, M_X)$ and it is refined by a another basis similar to the one described above. Thus, we finally get the desired equivalence
\[
\lim_{(A, I, M_{\spf A})\in (X, M_X)_{\Prism}}\textnormal{Vect}(A[1/I]^{\wedge}_p)^{\phi_A=1} \simeq
\textnormal{Loc}_{\Z_p}(X, M_X). 
\]
\eproof

There is a derived version of Theorem \ref{etale_realization_is_equivalence} as in \cite{BS_crystal}. 
The $\infty$-category
\[
\mD^{\textnormal{perf}}((X, M_X)_{\Prism}, \mO_{\Prism}[1/\mI]_p^{\wedge})^{\phi=1}
\]
is defined by replacing vector bundles by perfect complexes in the definition of 
\[ \textnormal{Vect}((X, M_X)_{\Prism}, \mO_{\Prism}[1/\mI]_p^{\wedge})^{\phi=1}.
\] 
Let $\mD^{(b)}((X, M_X)_{\eta}^{\diamondsuit}, \Z_p)$ denote the $\infty$-category of complexes of $\Z_p$-modules locally constant with perfect fibers; note that this is a hypercomplete sheaf of $\infty$-categories with respect to the quasi-pro-Kummer-\'etale topology if $X_{\eta}^{\diamondsuit}$ is quasicompact, as one can work with complexes globally concentrated on a fixed bounded range, cf. the proof of \cite[Proposition 5.11]{BM_arc}. 

\bt
Let $(X, M_X)$ be a  bounded fs log $p$-adic formal scheme. 
There is a natural equivalence
\[
\mD^{\textnormal{perf}}((X, M_X)_{\Prism}, \mO_{\Prism}[1/\mI]_p^{\wedge})^{\phi=1} \simeq 
\mD^{(b)}((X, M_X)_{\eta}^{\diamondsuit}, \Z_p). 
\]
\et

\bproof
If $X$ is quasicompact, the previous proof works, and then the case of quasiseparated $X$ follows by taking an affine Zariski open cover of $X$ and gluing. Finally, the general case follows by another gluing argument since any open of an affine formal scheme is (quasi)separated. 
\eproof

For a Laurent $F$-crystal $\mF$, the corresponding quasi-pro-Kummer-\'etale sheaf is called the \'etale realization of $\mF$ and denoted by $\mF_{\ett}$. For instance, $(\mO_{\Prism})_{\ett}$ is the constant sheaf $\Z_p$ on the quasi-pro-Kummer-\'etale site. Let us record the compatibility with pushforward in the simplest case:

\bp \label{etale_realization_commutes_with_pushforward}
Let $f\colon (X, M_X)\to (Y, M_Y)$ be a smooth map of  bounded fs log   $p$-adic formal schemes. 
If $f$ is proper, $Rf_{*}\mO_{\Prism}$ is an $F$-crystal of perfect complexes, and there is a natural isomorphism, 
\[
(Rf_{*}\mO_{\Prism})_{\ett} \cong
Rf_{\eta, *}\Z_p, 
\]
where $f_{\eta}\colon (X, M_X)_{\eta}^{\diamondsuit}\to (Y, M_Y)_{\eta}^{\diamondsuit}$. In particular, $Rf_{\eta, *}\Z_p$ is locally constant with perfect fibers. 
\ep

\br
It follows from the base change of log prismatic cohomology that $Rf_{\eta *}\Z_p$ commutes with base change $(Z, M_Z)\to (Y, M_Y)$. 
\er

\bproof
The Hodge--Tate comparison and the base change implies that $Rf_{*}\mO_{\Prism}$ is an $F$-crystal of perfect complexes. 

To show the second claim, we may assume, working \'etale locally on $Y$, that there is a sharp fs chart $P \to \Gamma (Y, M_Y)$. Take an object $(A, I ,M_A)^a$ of the saturated perfect log prismatic site of $(Y, M_Y)_{\Prism}$. Let $\widetilde{P}\to \Gamma (\spf A, M_{\spf A}^a)$ denote the pullback along 
\[\Gamma (\spf A, M_{\spf A})\to \Gamma (\spf A/I, M_{A, \spf A/I}^a).
\] Since this map is exact surjective and $P$ is sharp fs, there is a section $P\to \widetilde{P}$. Regard $(A, I, \widetilde{P})$ as a pre-log prism via the $\delta_{\log}$-structure induced by $\widetilde{P}\to \Gamma (\spf A, M_{\spf A})$. We may now apply Theorem  \ref{etale_comp_diaomnd_global}, and there is a functorial isomorphism
\[
(R\Gamma_{\Prism}((X, M_X)_{(A/I, M_A)}/(A, M_A))[1/d]/p^n)^{\phi=1} \cong
R\Gamma_{\qpket} ((X, M_X)^{\diamondsuit}_{(A/I, M_A), \eta}, \Z/p^n) 
\]
for every $n \geq 1$. 
As in the proof of Theorem \ref{etale_realization_is_equivalence}, we may assume, to determine the \'etale realization, that the generic fiber of $\spf A/I$ with the log structure is a strictly totally disconnected log perfectoid space $(\spa (S, S^+), M_S)$ that is quasi-pro-Kummer-\'etale over $(Y, M_Y)$. Then, the \'etale $\Z/p^n$-local system corresponds to the left hand side is locally constant with respect to finitely many closed and open subsets (as any \'etale cover of $S$ splits), and the right hand side recovered as the global section of the locally constant sheaf. This argument, combined with base change of log prismatic cohomology, also implies that $R\Gamma_{\qpket} ((X, M_X)^{\diamondsuit}_{\eta, (\spa (T, T^+),M_S)}, \Z/p^n)$ for all such quasi-pro-\'etale $\spa (T, T^+) \to \spa (S, S^+)$ is the value of the locally constant quasi-pro-(Kummer-)\'etale sheaf. Therefore, the restriction of $Rf_{\eta, *}\Z/p^n$ to $(\spa (S, S^+), M_S)$ is locally constant as a quasi-pro-Kummer-\'etale sheaf of complexes. We conclude that $(Rf_* \mO_{\Prism})_{\ett}$ identifies with $Rf_{\eta *}\Z_p$. 
\eproof



\newpage 
\section{Breuil--Kisin--Fargues modules} 
\label{section:BKF_module}

In this section, we prove that under the \'etale comparison, the log prismatic cohomology of a smooth proper fs log $p$-adic formal scheme has the structure of a Breuil--Kisin--Fargues module (cf. Definition \ref{definiton:BKF_module}). Throughout the section, we let $A = \Ainf =  W(\mO_C^\flat)$, where $C$ is an algebraically closed extension of $\Q_p$. Fix a compatible choice of $p$-power roots of unity and set 
\[ \epsilon = (1, \zeta_p, \zeta_{p^2}, ...) \in \mO_C^\flat;  \quad \mu = [\epsilon] -1, \: \: d=\xi = 1 + [\epsilon]^{\frac{1}{p}} + \cdots [\epsilon]^{\frac{p-1}{p}} \in A.\] 
Write $\fm \subset \mO_C$ for the maximal ideal.  
Let 
\[
(A, I, M_A) = (\Ainf, (\xi), M_A)
\]
be a saturated perfect ``log prism''; the associated log prism $(A, I, M_{\spf A}) = (A, I, M_A)^a$ is perfect and any perfect log prism with underlying prism $(A, I)$ arises in this way as any \'etale sheaf on $\spf A$ is constant.

Now let us first recall the setting of Theorem \ref{etale_comp_diaomnd_global} (and Corollary \ref{cor:etale_comparison_over_W(C_flat)}).
Suppose that $ (X, M_X)$ fits into the following cartesian diagram 
\[
\begin{tikzcd}
 (X_0, M_{X_0})  \arrow[d, "\pi_0"]  
& (X, M_{X}) \arrow[d] \arrow[l]
\\ 
(\spf \mO_C, M_0)^a  
& (\spf \mO_C, M_A)^a \arrow[l] 
\end{tikzcd}
\]
of log $p$-adic formal schemes, where 
\bi
\item The bottom map is induced by a map $M_0 \ra \Gamma(\spf A, M_{\spf A})$  from an fs monoid $M_0$. 
\item $(X_0, M_{X_0})$ is an  fs log $p$-adic formal scheme, smooth and proper over $\spf (\mO_C, M_0)^a$, such that the mod $p$ fiber of $\pi_0$ is of Cartier type and $X_0$ has the same underlying formal scheme as $X$. 
\ei 
Given the setup above, set 
\[M\coloneqq R\Gamma_{\Prism} ((X_0, M_{X_0})/(A, M_0)), \qquad \textup{and}  \quad  H^i_{\Prism}\coloneqq H^i(M) \:\:\: 
\] 
for each $i \ge 0$. 
Note that, by base change along the map $(A, I, M_0) \ra (A, I, M_A)$ of bounded pre-log prisms, we have  
\[
M = R\Gamma_{\Prism} ((X_0, M_{X_0})/(A, M_0)) \cong R\Gamma_{\Prism} ((X, M_{X})/(A, M_A)). 
\]

Let us also recall the definition of Breuil--Kisin--Fargues modules from \cite[Definition 4.22]{BMS1}.  
\bd \label{definiton:BKF_module}
A Breuil--Kisin--Fargues module is a finitely presented $\Ainf$-module $N$ equipped with a $\phi$-linear map $\phi =  \phi_N\colon N \ra N$, such that $N[\frac{1}{p}]$ is a free $\Ainf[\frac{1}{p}]$-module and that $\phi$ induces an isomorphism 
\[\phi\colon N [\frac{1}{\xi}] \isom N[\frac{1}{\phi(\xi)}].\]
\ed 

The main result we prove in this section is the following
\begin{theorem} \label{theorem:BKF_module}
Let $(X, M_X)$ be as above, then  $\phi^* H^i_{\Prism} = H^i_{\Prism} \otimes_{\Ainf, \phi} \Ainf$ with its Frobenius $\phi$ is a Breuil--Kisin--Fargues-module. 
\end{theorem}

We will apply \cite[Corollary 4.20]{BMS1}, and it suffices to show the following two propositions. 

\begin{proposition} \label{inverting_mu}
For every $i$,  
\[
H^i_{\Prism}\otimes_{\Ainf} \Ainf [1/\phi^{-1}(\mu)] \cong
H^i_{\qpket}((X, M_X)_{\eta}^{\diamondsuit}, \Z_p)\otimes_{\Z_p} \Ainf [1/\phi^{-1}(\mu)]. 
\]
\end{proposition}

\begin{proposition} \label{crys_free}
For every $i$,
\[
H^i (M \otimes^\L_{\Ainf, \phi} \Acrys[1/p])
\]
is a finite free $\Acrys [1/p]$-module. 
\end{proposition}

\subsection{Comparison over $\Ainf[\phi^{-1}(\mu)]$}
\noindent 

\noindent 
To prove Proposition \ref{inverting_mu}, we will construct suitable comparison maps using the Riemann--Hilbert correspondence in characteristic $p$ of the form of Bhatt--Lurie \cite{Bhatt_Lurie}. Set 
\[T\coloneqq R\Gamma_{\qpket}((X, M_X)_{\eta}^{\diamondsuit}, \Z_p).\] 
Recall that $T$ is a perfect complex of $\Z_p$-modules by Corollary \ref{cor:etale_comparison_over_W(C_flat)}. 

\begin{lemma} \label{lemma:RH}
There is a $\phi$-equivariant map $M\to T\otimes^\L_{\Z_p}\Ainf$ compatible with the \'etale comparison.  It induces a map
\[
(M_{\perf})^{\wedge}_p \to T\otimes^\L_{\Z_p}\Ainf, 
\] 
which becomes an isomorphism after inverting $\phi^{-1}(\mu)$. Here the source is the derived $p$-completed colimit perfection. Similarly, there is a $\phi$-equivariant map
\[
W(\fm^{\flat})\otimes^\L_{\Z_p} T \to (M_{\perf})^{\wedge}_p
\]
that becomes an isomorphism after inverting $\phi^{-1}(\mu)$.
Moreover, the resulting composite $W(\fm^{\flat})\otimes^\L_{\Z_p} T\to \Ainf\otimes_{\Z_p}T$ is induced by the inclusion $W(\fm^{\flat})\hookrightarrow \Ainf$.  
\end{lemma}

\begin{proof}
As $M$ and $T$ are derived $p$-complete, it suffices to construct a map 
\[
M\otimes^\L_{\Ainf}W_n (\mO_C^{\flat}) \to T \otimes^\L_{\Z_p} W_n (\mO_C^{\flat})
\]
for all $n\geq 1$. To construct these maps, we will use the Riemann--Hilbert correspondence in characteristic $p$ developed in \cite{Bhatt_Lurie} applied to $\spec \mO_C^{\flat}$. 

Let us summarize what we need. The derived functor of $\phi$-fixed points, working \'etale locally, gives a functor
\[
\textnormal{Sol}\colon
\mD^+_{\perf} (W_n (\mO_C^{\flat})[F]) \simeq \mD^+ (\text{Mod}^{\perf}_{W_n(\mO_C^{\flat})}) \to
\mD^+_{\ett} (\spec \mO_C^{\flat}, \Z/p^n), 
\]
where 
\bi
\item $\mD^+_{\perf} (W_n (\mO_C^{\flat})[F])$ denotes the full subcategory of the derived category of Frobenius $W_n (\mO_C^{\flat})$-modules whose Frobenii are quasi-isomorphisms, and 
\item $\mD^+ (\text{Mod}^{\perf}_{W_n(\mO_C^{\flat})})$ denotes the derived category of Frobenius $W_n (\mO_C^{\flat})$-modules whose Frobenii are isomorphisms. 
\ei 
Also recall that a Frobenius $\mO_C^{\flat}$-module $N$ is holonomic if $N$ is the perfection of a finitely presented $\mO_C^{\flat}$-module with a Frobenius \cite[Definition 4.1.1]{Bhatt_Lurie}, and such Frobenius modules are algebraic by \cite[Proposition 4.2.1]{Bhatt_Lurie} in the sense of \cite[Definition 2.4.1]{Bhatt_Lurie}. A Frobenius $W_n(\mO_C^{\flat})$-module $N$ is algebraic if $N/pN$ is algebraic \cite[Definition 9.5.2]{Bhatt_Lurie}. Now let 
\[
\mD^b_{\text{alg}} (W_n (\mO_C^{\flat})[F])\subset \mD^b_{\perf} (W_n (\mO_C^{\flat})[F])
\]
denote the full subcategory of bounded objects with cohomology being algebraic Frobenius $W_n(\mO_C^{\flat})$-modules. 
The functor
\[
\text{Sol}\colon \mD^b_{\text{alg}} (W_n (\mO_C^{\flat})[F]) \to \mD^b_{\ett}(\spec \mO_C^{\flat}, \Z/p^n)
\]
is an equivalence:\footnote{See also \cite[Theorem 5.4, Remark 5.3]{Mathew_perfect}.} the case $n=1$ is proven in \cite[Theorem 12.1.5]{Bhatt_Lurie}. For general $n$, the functor is fully faithful by \cite[Proposition 9.5.6, Proposition 9.6.2]{Bhatt_Lurie}, then the essential surjectivity follows by induction using \cite[Theorem 9.6.1]{Bhatt_Lurie} and \cite[Proposition 9.6.2]{Bhatt_Lurie}. 

To apply the above result, observe that all the cohomology of $M\otimes^\L_{\Ainf}W_n (\mO_C^{\flat})$ are finitely presented since $W_n (\mO_C^{\flat})$ is coherent \cite[Proposition 3.24]{BMS1} and $M$ is a perfect complex. Therefore, we see that the (non-completed) colimit perfection 
\[(M_n)_{\perf}\coloneqq \big(M\otimes^\L_{\Ainf}W_n (\mO_C^{\flat}) \big)_{\perf}\] 
is algebraic.\footnote{Alternatively, the case $n=1$ follows from \cite[Theorem 12.4.1]{Bhatt_Lurie} and then \cite[Proposition 9.5.5]{Bhatt_Lurie} implies the general case by induction. In fact, this alternative argument works for a general base ring.} Also note that the perfect complex $T \otimes^\L_{\Z_p} W_n (\mO_C^{\flat})$ corresponds under the above equivalence to the constant sheaf $\ul{T_n}$ attached to $T\otimes^\L_{\Z_p}\Z/p^n$. Therefore, to construct the desired map, it suffices to construct a map of \'etale sheaves 
\[\text{Sol}((M_n)_{\perf})\to \ul{T_n}.\] Let $(M_n [1/d])_{\perf}$ denote the colimit perfection 
\[(M_n [1/d])_{\perf} \coloneqq \big(M\otimes^\L_{\Ainf}W_n (\mO_C^{\flat})[1/d]\big)_{\perf},\] then we can simply take the following composite
\[
\text{Sol}((M_n)_{\perf}) \to \text{Sol}((M_n [1/d])_{\perf}) \cong \ul{T_n}, 
\]
where the last identification comes from the \'etale comparison. 
Moreover, for the open immersion $j\colon \spec C^{\flat}\hookrightarrow \spec \mO_C^{\flat}$, we have a natural isomorphism
\[
Rj_! j^*\text{Sol}((M_n)_{\perf}) \cong Rj_!j^* \ul{T_n}
\]
and both correspond to $W(\fm^{\flat})\otimes^\L_{\Z_p}T$: it is easy to check that $\text{Sol} (W_n(\fm^{\flat})\otimes^\L_{\Z_p}T) \cong Rj_! j^* \ul{T_n}$. 
So, the natural maps $Rj_! j^*\text{Sol}((M_n)_{\perf})\to \text{Sol}((M_n)_{\perf})$ similarly induce the map
\[
W(\fm^{\flat})\otimes^\L_{\Z_p}T\to (M_{\perf})^{\wedge}_p. 
\]
To show that these maps are isomorphisms after inverting $\phi^{-1}(\mu)$, it suffices to show that the cones of their mod $p^n$ reductions are killed by $[\fm^{\flat}]$ as everything is $p$-complete here, cf. the proof of \cite[Theorem 5.7]{BMS1}. 
The cones of
\[
W_n(\fm^{\flat})\otimes^\L_{\Z_p}T \to (M_n)_{\perf}, \quad
(M_n)_{\perf} \to W_n (\mO_C^{\flat})\otimes^\L_{\Z_p}T
\]
are concentrated on $\spec \mO_C^{\flat}/\fm^{\flat}$ \cite[Theorem 5.3.1]{Bhatt_Lurie} (this can be checked degree-wise), and hence killed by $[\fm^{\flat}]$. 
\end{proof}

To complete the proof of Proposition \ref{inverting_mu}, it suffices to show the following. 

\begin{lemma} \label{lemma:inverting_mu_on_perfection}
Let $D$ denote the rank of $\Omega^1_{(X_0, M_{X_0})/(\mO_C, M_0)}$. 
There is a map 
\[
(M_{\perf})^{\wedge}_p \to \phi^{-1}(\mu)^{-D} \otimes_{A}^\L M
\]
such that the composite 
\[M\to (M_{\perf})^{\wedge}_p \to \phi^{-1}(\mu)^{-D} \otimes_A^\L M
\] is induced by $A = \Ainf\subset \phi^{-1}(\mu)^{-D}\Ainf$. Moreover, the multiplication by $\phi^{-1}(\mu)^D$ on $M_{\perf}$ factors as
\[
\phi^{-1}(\mu)^D\colon
M_{\perf} \to (M_{\perf})^{\wedge}_p \to \phi^{-1}(\mu)^{-D} \otimes_A^\L M \xrightarrow{\phi^{-1}(\mu)^D} M\to M_{\perf}. 
\]
In particular, we have canonical isomorphisms
\[
M[1/\phi^{-1}(\mu)] \cong M_{\perf} [1/\phi^{-1}(\mu)] \cong (M_{\perf})^{\wedge}_p [1/\phi^{-1}(\mu)].
\]
\end{lemma}

\begin{proof}
Recall that $M = R\Gamma (X_{\ett}, \Prism_{(X_0, M_{X_0})/(A, M_0)})$ and 
\[
\Prism_{(X_0, M_{X_0})/(A, M_0)}\xrightarrow{\cong} L\eta_{\phi^{-1}(d)} \phi_*\Prism_{(X_0, M_{X_0})/(A, M_0)}. 
\]
By properties of $L\eta$, we have a map
\[
\phi_*\Prism_{(X_0, M_{X_0})/(A, M_0)} \to \phi^{-1}(d)^{-D} \otimes_A^\L \Prism_{(X_0, M_{X_0})/(A, M_0)} 
\]
with a factorization of the multiplication by $\phi^{-1}(\mu)^D$:
\begin{multline*} 
\qquad 
\phi_*\Prism_{(X_0, M_{X_0})/(A, M_0)} \to \phi^{-1}(d)^{-D} \otimes_A^\L \Prism_{(X_0, M_{X_0})/(A, M_0)} \\  \xrightarrow{\phi^{-1}(\mu)^D} \Prism_{(X_0, M_{X_0})/(A, M_0)} \to \phi_*\Prism_{(X_0, M_{X_0})/(A, M_0)}. \qquad 
\end{multline*} 
Note that $\phi^{-1}(d)$ divides $\phi^{-1}(\mu)$, so the middle map is well defined. 
Iterating this construction, we get a map
\[
\phi_*^{N} \Prism_{(X_0, M_{X_0})/(A, M_0)} \to
\bigl(\prod_{j=1}^N \phi^{-j}(d)\bigr)^{-D} \otimes_A^\L \Prism_{(X_0, M_{X_0})/(A, M_0)}
\]
with a factorization
\begin{multline*}
\qquad 
\phi^{-1}(\mu)^D\colon 
\phi^N_*\Prism_{(X_0, M_{X_0})/(A, M_0)} \to \bigl(\prod_{j=1}^N \phi^{-j}(d)\bigr)^{-D} \otimes_A^\L \Prism_{(X_0, M_{X_0})/(A, M_0)} \\
\xrightarrow{\phi^{-1}(\mu)^D} \Prism_{(X_0, M_{X_0})/(A, M_0)} \to \phi^N_*\Prism_{(X_0, M_{X_0})/(A, M_0)} \qquad 
\end{multline*}
for every $N\geq 1$; note again that the product $\bigl(\prod_{j=1}^N \phi^{-j}(d)\bigr)$ divides $\phi^{-1}(\mu)$ for every $N$. These are compatible with varying $N$, so, by passing to the colimit, we get a map
\begin{equation} \label{eq:iterated_map_scaling_mu}
(\Prism_{(X_0, M_{X_0})/(A, M_0)})_{\perf} \to \phi^{-1}(\mu)^{-D} \otimes^\L_{A} \Prism_{(X_0, M_{X_0})/(A, M_0)} 
\end{equation}
with a factorization
\begin{multline*}
\qquad 
\phi^{-1}(\mu)^{D} \colon (\Prism_{(X_0, M_{X_0})/(A, M_0)})_{\perf} \to \phi^{-1}(\mu)^{-D} \otimes_A^\L (\Prism_{(X_0, M_{X_0})/(A, M_0)}) \\
\xrightarrow{\phi^{-1}(\mu)^D} (\Prism_{(X_0, M_{X_0})/(A, M_0)})\to (\Prism_{(X_0, M_{X_0})/(A, M_0)})_{\perf}. \qquad 
\end{multline*}
As the target of (\ref{eq:iterated_map_scaling_mu}) is derived $p$-complete, this map factors through the derived $p$-completion. 
Taking the cohomology and passing to derived completion, we finally obtain a map
\[
(M_{\perf})^{\wedge}_p \to \phi^{-1}(\mu)^{-D} \otimes_A^\L M
\]
with a factorization
\[
\phi^{-1}(\mu)^D\colon
M_{\perf} \to \phi^{-1}(\mu)^{-D} \otimes_A^\L M \xrightarrow{\phi^{-1}(\mu)^D} M\to M_{\perf} 
\]
as desired. 
\end{proof}

\bproof[Proof of Proposition \ref{inverting_mu}] 
This follows immediately from Lemma \ref{lemma:RH} and Lemma \ref{lemma:inverting_mu_on_perfection}. 
\eproof 

\begin{remark}
Our proof of Proposition \ref{inverting_mu} also works in the following setting: let $R$ be a perfectoid ring $\mO_C$ that is $p$-torsionfree and integrally closed, and suppose that $\spa (R[1/p], R)$ is strictly totally disconnected perfectoid. Suppose that $(X_0, M_{X_0})$ is defined over $\spf (R, M_0)^a$ satisfying similar conditions as above. The log prismatic cohomology is taken over the pre-log prism $(\Ainf (R), M_0)$ (or with a suitable log structure). In the setting of Proposition \ref{etale_realization_commutes_with_pushforward}, such a generalization relates the values of 
$Rf_{*}\mO_{\Prism}[1/\phi^{-1}(\mu)]$ and $Rf_{\eta,*}\Z_p \otimes^\L_{\Z_p}\widehat{\mO}^+_Y [1/\phi^{-1}(\mu)]$. 

\end{remark}

\subsection{Comparison over $\Acrys$}  \noindent 

\noindent 
Next we prove Proposition \ref{crys_free}.  Let $(\Acrys, (p), M_{\crys})$ denote the ``log prism'' associated to the pre-log prism $M_0\to \Ainf\to\Acrys$, which is equipped with a  Frobenius:
\[
(\Acrys, (p), M_{\crys}) \xrightarrow{\phi} (\Acrys, (p), M_{\crys}). 
\]

\begin{proposition} 
Let $(X_{0}, M_{X_0})_{\mO_C/p}$ denote the base change of $(X_{0}, M_{X_0})$ along  
\[\spec (\mO_C/p, M_{\crys})^a \ra \spf (\mO_C, M_0)^a. \] 
There is a Frobenius equivariant isomorphism
\[
\phi^* R\Gamma_{\Prism} ((X_0, M_{X_0})/ (\Ainf, M_0)) \otimes^\L_{\Ainf}\Acrys\cong
R\Gamma_{\crys} ((X_{0}, M_{X_0})_{\mO_C/p}/ (\Acrys, M_{\crys})), 
\]
and the Frobenius becomes an isomorphism after inverting $p$:
\begin{multline} \label{eq:Frob_isom_on_crys_after_inverting_p}
\qquad  \phi[1/p]  \colon \: \:  
\phi^* R\Gamma_{\crys} ((X_{0}, M_{X_0})_{\mO_C/p}/ (\Acrys, M_{\crys}))[\frac{1}{p}] \\
\isom  R\Gamma_{\crys} ((X_{0}, M_{X_0})_{\mO_C/p}/ (\Acrys, M_{\crys}))[\frac{1}{p}].\qquad  \quad 
\end{multline}
\end{proposition}

\bproof 
As the kernel of $\Acrys \to \mO_C/p$ has divided powers, this follows from the crystalline comparison \cite[Theorem 6.3]{Koshikawa}, base change, and Corollary \ref{cor:image_Frob_global}. 
\eproof 

Proposition \ref{crys_free} follows from the next assertion on a version of ``Hyodo--Kato isomorphism'', which is a logarithmic analogue of \cite[Proposition 13.21]{BMS1}. Let $k = \mO_C/\fm$ denote the residue field of $C$ and let $(k, N)$ denote the log ring associated with $(k, M_0)$; note that the image of $N\setminus N^{\times}$ is $0$ and $N^{\times} \cong k^{\times}$. 
Let $(Y, M_Y)$ denote the base change of $(X_0, M_{X_0})$ to $(k, N)$. 

\bp
Fix a section $k\to \mO_C/p$.  
There is an isomorphism
\[ 
 R\Gamma_{\crys}((Y, M_Y)/(W(k), N))\otimes^\L_{W(k)} \Acrys [\frac{1}{p}]
\cong 
R\Gamma_{\crys}( (X_{0}, M_{X_0})_{\mO_C/p}/(\Acrys, M_{\crys}))[\frac{1}{p}]. 
\]
In particular, each cohomology group 
$H^i_{\crys}((X_{0}, M_{X_0})_{\mO_C/p}/(\Acrys, M_{\crys}))[\frac{1}{p}]$
is a finite free $\Acrys[1/p]$-module. 
\ep 

\bproof 
Iterating the Frobenius map, we see that
\[
(\phi^n)^*  R\Gamma_{\crys} ((X_{0}, M_{X_0})_{\mO_C/p}/ (\Acrys, M_{\crys}))
\]
is the crystalline cohomology of the base change of $X_{0, \mO_C/p}$ along
\[
\phi^n \colon (\mO_C/p, M_{\crys}) \to (\mO_C/p^{1/p^n}, M_{\crys}) 
\xrightarrow{\phi^n} ( \mO_C/p, M_{\crys}). 
\]
For each $n \ge 1$, let $Y_{\mO_C / p^{1/p^n}}$ denote the base change of $Y$ across $(k, N)  \ra (\mO_C/p^{1/p^n}, N)^a$ (using the fixed section $k \ra \mO_C/p$).  As (the underlying scheme of) $X_{0, \mO_C/p}$ is qcqs and of finite presentation over $\mO_C/p$, its mod $p^{1/p^n}$ fiber is isomorphic to $Y_{\mO_C / p^{1/p^n}}$ for sufficiently large $n$. Moreover, for sufficiently large $n$, the image of $M_0\setminus M_0^{\times}$ is 0 in $\mO_C / p^{1/p^n}$ so that the section $k\to \mO_C / p^{1/p^n}$ can be upgraded to 
\[
(k, N) \to (\mO_C / p^{1/p^n}, N)^a \cong (\mO_C / p^{1/p^n}, M_0)^a. 
\]
Moreover, as $M_{X_0}$ has fine charts and $X_0$ is qcqs, again for sufficiently large $n$, we obtain an isomorphism of log schemes
\[
(Y, M_Y)_{(\mO_C / p^{1/p^n}, M_0)^a} \cong (X_0, M_{X_0})_{(\mO_C / p^{1/p^n}, M_0)^a}.  
\]
Therefore, (considering $n^{th}$ iterates of the map (\ref{eq:Frob_isom_on_crys_after_inverting_p})), we have the following natural isomorphism 
\[
R\Gamma_{\crys} ((Y,M_Y)_{\mO_C/ p} / (A_{\crys}, M_{\crys}))[\frac{1}{p}] \isom R\Gamma_{\crys} ((X_{0}, M_{X_0})_{\mO_C/p}/ (\Acrys, M_{\crys}))[\frac{1}{p}]. 
\]
Finally, note that there are maps of pre-log rings
\[
(W(k), N) \to (\Acrys, N) \xrightarrow{\phi^n} (\Acrys, M_{\crys}).  
\]
The proposition then follows by the base change of log crystalline cohomology.
\eproof 

This completes the proof of Proposition \ref{crys_free}, and hence that of Theorem \ref{theorem:BKF_module}.

\newpage 

\section{Some further applications} \label{section:apps}

In the final section
of the paper, let us prove the applications mentioned in the introduction. 

\subsection{Hodge polygon of the generic fiber} 
\noindent 

\noindent 
Let $K$ be a complete nonarchimedean extension of $\Q_p$ with ring of integers $\mO_K$ and residue field $k$. Let $C = C_K$ be a complete algebraic closure of $K$. Let $(\Ainf, (\xi), N_\infty)$ be a saturated perfect ``log prism''.  
Suppose that $\mO_K$ is equipped with an fs  pre-log structure $N \ra \mO_K$ and further suppose that $(\mO_K, N)$ admits a map 
\[
\iota\colon (\mO_K, N) \ra (\mO_C, N_\infty)^a.
\] Let $(X, M_X)$ be an fs log $p$-adic formal scheme, 
smooth and proper over $(\mO_K, N)$, such that its mod $p$ fiber is of Cartier type. Note that the base change $(X, M_X)_{\mO_C}$ of $(X, M_X)$ across the map 
\[(\spf \mO_C, N)^a \ra  (\spf \mO_K, N)^a
\] precisely fits into the setup of Section \ref{section:BKF_module}. 

Let us restate Theorem \ref{mainthm:Newton_Hodge} from the introduction. 
\bt  \label{theorem:Newton_Hodge} Let $(X, M_X)$ be as above. 
Let $X^{\log}_s \coloneqq (X, M_X)_{(k, N)^a}$ be the special fiber of $(X, M_X)$ over $\ul k = (k, N)$ and $X^{\log}_{\eta}$ be the (adic) generic fiber of $(X, M_X)$ over $\ul K = (K, N)$. Then for each $i$, the Newton polygon of the $F$-isocrystal 
\[H^i_{\textup{logcrys}}(X^{\log}_s/W(\ul k))[\frac{1}{p}] \] 
lies on or above the Hodge polygon of the $X_{\eta},$ defined by the Hodge numbers 
\[h^j \coloneqq \dim_K H^{i-j} (X^{\log}_{\eta}, \Omega^j_{X_\eta/\ul K}). \]
\et 

\bproof 
The proof is similar to the proof of  \cite[Theorem 2]{Yao_generic_fiber}. To ease notation, write $(Y, M_Y)$ for the base change of $(X, M_X)$ along $\spf (\mO_C, N)^a \ra \spf (\mO_K, N)^a$, 
and consider the Frobenius twisted prismatic cohomology group
\[
M \coloneqq \phi^* H^i_{\Prism} ((Y, M_Y)/(\Ainf, N_\infty)) \cong H^i (\phi^* \Prism_{(Y, M_Y)/(\Ainf, N_\infty)}). 
\]
This is a Breuil--Kisin--Fargues module by Theorem \ref{theorem:BKF_module}. In \cite[Section 3]{Yao_generic_fiber}, the second author associated to $M$ a Hodge polygon $\textup{Hdg}(M)$ and showed (using the crystalline comparison) that the Newton polygon $\textup{New} (\phi)$ of the $F$-isocrystal $H^i_{\textup{logcrys}}(X^{\log}_s/W(\ul k))[\frac{1}{p}]$ always lies on or above  $\textup{Hdg}(M)$. Thus it suffices to prove that $\textup{Hdg}(M)$ lies on or above the Hodge polygon $\textup{Hdg}(Y^{\log}_\eta) = \textup{Hdg} (X^{\log}_\eta)$ (in fact, the method proves that they coincide, see \cite[Remark 4.4]{Yao_generic_fiber}). To this end, the same proof of \cite[Theorem 4.2]{Yao_generic_fiber} carries over, using 
\be
\item The Nygaard filtration on $\phi^* \Prism_{(Y, M_Y)/(\Ainf, N_\infty)}$ satisfies 
\[
\grade_N^r (\phi^* \Prism_{(Y,M_Y)/(\Ainf, N_\infty)}) \cong \tau_{\le r} \cl \Prism_{(Y,M_Y)/(\Ainf, N_\infty)} \{r\}
\]
\item There exists a canonical isomorphism 
\[
\Omega^{r}_{(Y, M_Y)/(\mO_C, N_\infty)} \cong H^r (\cl \Prism_{(Y, M_Y)/(\Ainf, N_\infty)}) \{r\}.  
\]
\ee 
in place of Theorem 2.10 and Theorem 2.3 in \textit{loc. cit}. 
\eproof

\subsection{Torsion discrepancy} 
\noindent 

\noindent 
One of the original motivations of \cite{BMS1} to introduce prismatic cohomology over $\Ainf$ is to study relations between torsions in integral $p$-adic cohomologies. In this direction, let us record the following result on torsion relations, generalizing \cite{BMS1, CK_semistable}. 

Keep notations from the previous section. Let $(X, M_X)$ be a  proper fs log $p$-adic formal scheme which is smooth over $(\mO_K, N)$, such that its mod $p$ fiber is of Cartier type.  
Let $X_{\cl \eta}^{\log, \diamondsuit}$ denote the log diamond generic fiber of $(X, M_X)_{(\mO_C, N_\infty)}$ over $(\spd(C, \mO_C), N_\infty)^a$ and let $X^{\log}_s$ be the special fiber of $(X, M_X)$ over $\ul k = (k, N)$.
 
\bt \label{theorem:torsion}
 Then we have  
\be
\item For each $i$ and $n \ge 0$, we have 
\begin{align*} 
\textup{length}_{\Z_p} (H^i_{\qpket}(X_{\cl \eta}^{\log,\diamondsuit}, \Z_p)_{\textup{tor}}/p^n) & \le \textup{length}_{W(k)} (H^i_{\textup{logcris}}(X^{\log}_s/W(\ul{k}))_{\textup{tor}}/p^n) \\
\textup{length}_{\Z_p} (H^i_{\qpket}(X_{\cl \eta}^{\log, \diamondsuit}, \Z/p^n)) & \le \textup{length}_{W(k)} (H^i_{\textup{logcris}}(X^{\log}_s/W_n(\ul{k}))) 
\end{align*}
\item  For each $i$ and $n \ge 0$, we have 
\begin{align*} 
\textup{length}_{\Z_p} (H^i_{\qpket}(X_{\cl \eta}^{\log, \diamondsuit}, \Z_p)_{\textup{tor}}/p^n) & \le \textup{val}_{\mO_K} (H^i_{\textup{logdR}}((X, M_X) /(\mO_K, N))_{\textup{tor}}/p^n) \quad \\
\textup{length}_{\Z_p} (H^i_{\qpket}(X_{\cl \eta}^{\log, \diamondsuit}, \Z/p^n)) & \le  \textup{val}_{\mO_K} (H^i_{\textup{logdR}}((X, M_X)_{\mO_K/p^n}/(\mO_K/p^n, N)))
\end{align*}
where $\textup{val}_{\mO_K}$ denotes the normalized length defined in \cite[Subsection 7.10]{CK_semistable}. 
\item Moreover, for each $i$, the cohomology group $H^i_{\textup{logcris}}(X^{\log}_s/W(\ul{k}))$ is $p$-torsion free if and only if $H^i_{\textup{logdR}}((X, M_X) /(\mO_K, N))$ is $p$-torsion free, and (either) implies that $H^i_{\qpket}(X_{\cl \eta}^{\log, \diamondsuit}, \Z_p)$ is $p$-torsion free and that $H^i_{\Prism} ((X, M_X)_{(\mO_C, N_\infty)}/(\Ainf, N_\infty))$ is a free $\Ainf$-module. 
\ee
\et  

\bproof 
Without loss of generality we may assume that $K = C$ is algebraically closed. The proof of all parts are exactly the same as the proof \cite[Theorem 7.9, Theorem 7.12 and Theorem 7.7]{CK_semistable}, making use of Proposition \ref{inverting_mu}, together with the  de Rham and   crystalline comparison of log prismatic cohomology. 
\eproof 

\br 
Alternatively, for part (3) in Theorem \ref{theorem:torsion}, the if and only if statement follows from the claim that 
\[
\dim_{k} \big( H^i_{\textup{logdR}}((X, M_X) /\ul{\mO_C})_{\textup{tor}} \otimes_{\mO_C}k \big)  = \dim_{k} \big( H^i_{\textup{logcris}}(X^{\log}_s/W(\ul{k}))_{\textup{tor}} /p \big), 
\]
which in turn follows from the proof in \cite[Remark 7.8]{CK_semistable}. 
\er

\br A special case ($n = 1$) of part (1) in Theorem \ref{theorem:torsion} implies that 
\[
\dim_{\F_p} H^i_{\qpket}(X_{\cl \eta}^{\log, \diamondsuit}, \F_p) \le \dim_k H^i_{\textup{logdR}} (X^{\log}_s/\ul k).
\]
This inequality can be proven without using Proposition \ref{inverting_mu}, by the following sequence of inequalities 
\begin{align*}
    \dim_{\F_p} H^i_{\qpket}(X_{\cl \eta}^{\log, \diamondsuit}, \F_p) 
    & \le \dim_{C^\flat} H^i (R\Gamma_{\Prism} ((X,M_X)_{(\mO_C, N_\infty)}/(\Ainf, N_\infty))\otimes^\L_{\Ainf} C^\flat) \\
    & \le \dim_{  k} H^i (R\Gamma_{\Prism} ((X,M_X)_{(\mO_C, N_\infty)}/(\Ainf, N_\infty))\otimes^\L_{\Ainf}   k) \\
    & =  \dim_k H^i_{\textup{logdR}} (X^{\log}_s/\ul k).
\end{align*}
Here the last equality follows from the de Rham (or crystalline) comparison while the second inequality follows from the structure of finitely presented $\mO_C^\flat$-modules. For the first inequality, note that by the \'etale comparison (Theorem \ref{etale_comp_diaomnd_global}), we have 
\[
H^i_{\qpket}(X_{\cl \eta}^{\log, \diamondsuit}, \F_p) \cong H^i ((R\Gamma_{\Prism} ((X, M_X)_{(\mO_C, N_\infty)}/(\Ainf, N_\infty))\otimes^\L_{\Ainf} C^\flat)^{\phi = 1}),
\]
and for a perfect $C^\flat$-complex $L$ equipped with a Frobenius map, we always have \[\dim_{\F_p} H^i (L^{\phi = 1}) \le \dim_{C^\flat} H^i (L).
\]
\er

\subsection{$u^\infty$-torsion in log prismatic cohomology} \label{ss:boundary_u_torsion}  
\noindent

\noindent 
Finally, we discuss some results on $u^{\infty}$-torsion over the Breuil--Kisin prism in the logarithmic setup, and prove Theorem \ref{mainthm:BK_module_structure} from the introduction.   Recall the setup: 
\bi 
\item $K$ is a discretely valued extension of $\Q_p$ with perfect residue field $k$ with a  fixed uniformizer $\pi \in \mO_K$. Let $E$ be the Eisenstein polynomial of $\pi$. 
\item $\mathfrak S = W(k)[\![u]\!]$, equipped with a pre-log structure $\N \ra S$ sending $1 \mapsto u$.  $(\mathfrak S, \N)$ has the structure of a $\delta_{\log}$-ring of rank $1$ with the Frobenius structure sending $u \mapsto u^p$. This makes $(\mathfrak S, I, \N)$ a pre-log prism, with $I = (E)$. 
\item Let $(X, M_X)$ be a proper log $p$-adic formal scheme which is smooth over $(\mO_K, \pi^\N)$, such that its mod $p$ fiber is of Cartier type.
\item Let $C = C_K = \widehat{\cl K}$ be a complete algebraic closure of $K$. 
Let $X_{\cl \eta}^{\log, \diamondsuit}$ denote the log diamond generic fiber of $(X, M_X)$ over $\spd(C, \mO_{C})$ (equipped with the trivial log structure). 
\ei 

\br 
Note that the Frobenius $\phi_{\mathfrak S}: \mathfrak S \ra \mathfrak S$ makes the $\mathfrak S$ a finite free module over itself. Therefore, in our setup, the \'etale sheaf $\Prism_{(X, M_X)/(\mathfrak S, \N)}^{(1)}$ is isomorphic to the base change $\Prism_{(X, M_X)/(\mathfrak S, \N)} \otimes^{\L}_{\mathfrak S, \phi_{\mathfrak S}} \mathfrak S$ from the description in (\ref{eq:presheaf_twisted_by_Frob}), and $\Prism^{(1), \tu{pre}}$ is already a sheaf (compare with Proposition \ref{prop:decalage_global}).  
\er

 For each integer $n \in \Z_{\ge 1}$, write 
\[
\Prism_{n} := \Prism_{(X, M_X)/(\mathfrak S, \N)} \otimes_{\Z}^\L \Z/p^n
\]
for the derived mod $p^n$ reduction of the \'etale sheaf $\Prism_{(X, M_X)/(\mathfrak S, \N)}$. 
The following result is a generalization of \cite[Theorem 3.3]{LiuLi2}.

\bt \label{theorem:boundary_u_torsion}
Let $e$ be the ramification index of $K$. Fix an integer $i$. For each $n \in \Z_{\ge 1}$, let 
\[T_{u, n} = T^i_{u, n}\coloneqq H^i  (X_{\ett}, \Prism_n) [u^\infty]
\] 
be the $u^\infty$-torsion submodule in the $i^{th}$  cohomology of $\Prism_{n}$.  Then the following holds.
\be
\item Suppose that $e\cdot (i -1) < p-1$, then $T_{u, n} = 0$.
\item  Suppose that $e\cdot (i -1) = p-1$, then the annihilator  $\text{Ann} (T_{u, n}) \supset (p, u)$. 
\item Suppose that $e\cdot (i -1) < 2(p-1)$, then  $\text{Ann} (T_{u, n}) + (u)\supset (p^{i-1}, u)$
\ee
The same assertions hold for 
\[T_{u} = T^i_u\coloneqq H^i_{\Prism} ((X, M_X)/(\mathfrak S, \N)) [u^\infty]
\] 
in place of $T_{u, n}$. 
\et

Theorem \ref{mainthm:BK_module_structure} from the introduction follows as an immediate corollary. 

\bc Retain notations from above and suppose that 
$e \cdot (i -1) < p -1$, then there is a (non-canonical) isomorphism
\[
H^i_{\Prism}((X, M_X)/(\mathfrak S, \N)) \cong H^i_{\qpket}(X_{\cl \eta}^{\log, \diamondsuit}, \Z_p) \otimes_{\Z_p} \mathfrak S.
\]
Likewise, we have mod $p^n$ variant. In other words, we have an isomorphism 
\[
H^i(X_{\ett}, \Prism_n) \cong  H^i_{\qpket}(X_{\cl \eta}^{\log, \diamondsuit}, \Z/p^n) \otimes_{\Z_p} \mathfrak S.
\]
\ec 

\bproof 
For the first claim, the proof of \cite[Corollary 3.5]{LiuLi2} carries over \textit{verbatim}, making use of Theorem \ref{theorem:boundary_u_torsion} and the \'etale comparison Theorem \ref{etale_comp_diaomnd_global} and Corollary \ref{cor:etale_comparison_over_W(C_flat)} (see also Proposition \ref{inverting_mu}) in the logarithmic setting. The proof of the mod $p^n$ variant is similar. By Theorem \ref{theorem:boundary_u_torsion} and \cite[Proposition 2.6]{LiuLi2}, we know that 
\[
H^i(X_{\ett}, \Prism_n) \cong M^i_n \otimes_{\Z_p} \mathfrak S
\]
for some finitely generated $\Z_p$-module $M^i_n$. To identify $M_n^i$ with $ H^i_{\qpket}(X_{\cl \eta}^{\log, \diamondsuit}, \Z/p^n)$, we base change to $W(C^\flat)$ and apply the mod $p^n$ part of Corollary \ref{cor:etale_comparison_over_W(C_flat)}. 
\eproof

Now we turn to Theorem \ref{theorem:boundary_u_torsion}.  The proof of Theorem \ref{theorem:boundary_u_torsion} entirely parallels that of \cite{LiuLi2}, except we regard $\Prism_{(X, M_X)/(\mathfrak S, \N)}$ and the Nygaard filtration (on its Frobenius twist) as \'etale sheaves on $X$, while \cite{LiuLi2} makes use of the quasisyntomic site of $X$. We briefly explain how the proof works below. 

Let us start with a slightly more general setup. Let $(A, I, M_A)$ be a pre-log prism and $(X, M_X)$ be a log $p$-adic formal scheme that is smooth over $(A/I, M_A)$ with the mod $p$ fiber being of Cartier type. Recall the natural map 
\[
\phi\colon \textup{Fil}^i_{N} \Prism_{(X, M_X)/(A, M_A)}^{(1)}  \lra I^i \otimes^\L_{A} \Prism_{(X, M_X)/(A, M_A)}
\]
from (\ref{eq:Frob_on_global_Nygaard}). Composing with the map
\[
I^i \otimes^\L_{A} \Prism_{(X, M_X)/(A, M_A)} 
\xrightarrow{I^{-i} \otimes^\L } \Prism_{(X, M_X)/(A, M_A)}
\]
induced by tensoring with $I^{-i}$, we obtain the divided Frobenius 
\begin{equation} \label{eq:divided_Frob}
    \frac{\phi}{I^i}\colon \textup{Fil}^i_{N} \Prism_{(X, M_X)/(A, M_A)}^{(1)}  \lra  \Prism_{(X, M_X)/(A, M_A)}. 
\end{equation}

The following assertion is an analogue of \cite[Lemma 7.8]{LiuLi1}. 

\bl \label{lemma:isom_inj_for_nygaard_filtration}
The induced map on cohomology by the Frobenius $\phi$
\[
\psi(r, i)\colon  H^r (X_{\ett}, \textup{Fil}^i_{N} \Prism_{(X, M_X)/(A, M_A)}^{(1)}) \lra H^r(X_{\ett}, \Prism_{(X, M_X)/(A, M_A)}) \otimes_{A} I^i
\]
is injective when $r = i + 1$ and is an isomorphism when $r \le i$. The same assertions hold for their derived mod $p^n$ counterparts. More precisely, the induced map  
\[
H^r (X_{\ett}, \textup{Fil}^i_{N} \Prism_{(X, M_X)/(A, M_A)}^{(1)}/p^n) \lra H^r(X_{\ett}, \Prism_{(X, M_X)/(A, M_A)}/p^n) \otimes_{A} I^i
\]
is injective when $r = i + 1$ and is an isomorphism when $r \le i$. Here $-/p^n$ denotes the derived mod $p^n$ reduction. 
\el 

The proof is the same as in \cite[Lemma 7.8]{LiuLi1}. For completeness let us recall the argument. 

\bproof 
By Corollary \ref{cor:Nygaard_complete_for_smooth},  
$R\Gamma (X_{\ett}, \textup{Fil}^i_{N} \Prism_{(X, M_X)/(A, M_A)}^{(1)})$ is complete with respect to the Nygaard filtration $R\Gamma (X_{\ett}, \textup{Fil}^{i+\bullet}_{N} \Prism_{(X, M_X)/(A, M_A)}^{(1)})$. Consider the map 
\begin{equation} \label{eq:map_of_filtered_complex_prismatic}
R\Gamma (X_{\ett}, \textup{Fil}^i_{N} \Prism_{(X, M_X)/(A, M_A)}^{(1)}) \ra R \Gamma (X_{\ett}, \Prism_{(X, M_X)/(A, M_A)} \otimes^\L_{A} I^i)
\end{equation} 
as a map of filtered complexes (where the target is equipped with the $I$-adic filtration). It suffices to show that the cone of the $(i+k)^{\tu{th}}$-graded piece of the map (\ref{eq:map_of_filtered_complex_prismatic}) lives in $\mD^{> i+k} (A/I)$ for all $k$, but this follows from Corollry \ref{cor:global_Nygaard}. 
The mod $p^n$ version can be proved by a similar argument. 
\eproof

\bc 
The divided Frobenius (\ref{eq:divided_Frob}) induces an injection 
\[
H^{i} (X_{\ett}, \textup{Fil}^{i-1}_{N} \Prism_{(X, M_X)/(A, M_A)}^{(1)} / p^n) \hookrightarrow H^{i} (X_{\ett},   \Prism_{(X, M_X)/(A, M_A)}/p^n) 
\]
\ec 

\bproof 
This follows from Lemma \ref{lemma:isom_inj_for_nygaard_filtration}. 
\eproof 

Now let us return to the setup in the beginning of this subsection. In particular, we let $(A, I, M_A) = (\mathfrak S, (E), \N)$ for the rest of this subsection. Also retain notations from Theorem \ref{theorem:boundary_u_torsion}. Recall that we have fixed an integer $i$, and for each $n \ge 1$, we have the $u^{\infty}$-torsion submodule 
\[
T_{n, u}= H^i (X_{\ett},  \Prism_n) [u^\infty].
\]
The following is the analogue of  \cite[Proposition 3.1]{LiuLi2}. 
\bc \label{cor:image_of_I^(i+1)_on_Frob_twist}  Given the setup above,  we have  
\[
E^{i-1} \cdot \textup{Ann}(T_{n, u}) \subset \textup{Ann}(T_{n, u}) \otimes_{\mathfrak S, \phi_{\mathfrak S}} \mathfrak S. 
\]
In particular, the following assertion of \cite[Corollary 3.2]{LiuLi2} holds: suppose that \[\textup{Ann}(T_{n, u}) + (p) = (u^\alpha, p)\] for some $\alpha \in \N$, then 
\[\alpha(p-1) \le e (i-1).\]
\ec 

\bproof 
The second claim follows from the first by the exact same (formal algebra) argument of \cite[Corollary 3.2]{LiuLi2}. The first claim follows from the proof of \cite[Proposition 3.1]{LiuLi2}. More precisely, multiplication by $E^{i-1}$ on 
\[
H^i (X_{\ett},  \Prism_{(X, M_X)/(\mathfrak S, \N)}/p^n) \otimes_{\mathfrak S, \phi_{\mathfrak S}} \mathfrak S \cong   H^{i} (X_{\ett},  \Prism_{(X, M_X)/(\mathfrak S, \N)}^{(1)} / p^n) 
\]
factors as
\begin{multline} \label{eq:E^i-1_factors}
\quad     H^{i} (X_{\ett},  \Prism_{(X, M_X)/(\mathfrak S, \N)}^{(1)} / p^n) \otimes_{\mathfrak S} (E)^{i-1} \lra  \\ 
H^{i} (X_{\ett}, \textup{Fil}^{i-1}_{N} \Prism_{(X, M_X)/(\mathfrak S, \N)}^{(1)} / p^n) 
     \lra   H^{i} (X_{\ett},  \Prism_{(X, M_X)/(\mathfrak S, \N)}^{(1)} / p^n),
\end{multline}
induced from the maps of \'etale sheaves 
\[
 \Prism_{(X, M_X)/(\mathfrak S, \N)}^{(1)} \otimes_{\mathfrak S}^\L (E)^{i-1} \lra \textup{Fil}^{i-1}_{N} \Prism_{(X, M_X)/(\mathfrak S, \N)}^{(1)} \lra   \Prism_{(X, M_X)/(\mathfrak S, \N)}^{(1)}  
\]
(see Construction \ref{construction:global_Nygaard} for the first arrow).  
Note that $
H^{i} (X_{\ett}, \textup{Fil}^{i-1}_{N} \Prism_{(X, M_X)/(\mathfrak S, \N)}^{(1)} / p^n) $ is a submodule of $H^{i} (X_{\ett},   \Prism_{(X, M_X)/(\mathfrak S, \N)}/p^n)$ by 
Corollary \ref{cor:image_of_I^(i+1)_on_Frob_twist}. The assertion on annihilators then follows immediately.   
\eproof 

\bproof[Proof of Theorem \ref{theorem:boundary_u_torsion}] 
With the preparations above, the proof of \cite[Theorem 3.3]{LiuLi2} carries over \textit{verbatim}. 
\eproof



 \newpage

\bibliographystyle{plain}
\bibliography{ref}

\end{document}